\documentclass[12pt]{article}

\pdftrailerid{}
\usepackage{amsmath,amssymb,amsthm,eucal,tikz,tikz-cd}
\usepackage[pagebackref,colorlinks=true,allcolors=black,bookmarksopen,bookmarksdepth=3]{hyperref}
\usepackage[margin=1in]{geometry}

\usepackage{enumitem}
\setenumerate{label=(\roman*),topsep=1pt,itemsep=1pt,partopsep=0pt,parsep=0pt}

\newtheorem*{theorem*}{Theorem}
\newtheorem{theorem}{Theorem}[section]
\newtheorem{lemma}[theorem]{Lemma}
\newtheorem{corollary}[theorem]{Corollary}
\newtheorem{proposition}[theorem]{Proposition}
\newtheorem{conjecture}[theorem]{Conjecture}

\theoremstyle{definition}
\newtheorem{definition}[theorem]{Definition}
\theoremstyle{remark}
\newtheorem{remark}[theorem]{Remark}
\newtheorem*{remark*}{Remark}
\newtheorem{question}[theorem]{Question}
\newtheorem*{example*}{Example}
\newtheorem{example}[theorem]{Example}
\newtheorem{convention}[theorem]{Convention}

\numberwithin{equation}{section}

\renewcommand{\top}{\mathrm{top}}
\renewcommand{\H}{\mathcal H}
\renewcommand{\Re}{\operatorname{Re}}
\renewcommand{\Im}{\operatorname{Im}}

\newcommand{\geo}{\mathrm{geo}}
\newcommand{\HJ}{\mathcal{HJ}}
\newcommand{\1}{\mathbf 1}
\newcommand{\J}{\mathcal J}
\newcommand{\A}{\mathcal A}
\newcommand{\B}{\mathcal B}
\newcommand{\C}{\mathcal C}
\newcommand{\E}{\mathcal E}
\newcommand{\ainf}{A_\infty}
\newcommand{\SSS}{\mathcal S}
\newcommand{\LL}{\mathbb L}
\newcommand{\sL}{\mathcal L}

\newcommand{\ZZ}{\mathbb Z}
\newcommand{\CC}{\mathbb C}
\newcommand{\RRR}{\mathcal R}
\newcommand{\RR}{\mathbb R}
\newcommand{\R}{\mathbf R}
\newcommand{\PP}{\mathcal P}
\newcommand{\Q}{\mathcal Q}
\newcommand{\M}{\mathcal M}
\newcommand{\N}{\mathcal N}
\newcommand{\Mbar}{\overline{\mathcal M}}
\newcommand{\Cbar}{\overline{\mathcal C}}
\newcommand{\Rbar}{\overline{\mathcal R}}
\newcommand{\Sbar}{\overline{\mathcal S}}
\newcommand{\oo}{\mathfrak o}
\newcommand{\BM}{\mathrm{BM}}
\newcommand{\OO}{\mathcal O}
\newcommand{\OC}{{\mathcal O\mathcal C}}
\newcommand{\D}{\mathcal D}
\newcommand{\W}{\mathcal W}

\newcommand{\Ham}{\operatorname{Ham}}

\newcommand{\im}{\operatorname{im}}
\newcommand{\Nbd}{\operatorname{\mathcal{N}bd}}
\newcommand{\Ndg}{\operatorname{\mathsf N_{\mathsf{dg}}}}
\newcommand{\Ch}{\operatorname{\mathsf{Ch}}}
\newcommand{\id}{\operatorname{id}}

\newcommand{\Tw}{\operatorname{\mathsf{Tw}}}
\newcommand{\F}{\mathcal F}
\newcommand{\Ainftycat}{\mathsf{A_\infty\text{-}cat}}

\newcommand{\Pro}{\operatorname{\mathsf{Pro}}}
\newcommand{\cones}{\operatorname{\mathsf{cones}}}

\newcommand{\coker}{\operatorname{coker}}

\newcommand{\const}{\mathrm{const}}
\newcommand{\std}{\mathrm{std}}
\newcommand{\dR}{\mathrm{dR}}
\newcommand{\op}{\mathrm{op}}
\newcommand{\pt}{\mathrm{pt}}
\newcommand{\lf}{\mathrm{lf}}
\newcommand{\reg}{\mathrm{reg}}
\newcommand{\Hom}{\operatorname{Hom}}

\newcommand{\Spin}{\operatorname{Spin}}
\newcommand{\colim}{\operatornamewithlimits{colim}}
\newcommand{\hocolim}{\operatornamewithlimits{hocolim}}

\newcommand{\charfol}{\mathrm{char.fol.}}
\newcommand{\morse}{\mathrm{morse}}
\newcommand{\Pos}{\mathsf{Pos}}

\allowdisplaybreaks

\begin{document}

\title{Covariantly functorial wrapped Floer theory on Liouville sectors}

\author{
Sheel Ganatra\thanks{S.G.\ was partially supported by the National Science Foundation through a postdoctoral fellowship with grant number DMS--1204393 and under agreement number DMS--1128155.  Any opinions, findings and conclusions or recommendations expressed in this material are those of the authors and do not necessarily reflect the views of the National Science Foundation.},
John Pardon\thanks{This research was conducted during the period J.P.\ served as a Clay Research Fellow and was partially supported by a Packard Fellowship and by the National Science Foundation under the Alan T.\ Waterman Award, Grant No.\ 1747553.},
and
Vivek Shende\thanks{V.S.\ is partially supported by the NSF grant DMS--1406871 and a Sloan fellowship.}
}

\date{July 31, 2019}

\maketitle

\begin{abstract}
We introduce a class of Liouville manifolds with boundary which we call Liouville sectors.
We define the wrapped Fukaya category, symplectic cohomology, and the open-closed map for Liouville sectors, and we show that these invariants are covariantly functorial with respect to inclusions of Liouville sectors.
From this foundational setup, a local-to-global principle for Abouzaid's generation criterion follows.
\end{abstract}

\section{Introduction}

The goal of this paper is to introduce local methods in the study of Floer theory on Liouville manifolds.

We introduce a class of Liouville manifolds with boundary which we call \emph{Liouville sectors} (Definition \ref{sectordefintro}).
We define the wrapped Fukaya category, symplectic cohomology, and the open-closed map for Liouville sectors.
Moreover, we show that these invariants are \emph{covariantly functorial} with respect to (proper, cylindrical at infinity) inclusions of Liouville sectors.
From this foundational setup, a local-to-global principle for Abouzaid's generation criterion \cite{abouzaidcriterion} follows more or less immediately (Theorem \ref{localtoglobalnondegenerate}).

We now introduce the main results of the paper in more detail.

\subsection{Liouville sectors}

To do Floer theory on symplectic manifolds with boundary, one must establish sufficient control on when holomorphic curves may touch the boundary.
One particularly nice setting in which this is possible is given by the following definition, studied in \S\ref{secliouvillesector}.

\begin{definition}\label{sectordefintro}
A \emph{Liouville sector} is a Liouville manifold-with-boundary $X$ for which there exists a function $I:\partial X\to\RR$ such that:
\begin{itemize}
\item$I$ is \emph{linear at infinity}, meaning $ZI=I$ outside a compact set, where $Z$ denotes the Liouville vector field.
\item$dI|_\charfol>0$, where the characteristic foliation $C$ of $\partial X$ is oriented so that $\omega(N,C)>0$ for any inward pointing vector $N$.
\end{itemize}
On any Liouville sector, there is a canonical (up to contractible choice) symplectic fibration $\pi:X\to\CC_{\Re\geq 0}$ defined near $\partial X$.
For almost complex structures on $X$ making $\pi$ holomorphic (of which there is a plentiful supply), the projection $\pi$ imposes strong control on holomorphic curves near $\partial X$.
\end{definition}

\begin{example*}
For any compact manifold-with-boundary $Q$ (for instance a ball), its cotangent bundle $T^\ast Q$ is a Liouville sector.
\end{example*}

\begin{example*}
A punctured bordered Riemann surface $S$ is a Liouville sector iff every component of $\partial S$ is homeomorphic to $\RR$ (i.e.\ none is homeomorphic to $S^1$).
\end{example*}

\begin{example*}
Given a Liouville domain $\bar X_0$ and
a closed Legendrian $\Lambda\subseteq\partial\bar X_0$,
one may define a Liouville sector $X$ 
by taking the Liouville completion of $\bar X_0$ away from a standard neighborhood of $\Lambda$.
\end{example*}

\begin{example*}
More generally, given a Liouville domain $(\bar X_0,\lambda)$ and a hypersurface-with-boundary $F_0\subseteq\partial\bar X_0$ such that $(F_0, \lambda)$ is again a Liouville domain,
one may define a Liouville sector $X$ by
completing $\bar X_0$ away from a standard neighborhood of $F_0$.
In fact, every Liouville sector is of this form, uniquely so up to a contractible choice.
\end{example*}

\begin{example*}
To every Liouville Landau--Ginzburg model $\pi:E\to\CC$, one can associate a Liouville sector which, morally speaking, is defined by removing from $E$ the inverse image of a neighborhood of a ray (or half-plane) disjoint from the critical locus of $\pi$.
There are various ways of formalizing the notion of a Liouville
Landau--Ginzburg model (see \cite[\S 2]{maydanskiyseidel} for Lefschetz
fibrations); for us $E$ should be a Liouville manifold and $\pi$
should induce an embedding into $\partial_\infty E$ of the contact mapping torus of the monodromy action on the fiber $F$ (i.e.\ $(S^1\times F,dt+\lambda_F)$ if the monodromy is trivial).
This embedding furthermore extends to an open book decomposition if $\pi$ has
compact critical locus.  The associated Liouville sector is defined by
applying the previous example to $E$ and a fiber $\{t\}\times F_0$ inside $\partial_\infty E$.
\end{example*}

\begin{remark*}
The notion of a Liouville sector is essentially equivalent to Sylvan's notion of a \emph{stop} on a Liouville manifold \cite{sylvanthesis} (for every Liouville manifold $\bar X$ with stop $\sigma$, there is a Liouville sector $X=\bar X\setminus\sigma$, and every Liouville sector is of this form, uniquely in a homotopical sense).
The language of Liouville sectors has two advantages relevant for our work in this paper: (1) inclusions of Liouville sectors $X\hookrightarrow X'$ (which play a central role in this paper) are easier to talk about, and (2) being a Liouville sector is a \emph{property} rather than extra data, which makes geometric operations simpler and more clearly canonical.
\end{remark*}

\subsection{Wrapped Floer theory on Liouville sectors}

We generalize many basic objects of wrapped Floer theory from Liouville manifolds to Liouville sectors.
Specifically, we define the wrapped Fukaya category, symplectic cohomology, and the open-closed map for Liouville sectors.
The ``wrapping'' in these definitions takes place on the boundary at infinity $\partial_\infty X$ and is ``stopped'' when it hits $\partial X$.
An important feature in this setting is that these invariants are all \emph{covariantly functorial} for (proper, cylindrical at infinity) inclusions of Liouville sectors.
The key ingredient underlying these Floer theoretic constructions is the projection $\pi:X\to\CC_{\Re\geq 0}$ defined near $\partial X$ and the resulting control on holomorphic curves near $\partial X$.

In \S\ref{secwrapped}, we define the wrapped Fukaya category $\W(X)$ of a Liouville sector $X$, and we show that an inclusion of Liouville sectors $X\hookrightarrow X'$ induces a functor $\W(X)\to\W(X')$.
The wrapped Fukaya category of Liouville sectors generalizes the wrapped Fukaya category of Liouville manifolds as introduced by Abouzaid--Seidel \cite{abouzaidseidel}.
Note that our pushforward maps $\W(X)\to\W(X')$ for inclusions of Liouville sectors $X\hookrightarrow X'$ are distinct from (though related to) the Viterbo restriction functors $\W(X')\to\W(X)$ induced by inclusions of Liouville manifolds $X\hookrightarrow X'$ defined by Abouzaid--Seidel \cite{abouzaidseidel}.

To define $\W(X)$, we adopt the later techniques of Abouzaid--Seidel \cite{abouzaidseidelunpublished} in which $\W(X)$ is defined as the localization of a 
corresponding directed category $\OO(X)$ at a collection of continuation morphisms.
The key new ingredient needed to define $\W(X)$ for Liouville sectors
is the fact that holomorphic disks with boundary on
Lagrangians inside $X$ remain disjoint from a
neighborhood of $\partial X$.  This can be seen
from the holomorphic projection $\pi$.

\begin{example*}
For the Liouville sector $X$ associated to an exact symplectic Landau--Ginzburg model $\pi:E\to\CC$, the wrapped Fukaya category $\W(X)$ should be regarded as a definition of the Fukaya--Seidel category of $(E,\pi)$.
\end{example*}

\begin{example*}
The infinitesimally wrapped Fukaya category of Lagrangians in a Liouville manifold $\bar X$ asymptotic to a fixed Legendrian $\Lambda\subseteq\partial_\infty\bar X$ is a full subcategory of the wrapped Fukaya category $\W(X)$ of the Liouville sector $X$ obtained from $\bar X$ by removing a standard neighborhood of $\Lambda$ near infinity.
Namely, Lagrangians in $\bar X$ asymptotic to $\Lambda$ can be perturbed by the negative Reeb flow to define objects of $\W(X)$.
The positive Reeb flow from such objects falls immediately into the deleted neighborhood of $\Lambda$, so wrapping inside $\partial_\infty X$ exactly realizes ``infinitesimal wrapping''.
\end{example*}

\begin{example*}
The Legendrian contact homology algebra of $\Lambda\subseteq\partial_\infty\bar X$ with respect to the filling $\bar X$ is also expected to admit a description in terms of $\W(X)$.
Namely, near any point of $\Lambda$, there is a small Legendrian sphere linking $\Lambda$ which further bounds a small exact Lagrangian disk, whose endomorphism algebra in $\W(X)$ should be (by the philosophy of Bourgeois--Ekholm--Eliashberg \cite{bourgeoisekholmeliashberg}) the Legendrian contact homology of $\Lambda$ with loop space coefficients (see the argument sketched in Ekholm--Ng--Shende \cite[\S 6]{ekholmngshende} and in Ekholm--Lekili \cite[\S B]{ekholm-lekili}).
\end{example*}

\begin{remark*}
The wrapped Fukaya category $\W(X)$ of the Liouville sector $X=\bar X\setminus\sigma$ associated to a Liouville manifold $\bar X$ with stop $\sigma$ should coincide with the partially wrapped Fukaya category $\W_\sigma(\bar X)$ defined by Sylvan \cite{sylvanthesis}.
\end{remark*}

In \S\ref{secsymplecticcohomology}, we define the symplectic cohomology of a Liouville sector as the direct limit
\begin{equation}\label{shdefeqintro}
SH^\bullet(X,\partial X):=\varinjlim_{\begin{smallmatrix}H:X\to\RR\\H|_{\Nbd\partial X}=\Re\pi\end{smallmatrix}}HF^\bullet(X;H)
\end{equation}
(generalizing symplectic cohomology $SH^\bullet(X)$ of Liouville manifolds as introduced in Floer--Hofer \cite{floerhofersh}, Cieliebak--Floer--Hofer \cite{cieliebakfloerhofersh}, and Viterbo \cite{viterbosh}; additional references include Seidel \cite{seidelbiased}, Abouzaid \cite{abouzaidviterbo}, and Oancea \cite{oanceasurvey}).
We also define a map $H^\bullet(X,\partial X)\to SH^\bullet(X,\partial X)$ and show that an inclusion of Liouville sectors $X\hookrightarrow X'$ induces a map $SH^\bullet(X,\partial X)\to SH^\bullet(X',\partial X')$ (this map, which is compatible with the map $H^\bullet(X, \partial X) \to H^\bullet(X', \partial X')$, is distinct from, though related to, the restriction map introduced by Viterbo \cite{viterbosh}).
The condition that $H|_{\Nbd\partial X}=\Re\pi$ prevents Floer trajectories for $HF^\bullet(X;H)$ from passing through a neighborhood of $\partial X$.
More generally, for an inclusion $X\hookrightarrow X'$, if $H|_{\Nbd\partial X}=\Re\pi$ and $H|_{\Nbd\partial X'}=\Re\pi'$, then Floer trajectories with positive end inside $X$ must stay entirely inside $X$ (on the other hand, Floer trajectories with positive end inside $X'\setminus X$ can pass through $X$), and this is the key to defining the map $SH^\bullet(X,\partial X)\to SH^\bullet(X',\partial X')$.

To make the definition \eqref{shdefeqintro} rigorous is delicate for two reasons.
First, the function $\Re\pi$ is not linear at infinity, and so we must splice it together with a linear Hamiltonian in such a way that there are no periodic orbits near infinity.
Second, to bound Floer trajectories away from infinity (to prove compactness), we adopt the techniques of Groman \cite{groman} based on monotonicity and bounded geometry, and to apply these methods to Floer equations for continuation maps, we must choose families of Hamiltonians which are \emph{dissipative} in the sense of Groman \cite{groman}.
We do not know how to use the maximum principle (as usually used to construct symplectic cohomology on Liouville manifolds) to prove the needed compactness results.
Finally, let us remark that there should be another version of symplectic cohomology for Liouville sectors, say denoted $SH^\bullet(X)$, where we instead require $H|_{\Nbd\partial X}=-\Re\pi$ (though it would be nontrivial to extend our methods to make this definition precise).

\begin{remark*}
Sylvan has defined the partially wrapped symplectic cohomology $SH^\bullet_\sigma(\bar X)$ \cite{sylvanthesis} of a Liouville manifold $\bar X$ equipped with a stop $\sigma$, and we expect that it is isomorphic to $SH^\bullet(X,\partial X)$ for the Liouville sector $X=\bar X\setminus\sigma$.
\end{remark*}

In \S\ref{secopenclosed}, we define an open-closed map
\begin{equation}\label{ocintro}
\OC:HH_\bullet(\W(X))\to SH^{\bullet+n}(X,\partial X)
\end{equation}
for Liouville sectors, where $n=\frac 12\dim X$ (generalizing definitions given by Fukaya--Oh--Ohta--Ono \cite{foooI}, Seidel \cite{seideldeformations}, and Abouzaid \cite{abouzaidcriterion}) and we show that $\OC$ is a natural transformation of functors, meaning it commutes with the pushforward maps induced by inclusions of Liouville sectors (we adopt the convention whereby the subscript on Hochschild homology $HH_\bullet$ is a \emph{cohomological} grading).

To define the open-closed map, we adopt the methods of Abouzaid--Ganatra \cite{abouzaidganatra} in which the domain of $\OC$ is taken to be $HH_\bullet(\OO(X),\B(X))$, where $\OO(X)$ is the directed category whose localization is $\W(X)$, and $\B(X)$ is a certain geometrically defined $\OO(X)$-bimodule quasi-isomorphic as $\OO(X)$-bimodules to $\W(X)$ (properties of localization give a canonical isomorphism $HH_\bullet(\OO(X),\W(X))=HH_\bullet(\W(X))$; see Lemma \ref{HHquotientisomorphism}).

Functoriality of the open-closed map has an immediate application towards the verification of Abouzaid's generation criterion, which we turn to next.
The idea of localizing open-closed maps has appeared earlier in unpublished work of Abouzaid \cite{abouzaidcosheaves}, and Abouzaid's proof \cite{abouzaidviterbo} of Viterbo's theorem also served as an inspiration for our work.

\subsection{Local-to-global principle for Abouzaid's criterion}

Recall that a Liouville manifold $X$ is called \emph{non-degenerate} iff the unit $1\in SH^\bullet(X)$ lies in the image of $\OC$.
Recall also that a collection of objects $F\subseteq\W(X)$ is said to satisfy \emph{Abouzaid's criterion} \cite{abouzaidcriterion} iff the unit lies in the image of the restriction of $\OC$ to $HH_\bullet(\F)$ (where $\F\subseteq\W(X)$ denotes the full subcategory with objects $F$).
Abouzaid's criterion and non-degeneracy have many important consequences.
The main result of \cite{abouzaidcriterion} is that if $F$ satisfies Abouzaid's criterion, then $F$ split-generates $\W(X)$
(i.e.\ $\Pi\Tw\F\twoheadrightarrow\Pi\Tw\W(X)$ is essentially surjective; see Seidel \cite[(3j), (4c)]{seidelbook}).
Non-degeneracy of $X$ implies that the open-closed map and the closed-open map are both isomorphisms \cite[Theorem 1.1]{ganatrathesis} as conjectured by Kontsevich \cite{kontsevichhms}
(and the same for their $S^1$-equivariant versions \cite{ganatraS1action}).
Non-degeneracy of $X$ also implies that the category $\W(X)$ is homologically smooth \cite[Theorem 1.2]{ganatrathesis} and Calabi--Yau \cite{ganatraS1action}.

To state our local-to-global argument for verifying Abouzaid's criterion, we need the following definition.
Let $X$ be a manifold.
A family of codimension zero submanifolds-with-boundary $X_\sigma\subseteq X$ indexed by a poset $\Sigma$ (so $X_\sigma\subseteq X_{\sigma'}$ for $\sigma\leq\sigma'$) is called a \emph{homology hypercover} iff the map
\begin{equation}
\hocolim_{\sigma\in\Sigma}C_\bullet^\BM(X_\sigma)\to C_\bullet^\BM(X)
\end{equation}
hits the fundamental class $[X]\in H_\bullet^\BM(X)$.
Here $H_\bullet^\BM$ (Borel--Moore homology, also written $H_\bullet^\lf$) denotes the homology of locally finite singular chains.
Concretely, the homotopy colimit means
\begin{equation}
\hocolim_{\sigma\in\Sigma}C_\bullet^\BM(X_\sigma):=\bigoplus_{p\geq 0}\bigoplus_{\sigma_0\leq\cdots\leq\sigma_p\in\Sigma}C_{\bullet-p}^\BM(X_{\sigma_0})
\end{equation}
namely simplicial chains on the nerve of $\Sigma$ with coefficients given by $\sigma\mapsto C_\bullet^\BM(X_\sigma)$ (the differential on the right is the internal differential plus the sum over all ways of forgetting some $\sigma_i$).
By Poincar\'e duality, it is equivalent to ask that the map
\begin{equation}
\hocolim_{\sigma\in\Sigma}C^\bullet(X_\sigma,\partial X_\sigma)\to C^\bullet(X)
\end{equation}
hit the unit $1\in H^\bullet(X)$.

\begin{example*}
For every \emph{finite} cover of $X$ by $\{X_i\}_{i\in I}$, the family of all finite intersections $X_{i_0}\cap\cdots\cap X_{i_k}$ indexed by $\Sigma:=2^I\setminus\{\varnothing\}$ is a homology hypercover of $X$.
\end{example*}

\begin{remark*}
Instead of considering a family of submanifolds-with-boundary $X_\sigma\subseteq X$ indexed by a poset $\Sigma$, one could also consider a simplicial submanifold-with-boundary $X_\bullet\to X$.
The latter perspective is somewhat more standard (and essentially equivalent), though we have avoided it for reasons of exposition.
\end{remark*}

\begin{theorem}\label{localtoglobalnondegenerate}
Let $X$ be a Liouville manifold with a homology hypercover by Liouville sectors $\{X_\sigma\}_{\sigma\in\Sigma}$.
Let $F_\sigma\subseteq\W(X_\sigma)$ be collections of objects with $F_\sigma\subseteq F_{\sigma'}$ for $\sigma\leq\sigma'$, and
$ \F_\sigma $ the corresponding full subcategories.  Assume
\begin{equation}\label{localoc}
\OC_\sigma:HH_\bullet( \F_\sigma )\to SH^{\bullet+n}(X_\sigma,\partial X_\sigma)
\end{equation}
is an isomorphism for all $\sigma\in\Sigma$.
Then $F:=\bigcup_{\sigma\in\Sigma}F_\sigma\subseteq\W(X)$ satisfies Abouzaid's criterion.
\end{theorem}

Theorem \ref{localtoglobalnondegenerate} follows immediately from the functoriality of $\OC$; to be precise, it follows from the following commutative diagram:
\begin{equation}\label{keydiagram}
\begin{tikzcd}[column sep = small]
\vphantom{C}\smash{\hocolim\limits_{\sigma\in\Sigma}CC_{\bullet-n}(\F_\sigma)}\ar{r}\ar{d}&\smash{\hocolim\limits_{\sigma\in\Sigma}SC^\bullet(X_\sigma,\partial X_\sigma)}\ar{d}&\ar{l}\smash{\hocolim\limits_{\sigma\in\Sigma}C^\bullet(X_\sigma,\partial X_\sigma)}\ar{d}\\
\vphantom{C}\smash{CC_{\bullet-n}(\F)}\ar{r}&SC^\bullet(X)&\ar{l}C^\bullet(X).
\end{tikzcd}
\end{equation}
Indeed, if each local open-closed map \eqref{localoc} is an isomorphism, then the top left horizontal arrow in \eqref{keydiagram} is a quasi-isomorphism.
This implies that the image of the map $HH_{\bullet-n}(\F)\to SH^\bullet(X)$ contains the image of the composition $\hocolim_{\sigma\in\Sigma}C^\bullet(X_\sigma,\partial X_\sigma)\to C^\bullet(X)\to SC^\bullet(X)$, which in turn hits the unit since $\{X_\sigma\}_{\sigma\in\Sigma}$ is a homology hypercover of $X$.
Note that for this proof to make sense, we must provide a definitions of $\W$, $SC^\bullet$, and $\OC$ which are functorial on the chain level, up to coherent higher homotopy.

Note that to apply Theorem \ref{localtoglobalnondegenerate} (which is valid over any commutative coefficient ring), we do not need to know anything about $SH^\bullet(X)$ or the morphism spaces in $\W(X)$, both of which can be difficult to compute for general $X$.
In contrast, if $\partial_\infty X_\sigma$ is the contactization of a Liouville domain (which occurs often in practice for ``small'' sectors $X_\sigma$), then the map $H^\bullet(X_\sigma,\partial X_\sigma)\to SH^\bullet(X_\sigma,\partial X_\sigma)$ is an isomorphism (Lemma \ref{cylinderstopped} and Proposition \ref{hhshiso}) and it is often easy to compute the left hand side of \eqref{localoc} as well and see that this map is an isomorphism.
We now give some examples of interesting $X$ covered by such $X_\sigma$ (conjecturally any Weinstein manifold $X$ admits such a cover).

\begin{example}\label{cotangent}
Abouzaid \cite{abouzaidcotangent}  showed that the collection of cotangent fibers $T^\ast_qQ\subseteq T^\ast Q$ satisfies Abouzaid's criterion for any closed manifold $Q$.
This result can be deduced from Theorem \ref{localtoglobalnondegenerate} as follows.
Since $T^\ast Q$ admits a homology hypercover by copies of $T^\ast\!B$ ($B$ is the ball), it is enough to show that
\begin{equation}
\OC:HH_\bullet(CW^\bullet(T^\ast_0\!B))\to SH^{\bullet+n}(T^\ast\!B,\partial T^\ast\!B)
\end{equation}
is an isomorphism, where $T^\ast_0\!B\subseteq T^\ast\!B$ denotes the fiber.
The maps
\begin{align}
H^\bullet(T^\ast_0\!B)&\to HW^\bullet(T^\ast_0\!B)\\
H^\bullet(T^\ast\!B,\partial T^\ast\!B)&\to SH^\bullet(T^\ast\!B,\partial T^\ast\!B)
\end{align}
are both isomorphisms, since for certain nice choices of contact form, there are no Reeb orbits/chords (for the second map, combine Example \ref{ballfullystopped} and Proposition \ref{hhshiso}).
We conclude that the open-closed map for $T^\ast\!B$ is an isomorphism, using the general property that $\OC(\1_L)=i_![L]$ (see Proposition \ref{ocmorse}).

Note that the above argument does not say anything about whether or not the fiber (split-)generates $\W(T^\ast\!B)$.
Presumably split-generation could be shown by following through on Remark \ref{sectorcriterion}, and generation could be shown by generalizing Abouzaid's work \cite{abouzaidtwisted}.
\end{example}

\begin{figure}[hbt]
\centering
\includegraphics{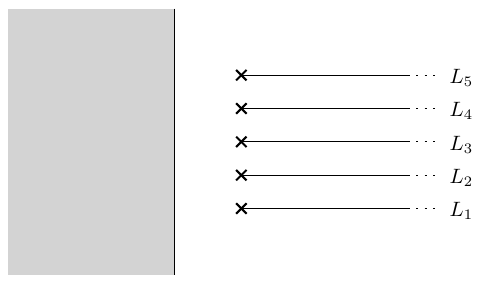}
\caption{A Lefschetz fibration.}\label{lefschetzfibration}
\end{figure}

\begin{example}\label{lefschetzexample}
Let $X$ be the Liouville sector associated to an exact symplectic (Liouville) Lefschetz fibration $\pi:E\to\CC$ (morally speaking, defined by removing $\pi^{-1}(\CC_{\Re<-N})$ from $E$).
The collection of Lefschetz thimbles $L_1,\ldots,L_m\subseteq X$ (as illustrated in Figure \ref{lefschetzfibration}) is an \emph{exceptional collection} inside $\W(X)$, meaning that
\begin{align}
HF^\bullet(L_i,L_i)&=\ZZ\\
HF^\bullet(L_i,L_j)&=0\quad\text{for }i>j.
\end{align}
Furthermore, $\partial_\infty X$ has no Reeb orbits, so the map $H^\bullet(X,\partial X)\xrightarrow\sim SH^\bullet(X,\partial X)$ is an isomorphism.
Using the identity $\OC(\1_L)=i_![L]$, we conclude that the open-closed map
\begin{equation}
HH_\bullet(\F^\to(\pi))\to SH^{\bullet+n}(X)
\end{equation}
is an isomorphism, where $\F^\to(\pi)\subseteq\W(X)$ denotes the full subcategory spanned by the Lefschetz thimbles $L_1,\ldots,L_m$ (originally introduced by Seidel \cite{seidelbook}).
(Justification for these assertions is provided in the body of the paper).
\end{example}

\begin{remark}\label{sectorstoo} \label{sectorcriterion}
The diagram \eqref{keydiagram} remains valid when $X$ is itself a Liouville sector.
However, to take advantage of it, one needs to first formulate the correct analogue of Abouzaid's criterion for Liouville sectors and their wrapped Fukaya categories.

The most naive generalization of Abouzaid's criterion to Liouville sectors, using the open-closed map \eqref{ocintro}, does not make sense since $SH^\bullet(X,\partial X)$ usually does not have a unit.
Rather, we suspect that the correct generalization of Abouzaid's criterion should involve the map
\begin{equation}
\OC:\Bigl[CC_{\bullet-n}(\W(X_0))\to CC_{\bullet-n}(\W(X))\Bigr]\to\Bigl[SC^\bullet(X_0,\partial X_0)\to SC^\bullet(X,\partial X)\Bigr]
\end{equation}
where the brackets indicate taking the cone of the map inside, and $X_0\subseteq X$ denotes a small closed regular neighborhood of $\partial X$ (cylindrical at infinity).
Note that $C^\bullet(X)$ is naturally quasi-isomorphic to $[C^\bullet(X_0,\partial X_0)\to C^\bullet(X,\partial X)]$, which naturally maps to the right side above, so there is a unit to speak of hitting.

A version of Abouzaid's criterion for the Fukaya--Seidel category \cite{seidelbook} of a symplectic Landau--Ginzburg model has been given earlier by Abouzaid--Ganatra \cite{abouzaidganatra}.
Their criterion is distinct from (though related to) the version proposed just above.
\end{remark}

\begin{remark}
We expect that Theorem \ref{localtoglobalnondegenerate} can be used to show that for any Weinstein manifold $X^{2n}$ admitting a singular Lagrangian spine $\LL^n\subseteq X^{2n}$ (see Remark \ref{spinecoreskeletonrmk}) with arboreal singularities in the sense of Nadler \cite{nadlerarboreal}, the ``fibers of the projection $X\to\LL$'' satisfy Abouzaid's criterion (generalizing Example \ref{cotangent}).
Such a Weinstein manifold should admit a homology hypercover by ``arboreal Liouville sectors'' $X_T^{2n}$ associated to rooted trees $T$ (defined in terms of the corresponding arboreal singularities).
We expect the arboreal sector $X_T$ to correspond to the Lefschetz fibration over $\CC$ whose fiber is a plumbing of copies of $T^\ast S^{n-1}$ according to $T$ and whose vanishing cycles are the zero sections, ordered according to the rooting of $T$ (this has now been proven by Shende \cite{shendearborlefschetz}).
Furthermore, the Lefschetz thimbles should correspond to the ``fibers of the projection $X\to\LL$'' (more precisely, they should generate the same full subcategory).
Example \ref{lefschetzexample} would then imply that the open-closed map for each $X_T$ is an isomorphism, so by Theorem \ref{localtoglobalnondegenerate} we would conclude that the ``fibers of the projection $X\to\LL$'' satisfy Abouzaid's criterion.
\end{remark}

\subsection{Acknowledgements}

This collaboration began while V.\,S.\ and J.\,P.\ visited the Institut Mittag--Leffler during the 2015 fall program on ``Symplectic geometry and topology'', and the authors thank the Institut for its hospitality.
We are grateful to Mohammed Abouzaid for useful discussions, for sharing
\cite{abouzaidcosheaves}, and for telling us why the Floer complexes in this
paper are cofibrant.  We also thank David Nadler and Zack Sylvan for useful
conversations and Thomas Massoni for comments on an earlier version.
Finally, we thank the referee for their careful reading of this long and technical paper.

\section{Liouville sectors}\label{secliouvillesector}

\subsection{Notation}

The notation $\Nbd K$ shall mean ``some neighborhood of $K$''.
``A neighborhood of infinity'' means ``the complement of a pre-compact set'' (i.e.\ $\Nbd{\{\infty\}}$ in the one-point compactification).
``At infinity'' shall mean ``over some neighborhood of infinity''.  We write $M^\circ$ for the interior of $M$.
The notation $s\gg 0$ shall mean ``$s$ sufficiently large'', and $s\ll 0$ means $-s\gg 0$.

We work in the smooth category unless otherwise stated.
A function on a closed subset of a smooth manifold is called smooth iff it can be extended to a smooth function defined in a neighborhood (however such an extension is not specified).

\subsection{Liouville manifolds}

References for Liouville manifolds include \cite{eliashberggromovconvexsymplectic,eliashbergplurisubharmonic,cieliebakeliashberg,seidelbiased}.

A \emph{Liouville vector field} $Z$ on a symplectic manifold $(X, \omega)$ is a vector field satisfying
$\sL_Z\omega=\omega$, or, equivalently, $\omega=d\lambda$ for $\lambda := \omega(Z, -)$.
Such $\lambda$ is called a \emph{Liouville form} and determines both $\omega$ and $Z$.
An \emph{exact symplectic manifold} is a manifold equipped with a Liouville form (equivalently, it is a symplectic manifold equipped with a Liouville vector field).

To any co-oriented contact manifold $(Y,\xi)$ one associates an exact symplectic manifold $(SY,\lambda)$ called the \emph{symplectization} of $Y$, defined as the total space of the bundle of positive contact forms, equipped with the restriction of the tautological Liouville $1$-form $\lambda$ on $T^\ast Y$.
Equipping $Y$ with a positive contact form $\alpha$ induces a trivialization $(SY,\lambda)=(\RR_s\times Y,e^s\alpha)$ in which $Z=\frac\partial{\partial s}$.
Henceforth, we omit the adjectives ``co-oriented'' and ``positive'' for contact manifolds/forms, though they should be understood as always present.
An exact symplectic manifold $X$ is the symplectization of a contact manifold iff there is a diffeomorphism $X=\RR_s\times Y$ identifying $Z$ with $\frac\partial{\partial s}$.

A \emph{Liouville domain} is a compact exact symplectic manifold-with-boundary whose Liouville vector field is outward pointing along the boundary; the restriction of $\lambda$ to the boundary of a Liouville domain is a contact form.
A \emph{Liouville manifold} is an exact symplectic manifold which is ``cylindrical and convex at infinity'', meaning that the following two equivalent conditions are satisfied:
\begin{itemize}
\item There is a Liouville domain $X_0\subseteq X$ such that the positive Liouville flow of $\partial X_0$ is defined for all time and the resulting map $X_0\cup_{\partial X_0}(\RR_{\geq 0}\times\partial X_0)\to X$ is a diffeomorphism (equivalently, is surjective).
\item There is a map from the ``positive half'' of a symplectization $(\RR_{s\geq 0}\times Y,e^s\alpha)\to(X,\lambda)$ respecting Liouville forms and which is a diffeomorphism onto its image, covering a neighborhood of infinity.
\end{itemize}
The Liouville flow defines contactomorphisms between different choices of $\partial X_0$ and/or $Y$, so there is a well-defined contact manifold $(\partial_\infty X,\xi)$ which we regard as the ``boundary at infinity of $X$'' (not to be confused with the actual boundary $\partial X$, which is not present now but will be later).
There is a canonical embedding of the (full) symplectization of $(\partial_\infty X,\xi)$ into $X$ as an open subset, and there is a canonical bijection between Liouville domains $X_0\subseteq X$ whose completion is $X$ and contact forms on $\partial_\infty X$.

An object living on a Liouville manifold is called \emph{cylindrical} iff it is invariant under the Liouville flow near infinity.

\begin{remark}\label{spinecoreskeletonrmk}
It is natural to view a Liouville domain/manifold $X$ as a ``thickening'' of the locus $\LL\subseteq X$ of points which do not escape to infinity under the Liouville flow (e.g.\ regarding a punctured Riemann surface as a thickening of a ribbon graph is a special case of this).
Under certain assumptions on the Liouville flow (e.g.\ if $X$ is Weinstein), this $\LL\subseteq X$ is a \emph{singular isotropic spine} for $X$ (``spine'' carries its usual meaning, e.g.\ as in Zeeman \cite{zeeman}, namely that $X$ deforms down to a small regular neighborhood of $\LL$).
It is also common to call $\LL$ the \emph{core} or \emph{skeleton} of $X$.
\end{remark}

\begin{example}
The manifold $X = \RR^{2n}$ equipped with the standard symplectic form $\omega:=\sum_{i=1}^ndx_i\wedge dy_i$ can
be given the structure of a Liouville manifold by taking the vector field $Z$ to be the 
generator of radial expansion $Z:=\frac 12\sum_{i=1}^n(x_i\frac\partial{\partial x_i}+y_i\frac\partial{\partial y_i})$.  In this case, $X_0$ can be chosen as the unit ball, and the core is the origin.
\end{example}

\begin{example}
The cotangent bundle of a compact manifold $X = T^\ast Q$, equipped
with the tautological Liouville $1$-form $\lambda$,
is a Liouville manifold in which $Z$ is the generator of fiberwise radial
dilation.  We can choose $X_0$ as the unit codisk bundle, and the core is the zero section.

If $Q$ is non-compact, then $T^\ast Q$ is not a Liouville manifold when equipped with the tautological Liouville form.
However, if $Q$ is the interior of a compact manifold with boundary $\overline Q$, then $T^\ast Q$ may be given the structure of a Liouville manifold by modifying the Liouville form appropriately near $\partial\overline Q$ to make it convex.
\end{example}

\subsection{Hamiltonian flows}\label{hamiltonians}

To a function $H:X\to\RR$ on a symplectic manifold $X$, one associates the Hamiltonian vector field $X_H$ defined by
\begin{equation}
\omega(X_H,\cdot)=-dH.
\end{equation}
When $X=SY$ is the symplectization of a contact manifold $Y$, we say $H$ is \emph{linear} iff $ZH=H$, in which case $X_H$ commutes with $Z$.
The following spaces are in canonical bijection:
\begin{itemize}
\item The space of functions $H$ on $X$ satisfying $ZH=H$.
\item The space of symplectic vector fields on $X$ commuting with $Z$.
\item The space of sections of $TY/\xi$.
\item The space of contact vector fields on $Y$.
\end{itemize}
(More generally, this holds for $X$ Liouville with $Y=\partial_\infty X$ and $H$ defined near infinity.)
The contact vector field associated to a section $f$ of $TY/\xi$ is denoted $V_f$.
In the presence of a contact form $\alpha$, sections of $TY/\xi$ are identified with real valued functions via pairing with $\alpha$, so we may write $V_f$ for functions $f$.
The Reeb vector field $\R_\alpha:=V_1$ is the contact vector field associated to the constant function $1$ (equivalently, $\R_\alpha$ is defined by the properties $\alpha(\R_\alpha)=1$ and $d\alpha(\R_\alpha,\cdot)=0$); more generally $V_f=\R_{f^{-1}\alpha}$ for $f>0$.

\subsection{Liouville sectors}\label{sectordefsec}

A Liouville manifold-with-boundary $X$ is defined analogously to a Liouville manifold: it is an exact symplectic manifold-with-boundary for which a neighborhood of infinity is given by the positive half of the symplectization of a contact manifold-with-boundary $\partial_\infty X$.
The Liouville vector field is allowed to be non-tangent to $\partial X$ over a compact set (and so, in particular, it is not required to be complete, except at infinity).
Because of this, there may be no embedding of the full symplectization of $\partial_\infty X$ into $X$.

Floer theory on Liouville manifolds-with-boundary is not well-behaved in general, since holomorphic curves can touch the boundary.
This situation can be remedied by introducing appropriate assumptions on the characteristic foliation of the boundary.
This leads to the notion of a \emph{Liouville sector}, which we now introduce and proceed to study from a purely symplectic geometric viewpoint.

In order to state the definition, let us recall that a hypersurface in a symplectic manifold carries a canonical one-dimensional foliation, called the \emph{characteristic foliation}, whose tangent space is the kernel of the restriction of the symplectic form to the hypersurface (as is standard, we shall abuse terminology and use the words ``characteristic foliation'' to refer to this kernel as well).
Recall also that a hypersurface in a contact manifold is said to be \emph{convex} iff there is a contact vector field defined in its neighborhood which is transverse to it.

\begin{definition}\label{sectordef}
A \emph{Liouville sector} is Liouville manifold-with-boundary
satisfying the following equivalent conditions:
\begin{itemize}
\item For some $\alpha>0$, there exists $I:\partial X\to\RR$ with $ZI=\alpha I$ near infinity and $dI|_\charfol>0$.
\item For every $\alpha>0$, there exists $I:\partial X\to\RR$ with $ZI=\alpha I$ near infinity and $dI|_\charfol>0$.
\item The boundary of $\partial_\infty X$ is convex
and there is a diffeomorphism $\partial X=\RR\times F$ sending the characteristic foliation of $\partial X$ 
to the foliation of $\RR\times F$ by leaves $\RR\times\{p\}$.
\end{itemize}
In the first two conditions, the characteristic foliation $C$ is oriented so that $\omega(N,C)>0$ for any inwarding pointing vector $N$.
Note that $dI|_\charfol>0$ is equivalent to the Hamiltonian vector field $X_I$ being outward pointing along $\partial X$.
We call such $I$ ``$\alpha$-defining functions'' for $\partial X$ (with the convention that $\alpha=1$ if omitted).
Note that the space of $\alpha$-defining functions is convex, and thus either empty or contractible.

Observe that being a Liouville sector is an open condition, i.e.\ it is preserved under small deformations within the class of Liouville manifolds-with-boundary.

An ``inclusion of Liouville sectors'' $i:X\hookrightarrow X'$ shall mean a proper map which is a diffeomorphism onto its image, satisfying $i^\ast\lambda'=\lambda+df$ for compactly supported $f$.
A ``trivial inclusion of Liouville sectors'' is one for which $i(X)$ may be deformed to $X'$ through Liouville sectors included into $X'$.
\end{definition}

\begin{lemma} \label{lem:sectordef}
The conditions in Definition \ref{sectordef} are equivalent.
\end{lemma}

\begin{proof}
To prove the equivalence of the first two conditions, suppose we have an $\alpha$-defining function $I:\partial X\to\RR$ and let us produce an $\alpha'$-defining function $I'$ by smoothing $\frac I{\left|I\right|}\left|I\right|^{\alpha'/\alpha}$ as follows.
Let $Y$ be a contact type hypersurface in $X$, far out near infinity, projecting diffeomorphically onto $\partial_\infty X$ via the forward Liouville flow ($Y$ meets $\partial X$ transversely).
Since $Z$ is tangent to $\{I=0\}$ near infinity, we conclude that $Y\cap\partial X$ and $\{I=0\}$ are transverse submanifolds of $\partial X$.
Now we also know that the characteristic foliation of $\partial X$ is transverse to $\{I=0\}$, so combining these two facts we can modify $Y$ locally near $Y\cap\{I=0\}$ so that $Y\cap\partial X$ is tangent to the characteristic foliation of $\partial X$ in a neighborhood of $Y\cap\{I=0\}$.
Since the characteristic foliation of $\partial X$ is now tangent to $Y\cap\partial X$ near $Y\cap\{I=0\}$, we can smooth the restriction of $\frac I{\left|I\right|}\left|I\right|^{\alpha'/\alpha}$ to $Y\cap\partial X$ near $Y\cap\{I=0\}$ so as to make its differential positive on the characteristic foliation, and then extend it to the positive half of the symplectization $\RR_{\geq 0}\times Y\subseteq X$ by the scaling property $ZI'=\alpha'I'$ (this extension remains positive on the characteristic foliation since $Z$ preserves the characteristic foliation).
It is straightforward to extend this smoothing of $\frac I{\left|I\right|}\left|I\right|^{\alpha'/\alpha}$ over $\RR_{\geq 0}\times Y$ to all of $X$ since the characteristic foliation (on which $dI$ is positive) is transverse to the non-smooth locus $\{I=0\}$.

To see that the first two conditions imply the third, observe that for a defining function $I:\partial X\to\RR$ (meaning $\alpha=1$), its Hamiltonian vector field $X_I$ gives a contact vector field on $\partial_\infty X$, which is outward pointing since $CI>0$.
By assumption, $dI$ is positive on the characteristic foliation, thus $I$ is in particular a submersion.  Along with the control in $I$ near infinity, it follows
that there is a diffeomorphism $\partial X=\RR\times I^{-1}(0)$ as desired.

Finally, suppose that the third condition is satisfied, and let us construct a defining function $I$.
Since $\partial_\infty X$ has convex boundary, there exists a function $I:\partial X\to\RR$ defined near infinity satisfying $ZI=I$ and $CI>0$.
Using the diffeomorphism $\partial X=\RR\times F$, suppose $I$ is defined over a neighborhood of the complement of $(-N,N)\times U$ for some pre-compact open $U\subseteq F$.
Now $I$ can be (re)defined on $(-N,N)\times U$ so that that $CI>0$ iff $I(N,p)>I(-N,p)$ for all $p\in U$, and this can be achieved by taking $N<\infty$ sufficiently large.
\end{proof}

\begin{question}\label{convexredundant}
Suppose $X$ is a Liouville manifold-with-boundary and there is a diffeomorphism $\partial X=\RR\times F$ sending the characteristic foliation to the foliation by leaves $\RR\times\{p\}$.
Is $X$ a Liouville sector?
\end{question}

\begin{example}
If $Q$ is a compact manifold-with-boundary, then $T^\ast Q$ is a Liouville sector.
Indeed, any vector field on $Q$ lifts to a Hamiltonian vector field on $T^\ast
Q$ (with linear Hamiltonian), and the lift of a vector field transverse to
$\partial Q$ thus certifies that $T^\ast Q$ is a Liouville sector.
Furthermore, if $Q_0 \hookrightarrow Q_1$ is a codimension zero embedding of a compact
manifolds-with-boundary, then $T^\ast Q_0 \hookrightarrow T^\ast Q_1$ is an inclusion
of Liouville sectors.
\end{example}

\begin{remark}[Open Liouville sectors]\label{opensector}
An \emph{open Liouville sector} (or perhaps an ``ind-(Liouville sector)'') is a pair $(X,\partial_\infty X)$ where $X$ is an exact symplectic manifold and $\partial_\infty X$ is a contact manifold (both without boundary), together with a germ near $+\infty$ of a (codimension zero) embedding of the symplectization $S\partial_\infty X$ into $X$ (strictly respecting Liouville forms), such that the pair $(X,\partial_\infty X)$ is exhausted by Liouville sectors.
Being exhausted by Liouville sectors means that every subset of $X$ which away from a compact subset of $X$ equals the cone over a compact subset of $\partial_\infty X$, is contained in a Liouville sector $X_0\subseteq X$ with $\partial_\infty X_0\subseteq\partial_\infty X$.

For example, $T^*\RR^n$ is an open Liouville sector, as is more generally $T^*M$ for
any (not necessarily compact) manifold $M$; an exhaustion is given by the family of Liouville subsectors $T^*M_0$ for compact codimension zero submanifolds-with-boundary $M_0\subseteq M$ (which obviously exhaust $M$).
For any Liouville sector $X$, its interior $X\setminus\partial X$ is an open Liouville sector.

A (codimension zero) inclusion of open Liouville sectors $i:(X,\partial_\infty X)\hookrightarrow(Y,\partial_\infty Y)$ is simply a map of pairs in the obvious sense (i.e.\ compatible with the embeddings $S\partial_\infty X\hookrightarrow X$ and $S\partial_\infty Y\hookrightarrow Y$ defined near infinity) satisfying $i^\ast\lambda_Y=\lambda_X+df$ where the support of $f$ does not approach $\partial_\infty X$.
For example, for any open
inclusion of manifolds $M\hookrightarrow N$, the inclusion $T^*M\hookrightarrow
T^*N$ is an inclusion of open Liouville sectors.
For any inclusion of Liouville sectors $X\hookrightarrow X'$, the associated inclusion of their interiors $X\setminus\partial X\hookrightarrow X'\setminus\partial X'$ is an inclusion of open Liouville sectors.
\end{remark}

\begin{lemma}\label{deformboundary}
Let $X$ be a Liouville sector, and let $Y_t$ be a $1$-parameter family of contact manifolds with convex boundary, where $Y_0=\partial_\infty X$.
There exists a corresponding $1$-parameter family of Liouville sectors $X_t$ and isomorphisms $\partial_\infty X_t=Y_t$, specializing to $X_0=X$ and $Y_0=\partial_\infty X$.
\end{lemma}

\begin{proof}
By Gray's theorem, $Y_t$ for $t$ close to $0$ may be viewed simply as a deformation of the boundary of $Y_0$ (i.e.\ the contact structure is fixed).
Now consider an arbitrary deformation of the boundary of the symplectization of $Y_0$, which is fixed for $s\ll 0$ and which follows $\partial Y_t$ for $s\gg0$.
On this deformation, for $t$ sufficiently small, we may splice together the defining function $I_0:\partial(SY_0)\to\RR$ for $s\ll 0$ with the defining functions $I_t:\partial(SY_t)\to\RR$ for $s\gg0$.
This proves the result for $t$ sufficiently small.
Now the general case follows from a compactness argument.
\end{proof}

\begin{definition}\label{lsfiber}
Let $X$ be a Liouville sector.  The symplectic reduction $F:=(\partial X)/C$ (quotient by the characteristic foliation) is a smooth manifold, and there is a diffeomorphism $\partial X=\RR\times F$ in which the leaves of the characteristic foliation are $\RR\times\{p\}$ (see Definition \ref{sectordef}).
By Cartan's formula, the restriction of the symplectic form $\omega|_{T\partial X}$ is pulled back from the projection $\partial X\to F$; moreover, the restriction of the Liouville form $\lambda|_{T\partial X}$ is (locally) pulled back from $F$ near infinity (more precisely, over the locus where $Z$ is tangent to $\partial X$).
Choosing a section of the projection $\partial X\to F$ thus defines a Liouville form on $F$ which is well-defined up to adding $df$ for compactly supported $f$.
Note that $F$ is a Liouville manifold when equipped with any/all of these $\lambda$; convexity at infinity may be seen by using the embedding $F=I^{-1}(0)\subseteq\partial X$ for any $\alpha$-defining function $I$.

In particular, there are two Liouville forms on $F$, denoted $\lambda_\infty$ and $\lambda_{-\infty}$, obtained by embedding $F=I^{-1}(a)\subseteq\partial X$ for $a\to\pm\infty$ and any $\alpha$-defining function $I$.
We have
\begin{equation}\label{Cintegraldifference}
\lambda_\infty-\lambda_{-\infty}=d\int_C\lambda
\end{equation}
where $\int_C\lambda$ denotes the compactly supported function $F\to\RR$ obtained by integrating $\lambda$ over the leaves of the characteristic foliation (i.e.\ the fibers of the projection $\partial X\to F$).
We say that \emph{$\partial X$ is exact} or \emph{$X$ has exact boundary} iff $\int_C\lambda\equiv 0$, which implies $\lambda_\infty=\lambda_{-\infty}$ 
(and for $\dim X\geq 4$, the converse implication holds as well).
We will see in Proposition \ref{deformexact} that every Liouville sector can be deformed so that its boundary becomes exact.
\end{definition}

\begin{lemma}\label{sectorboundarytangent}
Let $X$ be a Liouville sector with exact boundary.
There exists a compactly supported function $f$ such that $Z_{\lambda+df}$
(the Liouville vector field associated to the Liouville form $\lambda+df$)
is everywhere tangent to $\partial X$.
\end{lemma}

\begin{proof}
We have $Z_{\lambda+df}=Z_\lambda-X_f$, which is tangent to $\partial X$ if and only if
\begin{equation}\label{hameqn}
df|_C=-\lambda|_C
\end{equation}
where $C$ denotes the characteristic foliation of $\partial X$.
Note that $\lambda|_C$ has compact support since $Z_\lambda$ is already tangent to $\partial X$ near infinity.

Since $X$ is a Liouville sector, there is a diffeomorphism $\partial X=\RR\times F$ sending the characteristic foliation to the foliation by $\RR\times\{p\}$.
It follows from this normal form that there is at most one function $f:\partial X\to\RR$ of compact support satisfying \eqref{hameqn}.
The existence of such an $f$ is equivalent to the vanishing of the integral of $\lambda$ over every leaf of $C$, which is precisely the definition of $\partial X$ being exact.
\end{proof}

\subsection{Constructions of Liouville sectors}

We now develop tools for constructing Liouville sectors, and we use these tools to give more examples of Liouville sectors.

\begin{remark}\label{sectorsmoothcorners}
Constructions of Liouville sectors sometimes involve ``smoothing corners'' to convert a Liouville manifold-with-corners into a Liouville manifold-with-boundary.
We therefore record here the convenient fact that, to show that the result is a Liouville sector, it is enough to check the existence of a defining function $I$ before smoothing the corners (in which case the condition $dI|_\charfol>0$ is imposed over every closed face).
In fact, (any smooth extension of) the same function $I$ will do the job.
To see this, simply note that (the positive ray of) the characteristic foliation at a point of the smoothed boundary lies in the convex hull of (the positive rays of) the characteristic foliations of the faces of the nearby cornered boundary; hence positivity of $dI|_\charfol$ is preserved by the smoothing process.
\end{remark}

\begin{lemma}\label{suturedliouvillemanifold}
Let $\bar X_0$ be a Liouville domain, and let $A\subseteq\partial\bar X_0$ be a codimension zero submanifold-with-boundary such that there exists a function $I:A\to\RR$ with $\R_\lambda I>0$ such that the contact vector field $V_I$ is outward pointing along $\partial A$.
Then
\begin{equation}\label{suturedliouvillemanifoldeq}
X:=\bar X\setminus(\RR_{>0}\times A^\circ)
\end{equation}
is a Liouville sector, where $\bar X$ denotes the Liouville completion of $\bar X_0$.
\end{lemma}

\begin{proof}
The linear extension of $-I$ is a defining function for $X$.
\end{proof}

\begin{definition}\label{liouvillesectorcreation}
A \emph{sutured Liouville domain} $(\bar X_0,F_0)$ is a Liouville domain $\bar X_0$ together with a codimension one submanifold-with-boundary $F_0\subseteq\partial\bar X_0$ such that $(F_0,\lambda)$ is a Liouville domain.
Similarly, a \emph{sutured Liouville manifold} is a Liouville manifold $\bar X$ together with a codimension one submanifold-with-boundary $F_0\subseteq\partial_\infty\bar X$ and a contact form $\lambda$ defined over $\Nbd F_0$ such that $(F_0,\lambda)$ is a Liouville domain.
(Compare with the notion of a ``Weinstein pair'' from \cite{eliashbergweinsteinrevisited}.)

Given a sutured Liouville domain $(\bar X_0,F_0)$, the Reeb vector field of $\lambda$ is transverse to $F_0$ since $d\lambda|_{F_0}$ is symplectic, and thus determines a local coordinate chart $F_0\times\RR_{\left|t\right|\leq\varepsilon}\hookrightarrow\partial\bar X_0$ in which the contact form $\lambda$ equals $dt+\lambda|_{F_0}$.
The contact vector field associated to the function $t$ is given by $t\frac\partial{\partial t}+Z_{\lambda|_{F_0}}$ which is outward pointing along $\partial(F_0\times\RR_{\left|t\right|\leq\varepsilon})$.
We conclude that a sutured Liouville domain $(\bar X_0,F_0)$ in the present
sense determines a codimension zero submanifold $A=F_0\times\RR_{\left|t\right|\leq\varepsilon}$ of $\partial\bar X_0$, which satisfies the
hypotheses of Lemma \ref{suturedliouvillemanifold} (witnessed by the function
$I=t$). In particular, (the conclusion of Lemma \ref{suturedliouvillemanifold}
implies) a sutured Liouville domain $(\bar X_0,F_0)$ gives rise to a Liouville
sector.

We will see in Lemma \ref{suturedequivalence} that every Liouville sector arises from a unique (in the homotopical sense) sutured Liouville domain.
\end{definition}

\begin{example}\label{removelegendrian}
If $\Lambda\subseteq\partial_\infty\bar X$ is a Legendrian, by the Weinstein neighborhood theorem, there are (homotopically unique) coordinates near $\Lambda$ given by $\RR_t\times T^\ast\Lambda$ with contact form $dt+\lambda$.
Choosing $F_0=D^\ast\Lambda$ gives a sutured Liouville domain and thus a Liouville sector $X$, which we think of informally as being obtained from $\bar X$ by removing a small regular neighborhood of $\Lambda$.

It would be of interest to generalize this construction to sufficiently nice (e.g.\ sub-analytic) singular Legendrian $\Lambda$, however this requires constructing a convex neighborhood of such $\Lambda$.
\end{example}

\begin{remark}
The notion of the skeleton $\LL\subseteq X$ of a Liouville domain/manifold (see Remark \ref{spinecoreskeletonrmk}) admits a natural generalization to sutured Liouville domains/manifolds.
Namely, given a sutured Liouville domain $(\bar X_0,F_0)$, we consider the loci $\LL_0\subseteq\bar X_0$ and $\LL\subseteq\bar X$ of points which do not escape to the complement of the skeleton of $F_0$ at infinity.
Note that $\LL$ is necessarily non-compact unless $F_0$ is empty.
As before, under certain assumptions on the Liouville flow on $\bar X$ and $F_0$ (e.g.\ if both are Weinstein), then $\LL_0\subseteq\bar X_0$ (and $\LL\subseteq\bar X$) is a singular isotropic spine.
We will call $\LL_0\subseteq\bar X_0$ the \emph{skeleton of $\bar X_0$ relative to $F_0$} or simply the \emph{relative skeleton} of the sutured Liouville domain $(\bar X_0,F_0)$; analogous terminology applies to $\LL\subseteq\bar X$.
It is reasonable to regard such a skeleton as also being associated to the corresponding Liouville sector.
\end{remark}

\begin{definition}
An \emph{open book decomposition} of a contact manifold $Y$ consists of a \emph{binding} $B\subseteq Y$ (a codimension two submanifold), a tubular neighborhood $B\times D^2\subseteq Y$,
a submersion $\pi:Y\setminus B\to S^1$ standard over $B\times D^2$, and a contact form $\alpha$ on $Y$ such that the \emph{pages} of the open book $(\pi^{-1}(a),d\alpha)$ are symplectic, and $\alpha=(1+\frac 12r^2)^{-1}(\alpha|_B+\lambda_{D^2})$ over $B\times D^2$, where $\lambda_{D^2}:=\frac 12(x\,dy-y\,dx)=\frac 12r^2d\theta$.
Experts will note that we could equivalently use any smooth radial function with negative radial derivative (for $r>0$) in place of $(1+\frac 12r^2)^{-1}$.
The particular choice $(1+\frac 12r^2)^{-1}$ has the nice property that the Reeb vector field of $\alpha$ is given by $\R_\alpha=\R_{\alpha|_B}+\frac\partial{\partial\theta}$ over $B\times D^2$ (see \eqref{bindingflow}).
\end{definition}

\begin{lemma}\label{openbookconvex}
Let $Y$ be a contact manifold equipped with an open book decomposition $(B,\pi,\alpha)$.
Let $Q\subseteq Y$ be a hypersurface which outside $B\times D^2$ coincides with $\pi^{-1}(\{\theta_1\cup\theta_2\})$ and which inside $B\times D^2$ is given by $B\times\gamma$ where $\gamma$ is a simple arc in $D^2$ connecting $\theta_1,\theta_2\in\partial D^2$.
Then $Q$ is convex (i.e.\ there is a contact vector field transverse to $Q$).
\end{lemma}

\begin{proof}
With respect to the contact form $\lambda_B+\frac 12r^2d\theta$ on $B\times D^2$, the contact vector field for a contact Hamiltonian $f:D^2\to\RR$ is given by
\begin{equation}\label{bindingflow}
(f-Z_{D^2}f)\R_B+X_f
\end{equation}
where $Z_{D^2}=\frac 12r\frac\partial{\partial r}$ is the Liouville vector field of $\lambda_{D^2}=\frac 12r^2d\theta$ and $X_f$ denotes the Hamiltonian vector field of $f$ with respect to the area form $\omega_{D^2}=r\,dr\,d\theta=dx\,dy$.
This contact vector field is transverse to $Q\cap(B\times D^2)$ exactly when the restriction of $f$ to $\gamma$ has no critical points.
We can thus arrange that $f=\pm(1+\frac 12r^2)^{-1}$ near $\theta_1$ and $\theta_2$ respectively, and hence it extends to the rest of $Q$ as plus/minus the Reeb vector field of $\alpha$, which is transverse to $Q$ as desired.
\end{proof}

\begin{example}\label{singlepageremoval}
Let $\bar X$ be a Liouville manifold, and suppose $\partial_\infty\bar X$ is equipped with an open book decomposition $(B,\pi,\alpha)$.
A choice of page $F_0:=\pi^{-1}(t)\setminus(B\times D^2_\delta)$ 
determines a sutured Liouville domain and thus a Liouville sector $X$.
Note that for any other page $F_0':=\pi^{-1}(t')\setminus(B\times D^2_\delta)$, Lemma \ref{openbookconvex} implies that $\partial_\infty X=\partial_\infty\bar X\setminus N_\varepsilon F_0$ can be deformed to $N_\varepsilon F_0'$ through codimension zero submanifolds-with-boundary of $\partial_\infty\bar X$ with convex boundary (namely, one deforms the complement of a neighborhood of $t\in S^1$ to a neighborhood of $t'\in S^1$ and takes the inverse image under $\pi$, smoothing the boundary appropriately near the binding $B$).
Using Lemma \ref{deformboundary}, this deformation can be lifted to a deformation of $X$, so $\partial_\infty X$ is, up to deformation, a regular neighborhood of a complementary page.
\end{example}

\begin{example}\label{landauginzburgsector}
Let $E$ be a Liouville manifold equipped with a ``superpotential'' $\pi:E\to\CC$ (the pair $(E,\pi)$ is called an exact symplectic (Liouville) Landau--Ginzburg model).
The map $\pi$ determines an embedding of $(S^1\times F_0,dt+\lambda)$ into $\partial_\infty E$, where $(F_0,\lambda)$ is a Liouville domain whose completion is the generic fiber of $\pi$ and the $S^1$ factor corresponds to the angular coordinate of $\CC$.
Applying the construction of Example \ref{removelegendrian} to a fiber $\{t\}\times F_0$ gives rise to a Liouville sector $X$ associated to $\pi:E\to\CC$.
One should think of $X$ as being obtained from $E$ by removing the inverse image of a neighborhood of a ray at angle $t$ in $\CC$.
When the critical locus of $\pi$ is compact, the embedding $S^1\times F_0\subseteq\partial_\infty E$ extends to an open book decomposition of $\partial_\infty E$, and hence the conclusion of Example \ref{singlepageremoval} applies.
\end{example}

\begin{lemma}
Let $X$ and $Y$ be Liouville sectors whose Liouville vector fields are everywhere tangent to the boundary.
The product $X\times Y$ is also a Liouville sector. 
\end{lemma}

Recall that Lemma \ref{sectorboundarytangent} provides Liouville vector fields which are everywhere tangent to the boundary on any Liouville sector with exact boundary, and we will see later in Proposition \ref{deformexact} that every Liouville sector can be canonically deformed to have exact boundary.  Because of this, we may abuse notation and write $X\times Y$ for the Liouville sector obtained by performing such a deformation on $X$ and $Y$ and then taking their product.

\begin{remark}
The ``stabilization operation'' of passing from a Liouville sector $X$ to $X\times T^\ast[0,1]$ is of particular interest, and should induce an equivalence on Floer theoretic invariants (as a consequence of a K\"unneth formula).
The stabilization operation for Landau--Ginzburg models, namely passing from $\pi:E\to\CC$ to $\pi+z^2:E\times\CC\to\CC$, should be a special case of this.
More generally, the sum of Landau--Ginzburg models $\pi+\pi':E\times E'\to\CC$ should be a special case of the product of Liouville sectors.
\end{remark}

\begin{proof}
The product $X\times Y$ is a Liouville manifold-with-corners.
By Remark \ref{sectorsmoothcorners}, it is enough to verify the existence of a defining function on $X\times Y$ before smoothing the corners.

Fix defining functions $I_X:\partial X\to\RR$ and $I_Y$, and extend them to all of $X$ and $Y$ maintaining linearity at infinity (we could cut them off so they are supported in neighborhoods of the respective boundaries, though this is irrelevant for the present argument).

We now consider the function $I_X+I_Y$ on $X\times Y$.
Its differential is clearly positive on the characteristic foliation of $\partial(X\times Y)$, since the characteristic foliation of $\partial X\times Y$ 
is simply $C_X\oplus\{0\}\subseteq T \partial X \oplus TY = T(\partial X \times Y)$, and similarly for $X\times\partial Y$.
However, $I_X+I_Y$ may not be linear at infinity for the Liouville vector field $Z_{X\times Y}=Z_X+Z_Y$.
There are two disjoint ``problem'' regions, namely a compact locus in $X$ times a neighborhood of infinity in $Y$, and vice versa.
We will deal with these separately, and by symmetry it suffices to deal with the first one.

Fix a contact type hypersurface in $Y$ close to infinity mapping diffeomorphically onto $\partial_\infty Y$ (equivalently, fix a large contact form on $\partial_\infty Y$).
Consider the restriction of $I_X+I_Y$ to $X\times\partial_\infty Y$ viewed as the corresponding contact type hypersurface in $X\times Y$.
Define a new function $I_{X\times Y}:X\times Y\to\RR$ by extending $I_X+I_Y$ to be linear outside this contact type hypersurface (and smoothing the result).
Note that $I_{X\times Y}$ agrees with $I_X+I_Y$ except over the bad locus where $I_X+I_Y$ is not linear at infinity.
It is enough to show that the Hamiltonian vector field of $I_{X\times Y}$ is outward pointing along the boundary, and it is enough to check this before doing the smoothing.

So, let us calculate the Hamiltonian vector field of $I_{X\times Y}$.
Initially, we have coordinates $(X\times\RR\times\partial_\infty Y,\lambda_X+e^s\alpha_Y)$ for $(X\times Y,\lambda_X+\lambda_Y)$; in these coordinates $I_X+I_Y$ equals $I_X+e^sA_Y$ for a function $A_Y:\partial_\infty Y\to\RR$.
We change coordinates to $(X\times\RR\times\partial_\infty Y,e^s(\lambda_X+\alpha_Y))$; note that these describe the same exact symplectic manifold in view of the common contact type hypersurface $\{s=0\}$ in both manifolds and the completeness of their respective Liouville vector fields (note that this argument uses crucially the fact that $Z_X$ and $Z_Y$ are tangent to $\partial X$ and $\partial Y$, respectively).
In the latter coordinates, the function $I_{X\times Y}$ is given by $e^s(I_X+A_Y)$ (for $s\geq 0$), assuming that $\{s=0\}$ is the contact type hypersurface chosen to define $I_{X\times Y}$ from $I_X+I_Y$.
Let the Hamiltonian vector field of $I_Y=e^sA_Y$ on $(\RR\times\partial_\infty Y,e^s\alpha_Y)\subseteq(Y,\lambda_Y)$ be given by $V_Y+f\frac\partial{\partial s}$ for a contact vector field $V_Y$ on $Y$ and a function $f:Y\to\RR$.
Now a calculation shows that the Hamiltonian vector field of $e^s(I_X+A_Y)$ on $(X\times\RR\times\partial_\infty Y,e^s(\lambda_X+\alpha_Y))$ is given by
\begin{equation}\label{linearizedhamsum}
X_{I_X}+V_Y+f\frac\partial{\partial s}-fZ_X+(I_X-Z_XI_X)\R_{\alpha_Y}
\end{equation}
where $\R_{\alpha_Y}$ denotes the Reeb vector field of the contact form $\alpha_Y$ on $\partial_\infty Y$.
We know that $X_{I_X}$ is outward pointing along $\partial X$, and $V_{I_Y}$ is outward pointing along $\partial(\partial_\infty Y)$ by assumption.
The next two terms are both tangent to the boundary.
The third term converges to zero as $\alpha_Y$ becomes large, and hence we conclude that, for sufficiently large $\alpha_Y$, the vector field \eqref{linearizedhamsum} is outward pointing along the boundary, as desired (note that the other terms in \eqref{linearizedhamsum} are unchanged by scaling $\alpha_Y$, and that we only need to check the property of being outward pointing over the compact set $\{0\}\times\partial_\infty Y$ times a large compact subset of $X$ outside which $Z_XI_X=I_X$).
\end{proof}

\begin{lemma}\label{twosutures}
Every pair $(\bar X_0,A)$ satisfying the hypotheses of Lemma \ref{suturedliouvillemanifold} arises, up to deformation, from a unique (in the homotopical sense) sutured Liouville domain.
\end{lemma}

\begin{proof}
Let $A$ be an odd-dimensional manifold-with-corners equipped with a $1$-form $\lambda$ with $d\lambda$ maximally non-degenerate.
We say that $(A,\lambda)$ is \emph{matched} iff the following are satisfied:
\begin{itemize}
\item $\partial A$ is the union of two faces $(\partial A)_\pm$ meeting transversely along the corner locus.
\item The characteristic foliation of $d\lambda$ is positively/negatively transverse to $(\partial A)_\pm$, respectively.
\item Following the characteristic foliation defines a diffeomorphism $(\partial A)_+\xrightarrow\sim(\partial A)_-$.
\item $((\partial A)_\pm,\lambda|_{(\partial A)_\pm})$ are Liouville domains.
\end{itemize}
Being matched is clearly an open condition.

If $(A,\lambda)$ is matched, then the image $F_0\subseteq A$ of \emph{any} section of $A\to A/C$ (quotient by the characteristic foliation) is a Liouville domain (when equipped with the restriction of $\lambda$).
Conversely, to check that $((\partial A)_\pm,\lambda|_{(\partial A)_\pm})$ are Liouville domains, it is enough to check that any such $F_0$ is a Liouville domain.
If $\lambda$ is a contact form, then choosing an $F_0$ provides a unique embedding $A\subseteq F_0\times\RR_t$ in which $\lambda=\lambda|_{F_0}+dt$ and
\begin{equation}\label{Agineq}
A=\{g_-\leq t\leq g_+\}
\end{equation}
for $g_\pm:F_0\to\RR$ where $\pm g_\pm$ are positive on the interior and vanish transversely on the boundary.
Conversely, \eqref{Agineq} is matched for any Liouville domain $F_0$ and any such $g_\pm$.

If $(A,\lambda)$ is matched and $\lambda$ is a contact form, then there exists a function $I:A\to\RR$ with $\R_\lambda I>0$ whose contact vector field $V_I$ is outward pointing along $\partial A$ (as in the hypothesis of Lemma \ref{suturedliouvillemanifold}).
Indeed, let $I=f(t)$ in contactization coordinates $A\subseteq F_0\times\RR_t$.
Since $\R_\lambda=\frac\partial{\partial t}$, we must have $f'(t)>0$.
The contact vector field associated to $I$ is given by $V_I=f(t)\frac\partial{\partial t}+f'(t)Z_{\lambda|_{F_0}}$, which is outward pointing for, say, $f(t):=\tan^{-1}(Nt)$ for sufficiently large $N<\infty$ (more precisely, we just need $f(0)=0$, $f'(t)>0$, and $f'(t)/f(t)$ decaying sufficiently rapidly away from $t=0$).

A sutured Liouville manifold is ``the same'' as a pair $(\bar X_0,A)$ satisfying the hypothesis of Lemma \ref{suturedliouvillemanifold} for which, in addition,  $A$ is matched.
Indeed, the above discussion shows that the space of allowable $F_0$ inside a matched $A$ is contractible, and so is the space of matched $A\subseteq\partial\bar X_0$ containing a given fixed $F_0\subseteq\partial\bar X_0$.
We conclude that it is enough to show that every pair $(\bar X_0,A)$ satisfying the hypothesis of Lemma \ref{suturedliouvillemanifold} may be canonically deformed to make $A$ matched.

Let $(\bar X_0,A)$ be given, and let us specify a canonical deformation which makes $A$ matched.
Let $X$ denote the Liouville sector \eqref{suturedliouvillemanifoldeq} associated to $(\bar X_0,A)$, and fix a defining function $I:\partial X\to\RR$ which is the linear extension of a defining function $I|_A:A\to\RR$.
There are Liouville manifolds $F_\pm:=I^{-1}(s)$ for $s\to\pm\infty$, which are identified via the characteristic foliation of $\partial X$ (with the caveat that, over a compact set, this identification depends on how we smooth the corners of $\partial X$).
We choose large Liouville domains $(F_0)_\pm\subseteq F_\pm$ (identified under $F_+\xrightarrow\sim F_-$) and functions $f_\pm:(F_0)_\pm\to\RR$ such that $\pm f_\pm$ are positive on the interior and vanish transversely on the boundary.
Now the locus
\begin{equation}
A':=\{f_-\leq I\leq f_+\}\subseteq\partial X
\end{equation}
is well-defined once $(F_0)_\pm$ and $\pm f_\pm$ are taken sufficiently large.
It is clear from the structure of the characteristic foliation of $\partial X$ that $A'$ is matched.
We claim that, for suitable $f_\pm$, the Liouville vector field is outward pointing along $\partial A'$.
The Liouville vector field is given by $Z_{F_\pm}+I\frac\partial{\partial I}$ (outside a compact set), which is outward pointing along $\partial A'$ iff $\pm(f_\pm-Z_{F_\pm}f_\pm)>0$, which is easy to achieve.
We conclude that the Liouville vector field demonstrates that the region $A'\setminus A^\circ$ is a cylinder, and in particular there is a deformation $A\subseteq A_r\subseteq A'$ from $A_0=A$ to $A_1=A'$ such that the Liouville vector field is outward pointing along $\partial A_r$ for all $r\in[0,1]$.
We would like to follow this deformation $A_r$ of $A$ with a corresponding deformation of Liouville domains $(\bar X_0)_r\subseteq\bar X$ with $A_r\subseteq\partial(\bar X_0)_r$.
This is, of course, not possible on the nose since the Liouville vector field is tangent to the finite cylindrical region $A_r\setminus A^\circ$ rather than being outward pointing.
This is easily remedied, however, simply by perturbing the cylindrical region $A_r\setminus A^\circ$ keeping its upper boundary $\partial A_r$ fixed and pushing its lower boundary $\partial A$ inwards inside $A$.
We have thus defined a deformation of $(\bar X_0,A)$ which makes $A$ matched (let us also point out that $\bar X$ and $A\subseteq\partial_\infty\bar X$ remain fixed throughout the deformation, i.e.\ we are deforming only the contact form).

Now it remains to show that if $A$ is already matched, then the deformation described above can be taken so that $A_r$ is matched for all $r\in[0,1]$.
Suppose $A$ is presented as in \eqref{Agineq}, and fix $I:=e^st$.
We consider the deformation
\begin{equation}
A_r:=A\cup_{\partial A}(\partial A\times\RR_{0\leq s\leq r})\subseteq\partial X
\end{equation}
for $r\geq 0$.
Using the fact that the characteristic foliation of $(\partial A)_\pm$ is spanned by $\frac\partial{\partial s}-Z_{\lambda|_{(\partial A)_\pm}}$, we see that each $A_r$ is matched.
Now $F_\pm$ is the Liouville completion of $(\partial A)_\pm$, and the deformation $A_r$ is of the form specified earlier with $f_\pm:F_\pm\to\RR$ given by $f_\pm(x)=e^rg_\pm(\Phi^{-r}_{Z_{F_\pm}}(x))$.
Note that we may assume without loss of generality that $\pm(g_\pm-Z_{F_\pm}g_\pm)>0$, which implies the same for $f_\pm$.
\end{proof}

\subsection{Product decomposition near the boundary}

We now show that for every Liouville sector $X$, there is a canonical (up to contractible choice) identification near the boundary between $X$ and a product $F\times\CC_{\Re\geq 0}$.
More precisely, every $\alpha$-defining function, extended to a cylindrical (i.e.\ $Z$-invariant near infinity) neighborhood of $\partial X$, determines uniquely such coordinates.
Of particular interest is the resulting projection $\pi:\Nbd^Z\partial X\to\CC_{\Re\geq 0}$ for $\alpha=\frac 12$.
As we will see in \S\ref{bdryescape}, this function $\pi$ gives strong control on holomorphic curves near $\partial X$.

Equip $\CC$ with its standard symplectic form $\omega_\CC:=dx\,dy$ for $z=x+iy$ and the family of Liouville vector fields 
\begin{equation}\label{alphaLiouvillefield}
Z_\CC^\alpha:=(1-\alpha)\cdot x\frac\partial{\partial x}+\alpha\cdot y\frac\partial{\partial y}
\end{equation}
with associated Liouville forms $\lambda_\CC^\alpha$, parameterized by $\alpha\in\RR$.
When $\alpha=\frac 12$, we will simply write $Z_\CC:=\frac 12(x\frac\partial{\partial x}+y\frac\partial{\partial y})$ and $\lambda_\CC$ for the standard radial Liouville structure on $\CC$.
When $\alpha=1$, we will write $T^\ast\RR$ (equipped with its standard Liouville structure $\lambda_{T^\ast\RR}=p\,dq$ with $q=x$ and $p=-y$) in place of $\CC$.

\begin{remark}
The significance of the particular value $\alpha=\frac 12$ in the definition of $\pi$ is that the complex structure $J_\CC$ is invariant under the radial Liouville vector field $Z_\CC$ on $\CC$ (which is \eqref{alphaLiouvillefield} with $\alpha = \frac 12$).
This allows us to find an abundance of \emph{cylindrical} almost complex structures $J$ on $X$ for which $\pi$ is $(J,J_\CC)$-holomorphic.
The usefulness of such almost complex structures will be made clear in \S\ref{compactnessdiscussion} (in a word, cylindricity prevents holomorphic curves from escaping to infinity, and holomorphicity of $\pi$ prevents holomorphic curves from escaping to $\partial X$).
\end{remark}

Though we will only need the cases $\alpha=1$ and $\alpha=\frac 12$ of the next result, we state it for general $\alpha>0$.

\begin{proposition}\label{productnbhd}
Let $X$ be a Liouville sector.
Every $\alpha$-defining function $I:\Nbd^Z\partial X\to\RR$ extends to a unique identification (valid over a cylindrical neighborhood of the respective boundaries):
\begin{equation}
(X,\lambda_X)=(F\times\CC_{\Re\geq 0},\lambda_F+\lambda_\CC^\alpha+df)
\end{equation}
in which $I=y$ is the imaginary part of the $\CC_{\Re\geq 0}$-coordinate, $(F,\lambda_F)$ is a Liouville manifold, and $f:F\times\CC_{\Re\geq 0}\to\RR$ satisfies the following properties:
\begin{itemize}
\item $f$ is supported inside $F_0\times\CC$ for some Liouville domain $F_0\subseteq F$.
\item $f$ coincides with some $f_{\pm\infty}:F\to\RR$ for $\left|I\right|$ sufficiently large.
\end{itemize}
\end{proposition}

\begin{proof}
There is a unique function $R:\Nbd^Z\partial X\to\RR$ satisfying $R|_{\partial X}=0$, $X_IR\equiv -1$, and $ZR=(1-\alpha)R$.
Indeed, the first two conditions define $R$ uniquely over $\Nbd\partial X$ as the function ``time it takes to hit $\partial X$ under the flow of $X_I$''.
Differentiating $\omega(X_R,X_I)=1$ with respect to $Z$ shows that this $R$ satisfies $ZR=(1-\alpha)R$ outside a compact set (here we use the fact that $R=0$ over $\partial X$; the identity $\sL_ZX_H=X_{ZH-H}$ is also helpful).
Extending $R$ to a cylindrical neighborhood of $\partial X$ by maintaining the property $ZR=(1-\alpha)R$ preserves the relation $X_IR\equiv-1$.

Now
\begin{equation}
R+iI:\Nbd^Z\partial X\to\CC_{\Re\geq 0}
\end{equation}
is a symplectic fibration since $\omega(X_R,X_I)$ is non-vanishing.
Since $\omega(X_R,X_I)$ is in fact constant, the induced symplectic connection is flat.
The vector fields $X_I$ and $X_R$ on $\Nbd^Z\partial X$ are horizontal with respect to this connection.
Let $F:=(I,R)^{-1}(0,0)$ as a symplectic manifold (compare Definition \ref{lsfiber}), so the symplectic connection provides a germ of a symplectomorphism $X=F\times\CC_{\Re\geq 0}$ near the respective boundaries.

We now compare Liouville vector fields.
Define $\lambda_F$ as the restriction of $\lambda_X$ to the fiber $(I,R)^{-1}(0,0)=F$.
We therefore have
\begin{equation}
\lambda_X=\lambda_F+\lambda_\CC^\alpha+df
\end{equation}
for some function $f:\Nbd\partial X\to\RR$ (indeed, $\lambda_X-(\lambda_F+\lambda_\CC^\alpha)$ is closed, and to check exactness it is, for topological reasons, enough to check on the fiber over $(0,0)$ where it vanishes by definition).
Since $(R+iI)_\ast Z_X=Z_\CC^\alpha$ outside a compact set (this is a restatement of $ZI=\alpha I$ and $ZR=(1-\alpha)R$ near infinity), we conclude that $X_f$ is purely vertical outside a compact set, which is equivalent to saying that $f$ is locally independent of the $\CC_{\Re\geq 0}$-coordinate outside a compact set.
We conclude that $f$ satisfies the desired properties.

Finally, note that the identification thus constructed over a neighborhood of the boundary automatically extends to a cylindrical neighborhood of the boundary by integrating the Liouville flow.
\end{proof}

\begin{definition}\label{cprojdef}
The notation $\pi:\Nbd^Z\partial X\to\CC_{\Re\geq 0}$ shall refer to the special case $\alpha=\frac 12$ of the projection to $\CC_{\Re\geq 0}$ from Proposition \ref{productnbhd}.
\end{definition}

\begin{example}
Consider the Liouville sector $\CC_{\Re\geq 0}$ equipped with its usual symplectic form $\omega_\CC=dx\wedge dy$ and radial Liouville vector field $Z_\CC=\frac 12(x\frac\partial{\partial x}+y\frac\partial{\partial y})$.
We may take $I=y$ and $R=x$, so $\pi:=R+iI=x+iy:\CC_{\Re\geq 0}\to\CC$ is simply the usual inclusion.
\end{example}

\begin{proposition}\label{deformexact}
Any Liouville sector $X$ may be deformed to make $\partial X$ exact.
In fact, this deformation is homotopically unique.
\end{proposition}

Note that the proof we give involves a nontrivial deformation at infinity.
It seems likely that it is not possible in general to achieve exactness through a compactly supported deformation once $\dim X\geq 4$.

\begin{proof}
Let $I:X\to\RR$ be a defining function, and consider the resulting product neighborhood decomposition from Proposition \ref{productnbhd}:
\begin{equation}
X=F\times T^\ast\RR_{\geq 0}
\end{equation}
so that $\lambda_X=\lambda_F+\lambda_{T^\ast\RR_{\geq 0}}+df$ over a neighborhood of $\partial X$.

Now we let $\varphi:\RR_{\geq 0}\to[0,1]$ be a cutoff function which is zero near $t=0$ and which equals one for $t\geq\varepsilon$.
We consider the deformation of Liouville forms
\begin{equation}
\lambda_F+\lambda_{T^\ast\RR_{\geq 0}}+a\cdot d((1-\varphi(t))f)+d(\varphi(t)f)
\end{equation}
for $a\in[0,1]$.
For $a=1$ this is our original Liouville form on $X$, and for $a=0$ this is a Liouville form on $X$ making $\partial X$ exact.
It thus suffices to check that this gives a deformation of Liouville manifolds-with-boundary.
In other words, we just need to check convexity at infinity.
A calculation shows that for sufficiently large Liouville domains $F_0\subseteq F$ and sufficiently large $N<\infty$, the Liouville vector field is outward pointing along the boundary of $F_0\times\{\left|I\right|\leq N\}$ for all $a\in[0,1]$, which is sufficient.
\end{proof}

\subsection{Convex completion}\label{horizcomplete}

For every Liouville sector $X$, we describe a Liouville manifold $\bar X$, called the \emph{convex completion} of $X$, along with an inclusion of Liouville sectors $X\hookrightarrow\bar X$.
We show that the convex completion of the Liouville sector $X$ associated to a sutured Liouville manifold $(\bar X,F_0)$ is exactly $\bar X$ and that the inclusion $X\hookrightarrow\bar X$ is the obvious one.

To define $\bar X$, begin with the identification $X=F\times\CC_{\Re\geq 0}$ near the boundary from Proposition \ref{productnbhd} with $\alpha=\frac 12$ (i.e.\ using the radial Liouville vector field on $\CC$), and glue on $F\times\CC_{\Re\leq 0}$ in the obvious manner:
\begin{equation}
\bar X:=X{\textstyle\bigcup\limits_{\Nbd\partial(F\times\CC_{\Re\geq 0})}}\Nbd(F\times\CC_{\Re\leq 0}).
\end{equation}
This defines $\bar X$ as a symplectic manifold.
To fix a primitive
\begin{equation}
\lambda_{\bar X}=\lambda_F+\lambda_\CC+df,
\end{equation}
we just need to extend $f:\Nbd\partial(F\times\CC_{\Re\geq 0})\to\RR$ to $f:\Nbd(F\times\CC_{\Re\leq 0})\to\RR$.
Lemma \ref{horizcompletionisconvex} defines a contractible space of extensions $f$ for which $\bar X$ is a Liouville manifold, thus completing the definition of $\bar X$.

\begin{remark}
The convex completion is `complete' in two unrelated senses: both the Liouville vector field $Z_{\bar X}$ and the Hamiltonian vector field $X_I$ are complete in the positive direction.
\end{remark}

\begin{remark}
One may also define the ``$\alpha$-completion'' of a Liouville sector for any $\alpha>0$ by considering instead the Liouville form $\lambda_\CC^\alpha$ on $\CC$.
For $\alpha\in(0,1)$, the $\alpha$-completion is convex (and is, up to deformation, the convex completion), corresponding to the fact that $Z_\CC^\alpha$ has an index zero critical point at the origin.
For $\alpha>1$, the Liouville vector field $Z_\CC^\alpha$ has index one at the origin, and hence the $\alpha$-completion is not convex in any reasonable sense (if, however, one has two Liouville sectors with the same $F$ and desires to glue them together via charts $F\times\CC_{\Re\geq 0}$ and $F\times\CC_{\Re\leq 0}$ from Proposition \ref{productnbhd}, it is natural to take $\alpha>1$ to define this gluing).
The borderline case $\alpha=1$ corresponds to attaching an end $T^\ast\RR_{\leq 0}\times F$ (for instance the $1$-completion of $T^\ast B^n$ is $T^\ast\RR^n$).
This $\alpha=1$ completion is an open Liouville sector in the sense of Remark \ref{opensector}.
\end{remark}

\begin{lemma}\label{horizcompletionisconvex}
The convex completion $\bar X$ equipped with the Liouville form described above is a Liouville manifold, provided the extension $f:F\times\CC_{\Re<\varepsilon}\to\RR$ is bounded uniformly in $C^\infty$ and satisfies the two bulleted properties from Proposition \ref{productnbhd}.
\end{lemma}

\begin{proof}
We study the restriction of the Liouville vector field $Z_{\bar X}=Z_F+Z_\CC-X_f$ to the boundary of
\begin{equation}\label{completionconvex}
X\cup\Bigl(F_0\times\CC_{\begin{smallmatrix}-N\leq\Re\leq 0\\\left|\Im\right|\leq M\end{smallmatrix}}\Bigr),
\end{equation}
where $F_0\subseteq F$ is a Liouville domain such that $f$ is supported inside $F_0\times\CC_{\Re<\varepsilon}$, and $M<\infty$ is such that $f=f_{\pm\infty}$ for $\left|\Im\right|\geq M$.

Over the piece of the boundary lying over the imaginary axis, the Liouville vector field equals $Z_F+Z_\CC$, which is tangent to the boundary.
The remaining compact part of the boundary is naturally divided into the following three pieces:
\begin{equation}
\partial F_0\times\CC_{\begin{smallmatrix}-N\leq\Re\leq 0\\\left|\Im\right|\leq M\end{smallmatrix}},\qquad
F_0\times\CC_{\begin{smallmatrix}-N\leq\Re\leq 0\\\left|\Im\right|=M\end{smallmatrix}},\qquad
F_0\times\CC_{\begin{smallmatrix}\Re=-N\\\left|\Im\right|\leq M\end{smallmatrix}}.
\end{equation}
Over the first piece, the Liouville vector field also equals $Z_F+Z_\CC$, which is outward pointing.
Over the second piece, the Liouville vector field equals $Z_F+Z_\CC-X_{f_{\pm\infty}}$, which is also outward pointing (note that $X_{f_{\pm\infty}}$ is tangent to the $F$ factor).
Over the third and final piece of the boundary, the Liouville vector field equals $Z_F+Z_\CC-X_f$ which is outward pointing iff $-\frac{\partial f}{\partial I}+N\geq 0$, which holds as long as $N$ is sufficiently large by the assumed $C^\infty$-boundedness of $f$.

It follows that for sufficiently large $N<\infty$, sufficiently large $M<\infty$, and sufficiently large $F_0\subseteq F$, the Liouville vector field is outward pointing along the boundary of \eqref{completionconvex} (other than that above the imaginary axis), and hence by removing from \eqref{completionconvex} an appropriately chosen open positive half of the symplectization of $\partial_\infty X$, we obtain a Liouville domain.
This is enough to imply the desired result.
\end{proof}

\begin{lemma}\label{suturedequivalence}
Every Liouville sector arises, up to deformation, from a unique (in the homotopical sense) sutured Liouville manifold.
Moreover, the convex completion of the Liouville sector $X$ associated to a sutured Liouville manifold $(\bar X,F_0)$ coincides with $\bar X$, and the inclusion $X\hookrightarrow\bar X$ is the obvious one.
\end{lemma}

\begin{proof}
We instead prove the statement for pairs $(\bar X_0,A)$ satisfying the hypotheses of Lemma \ref{suturedliouvillemanifold}.
This is equivalent in view of Lemma \ref{twosutures}.

Given a Liouville sector $X$, we consider \eqref{completionconvex}.
The vector field $X_I=-\frac\partial{\partial R}$ is not quite outward pointing along the boundary of \eqref{completionconvex}, since it is tangent to the part of the boundary over the strip $-N<\Re<0$.
However, this is easily fixed by shrinking $M$ and $F_0$ slightly as the real part gets more negative.
It follows that $X_I$ is outward pointing along the boundary of this perturbed version of \eqref{completionconvex}, and thus \eqref{completionconvex} is a Liouville sector deformation equivalent to $X$.
On the other hand, \eqref{completionconvex} arises from a pair $(\bar X_0,A)$ in which $\bar X_0\subseteq\bar X$ is a Liouville domain (whose completion is $\bar X$) and $A$ is the closure of (the small perturbation of) the part of the boundary of \eqref{completionconvex} not lying on $i\RR$.
We conclude that, up to deformation, every Liouville sector arises from a pair $(\bar X_0,A)$ satisfying the hypotheses of Lemma \ref{suturedliouvillemanifold}.

Now consider the construction above in the case that $X$ is presented as the Liouville sector \eqref{suturedliouvillemanifoldeq} associated to a given pair $(\bar X_0^0,A^0)$ satisfying the hypotheses of Lemma \ref{suturedliouvillemanifold}.
We may take $I$ on $\partial X$ to satisfy $ZI=\frac 12I$ everywhere, which implies $\frac{\partial f}{\partial R}\equiv 0$ (this is a calculation), and we choose the unique extension of $f$ to $F\times\CC_{\Re\leq\varepsilon}$ maintaining this property.
In this case, the deformation from $X$ to \eqref{completionconvex} described above corresponds canonically to a deformation of pairs $(\bar X_0^t,A^t)$, where $\bar X_0^t\subseteq\bar X$ are Liouville domains (with completion $\bar X$) and $A^t$ are (the images inside $\partial\bar X_0^t$ of) $\partial_\infty(F\times\CC_{\Re\leq 0},\lambda_F+\lambda_\CC+df)\subseteq\partial_\infty\bar X$.
We conclude that every Liouville sector arises from a homotopically unique pair $(\bar X_0^0,A^0)$ satisfying the hypotheses of Lemma \ref{suturedliouvillemanifold}.
\end{proof}

\subsection{Stops and Liouville sectors}\label{withstops}

The notion of a Liouville sector is essentially equivalent to Sylvan's notion of a stop.

\begin{definition}[Sylvan \cite{sylvanthesis}]
A \emph{stop} on a Liouville manifold $X$ is a map $\sigma:F\times\CC_{\Re\leq\varepsilon}\to X$
which is a proper, codimension zero embedding, where $F$ is a Liouville manifold, satisfying $\sigma^\ast\lambda_X=\lambda_F+\lambda_\CC+df$ for some compactly supported $f$.
Here $\CC_{\Re\leq\varepsilon}$ has the standard symplectic form $dx\,dy$ and standard radial Liouville vector field $\frac 12(x\frac\partial{\partial x}+y\frac\partial{\partial y})$.
\end{definition}

If $(X,\sigma)$ is a Liouville manifold with a stop, then there is a Liouville sector $X\setminus\sigma$ with exact boundary obtained by removing $\sigma(F\times\CC_{\Re<0})$ from $X$ (the ``imaginary part'' function $y$ from the $\CC$ factor gives a $\frac 12$-defining function).
Conversely, let $X$ be a Liouville sector with exact boundary (recall that every Liouville sector may be deformed so it becomes exact by Proposition \ref{deformexact}).
Its convex completion $\bar X$ comes with an embedding of $F\times\CC_{\Re\leq\varepsilon}$ on which the Liouville form is given by $\lambda_F+\lambda_\CC+df$, for $f:F\times\CC_{\Re\leq\varepsilon}\to\RR$ as in \S\ref{horizcomplete}.
The difference $f_\infty-f_{-\infty}$ is given by $\int_C\lambda$ as before (see \eqref{Cintegraldifference}); in particular $f_\infty=f_{-\infty}$ since $\partial X$ is exact.
Thus if we replace $\lambda_F$ with $\lambda_F+df_\infty$, we may take $f$ to have compact support.
This defines a stop $\sigma$ on $\bar X$ such that $X=\bar X\setminus\sigma$.

\subsection{Cutoff Reeb dynamics}\label{reebdynamicsboundary}

To understand wrapping on Liouville sectors, we need to have precise control on the Reeb vector field near the boundary.
We now develop the necessary understanding, in the general context of a contact manifold $Y$ with convex boundary.

We are interested in what we shall call \emph{cutoff contact vector fields}, namely those contact vector fields on $Y$ which vanish on $\partial Y$ (this is, of course, the Lie algebra of the group of contactomorphisms of $Y$ fixing $\partial Y$ pointwise).
For any function $f:Y\to\RR_{\geq 0}$ vanishing transversely precisely on $\partial Y$, the space of cutoff contact vector fields corresponds to the space of contact Hamiltonians of the form $H = f^2G$ for smooth sections $G:Y\to TY/\xi$.
A \emph{cutoff Reeb vector field} shall mean a cutoff contact vector field whose contact Hamiltonian $H = f^2G$ satisfies $G>0$ over all of $Y$ (including $\partial Y$).

\begin{lemma}\label{compactcutoff}
For every cutoff Reeb vector field $\R$ on a contact manifold $Y$, there exists a compact subset of the interior of $Y$ which intersects all periodic orbits of $\R$.
\end{lemma}

\begin{proof}
Work on the symplectization $X=SY$, in which convexity of $\partial Y$ means there is a linear function $I$ with $X_I$ outward pointing.
Let $H$ be the linear Hamiltonian generating the cutoff Reeb vector field $\R$, meaning that $X_H$ is the tautological lift of $\R$ from $Y$ to $X=SY$.
Since $H$ vanishes to order two along $\partial X$ with positive second derivative, we conclude that $X_IH$ vanishes to first order along $\partial X$ and is negative just inside.
Since $X_IH$ is linear, we conclude that $X_IH<0$ over $(\Nbd^Z\partial X)\setminus\partial X$; equivalently,
\begin{equation}\label{XHIlowerbound}
X_HI>0\quad\text{over }(\Nbd^Z\partial X)\setminus\partial X.
\end{equation}
In particular, there can be no periodic orbits of $X_H$ entirely contained in $(\Nbd^Z\partial X)\setminus\partial X$, which
implies the desired result.
\end{proof}

For certain purposes, we will need cutoff Reeb vector fields $\R$ whose dynamics near the boundary satisfy even stronger requirements (for example, we will want there to exist a compact subset of the interior of $Y$ which \emph{contains} all periodic orbits of $\R$).
We are not able to show that an arbitrary cutoff Reeb vector field conforms to such strong requirements; instead, we will simply write down a sufficient supply of such $\R$.

We proceed to write down a canonical (up to contractible choice) family of cutoff Reeb vector fields near $\partial Y$ with excellent dynamics.
Let us first choose a contact vector field $V$ outward pointing along $\partial Y$ (this is a convex, and hence contractible, choice).
This determines unique coordinates $\partial Y\times\RR_{t\geq 0}\to Y$ under which $V=-\frac\partial{\partial t}$, defined near $\partial Y$.
The choice of $V$ also divides $\partial Y$ into two pieces $(\partial Y)_\pm$ namely where $V$ is positively/negatively transverse to $\xi$.
They meet along the locus 
\begin{equation}
    \Gamma_{\partial Y} :=  \{V\in\xi\}
\end{equation}
(called the `dividing set') which is a \emph{transversely cut out submanifold} of $\partial Y$ for \emph{any} choice of $V$ 
(this is a standard easy fact in the study of convex hypersurfaces, see e.g.\ \cite[I.3.B(1)]{girouxthesis} for a proof in general and \cite{geigesbook} for the three-dimensional case).

In a neighborhood of (any compact subset of) $(\partial Y)_-$, there is a unique contact form $\alpha$ with $\alpha(V)=1$, namely
\begin{equation}\label{invtform}
\alpha=\lambda+dt
\end{equation}
where $\lambda$ is a Liouville form on $(\partial Y)_-$.
Now the contact vector field associated to the contact Hamiltonian $M(t)$ (that is, the Reeb vector field of $M(t)^{-1}\alpha$) is given by
\begin{equation}
V_M=M'(t)Z_\lambda+M(t)\frac\partial{\partial t}.
\end{equation}
We now observe that as long as $M(t)>0$ and $M'(t)>0$, this vector field has excellent dynamics.
Namely: the function $t$ is monotonically increasing along trajectories, and all backwards trajectories converge to the locus $\{\lambda=0\}=\{\xi\subseteq T\partial Y\}$, which is a compact subset of $(\partial Y)_-$.
A similar story applies over $(\partial Y)_+$, except that the vector field is attracting instead of repelling.

The locus $\Gamma_{\partial Y} =  \{V\in\xi\}\subseteq\partial Y$ is more interesting; our first task is to write down an explicit contact form in a neighborhood of $\Gamma_{\partial Y}$.
Recall that the characteristic foliation of the hypersurface $\partial Y$ inside the contact manifold $Y$ is by definition the kernel of the restriction of $d\alpha$ to $\xi\cap T\partial Y$ for any contact form $\alpha$ on $Y$.
There is a canonical projection $\pi:\partial Y\to \Gamma_{\partial Y}$
defined in a neighborhood of $\Gamma_{\partial Y}$ by following the characteristic foliation, which is transverse to $\Gamma_{\partial Y}$.
Furthermore, we have $\xi\cap T\partial Y=(d\pi)^{-1}(\xi\cap T\Gamma_{\partial Y})$
(to see this, consider the Lie derivative with respect to any vector field tangent to the characteristic foliation of $\partial Y$ of the restriction to $\partial Y$ of any contact form on $Y$).
Thus $\xi$ is determined by the projection $\pi$, the contact structure $\xi\cap T\Gamma_{\partial Y}$
on $\Gamma_{\partial Y}$, 
and the section 
$[-V] = [\frac\partial{\partial t}]\in TY|_{\partial Y}/\xi=T\partial Y/(\xi\cap T\partial Y)=\pi^\ast(T\Gamma_{\partial Y}/(\xi\cap T\Gamma_{\partial Y}))$
over $\partial Y$.
More explicitly, choose a contact form
$\mu$ for the contact structure 
$\xi\cap T\Gamma_{\partial Y}$ on $\Gamma_{\partial Y}$ 
(a contractible choice).
With respect to this choice, the section $[\frac\partial{\partial t}]$ of $TY|_{\partial Y}/\xi=\pi^\ast(T\Gamma_{\partial Y}/(\xi\cap T\Gamma_{\partial Y}))\xrightarrow\mu\RR$ is just a
function on $\partial Y$ vanishing transversely on $\Gamma_{\partial Y}$.
Denoting
this function by $u$, we have local coordinates on $Y$ near $\Gamma_{\partial Y}$
given
by
\begin{equation}\label{YcoordsnearGamma}
    \Gamma_{\partial Y}\times\RR_{\left|u\right|\leq\varepsilon}\times\RR_{t\geq 0}
\end{equation}
such that the contact vector field $V=-\frac\partial{\partial t}$ and the contact form is given by
\begin{equation}
\mu+u\,dt.
\end{equation}
We wish to consider instead the scaled contact form
\begin{equation}\label{keycontactform}
\alpha:=\psi(u)^{-1}[\mu+u\,dt]
\end{equation}
for a function $\psi:\RR\to\RR_{>0}$ satisfying
\begin{alignat}{2}
\psi(u)&=\left|u\right|&\qquad&\text{for }\left|u\right|\geq\frac\varepsilon 2\\
u\psi'(u)&>0&\qquad&\text{for }u\ne 0
\end{alignat}
(the choice of $\psi$ is convex and hence contractible).
Note that for $u\geq\frac\varepsilon 2$ (resp.\ $u\leq-\frac\varepsilon 2$), this scaled contact form $\alpha$ satisfies $\alpha(V)=1$ (resp.\ $\alpha(V)=-1$) and thus is of the form \eqref{invtform} (resp.\ its negative).
Now the contact vector field associated to a contact Hamiltonian $M(t)$ is given by
\begin{equation}\label{complicatedreeb}
V_M=\psi'(u)M(t)\frac\partial{\partial t}-\psi(u)M'(t)\frac\partial{\partial u}+[\psi(u)-u\psi'(u)]M(t)\R_\mu.
\end{equation}
As before, given $M(t)>0$ and $M'(t)>0$, this vector field has excellent dynamics.
Specifically, $u$ is monotonically decreasing along trajectories; furthermore, $t$ is increasing for $u>0$ and decreasing for $u<0$.
The resulting dynamics are illustrated in Figure \ref{cutoffdynamics}.
Note that these dynamics prohibit any periodic orbit of the Reeb vector field from intersecting this neighborhood of $\partial Y$.

\begin{figure}[hbt]
\centering
\includegraphics{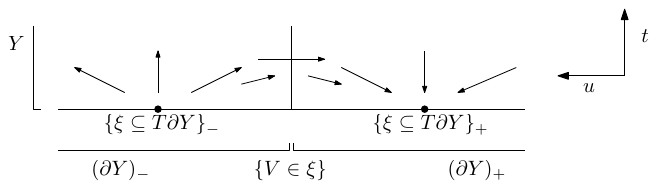}
\caption{Dynamics of the cutoff Reeb vector field near $\partial Y$.}\label{cutoffdynamics}
\end{figure}

We summarize the main properties of the above construction as follows.
Let us call a function $M:\RR_{\geq 0}\to\RR_{\geq 0}$ \emph{admissible} if $M(0)=0$, $M'(0)=0$, $M''(0)>0$, and $M'(t)>0$ for $t>0$.
(Usually, $M$ will only be defined on a connected interval inside $\RR_{\geq 0}$ containing zero).

\begin{proposition}\label{cutofflemma}
Let $Y$ be a contact manifold with convex boundary.
There exists a canonically defined contractible family of pairs consisting of a choice of coordinates $\partial Y\times\RR_{t\geq 0}\to Y$ near $\partial Y$ and a $\frac\partial{\partial t}$-invariant contact form $\alpha$, such that for any admissible function $M:\RR_{\geq 0}\to\RR_{\geq 0}$, the dynamics of the associated cutoff Reeb vector field defined over $\Nbd\partial Y$ satisfy the following property.
For any trajectory $\gamma:\RR\to\partial Y\times\RR_{t\geq 0}$, we have:
\begin{itemize}
\item If $dt(\gamma'(\tau_0))\geq 0$, then $dt(\gamma'(\tau))>0$ for all $\tau<\tau_0$.
\item If $dt(\gamma'(\tau_0))\leq 0$, then $dt(\gamma'(\tau))<0$ for all $\tau>\tau_0$.
\end{itemize}
In particular, for all sufficiently small $\delta>0$, no trajectory enters the region $\partial Y\times\RR_{0\leq t\leq\delta}$ and then exits.\qed
\end{proposition}

We now consider the cutoff Reeb dynamics on a very special class of contact manifolds with convex boundary, namely contactizations of Liouville domains.
Recall that for a Liouville domain $(F_0,\lambda)$, its \emph{contactization} is the contact manifold-with-boundary given by (smoothing the corners of) $([-1,1]\times F_0,dt+\lambda)$.
The contactization has convex boundary, as demonstrated by the contact vector field $t\frac\partial{\partial t}+Z_\lambda$.

\begin{lemma}\label{cylinderstopped}
Every contactization can be deformed through contact manifolds with convex boundary so as to admit a cutoff Reeb vector field with no periodic orbits.
\end{lemma}

Recall from Lemma \ref{deformboundary} that for any Liouville sector $X$, deformations of $\partial_\infty X$ always lift to deformations of $X$.

\begin{proof}[Proof \#1]
We consider a cutoff Reeb vector field $V_\varphi$ induced by a Hamiltonian $\varphi$ defined on (a smoothing of) $[-1,1]\times F_0$.
A calculation shows that $dt(V_\varphi)=\varphi-Z_\lambda\varphi$.
Hence to ensure that $V_\varphi$ has no periodic orbits, it is enough to choose $\varphi$ such that $\varphi-Z_\lambda\varphi>0$ on the interior of our contact manifold.
Since $\varphi>0$ on this region, this inequality may equivalently be written as $Z_\lambda\log\varphi<1$.
It is straightforward to define such a positive cutoff Hamiltonian function $\varphi$ on (a smoothing of) $[-1,1]\times F_0$ (in fact, we can even achieve the much stronger property $Z_\lambda\varphi<0$).
\end{proof}

\begin{proof}[Proof \#2]
The Liouville sector $\CC_{\Re\geq 0}\times F$ has a linear Hamiltonian $(\Re)^2$, whose Hamiltonian vector field $x\frac\partial{\partial y}$ clearly has no closed orbits.
It is thus enough to show that $\partial_\infty(\CC_{\Re\geq 0}\times F)$ is deformation equivalent to $[0,1]\times F_0$ through contact manifolds with convex boundary.
This follows from Lemma \ref{openbookconvex}.
\end{proof}

\begin{example}\label{ballfullystopped}
The boundary at infinity of $T^\ast B^n$ is contactomorphic (up to deformation) to the contactization $[0,1]\times D^\ast S^{n-1}$.
\end{example}

\begin{corollary}
If $(\bar X,F_0)$ is a sutured Liouville manifold, then any exact deformation of the Liouville form $\lambda|_{F_0}$ can be realized by a deformation of $(F_0,\lambda)$ supported in an arbitrarily small neighborhood of $F_0$.
\end{corollary}

\begin{proof}
Let $F_0\subseteq F_0^+\subseteq\bar X$ be a slightly larger Liouville domain, and consider the embedding $F_0^+\times\RR_{\left|t\right|\leq\varepsilon}\subseteq\bar X$ in which $\lambda=dt+\lambda|_{F_0^+}$.
The first proof of Lemma \ref{cylinderstopped} provides a cutoff contact Hamiltonian $\varphi$ defined on a smoothing of $F_0^+\times\RR_{\left|t\right|\leq\varepsilon}$ inside $\bar X$ such that the Reeb vector field of $\varphi^{-1}\lambda$ pairs positively with $dt$.
Note that we may further require that $\varphi\equiv 1$ over a neighborhood of $F_0$.
The Reeb vector field of $\varphi^{-1}\lambda$ thus provides an embedding of $F_0\times\RR_\tau$ into $\bar X$ under which the pullback of $\varphi^{-1}\lambda$ is given by $\lambda|_{F_0}+d\tau$.
Now deforming $F_0$ to the graph of a function $\tau:F_0\to\RR$ gives the desired result.
\end{proof}

\begin{conjecture}
If $\partial_\infty X$ and $\partial_\infty X'$ both admit cutoff Reeb vector fields with no periodic orbits, then so does $\partial_\infty(X\times X')$.
\end{conjecture}

This conjecture is rather strong.
It may be more reasonable to expect that taking product preserves the existence of a sequence of cutoff contact Hamiltonians $\varphi_1,\varphi_2,\ldots$ (with $f^{-2}\varphi_i>0$ bounded uniformly away from zero) and real numbers $a_1,a_2,\ldots\to\infty$ such that $V_{\varphi_i}$ has no periodic orbit of action $\leq a_i$.

\begin{conjecture}\label{stoppedmeanscontactization}
If a contact manifold with convex boundary $Y$ admits a cutoff Reeb vector field with no periodic orbits, then $Y$ is deformation equivalent to a contactization.
\end{conjecture}

Note that this conjecture (which is the converse of Lemma \ref{cylinderstopped}) is even stronger than the Weinstein conjecture that every Reeb vector field on a closed contact manifold has a periodic orbit \cite{weinsteinconjecture}.
The Weinstein conjecture was proven by Taubes in dimension three \cite{taubesweinstein}, and Conjecture \ref{stoppedmeanscontactization} has also been proven in dimension three by Colin--Honda \cite[Corollary 4.7]{colinhonda}.

\subsection{Holomorphic curves in Liouville sectors}\label{compactnessdiscussion}

When proving compactness results for holomorphic curves in Liouville sectors, there are two main concerns to address, namely curves crossing $\partial X$ and curves escaping to infinity.
We now discuss these in turn, with the goal of giving a general preview of the relevant techniques before they are applied in specific settings in \S\S\ref{secwrapped},\ref{secsymplecticcohomology},\ref{secopenclosed}.
We will usually omit the adjectives ``$\omega$-compatible'' and ``cylindrical near infinity'', though they should always be understood even when unwritten.

\subsubsection{Preventing crossing \texorpdfstring{$\partial X$}{dX}}\label{bdryescape}

The key to preventing holomorphic curves from crossing $\partial X$ is the
function $\pi:\Nbd^Z\partial X\to\CC_{\Re\geq 0}$ from Definition \ref{cprojdef}.
Given such a $\pi$, there is an
abundance of cylindrical almost complex structures $J$ for which $\pi$ is
$J$-holomorphic, and the space of such $J$ is contractible
(indeed, $J$ makes $\pi$ holomorphic iff it preserves the decomposition $TX=(\ker d\pi)\oplus(\ker d\pi)^{\perp_\omega}$ agreeing with $J_\CC$ on the latter $(\ker d\pi)^{\perp_\omega}=\pi^\ast T\CC$, and these conditions are preserved under the Liouville flow near infinity, compare Lemma \ref{productnbhd}).
For the
corresponding $J$-holomorphic curves, it is well understood that
$\pi^{-1}(\CC_{\left|\Re\right|<\varepsilon})$ acts as a ``barrier'', a fact
we recall here:

\begin{lemma}\label{pibarrierlemma}
Suppose $\pi: \Nbd^Z \partial X \to \CC_{\Re\geq 0}$ is $J$-holomorphic, and let $u:\Sigma\to X$ be a $J$-holomorphic map.
Suppose that $u^{-1}(\pi^{-1}(\CC_{\left|\Re\right|\leq\varepsilon}))$ is compact and disjoint from $\partial\Sigma$.
Then $u^{-1}(\pi^{-1}(\CC_{\left|\Re\right|\leq\varepsilon}))$ is empty, except possibly for closed components of $\Sigma$ over which $u$ is constant.
\end{lemma}

In the situations of interest for us, $\Sigma$ never has any closed components, so we simply conclude that $u(\Sigma)$ is disjoint from $\pi^{-1}(\CC_{\left|\Re\right|<\varepsilon})$.

\begin{proof}
We consider the holomorphic map $\pi\circ u:u^{-1}(\pi^{-1}(\CC_{\left|\Re\right|<\varepsilon}))\to\CC_{\left|\Re\right|<\varepsilon}$.
By compactness, its image $\im:=(\pi\circ u)(u^{-1}(\pi^{-1}(\CC_{\left|\Re\right|<\varepsilon})))$ is a closed and bounded subset of $\CC_{\left|\Re\right|<\varepsilon}$.
Now the open mapping theorem from classical single-variable complex analysis implies that $\im$ is open (which implies the desired result $\im=\varnothing$, as the only closed, open, bounded subset of $\CC_{\left|\Re\right|<\varepsilon}$ is the empty set), except for possibly when $\pi\circ u$ is locally constant.
By analytic continuation, if $\pi\circ u$ is locally constant at some point of $\Sigma$, it is constant on the entire connected component of $\Sigma$ containing that point.
In particular, this component of $\Sigma$ is mapped entirely into $\CC_{\left|\Re\right|<\varepsilon}$, so the hypotheses imply this is a closed component of $\Sigma$.
Now $u$ is constant on this component since $X$ is exact.
\end{proof}

An alternative (though in some sense related) proof of Lemma \ref{pibarrierlemma} proceeds by integrating $\pi^\ast(\varphi(\Re z)\cdot\omega_\CC)$ for some $\varphi:\RR\to[0,1]$ supported inside $[-\varepsilon,\varepsilon]$ over the holomorphic map in question and using the fact that $H^2_\dR(\CC_{\left|\Re\right|\leq\varepsilon},\partial\CC_{\left|\Re\right|\leq\varepsilon})=0$.
This latter argument is important since it generalizes to Floer trajectories with respect to Hamiltonians $H$ whose restriction to $\pi^{-1}(\CC_{\left|\Re\right|\leq\varepsilon})$ coincides with $\Re\pi$.
Namely, one can show that such Floer trajectories can only pass through $\partial X$ ``in the correct direction''; see Lemma \ref{scconfinecurves} for a precise statement.
It could be an important technical advance to enlarge the class of Hamiltonians defined near $\partial X$ for which similar confinement results for holomorphic curves can be proven.

\subsubsection{Preventing escape to infinity using monotonicity}\label{inftyescape}

The key to preventing holomorphic curves from escaping to infinity is
\emph{monotonicity inequalities}, which, given a holomorphic curve $u:\Sigma\to X$
passing through a point $p$, provide a lower bound
\begin{equation}\label{monotonicitylowerbound}
\int_{u^{-1}(B_\varepsilon(p))}u^*\omega\geq\const\cdot\varepsilon^2
\end{equation}
(assuming $u(\partial\Sigma)$ lies outside the ball $B_\varepsilon(p)$), for all sufficiently small $\varepsilon>0$ and some positive constant, both depending on the local geometry of $X$ (as an almost K\"ahler manifold) near $p$.
For precise statements of the monotonicity inequalities which we will use, we refer the reader to Sikorav \cite[Propositions 4.3.1 and 4.7.2]{sikorav}.

To use monotonicity inequalities to effectively control holomorphic curves, we need to ensure that the target almost K\"ahler manifold has \emph{(globally) bounded geometry} in the following sense:

\begin{definition}\label{bddgeodefn}
An almost K\"ahler manifold $(X,\omega,J)$ is said to have \emph{bounded geometry} iff there exist $\varepsilon>0$, $M_0,M_1,\ldots<\infty$, and a collection of coordinate charts $\varphi_\alpha:B(1)\to X$ such that $X=\bigcup_\alpha\varphi_\alpha(B(\frac 12))$ and
\begin{align}
\left\|\varphi_\alpha^\ast\omega\right\|_{C^r}&\leq M_r,\\
\left\|\varphi_\alpha^\ast J\right\|_{C^r}&\leq M_r,\\
(\varphi_\alpha^\ast\omega)(v,(\varphi_\alpha^\ast J)v)&\geq\varepsilon\cdot g_\std(v,v).
\end{align}
Note that these conditions imply, in particular, that $X$ is complete when equipped with the metric $\omega(\cdot,J\cdot)$.\footnote{A common equivalent formulation of the notion of bounded geometry is to ask that the curvature and all its derivatives be bounded, the metric be complete, and the injectivity radius be bounded from below.  Bounded geometry in the sense of Definition \ref{bddgeodefn} clearly implies bounded geometry in this sense, but the converse is, while standard and well known, not completely trivial; for this reason, it is logically simpler to employ Definition \ref{bddgeodefn} as we have formulated it here.}
\end{definition}

To apply monotonicity arguments to holomorphic curves with Lagrangian boundary conditions, one further needs to know that the Lagrangians in question also have bounded (extrinsic) geometry; precisely, this means there exist charts as in Definition \ref{bddgeodefn} such that each $\varphi_\alpha^{-1}(L)$ is either empty or the intersection of $B(1)$ with a linear subspace.

As we now observe in Lemmas \ref{bddgeo} and \ref{domainacsbddgeo} below, for Liouville sectors, almost
complex structures of bounded geometry are easy to come by: any cylindrical $J$ has this property,
and such $J$ will suffice for all of our arguments.  Moreoever, cylindrical
Lagrangians have bounded geometry with respect to such $J$, which is the only
case we will need.

\begin{lemma}\label{bddgeo}
Let $X$ be a Liouville sector, and let $J$ be cylindrical.  Then $(X,\omega,J)$ has bounded geometry.  If $L$ is cylindrical at infinity, then the same is true for $(X,\omega,J;L)$.
\end{lemma}

\begin{proof}
Obviously $(X,\omega,J)$ has bounded geometry over any compact subset of $X$.
By scaling via the Liouville flow, we observe that the geometry of $(X,\omega,J)$ near a point $p$ close to infinity is the same as the geometry of $(X,\lambda\omega,J)$ near a point $q$ contained in a compact subset of $X$, for some real number $\lambda\geq 1$.
Now as $\lambda\to\infty$, the structure $(X,\lambda\omega,J)$ near $q$ just converges to the tangent space of $X$ at $q$ equipped with its linear almost K\"ahler structure.
In particular, the family of all scalings by $\lambda\geq 1$ has uniformly bounded geometry, which thus implies the desired result.
The same reasoning applies to cylindrical Lagrangian submanifolds.
\end{proof}

The ``family version'' of Lemma \ref{bddgeo} relevant for studying with holomorphic curves with respect to a domain dependent almost complex structure is the following.

\begin{lemma}\label{domainacsbddgeo}
Let $X$ be a Liouville sector, and let $J:D^2\to\J(X)$ be a family of cylindrical almost complex structures (meaning cylindrical outside a uniformly chosen compact subset of $X$).
Then $(D^2\times X,\omega_{D^2}+\omega_X,j_{D^2}\oplus J)$ has bounded geometry.
If $L\subseteq X$ is a cylindrical Lagrangian, then $\partial D^2\times L$ has bounded geometry inside $D^2\times X$.
\end{lemma}

Note that we make no claims about geometric boundedness of ``moving Lagrangian boundary conditions'' (outside of the case that they are fixed at infinity).

\begin{proof}
We follow the proof of Lemma \ref{bddgeo}.
Geometric boundedness holds trivially over any compact subset, and to show geometric boundedness everywhere, we just need to understand $(D^2\times X,\omega_{D^2}+\lambda\omega_X,j_{D^2}\oplus J)$ over a compact subset of $D^2\times X$ but for arbitrarily large $\lambda\geq 1$.
Again, the limit as $\lambda\to\infty$ is nice, so we are done.
\end{proof}

To derive \emph{a priori} bounds on holomorphic curves using monotonicity arguments, one needs certain additional conditions near the punctures/ends.  For example, at boundary punctures, one needs the Lagrangian boundary conditions on either side to be uniformly separated at infinity, see Proposition \ref{wcompactness}.
For Floer cylinders $\RR\times S^1$ with an
$\RR$-invariant Floer equation with Hamiltonian term $H_t$, it is enough to know that the integrated flow $\Phi:X\to X$ of $X_{H_t}$ over $S^1$ enjoys a lower bound $d(x,\Phi(x))\geq\varepsilon>0$ near infinity, see Proposition \ref{compactness} (this is, of course, related to the Lagrangian boundary conditions setting since Floer cylinders for $H_t$ in $X$ are in bijection with holomorphic strips in $X^- \times X$ between $\Delta_X$ and $\Gamma_{\Phi}$ with respect to a suitable almost complex structure).

Monotonicity inequalities can also be used to derive \emph{a priori} bounds on Floer continuation maps (i.e.\ with varying Hamiltonian term), however this is considerably more subtle and is due to recent work of Groman \cite{groman}, requiring careful choice of ``dissipative'' Floer data near infinity.  We will recall the arguments relevant for our work in Proposition \ref{compactness}.

\subsubsection{Preventing escape to infinity using pseudo-convexity}\label{productpseudoconvexinterpolation}

The maximum principle for holomorphic curves with respect to almost complex structures of contact type plays an important role in Floer theory on Liouville manifolds (see e.g.\ \cite{seidelbiased,abouzaidseidel,rittertqft}).
We will make use of similar arguments in the setting of Liouville sectors, so we recall below the basic definitions and result.

\begin{definition}
An $\omega$-compatible almost complex structure $J$ is said to be of \emph{contact type} with respect to a positive linear Hamiltonian $r:X\to\RR_{>0}$ iff $dr=\lambda\circ J$.
This condition is equivalent to requiring that $J(Z_\lambda)=X_r$ and that $J$ stabilizes the contact distribution $\xi=\ker\lambda$ of the level sets of $r$.
\end{definition}

There is an abundance of almost complex structures of contact type for any given $r$ (and moreover they form a contractible space).
Note, however, that we do not claim that on a Liouville sector $X$ we can find almost complex structures of contact type which make $\pi$ holomorphic as in \S\ref{bdryescape}.
This makes arguments using the maximum principle on Liouville sectors somewhat subtle, since the best we can do is choose almost complex structures which are of contact type over a large compact subset of the interior of $\partial_\infty X$.
It could be an important technical advance to identify a reasonable class of almost complex structures on Liouville sectors which make $\pi$ holomorphic and with respect to which one can prove a maximum principle.

\begin{lemma}[Folklore]\label{maxprinciple}
Let $u:\Sigma\to X$ be $J$-holomorphic and satisfy $u^\ast\lambda|_{\partial\Sigma}\leq 0$ over $u^{-1}(\{r>a\})$, where $J$ is of contact type with respect to $r$ over $\{r>a\}$.
Then $u$ is locally constant over $u^{-1}(\{r>a\})$.
\end{lemma}

Note that $u^\ast\lambda|_{\partial\Sigma}=0$ whenever $u(\partial\Sigma)\subseteq L$ for Lagrangian $L\subseteq X$ which is cylindrical over $\{r>a\}$.
More generally, the hypothesis $u^\ast\lambda|_{\partial\Sigma}\leq 0$ is satisfied whenever $u$ satisfies ``non-negatively moving cylindrical Lagrangian boundary conditions'' over $\{r>a\}$ (moving cylindrical Lagrangian boundary conditions $\{L_t\}_{t\in\partial\Sigma}$ are said to be moving non-negatively when $\lambda(\partial_tL_t)\leq 0$; the sign is due to the fact that $\partial\Sigma$ is oriented counterclockwise while we judge positivity/negativity of moving Lagrangian boundary conditions in the clockwise direction).

\begin{proof}
This proof is due to Abouzaid--Seidel \cite[Lemma 7.2]{abouzaidseidel}.
Let $\kappa:\RR\to\RR_{\geq 0}$ be any smooth function satisfying $\kappa'(r)\geq 0$ and vanishing for $r\leq a$.
Stokes' theorem gives
\begin{equation}
0\leq\int_\Sigma\kappa(r(u))\cdot u^\ast\omega=\int_{\partial\Sigma}\kappa(r(u))\cdot u^\ast\lambda-\int_\Sigma\kappa'(r(u))\cdot u^\ast(dr\wedge\lambda)\leq 0,
\end{equation}
where the second inequality holds because $dr\wedge\lambda$ is $\geq 0$ on complex lines (since $J$ is of contact type) and $u^\ast\lambda|_{\partial\Sigma}\leq 0$ (by hypothesis).
The result follows immediately.
\end{proof}

\section{Wrapped Fukaya category of Liouville sectors}\label{secwrapped}

For any Liouville sector $X$, we define an $\ainf$-category $\W(X)$ called its wrapped Fukaya category.
For an inclusion of Liouville sectors $X\hookrightarrow X'$, we define an $\ainf$-functor $\W(X)\to\W(X')$;
more generally, for a diagram of Liouville sectors $\{X_\sigma\}_{\sigma\in\Sigma}$ indexed by a finite poset $\Sigma$, we define a diagram of $\ainf$-categories $\{\W(X_\sigma)\}_{\sigma\in\Sigma}$.
In fact, we also define a model of $\W(X)$ which is a strict functor from all Liouville sectors to $\ainf$-categories.

In all these definitions, it is also possible to consider only the full subcategories spanned by chosen collections of Lagrangians
(provided that, in the latter cases, the Lagrangians chosen for a given $X$ are also chosen for any $X'$ into which $X$ is included).
Since these $\ainf$-categories and $\ainf$-functors are chain level objects, they depend on a number of choices (almost complex structures, etc.), though they are well-defined up to quasi-equivalence.
In particular, the wrapped Floer cohomology groups $HW^\bullet(L_0,L_1)$, their product $HW^\bullet(L_0,L_1)\otimes HW^\bullet(L_1,L_2)\to HW^\bullet(L_0,L_2)$, and their pushforward $HW^\bullet(L_0,L_1)_X\to HW^\bullet(L_0,L_1)_{X'}$ are all well-defined.

Officially, we work with $\ZZ$ coefficients and $\ZZ/2$-grading.
On the other hand, issues of coefficients/gradings are mostly orthogonal to the main point of our discussion, and a much more general setup is certainly possible in this regard.

To define $\W$, we adopt the method due to Abouzaid--Seidel \cite{abouzaidseidelunpublished} whereby one first defines a directed $\ainf$-category $\OO$ together with a collection $C$ of ``continuation morphisms'' in $\OO$, and then defines $\W:=\OO[C^{-1}]$ as the localization of $\OO$ at $C$.
An advantage of this definition is that it is very efficient in terms of the complexity of the Floer theoretic input (e.g.\ it does not require the construction of coherent systems of higher homotopies between chain level compositions of continuation maps/elements).
The key result in making this approach work is Lemma \ref{wrappingcalculatesW} (due to Abouzaid--Seidel) which asserts that the morphism spaces in the category $\W$ defined by localization are indeed isomorphic to the wrapped Floer cohomology groups defined by wrapping Lagrangians geometrically.

When doing Floer theory on Liouville sectors, the following convention is convenient, and will be in effect for the remainder of this section.
It is harmless because we will show that trivial inclusions of Liouville sectors induce quasi-isomorphisms/equivalences on all the invariants we define.

\begin{convention}\label{choosingpi}
By ``Liouville sector'' we will mean ``Liouville sector equipped with a choice of $\pi:\Nbd^Z\partial X\to\CC$ as in Definition \ref{cprojdef}''.
Inclusions of Liouville sectors $i:X\hookrightarrow X'$ are required to satisfy either $i(X)\cap\partial X'=\varnothing$ or $i(X)=X'$, and in the latter case we require $\pi=\pi'\circ i$.
\end{convention}

\subsection{\texorpdfstring{$\ainf$}{A\_infinity}-categories}

We work throughout with cohomologically unital $\ZZ/2$-graded $\ainf$-categories over $\ZZ$, with cofibrant morphism complexes.
The same assumptions (cohomological unitality, $\ZZ/2$-grading, $\ZZ$ coefficients, and cofibrancy) apply to all $\ainf$-modules and $\ainf$-bimodules as well.
Morphisms of the above ($\ainf$-functors and morphisms of $\ainf$-modules and bimodules) must also be cohomologically unital and must respect gradings and coefficients.
Gradings will be cohomological, and all $\ainf$-categories in this paper are small (i.e.\ they have a \emph{set} of objects).
In fact, all the $\ainf$-categories, functors, modules, bimodules, and morphisms thereof which we consider will be strictly unital, however in reasoning about them it is only cohomological unitality which is ever relevant.

For the basic defintions of $\ainf$-categories, $\ainf$-modules, and
$\ainf$-bimodules (about which we assume some basic familiarity), we refer the reader to
\cite{seidelbook,seidelsubalgebras,tradlerinfinity,ganatrathesis,sheridanformulae}.
Although these and most other references work over a field, the notions which
we will use make perfect sense over any commutative ring.
Cofibrancy (which is automatic over a field) is defined in Definition \ref{cofibrantdef}.
Cofibrancy is significant for many reasons in the theory of $\ainf$-categories, however its sole significance for us (other than in Lemma \ref{cofibrantqisoinvert}) is that tensor products of cofibrant complexes are well-behaved (since a cofibrant complex is, in particular, K-flat).

For an $\ainf$-category $\C$ and a pair of objects $X,Y\in\C$, we write $\C(X,Y)$ for the associated morphism space in $\C$.
We adopt the ``forward composition'' convention, meaning we write composition as
\begin{equation}
\mu^k: \C(X_0,X_1)\otimes\C(X_1,X_2) \otimes \cdots \otimes \C(X_{k-1},X_k)\to\C(X_0,X_k)[2-k].
\end{equation}
For details on the associated sign conventions for the $\ainf$ relations, see \cite{tradlerinfinity,seidelsubalgebras,sheridanformulae}.

\subsubsection{Cofibrant complexes}

\begin{definition}\label{cofibrantdef}
We fix any class $\mathfrak C$ of isomorphism classes of $\ZZ/2$-graded complexes over
$\ZZ$, which we call \emph{cofibrant} complexes, satisfying the following
properties: 
\begin{enumerate}
\item\label{cofibrantfree}A free module concentrated in degree zero is cofibrant.
\item\label{cofibrantshift}A shift of a cofibrant complex is cofibrant.
\item\label{cofibrantdirectlimit}If $\{C_i\}_{i\in I}$ is a directed system where $I$ is well-ordered and $0\to\varinjlim_{i'<i}C_{i'}\to C_i\to K_i\to 0$ is degreewise split with $K_i$ cofibrant for all $i\in I$, then $\varinjlim_iC_i$ is cofibrant.
\item\label{cofibranttensor}A tensor product of cofibrant complexes is cofibrant.
\item\label{cofibranttensorexact}If $C$ is cofibrant then $-\otimes C$ preserves acyclicity (``$C$ is K-flat'').
\item\label{cofibrantkprojective}If $C$ is cofibrant then $\Hom(C,-)$ preserves acyclicity (``$C$ is K-projective'').
\end{enumerate}
The existence of such a class of complexes $\mathfrak C$ is essentially due to Spaltenstein \cite{spaltenstein}.
We recall the argument in Lemma \ref{cofibrantexist} below.
\end{definition}

\begin{lemma}[{Spaltenstein \cite[0.11 Lemma]{spaltenstein}}]\label{inverseexact}
Let $\{A_i\}_{i\in I}$ be an inverse system of complexes of abelian groups, indexed by a well-ordered set $I$.
Suppose further that for all $i\in I$, the following sequence is exact with $K_i$ acyclic:
\begin{equation}
\smash{0\to K_i\to A_i\to\varprojlim_{i'<i}A_{i'}\to 0.}
\end{equation}
Then $\varprojlim_iA_i$ is acyclic.
\end{lemma}

\begin{lemma}\label{cofibrantexist}
The class $\mathfrak C$ of K-projective complexes satisfies the conditions in Definition \ref{cofibrantdef}.
\end{lemma}

\begin{proof}
By definition, $C\in\mathfrak C$ if and only if \ref{cofibrantkprojective} $\Hom(C,-)$ preserves acyclicity.
Clearly \ref{cofibrantfree} free modules in degree zero are in $\mathfrak C$, and clearly \ref{cofibrantshift} $\mathfrak C$ is closed under shift.
The adjunction $\Hom(C\otimes C',-)=\Hom(C,\Hom(C',-))$ shows that \ref{cofibranttensor} $\mathfrak C$ is closed under tensor products.
This adjunction together with the fact that a complex $K$ is acyclic if and only if $\Hom(K,I)$ is acyclic for every injective $\ZZ$-module $I$ (which holds since the category of $\ZZ$-modules has enough injectives) implies that \ref{cofibranttensorexact} $-\otimes C$ preserves acyclicity if $\Hom(C,-)$ does.
Finally, property \ref{cofibrantdirectlimit} follows from Lemma \ref{inverseexact}.
\end{proof}

\begin{remark}
The Floer complexes appearing in this paper are all equipped with an action filtration, which implies they are cofibrant by properties \ref{cofibrantfree}--\ref{cofibrantdirectlimit}.
\end{remark}

\begin{lemma}\label{cofibrantqisoinvert}
Let $f:A\to B$ be a quasi-isomorphism.
\begin{enumerate}
\item If $B$ is cofibrant, then there exists $g:B\to A$ such that $fg:B\to B$ is chain homotopic to the identity.
\item If $A$ is also cofibrant, then $gf:A\to A$ is also chain homotopic to the identity.
\end{enumerate}
\end{lemma}

\begin{proof}
If $B$ is cofibrant, then $\Hom(B,A)\to\Hom(B,B)$ is a quasi-isomorphism, so we may lift the identity element of $\Hom(B,B)$ to a cycle homologous to it.  In other words, we get a chain map $g:B\to A$ such that $fg$ is chain homotopic to the identity.  If $A$ is also cofibrant, then applying the first part to $g$, we get $h$ with $gh$ homotopic to the identity.  Now $gf$ is homotopic to $gfgh$, which is homotopic to $gh$, which is homotopic to the identity.
\end{proof}

\subsubsection{\texorpdfstring{$\ainf$}{A\_infinity}-modules and bimodules}\label{secbimodules}

It will be helpful to recall some terminology and foundational results about
$\ainf$-bimodules over a pair of $\ainf$-categories, as well left and right
modules over a single $\ainf$-category (which arise as
the special case of bimodules where one of the categories is $\ZZ$).

We recall that for $\ainf$-categories $\C$ and $\D$, there is
a dg-category, denoted $[\C, \D]$, of \emph{$(\C,\D)$-bimodules}
\cite{tradlerinfinity, seidelsubalgebras, ganatrathesis}, whose
objects are cohomologically unital $\ainf$-bilinear functors $\B: \C^\op \times \D \to \Ch$ (in the
sense of \cite{lyubashenkomultilinear}; see also \cite{sheridanformulae}).
Concretely, a bimodule $\B$ consists of a cochain complex $\B(X,Y)$ for every pair of objects $X\in\C$, $Y\in\D$, equipped
with higher multiplication maps
\begin{multline}
\mu^{k|1|\ell}_{\B}: \C(X_k, X_{k-1}) \otimes \cdots
\otimes \C(X_1,X_0) \otimes \B(X_0,Y_0) \otimes \D(Y_0, Y_1) \otimes \cdots
\otimes \D(Y_{\ell-1},Y_\ell)\\
\to \B(X_k, Y_\ell)[1-k-\ell]
\end{multline}
satisfying the natural $\ainf$ relations.
As mentioned earlier, we require that each cochain complex $\B(X,Y)$ be cofibrant.

The space of morphisms in $[\C, \D]$ of degree $s$ from a bimodule $\PP$ to a
bimodule $\Q$ is, as a $\ZZ$-module, the direct product over all
pairs of tuples $(X_0, \ldots, X_k)$ and $(Y_0, \ldots, Y_\ell)$
of all maps
\begin{multline}
F^{k|1|\ell}: \C(X_k, X_{k-1}) \otimes \cdots
\otimes \C(X_1,X_0) \otimes \PP(X_0,Y_0) \otimes \D(Y_0, Y_1) \otimes \cdots
\otimes \D(Y_{\ell-1},Y_\ell)\\
\to \Q(X_k, Y_\ell)[s - k - \ell].
\end{multline}
These are called the degree $s$ bimodule `pre-homomorphisms' in the language
of \cite[Section 2]{seidelsubalgebras}.
There is a natural differential on the morphism space measuring the failure
of a collection $\{F^{k|1|l}\}$ to satisfy the ``$\ainf$-bimodule morphism
relations''; see \cite[eq.\ (2.8)]{seidelsubalgebras}. A \emph{bimodule homomorphism}
$\PP\to\Q$ is a closed degree zero pre-homomorphism. Recall that the condition of being a bimodule
homomorphism implies that $F^{0|1|0}:\PP(X,Y) \to \Q(X, Y)$ is a cochain map for
all $X,Y$. A bimodule homomorphism is called a \emph{quasi-isomorphism} iff $F^{0|1|0}$ 
is a quasi-isomorphism of co-chain complexes for each $X,Y$.

A $(\C,\C)$-bimodule will often simply be called a $\C$-bimodule.
There is a canonical $\C$-bimodule, the \emph{diagonal bimodule}
$\C_{\Delta}(X,Y):=\C(X,Y)$, with bimodule multiplication maps induced by
$\ainf$-multiplication (up to a sign twist, rather; see \cite[eq.\ (2.20)]{seidelsubalgebras}). By abuse of notation, we will simply denote this
bimodule by $\C$.

A left $\C$-module (resp.\ right $\D$-module) is simply a $(\C,\D)$-bimodule in which $\D=\ZZ$ (resp.\ $\C=\ZZ$) is the $\ainf$-category with a single object $\ast$ with endomorphism algebra $\ZZ$, and we shall write
\begin{align}
\M(X)& := \M( X, \ast),\\
\N(Y) &:= \N(\ast, Y).
\end{align}
Put another way, although left (resp.\ right) $\ainf$-modules are typically
defined as functors from $\C^\op$ (resp.\ $\D$) to $\Ch$, there is no
difference in thinking of them as bilinear functors $\C^{\op} \times \ZZ$
(resp.\ $\ZZ \times \D$) to $\Ch$ (compare \cite[p.\ 11]{seidelsubalgebras}).
In particular, the present discussion of $\ainf$-bimodules specializes
immediately to a discussion of left/right $\ainf$-modules. As a particularly
degenerate special case, we note that by the same convention $(\ZZ,\ZZ)$-bimodules are simply cochain complexes.

Let $\C_0$, $\C$, and $\C_1$ be $\ainf$-categories. Suppose $\PP$ is
a $(\C_0,\C)$-bimodule and $\Q$ is a $(\C,\C_1)$-bimodule.
We denote (the bar model of) the \emph{derived tensor product} of $\PP$ and $\Q$ over $\C$ by
\begin{equation}
    \PP \otimes_\C \Q,
\end{equation}
which concretely is the $(\C_0,\C_1)$-bimodule which associates to a pair
of objects $(X,Z)$ the chain complex
\begin{equation}\label{barcomplextensorproduct}
(\PP \otimes_{\C} \Q )(X,Z) := \bigoplus_{\begin{smallmatrix}k\geq 0\\Y_0, \ldots, Y_k \in \C\end{smallmatrix}} \PP(X,Y_0) \otimes \C(Y_0, Y_1)[1] \otimes \cdots \otimes \C(Y_{k-1}, Y_k)[1] \otimes \Q(Y_k, Z)
\end{equation}
with differential and bimodule maps given by summing over ways to contract tensor chains
by the $\ainf$ and bimodule structure maps, see \cite{seidelsubalgebras,ganatrathesis} for more details.
Note that this construction preserves
cofibrancy (i.e.\ if the morphism spaces in $\C$ and the chain complexes associated to $\PP$ and $\Q$ are all cofibrant, then the above chain complexes are too).
As noted above, this construction (and the results that follow) carries over to modules as well, so in
particular, the tensor product of a right $\C$-module (i.e.\ a $(\ZZ,\C)$-bimodule)
with a $(\C,\D)$-bimodule is a right $\D$-module, etc.

For any $(\C_0, \C)$-bimodule $\PP$ or any $(\C,
\C_0)$-bimodule $\Q$, there are canonical bimodule homomorphisms
\begin{align}
    \label{tensordiagonalright}\PP \otimes_{\C} \C &\to \PP,\\
    \label{tensordiagonalleft} \C \otimes_{\C}  \Q & \to \Q.
\end{align}
For instance, on the level of chain complexes, for any pair of objects $(X,Z)$,
the map  \eqref{tensordiagonalright} from
\begin{equation}
(\PP \otimes_{\C} \C )(X,Z)
:=\bigoplus_{\begin{smallmatrix}k\geq 0\\ Y_0, \ldots, Y_k \in
    \C\end{smallmatrix}} \PP(X,Y_0) \otimes \C(Y_0, Y_1)[1] \otimes \cdots \otimes
\C(Y_{k-1}, Y_k)[1] \otimes \C(Y_k, Z)
\end{equation}
to $\PP(X,Z)$ is given by contracting
using the bimodule structure map of $\PP$ (up to an overall sign correction,
compare \cite[eq.\ (2.21)--(2.24)]{seidelsubalgebras}). The following lemma is standard,
see \cite{seidelsubalgebras, ganatrathesis} for more details.

\begin{lemma}\label{identityidempotent}
    The canonical maps \eqref{tensordiagonalright} and \eqref{tensordiagonalleft} are quasi-isomorphisms.
\end{lemma}

\begin{proof}
Without loss of generality, we discuss \eqref{tensordiagonalright} only (the case of \eqref{tensordiagonalleft} is identical, and, in fact, follows from the case of \eqref{tensordiagonalright} by passing to the ``opposite bimodule'').

The cone of \eqref{tensordiagonalright} evaluated at $(X,Z)$ is given by
\begin{equation}\label{conecomplextrivialtensor}
\bigoplus_{\begin{smallmatrix}k\geq 0\\ Y_1, \ldots, Y_k \in \C\end{smallmatrix}} \PP(X,Y_1) \otimes \C(Y_1, Y_2)[1] \otimes \cdots \otimes \C(Y_{k-1}, Y_k)[1]\otimes\C(Y_k,Z)[1]
\end{equation}
where the $k=0$ term is by definition $\PP(X,Z)$.
If $\C$ and $\PP$ are strictly unital, then $-\otimes\1_Z$ is a contracting homotopy of this complex, where $\1_Z\in\C(Z,Z)$ denotes the strict unit.
In the general cohomologically unital case, argue as follows.

For any cycle $c=\sum_{k\geq 0}c_k$ in \eqref{conecomplextrivialtensor}, let $k_{\max}(c)$ denote the maximum $k$ for which $c_k$ is nonzero.
To show that \eqref{conecomplextrivialtensor} is acyclic, it suffices to show that for every cycle $c$, there exists another cycle $c'$ cohomologous to $c$ with $k_{\max}(c')<k_{\max}(c)$.
If
\begin{equation}
[c_{k_{\max}(c)}]\in H^\bullet\biggl(\bigoplus_{Y_1,\ldots,Y_{k_{\max}(c)}\in\C}\PP(X,Y_1)\otimes\C(Y_1, Y_2)\otimes\cdots \otimes\C(Y_{k_{\max}(c)},Z)\biggr)
\end{equation}
vanishes, then the existence of a suitable $c'$ is clear.
Thus it is enough to show that every cycle $c$ is cohomologous to a cycle $c'$ with $k_{\max}(c')=k_{\max}(c)$ and $[c'_{k_{\max}(c')}]=0$.

For any given cycle $c$, consider $c':=c-d(c\otimes\1_Z)$ where $\1_Z\in\C(Z,Z)$ denotes any cycle whose class in cohomology is the cohomological unit (clearly $c$ and $c'$ are cohomologous).
Note that since both $c$ and $\1_Z$ are cycles, the chain $d(c\otimes\1_Z)$ may be obtained from $c\otimes\1_Z$ by contracting just those strings of consecutive tensor factors which intersect both $c$ and $\1_Z$ nontrivially.
In particular, we conclude that the length $k_{\max}(c)+1$ term of $c'$ vanishes, and the length $k_{\max}(c)$ term of $c'$ is given by $c_k-(\id^{\otimes k}\otimes\mu^2)(c_k\otimes\1_Z)$.
Hence it is enough to show that the map
\begin{multline}
\PP(X,Y_1)\otimes\C(Y_1, Y_2)\otimes\cdots \otimes\C(Y_{k_{\max}(c)},Z)\\
\xrightarrow{\id^{\otimes k}\otimes\mu^2(-,\1_Z)}\PP(X,Y_1)\otimes\C(Y_1,Y_2)\otimes\cdots \otimes\C(Y_{k_{\max}(c)},Z)
\end{multline}
acts as the identity on cohomology for all $Y_1,\ldots,Y_{k_{\max}(c)}\in\C$.
This follows from the fact that both maps
\begin{align}
&\mu^2(-,\1_Z):\C(Y_{k_{\max}(c)},Z)\to\C(Y_{k_{\max}(c)},Z)\\
&\mu^2(-,\1_Z):\PP(X,Z)\to\PP(X,Z)
\end{align}
are homotopic to the identity by Lemma \ref{multhunitisidhtpy} below.
\end{proof}

\begin{lemma}\label{multhunitisidhtpy}
Let $\M$ be a $\C$-module and let $\1_X\in\C(X,X)$ be any cycle representing the cohomology unit.
The map $\mu^2(-,\1_X):\M(X)\to\M(X)$ is chain homotopic to the identity map.
\end{lemma}

\begin{proof}
First note that the chain homotopy class of the map $\mu^2(-,\1_X)$ is independent of the choice of cycle $\1_X$ representing the cohomology unit.
Next, note that since $\mu^2(\1_X,\1_X)$ is cohomologous to $\1_X$, the $\mu^3$ operation provides a chain homotopy between $\mu^2(-,\1_X)\circ\mu^2(-,\1_X)$ and $\mu^2(-,\1_X)$ (in other words, $\mu^2(-,\1_X)$ is idempotent up to chain homotopy).
Finally, recall that since $\mu^2(-,\1_X)$ is a quasi-isomorphism (by cohomological unitality), cofibrancy of $\M(X)$ implies by Lemma \ref{cofibrantqisoinvert} that $\mu^2(-,\1_X)$ is invertible up to homotopy.
This completes the proof since idempotent and invertible implies identity.
\end{proof}

For a  $(\C,\D)$-bimodule $\B$ and a pair of functors $g: \C'\to \C$ and $h: \D'\to \D$, we denote by
$(g,h)^*\B$ the \emph{two-sided pull-back} (see e.g.\ \cite{ganatrathesis}) of $\B$ along
$g$ and $h$.  On the level of chain complexes, $(g,h)^*\B(X,Y):= \B(gX,hY)$.
Thinking of the bimodule $\B$ as a bilinear functor $\C^\op \times \D \to
\Ch$, the two sided pull-back is simply the composition $\B\circ(g^\op,h)$.

For $\PP$ a $(\C_0,\C)$-bimodule, $\Q$ a $(\C,\C_1)$-bimodule, and $f: \A \to \C$ an $\ainf$-functor, we denote by
\begin{equation}
    \PP \otimes_{\A} \Q
\end{equation}
the tensor product $(\id_{\C_0}, f)^* \PP \otimes_{\A} (f,
\id_{\C_1})^*\Q$. In the cases we study below, $f$ will be an inclusion on the
level of morphism complexes with no higher order functor operations (which we call a \emph{naive inclusion})
justifying our omission of $f$ from the notation. We note that $f$ always induces
a canonical bimodule homomorphism
\begin{equation}\label{eq:subtensormap}
    \PP \otimes_{\A} \Q \to \PP \otimes_{\C} \Q.
\end{equation}
If $f$ is a naive inclusion, this map is just for each pair of objects $(X,Y)$,
the levelwise inclusion of the complexes \eqref{barcomplextensorproduct}.
Finally, let us discuss a circumstance under which this map is a quasi-isomorphism.

\begin{lemma}\label{splitgeneratorquasiisomorphism}
Let $f: \A\to\C$ be cohomologically fully faithful and split-generating (e.g.\ it could be a quasi-equivalence, or it could be the inclusion of a full subcategory spanned by split-generators).
Then the map
\begin{equation}
\PP\otimes_\A\Q\to\PP\otimes_\C\Q
\end{equation}
is a quasi-isomorphism for all $(\C_0, \C)$-bimodules $\PP$ and $(\C, \C_1)$-bimodules $\Q$.
\end{lemma}

\begin{proof}
The result in the case $\A\to\C$ is bijective on objects is clear.

The result in the case $\A\to\C$ is the inclusion of a full subcategory may be seen as follows.
Note that, by virtue of the quasi-isomorphism $\PP\otimes_\C\C\otimes_\A\C\otimes_\C\Q\to\PP\otimes_\A\Q$ (and the same with $\C$ in place of $\A$) from Lemma \ref{identityidempotent}, it is enough to show that $\C\otimes_\A\C\to\C\otimes_\C\C$ is a quasi-isomorphism.
Further composing with the quasi-isomorphism $\C\otimes_\C\C\to\C$, we see that it is enough to show that $\C\otimes_\A\C\to\C$ is a quasi-isomorphism.
When restricted to inputs $(X,Y)$ with $Y\in\A$, this map is simply $\C\otimes_\A\A\to\C$, which is a quasi-isomorphism by Lemma \ref{identityidempotent}.
This implies the same for any $Y$ split-generated by $\A$, namely all $Y\in\C$.

These two special cases suffice to treat the general case, as we now argue.
We may factor $\A\to\C$ as $\A\to\im(\A)\to\C$.
Since $\im(\A)\to\C$ is the inclusion of a full subcategory, it is enough to treat the case of $\A\to\im(\A)$, i.e.\ we may assume $\A\to\C$ is surjective on objects.
Now, choose a full subcategory $\tilde\C\subseteq\A$ mapping bijectively to $\C$, so there is a factorization $\tilde\C\to\A\to\C$.
Now the result for $\A\to\C$ follows from the result for $\tilde\C\to\C$ (bijective on objects) and $\tilde\C\to\A$ (full subcategory; note $\tilde\C\to\A$ is essentially surjective since $\A\to\C$ is cohomologically fully faithful).
\end{proof}

\subsubsection{Quotients and localization}\label{alocal}

We review some basic elements of the theory of quotients and localizations of $\ainf$-categories due to Lyubashenko--Ovsienko \cite{lyubashenkoovsienko}, Lyubashenko--Manzyuk \cite{lyubashenkomanzyuk}, and Drinfeld \cite{drinfelddgquotient}, generalizing much earlier work of Verdier \cite{verdierderivedcategories} on quotients and localizations of triangulated categories.
Our aim is both to make this article self-contained
and to verify that the (rather elementary) results we need remain valid for
$\ainf$-categories over a general commutative ring (which, as stated above, carry the
requirement that all morphism, module, bimodule complexes are
cofibrant).

\begin{definition}\label{ainfquotient}
Let $\C$ be an $\ainf$-category, and let $\A$ be a set of objects of $\C$ (meaning, $\A$ is a set, and for each element of $\A$ there is specified an object of $\C$).
The \emph{quotient $\ainf$-category} $\C/\A$ is defined to
have the same objects as $\C$ and morphism spaces
\begin{align}\label{quotientdefcomplex}
(\C/\A)(X,Y):={}&\bigoplus_{\begin{smallmatrix}p\geq 0\\A_1,\ldots,A_p\in\A\end{smallmatrix}}\C(X,A_1)\otimes \C(A_1, A_2)[1] \otimes\cdots\otimes\C(A_p,Y)[1]\\
={}&[(\C\otimes_\A\C)[1]\to\C](X,Y)\nonumber
\end{align}
with the usual bar differential (note that the $p=0$ term is, by convention, $\C(X,Y)$).
It is important that $\A$ be a \emph{set} so that the direct sum \eqref{quotientdefcomplex} makes sense.
Note that this construction preserves strict unitality and cofibrancy.
More generally, we may quotient $\C$ by any set $\A$ of objects of $\Tw\C$ (note that, although $\A$ must be a set, there is no need to fix a specific small model for $\Tw\C$; also note that when $\C$ is strictly unital or cohomologically unital, so is $\Tw\C$).

Let $\M$ be a right $\C$-module.  There is a $\C/\A$-module $\M/\A$ defined as
\begin{align}\label{modulequotientdefcomplex}
(\M/\A)(Y):={}&\bigoplus_{\begin{smallmatrix}p\geq 0\\A_1,\ldots,A_p\in\A\end{smallmatrix}}\M(A_1)\otimes\C(A_1,A_2)[1]\otimes\cdots\otimes\C(A_p,Y)[1]\\
={}&[\M\otimes_\A\C[1]\to\M](Y).\nonumber
\end{align}
with the usual bar differential (note again that the $p=0$ term in the sum is $\M(Y)$).  If $\M$ is a left $\C$-module, then the analogously
defined $\C/\A$-module is denoted $\A\backslash\M$.  Bimodules can be
quotiented on both sides: if $\B$ is a $(\C,\C')$-bimodule, then we can form
the $(\C/\A,\C'/\A')$-bimodule $\A\backslash\B/\A'$.
\end{definition}

\begin{lemma}
The quotient construction in Definition \ref{ainfquotient} preserves cohomological unitality, and quotient functors are cohomologically unital.
\end{lemma}

\begin{proof}
It is enough to show that if $\1_X\in\C(X,X)$ is a cycle representing the cohomological unit, then $\mu^2(-,\1_X):(\M/\A)(X)\to(\M/\A)(X)$ acts as the identity on cohomology.
Since $\mu^2(-,\1_X)\circ\mu^2(-,\1_X)$ is chain homotopic to $\mu^2(-,\1_X)$ (since $\1_X$ is a cohomological unit in $\C$), it is enough to show that $\mu^2(-,\1_X)$ is a quasi-isomorphism.
This then follows from a filtration argument.
\end{proof}

\begin{lemma}\label{quotientannihilatesA}
If $Y$ is split-generated by $\A$, then $(\M/\A)(Y)$ is acyclic.  In
particular, if $X$ or $Y$ is split-generated by $\A$, then $(\C/\A)(X,Y)$ is acyclic.\qed
\end{lemma}

\begin{proof}
We may as well assume that $\C$ contains no objects other than $\A$ and $Y$.
Now $\M\otimes_\A\C\to\M\otimes_\C\C$ is a quasi-isomorphism by Lemma \ref{splitgeneratorquasiisomorphism}, and $\M\otimes_\C\C\to\M$ is a quasi-isomorphism by Lemma \ref{identityidempotent}.
Thus $\M\otimes_\A\C\to\M$ is a quasi-isomorphism, and hence $\M/\A$ is acyclic.
\end{proof}

\begin{lemma}\label{localgood}
If $\M(A)$ is acyclic for all $A\in\A$, then the natural map
$\M(Y)\to(\M/\A)(Y)$ is a quasi-isomorphism for all $Y\in\C$.

In particular,
if $\C(X,A)$ is acyclic for all $A\in\A$ (``$X$ is left-orthogonal to $\A$''), then the natural map
$\C(X,Y)\to(\C/\A)(X,Y)$ is a quasi-isomorphism for all $Y\in\C$.
\end{lemma}

\begin{proof}
By considering the length filtration (i.e.\ the filtration by $p$) on the
quotient of \eqref{modulequotientdefcomplex} by the inclusion of $\M(Y)$, it is
enough to show that for $p\geq 1$ and $A_1,\ldots,A_p\in\A$, the complex
\begin{equation}
\M(A_1)\otimes\C(A_1,A_2)\otimes\cdots\otimes\C(A_p,Y)
\end{equation}
is acyclic.  Now note that the first term is acyclic and each remaining tensor
factor is cofibrant.
\end{proof}

\begin{corollary}\label{localizationsplitgenerate}
If $\A$ split-generates $\B$, then $\M/\A\to\M/\B$ is a quasi-isomorphism (in particular, $\C/\A\to\C/\B$ is a quasi-isomorphism).
\end{corollary}

\begin{proof}
The map in question is itself a quotient $\M/\A\to(\M/\A)/\B'=\M/\B$ where $\B'$ denotes (the image in $\C/\A$ of) the objects of $\B$ which are not in $\A$ (note that this latter quotient is the quotient of a $\C/\A$-module).
Thus by Lemma \ref{localgood}, it is enough to show that $\B'$ are zero objects in $\C/\A$, which in turn follows from Lemma \ref{quotientannihilatesA} since $\A$ split-generates $\B$.
\end{proof}

\begin{lemma}\label{quotienttensorproperty}
For any $(\C/\A)$-modules $\M$ and $\N$, the natural map $\M\otimes_\C\N\to\M\otimes_{\C/\A}\N$ is a quasi-isomorphism.
\end{lemma}

\begin{proof}
By Lemma \ref{identityidempotent}, we can replace $\M$ with $\M\otimes_{\C/\A}\C/\A$ and $\N$ with $\C/\A\otimes_{\C/\A}\N$.
Hence it is enough to show that $\C/\A\otimes_\C\C/\A\to\C/\A\otimes_{\C/\A}\C/\A$ is a quasi-isomorphism.
Using Lemma \ref{identityidempotent} again, it is thus enough to show that $\C/\A\otimes_\C\C/\A\to\C/\A$ is a quasi-isomorphism.

Explicitly, the bimodule $\C/\A\otimes_\C\C/\A$ is given by
\begin{equation}\label{caccaweirdtensor}
    \bigoplus_{\begin{smallmatrix}p,q,r\geq 0\\A_1,\ldots,A_p\in\A\\Z_0,\ldots,Z_q\in\C\\A_1',\ldots,A_r'\in\A\end{smallmatrix}}\C(-,A_1)\otimes\cdots\otimes\C(A_p,Z_0)\otimes \C(Z_0, Z_1) \cdots\otimes \C(Z_{q-1}, Z_q) \otimes \C(Z_q,A_1')\otimes\cdots\otimes\C(A_r',-),
\end{equation}
the bimodule $\C/\A$ is given by
\begin{equation}\label{causualtensor}
\bigoplus_{\begin{smallmatrix}s\geq 0\\A_1,\ldots,A_s\in\A\end{smallmatrix}}\C(-,A_1)\otimes\cdots\otimes\C(A_s,-),
\end{equation}
and the map $\C/\A\otimes_\C\C/\A\to\C/\A$ is given by summing up (with appropriate signs) all ways of applying $\mu^k$'s to subsequences containing all the $Z_i$'s in \eqref{caccaweirdtensor} and then viewing the surviving $A_i$'s and $A_i'$'s as a single sequence $A_1,\ldots,A_s$ as in \eqref{causualtensor}.
Said a bit differently, but equivalently, the map $\C/\A\otimes_\C\C/\A\to\C/\A$ is the composition
\begin{equation}
\C/\A\otimes_\C\C/\A=\A\backslash\C\otimes_\C\C/\A\xrightarrow\alpha\A\backslash\C/\A\xrightarrow\beta\C/\A,
\end{equation}
where $\alpha$ is simply induced by the obvious map $\C\otimes_\C\C\to\C$, and $\beta$ is a non-obvious map given by concatenating the two adjacent sequences of elements of $\A$.
The map $\alpha$ is a quasi-isomorphism by Lemma \ref{identityidempotent}.
The map $\beta$ has a section $\C/\A\to\A\backslash\C/\A$, namely the obvious inclusion $\M\to\A\backslash\M$ for $\M=\C/\A$.
This section is a quasi-isomorphism by Lemmas \ref{quotientannihilatesA} and \ref{localgood}, and hence $\beta$ is also a quasi-isomorphism.
\end{proof}

\begin{lemma}\label{localgoodpro}
If $\varinjlim_iH^\bullet\M_i(A)=0$ for all $A\in\A$ for a sequence
$\M_1\xrightarrow{f_1}\M_2\xrightarrow{f_2}\cdots$ of $\C$-modules, then the natural map
$\varinjlim_iH^\bullet\M_i(Y)\to\varinjlim_iH^\bullet(\M_i/\A)(Y)$ is an isomorphism for all $Y\in\C$.

In particular, if $\varinjlim_iH^\bullet\C(X_i,A)=0$ for all $A\in\A$ for
$X_1,X_2,\ldots\in\C$ and cycles $c_i\in\C(X_{i+1},X_i)$ (``$\varprojlim_iX_i\in\Pro\C$ is left-orthogonal to $\A$''), then the natural map
\begin{equation}
\varinjlim_iH^\bullet\C(X_i,Y)\to\varinjlim_iH^\bullet(\C/\A)(X_i,Y)
\end{equation}
is an isomorphism for all $Y\in\C$.
\end{lemma}

\begin{proof}
Apply Lemma \ref{localgood} to the mapping cone
\begin{equation}
\varinjlim_i\M_i:=
\biggl[\bigoplus_{i=1}^\infty\M_i[1]\xrightarrow{\bigoplus\1_i-f_i}\bigoplus_{i=1}^\infty\M_i\biggr]
\end{equation}
which is naturally a $\C$-module.
\end{proof}

\begin{definition}\label{localizationdef}
Let $\C$ be an $\ainf$-category, and let $W$ be a collection of morphisms in
$H^0\C$.  The \emph{localized $\ainf$-category} $\C[W^{-1}]$ is defined as
the quotient 
\begin{equation}
\C[W^{-1}]:=\C/\!\cones(W)
\end{equation}
where $\cones(W)$ denotes the set of all cones $[X\xrightarrow aY]$ where $a\in\C(X,Y)$ is a cycle representing an element of $W$.

For a right $\C$-module $\M$, we define $\M_{W^{-1}}:=\M/\!\cones(W)$ which is a right $\C[W^{-1}]$-module.
Similarly define $_{W^{-1}}\M$ for $\M$ a left $\C$-module, and define
${}_{W^{-1}}\M_{W^{\prime-1}}$ for $\M$ a $(\C,\C')$-bimodule.
\end{definition}

\subsection{Holomorphic curves and \texorpdfstring{$\ainf$ operations}{A\_infinity operations}}\label{wcurvessec}

We begin with a general discussion of holomorphic curves in Liouville sectors,
as well as the corresponding $\ainf$ operations, for sequences of mutually transverse Lagrangians.

Let $X$ be a Liouville sector and let $L_1,L_2\subseteq X$ be a pair of transverse cylindrical exact Lagrangians (we tacitly assume all Lagrangians to be disjoint from $\partial X$).
For the purpose of defining gradings/orientations, suppose further that these Lagrangians are equipped with $\Spin$ structures, namely equipped with lifts
\begin{equation}\label{spinlifts}
\begin{tikzcd}
{}&B\Spin\ar{d}\\
L\ar{r}{TL}\ar[dashed]{ru}&BO\ar{r}{(w_1,w_2)}&K(\ZZ/2,1)\times K(\ZZ/2,2)
\end{tikzcd}
\end{equation}
where we remark that $B\Spin\to BO$ is the $K(\ZZ/2,0)\times K(\ZZ/2,1)$ bundle trivializing $(w_1,w_2):BO\to K(\ZZ/2,1)\times K(\ZZ/2,2)$ (if we were content working without gradings and over characteristic $2$, such data could be omitted).
Associated to any such pair $L_1,L_2\subseteq X$ is a \emph{Lagrangian Floer complex}
\begin{equation}\label{lagnfloercomplex}
CF^\bullet(L_1,L_2):=\bigoplus_{p\in L_1\cap L_2}\oo_{L_1,L_2,p}
\end{equation}
which is, as a $\ZZ$-module, isomorphic
to the free abelian group generated by the finite set $L_1 \cap
L_2$. More intrinsically, \eqref{lagnfloercomplex} is the direct sum of the orientation lines $\oo_{L_1,L_2,p}$ associated to each intersection point $p \in L_1 \cap L_2$ (an \emph{orientation line} is a $\ZZ/2$-graded free $\ZZ$-module of rank one; the $\ZZ/2$-grading is relevant for the Koszul rule of signs and, equivalently, for the super tensor product).
The orientation line $\oo_{L_1,L_2,p}$ is extracted from index theory in the usual way, which we summarize as follows.
Consider the linearized $\bar\partial$-operator at the constant map $u:(-\infty,0]\times[0,1]\to X$ sending everything to $p$ and subject to Lagrangian boundary conditions $u(t,0)\in L_1$ and $u(t,1)\in L_2$.
We extend this Cauchy--Riemann operator with totally real boundary conditions to such an operator on the unit disk with a single negative boundary puncture (see Figure \ref{orientationlineslagr}).
The boundary conditions are extended by choosing a path in the Lagrangian Grassmannian of $T_pX$ from $TL_1$ to $TL_2$ compatible with the given lifts \eqref{spinlifts} (note that there are multiple such paths differing by an even multiple of the Maslov class, however the resulting orientation lines are canonically isomorphic, so we systematically elide this point).
Then $\oo_{L_1,L_2,x}:=\oo_D^\vee$ is defined as the dual of the Fredholm orientation line $\oo_D:=\oo_{\ker D}\otimes\oo_{\coker D}^\vee$ of this Cauchy--Riemann operator $D$ (the orientation line $\oo_V$ of a finite-dimensional vector space $V$ is the $\ZZ$-module generated by the two orientations on $V$ modulo the relation that their sum vanishes, placed in cohomological degree $-\dim V$; more succinctly, we could equivalently define $\oo_V:=H_\bullet(V,V\setminus 0)$).
The grading of $\oo_{L_1,L_2,x}$ is determined by the
usual Maslov index considerations (see \cite{seidelbook}).
For further background on the theory of gradings and orientations in Floer
theory, we refer the reader to Seidel \cite[\S 11]{seidelbook}
\cite{seidelgraded} and Abouzaid \cite[\S 1.4]{abouzaidviterbo} (let us also point out that we could just as well use $\ZZ/2N$-gradings in our setup).

\begin{figure}[hbt]
\centering
\includegraphics{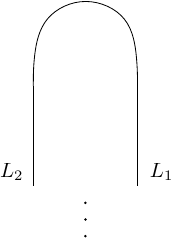}
\caption{The orientation line $\oo_{L_1,L_2,p}$ is defined as the dual of the Fredholm orientation line of Cauchy--Riemann operator on the punctured disk, illustrated here, extending the linearized $\bar\partial$-operator at the constant map $u:(-\infty,0]\times[0,1]\to X$ sending everything to $p$.}\label{orientationlineslagr}
\end{figure}

We now discuss the construction of the $\ainf$ operations $\mu^k$ on the Lagrangian Floer complexes \eqref{lagnfloercomplex}, of which the differential is the special case $k=1$.

For $k\geq 2$, let $\Rbar_{k,1}$ denote the Deligne--Mumford moduli space of
stable disks with $k+1$ marked points on the boundary labelled $x_1,\ldots,x_k,y$
in counterclockwise order, which inherits the structure of a smooth manifold
with corners from its embedding into $\Mbar_{0,k+1}$.  For $k=1$, define
$\Rbar_{k,1}$ as the stack $\pt/\RR$.  For $k\geq 1$, denote by
$\Sbar_{k,1}\to\Rbar_{k,1}$ the universal curve (note that
$\Sbar_{1,1}=[0,1]$).

The \emph{thin parts} of the fibers of $\Sbar_{k,1}\to\Rbar_{k,1}$ refers to a neighborhood (inside the total space $\Sbar_{k,1}$) of the boundary marked points $x_1,\ldots,x_k,y$ and the nodes of the fibers.
The complement of the thin parts is called the \emph{thick parts}.

Denote by $\J(X)$ the space of $\omega$-compatible cylindrical almost complex structures on $X$.
Families of almost complex structures $U\to\J(X)$ are always implicitly required to be uniformly cylindrical, in the sense that there exists a subset of $U\times X$ which is proper over $U$ outside which the family is $Z$-invariant.

Let $L_1,\ldots,L_N\subseteq X$ be a finite collection of mutually transverse cylindrical exact Lagrangians equipped with $\Spin$ structures.
For every sequence $1\leq i_0<\cdots<i_k\leq N$ ($k\geq 1$), fix families of ``universal strip-like coordinates'' (following \cite[\S 2b]{abouzaidseidel})
\begin{align}
\label{coordsI}\xi_{i_0,\ldots,i_k;j}^+:[0,\infty)\times[0,1]\times\Rbar_{k,1}&\to\Sbar_{k,1}\quad j=1,\ldots,k\\
\label{coordsII}\xi_{i_0,\ldots,i_k}^-:(-\infty,0]\times[0,1]\times\Rbar_{k,1}&\to\Sbar_{k,1}
\end{align}
and a family of almost complex structures
\begin{equation}\label{Jfamilyforainfty}
J_{i_0,\ldots,i_k}:\Sbar_{k,1}\to\J(X)
\end{equation}
satisfying the following properties:
\begin{itemize}
\item
The strip-like coordinates $\xi$ must be \emph{compatible with gluing} in the following sense.
For each $j=1,\ldots,\ell$, there is a boundary collar
\begin{equation}\label{collarimage}
\Rbar_{k,1}\times\Rbar_{\ell,1}\times(0,\infty]\to\Rbar_{k+\ell-1,1}
\end{equation}
defined, for sufficiently large $S\in(0,\infty]$, by gluing together the ends
at $x_j$ in the disk with $\ell+1$ punctures and $y$ in the disk with $k+1$ punctures
according to the parameter $S$ via the coordinates \eqref{coordsI}--\eqref{coordsII}.
Over the image of this boundary collar, the strip-like
coordinates on $\Rbar_{k,1}$ and $\Rbar_{\ell,1}$ determine ``glued'' strip-like
coordinates on $\Rbar_{k+\ell-1,1}$, and we require that these glued coordinates agree with the
strip-like coordinates specified on $\Rbar_{k+\ell-1,1}$.
Note that this gluing procedure also gives rise to strip-like coordinates in all thin parts of fibers.

(Though we will always write choices of coordinates in the form \eqref{coordsI}--\eqref{coordsII}, it is somewhat better to view these coordinates as only being well-defined up to translating in the $s$-coordinate, i.e.\ what is well-defined is the coordinate $t$ and the $1$-form $ds$ in a neighborhood of each puncture, and more generally over the thin parts.)
\item The almost complex structures $J$ must be \emph{compatible with gluing} via $\xi$ in the following sense:
(1) $J_{i_0,\ldots,i_k}$ must be $s$-invariant in the thin parts with respect to the strip-like coordinates $\xi$,
(2) $J_{i_0,\ldots,i_{k+\ell-1}}$ must, over (the inverse image in $\Sbar_{k+\ell-1,1}$ of) the image of \eqref{collarimage}, be given by the obvious gluing under $\xi$ of its restriction to $\Rbar_{k,1}\times\Rbar_{\ell,1}$, and
(3) $J_{i_0,\ldots,i_{k+\ell-1}}$ must, over (the inverse image in $\Sbar_{k+\ell-1,1}$ of) the image of $\Rbar_{k,1}\times\Rbar_{\ell,1}$ under \eqref{collarimage} coincide with the product of $J_{i_{j-1},\ldots,i_{k+j-1}}$ and $J_{i_0,\ldots,i_{j-1},i_{k+j-1},\ldots,i_{k+\ell-1}}$.
\item The almost complex structures $J$ must be \emph{adapted to $\partial X$}, meaning that $\pi:\Nbd^Z\partial X\to\CC$ is $J_{i_0,\ldots,i_k}$-holomorphic over $\pi^{-1}(\CC_{\left|\Re\right|\leq\varepsilon})$ for some $\varepsilon>0$.
\end{itemize}

\begin{lemma}\label{stripacsinduct}
Strip-like coordinates \eqref{coordsI}--\eqref{coordsII} and almost complex structures \eqref{Jfamilyforainfty} as above may be constructed by induction on the subset $\{i_0,\ldots,i_k\}$.
\end{lemma}

\begin{proof}
The only thing to check is that the various compatibility conditions agree on their overlap, and this follows from the compatibility conditions at earlier steps in the induction.
For more details, see Seidel \cite[II (9g,9i)]{seidelbook}.
\end{proof}

\begin{figure}[hbt]
\centering
\includegraphics{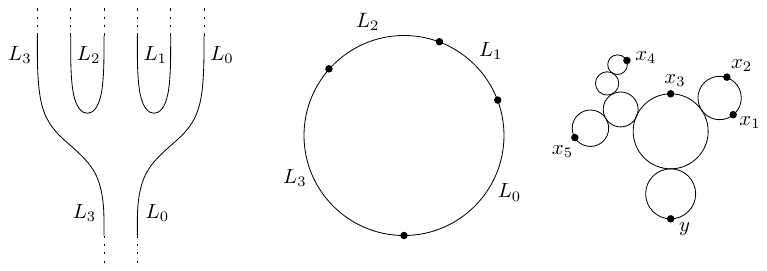}
\caption{Holomorphic curves comprising the moduli spaces $\Rbar_{k,1}(y;x_1,\ldots,x_k)$.}\label{mudisk}
\end{figure}

Fixing a choice of strip-like coordinates and almost complex structures as
above, we can now consider the compactified moduli spaces
\begin{equation}\label{modulidiscs_with_asymptotics}
\Rbar_{k,1}(y;x_1,\ldots,x_k)
\end{equation}
of stable $J_{i_0, \ldots, i_k}$-holomorphic maps $u: \Sigma \to X$ with Lagrangian boundary
conditions as in Figure \ref{mudisk}, mapping the marked points to the chosen intersections $y\in L_{i_0}\cap L_{i_k}$ and $x_s\in L_{i_{s-1}}\cap L_{i_s}$.  Note that it is the map which must be stable, not the domain (the domain merely has a unique stabilization map to a fiber of $\Sbar_{k,1}\to\Rbar_{k,1}$).
Note that $J_{i_0,\ldots,i_k}$-holormophicity of $\pi$ as above implies that any such holomorphic disk must be disjoint from $\pi^{-1}(\CC_{0\leq\Re\leq\varepsilon})$ since its boundary is disjoint from this region by  Lemma  \ref{pibarrierlemma} of \S\ref{bdryescape}.

To show that the moduli spaces
\eqref{modulidiscs_with_asymptotics}
are indeed compact (as their name would suggest), we need to recall the definition of the action functional for exact Lagrangians $L_0,L_1\subseteq X$:
\begin{align}
a:L_0\cap L_1&\to\RR\\
x&\mapsto f_0(x)-f_1(x)
\end{align}
where $\lambda_i|_{L_i}=df_i$ (the ambiguity in choosing $f_i$ will not concern us; note that we do not require $f_i$ to have compact support).
The energy of a strip $u:\RR\times[0,1]\to X$ with $u(s,0)\in L_0$, $u(s,1)\in L_1$, and $u(\pm\infty,t)=x^\pm$ is defined as
\begin{equation}
E(u):=\int_{\RR\times[0,1]}u^\ast d\lambda=a(x^+)-a(x^-).
\end{equation}
For pseudo-holomorphic $u$, the integrand is $\geq 0$, and $E(u)=0$ implies $u$ is constant.
Similar considerations apply to all disks as in Figure \ref{mudisk}.

\begin{proposition}\label{wcompactness}
The moduli spaces $\Rbar_{k,1}(y;x_1,\ldots,x_k)$ are compact.
\end{proposition}

\begin{proof}
It is enough to show that all stable disks $u:\Sigma\to X$ are contained \emph{a priori} inside a fixed compact subset of $X$ depending only on $y,x_1,\ldots,x_k$ (then the usual Gromov compactness arguments apply).

We first claim that, given any compact subset $B\subseteq\Sigma$ (a fiber of $\Sbar_{k,1}\to\Rbar_{k,1}$) disjoint from the boundary marked points and nodes of $\Sigma$ (i.e.\ $B$ is contained in the thick parts of $\Sigma$) and a point $p\in B$, the image $u(B)$ is bounded \emph{a priori} away from infinity in terms of $u(p)$, the energy $E$, and the geometry of $J$ restricted to $B$.
Indeed, this follows by applying monotonicity inequalities to the graph of $u$ inside $\Sigma\times X$ (see \S\ref{inftyescape}) and using the geometric boundedness from Lemma \ref{domainacsbddgeo}.

Let us now discuss the situation over the thin parts of $\Sigma$.
We work in the chosen strip-like coordinates $[0,\infty)\times[0,1]\to\Sigma$ (or $(-\infty,0]\times[0,1]\to\Sigma$ or $I\times[0,1]\to\Sigma$).
Outside a sufficiently large compact subset $K$ of $X$, all Lagrangians in question are uniformly separated.
Therefore, if $[a,b]$ is an interval such that $u([a,b]\times[0,1])$ lies entirely outside $K$, then we have
\begin{equation}\label{wcsdistanceintegral}
E\geq\int_{[a,b]}\int_{[0,1]}\left|\partial_tu\right|^2\geq\int_{[a,b]}\biggl(\int_{[0,1]}\left|\partial_tu\right|\biggr)^2\geq\const\cdot(b-a).
\end{equation}
It follows that there exists $L<\infty$ such that for every interval $[a,b]$ of length $\geq L$, there is some point in $[a,b]\times[0,1]\subseteq[0,\infty)\times[0,1]\subseteq\Sigma$ which gets mapped inside $K$ by $u$.
Now we apply the monotonicity argument in the previous paragraph to conclude.
\end{proof}

Choosing $J_{i_0,\ldots,i_k}$ generically (meaning ``belonging to an unspecified comeagre set'') guarantees all moduli spaces $\Rbar_{k,1}(y;x_1,\ldots,x_k)$ are all cut out transversely.\footnote{Warning: some papers use the word ``generic'' to mean ``achieves transversality''.  This language inevitably causes confusion when trying to express the key (nontrivial!)\ fact that, in favorable cases, ``a generic choice of almost complex structure achieves transversality''.}
We merely sketch the argument as it is standard: perturb by induction on $k$; the induction hypothesis and gluing guarantee transversality over $\Nbd\partial\Rbar_{k,1}$, and then we just perturb over the rest (for more details, see Seidel \cite[II (9k)]{seidelbook}).

Counting holomorphic disks in the dimension zero part of $\Rbar_{k,1}(y;x_1,\ldots,x_k)$ defines operations
\begin{equation}
CF^\bullet(L_{i_0},L_{i_1})\otimes\cdots\otimes CF^\bullet(L_{i_{k-1}},L_{i_k})\otimes\oo_{\Rbar_{k,1}}\to CF^\bullet(L_{i_0},L_{i_k})
\end{equation}
which (with suitable signs as in \cite[(12.24)]{seidelbook}) can be interpreted as maps
\begin{equation}
\label{muoperation}
\mu^k:CF^\bullet(L_{i_0},L_{i_1})\otimes\cdots\otimes CF^\bullet(L_{i_{k-1}},L_{i_k})\to CF^\bullet(L_{i_0},L_{i_k})[2-k]
\end{equation}
satisfying the $\ainf$ relations, by the usual arguments considering the boundary of the moduli spaces of dimension one (see \cite[(12d,12g)]{seidelbook}, in particular Proposition 12.3 therein, for a discussion of signs).

For future reference, we record here the definition of the action functional for holomorphic maps $u:\Sigma\to X$ with ``moving Lagrangian boundary conditions'', meaning that $u(s)\in L_s$ for a family of cylindrical exact Lagrangians $L_s\subseteq X$ for $s\in\partial\Sigma$.
Specifically, the following expression for the geometric energy of such a map will be crucial:
\begin{equation}\label{geoenergymoving}
E^\geo(u):=\int_\Sigma u^\ast d\lambda=\sum_ia(x_i^+)-\sum_ia(x_i^-)+\int_{\partial\Sigma}H_s(u)
\end{equation}
where the Hamiltonians $H_s:L_s\to T^\ast_s\partial\Sigma$ for $s\in\partial\Sigma$ are defined by the property that, choosing $f_s:L_s\to\RR$ satisfying $df_s=\lambda|_{L_s}$ with respect to which we define the action $a$, we have $H=\lambda|_{\mathcal L}-df$ as $1$-forms on $\mathcal L:=\bigcup_{s\in\partial\Sigma}L_s$ (of course, this means the total space of the abstract family mapping to $X$, not a union of subsets of $X$).
This $H_s$ is also the Hamiltonian function generating the isotopy, meaning that $X_{H_s(\frac\partial{\partial s})}=\frac{\partial L_s}{\partial s}\in TX/TL_s$.

\subsection{Floer cohomology and continuation elements}\label{continuationsection}

Let $X$ be a Liouville sector.
As a consequence of the discussion in \S\ref{wcurvessec}, we have Floer cohomology groups $HF^\bullet(L_1,L_2)$ for any transverse exact Lagrangians $L_1,L_2\subseteq X$ (cylindrical at infinity).
We also have composition maps $HF^\bullet(L_1,L_2)\otimes
HF^\bullet(L_2,L_3)\to HF^\bullet(L_1,L_3)$ when $L_1,L_2,L_3$ are mutually
transverse.  Let us first observe that $HF^\bullet$ satisfies a locality
property with respect to inclusions of Liouville sectors:

\begin{lemma}\label{lemma:localityHF}
    If $X\hookrightarrow X'$ is an inclusion of Liouville sectors, then the
    Floer cohomology of a pair of Lagrangians inside $X$ is canonically
    isomorphic to their Floer cohomology taken inside $X'$.
\end{lemma}

\begin{proof}
    Choose a $J$ on $X$ as in \S\ref{wcurvessec}, and extend it to $X'$. Now Lemma
    \ref{pibarrierlemma}
    implies that the relevant holomorphic disks (in $X'$) with boundary on the Lagrangians (in $X$) must lie entirely inside $X$, since $J$ was chosen to make $\pi:
    \Nbd^Z \partial X \to \CC$ holomorphic.
\end{proof}

Let us also observe that $HF^\bullet(L,K)$ satisfies a sort of deformation
invariance property with respect to some (non-compactly supported) deformations.

\begin{lemma}\label{simultaneousdeformation}
For any simultaneous deformation of cylindrical exact Lagrangians $(L_t,K_t)$ with $L_0\pitchfork K_0$ and $L_1\pitchfork K_1$ such that $L_t$ and $K_t$ are disjoint at infinity for all $t\in[0,1]$, there is an induced isomorphism $HF^\bullet(L_0,K_0)=HF^\bullet(L_1,K_1)$.
\end{lemma}

\begin{proof}
Any such deformation may be factored as a composition of a global Hamiltonian
isotopy of $X$ (fixed near $\partial X$, cylindrical at infinity) followed by a
compactly supported deformation of each of the Lagrangians. The invariance of
$HF^\bullet$ with respect to either of these types of deformations is standard.
More precisely, both types of deformations induce Floer theoretically defined continuation maps between Floer cohomology groups, whose actions on cohomology are canonical and intertwine composition of isomorphisms with concatenation of isotopies.
\end{proof}

To understand the relationship between $HF^\bullet(L_0, K_0)$ and $HF^\bullet(L_1, K_1)$ when $L_t$ and $K_t$ intersect near
$\infty$ for some $t$, we need the notion of a \emph{positive} isotopy $L_t$ (and, correspondingly, a \emph{negative} isotopy $K_t$).

\begin{definition}
An isotopy of Lagrangians $L_t$ is called \emph{positive} (resp.\ \emph{non-negative}, \emph{negative}, \emph{non-positive}) near infinity iff for some (equivalently, any) contact form $\alpha$ on $\partial_{\infty} X$, we have $\alpha(\partial_t\partial_{\infty} L_t) > 0$ (resp.\ $\geq 0$, $<0$, $\leq 0$) for all $t$.
\end{definition}

Let us write $HF^\bullet(L,L)$ to mean $HF^\bullet(L^+,L)$ where $L^+$ denotes an unspecified sufficiently small transverse pushoff of $L$, which is positive near infinity.
Lemma \ref{simultaneousdeformation} implies that $HF^\bullet(L,L)$ is independent of the choice of $L^+$ up to canonical isomorphism (so the notation $HF^\bullet(L,L)$ is justified).
We also have that $HF^\bullet(L,L)$ is an associative algebra, and that $HF^\bullet(L,K)$ and $HF^\bullet(K,L)$ are left- and right-modules over $HF^\bullet(L,L)$, respectively.
In fact, as we now show, $HF^\bullet(L,L)$ is unital and $HF^\bullet(L,K)$ and $HF^\bullet(K,L)$ are unital modules over it.

\begin{proposition}\label{wrappingunits}
The associative algebra $HF^\bullet(L,L)$ has a unit, and the modules $HF^\bullet(L,K)$ and $HF^\bullet(K,L)$ over it are unital.
\end{proposition}

\begin{proof}
This is well-known for Liouville manifolds,
and we may deduce the result for Liouville sectors from the case of Liouville manifolds as follows.
Namely, if $X\hookrightarrow X'$ is an inclusion of Liouville sectors, then the result for $X'$ implies the result for $X$, and every Liouville sector $X$ embeds into its convex completion $\bar X$ (see \S\ref{horizcomplete}) which is a Liouville manifold.

Alternatively, the proof for Liouville manifolds can be implemented directly in the Liouville sector setting.
The only subtle point in this adaptation is to show compactness, namely to show that the relevant disks are bounded \emph{a priori} away from infinity.
Recall that one considers holomorphic disks with non-negatively (in the clockwise direction) moving Lagrangian boundary conditions (and note that Proposition \ref{wcompactness} concerns only fixed Lagrangian boundary conditions).
We consider almost complex structures which are of contact type near infinity over a sufficiently large compact subset $V$ of the interior of $\partial_\infty X$ (compare \S\ref{productpseudoconvexinterpolation}) containing the moving Lagrangian boundary conditions.
Now the geometric energy of such disks is bounded by their topological energy since the Lagrangians move non-negatively at infinity (see \eqref{geoenergymoving}).
The arguments from Proposition \ref{wcompactness} based on monotonicity thus imply that such disks are bounded away from infinity, except possibly near the region swept out by the moving Lagrangian boundary conditions (which we do not claim have bounded geometry, compare Lemma \ref{domainacsbddgeo}).
This locus is contained in $V$, and hence the maximum principle applies to deal with it as well.

Yet a third possible argument would be to adapt the methods of Groman \cite{groman} based solely on monotonicity, as we will do in \S\ref{secsymplecticcohomology} to define symplectic cohomology of Liouville sectors and in \S\ref{secopenclosed} to define the open-closed map.
\end{proof}

\begin{remark}\label{hfmorse}
The proof(s) of Proposition \ref{wrappingunits} generalize directly to define a map $H^\bullet(L)\to HF^\bullet(L,L)$ sending the unit $[L]\in H^\bullet(L)$ to the unit $\1_L\in HF^\bullet(L,L)$ (e.g.\ by counting holomorphic disks with moving Lagrangian boundary conditions and one point constraint on $L$).
This map is well-known to be an isomorphism (it should moreover be an algebra map, though this is technically more difficult to show; see Fukaya--Oh \cite{fukayaoh} and Abouzaid \cite{abouzaidplumbing} for proofs in closely related settings).
\end{remark}

The existence of a unit inside $HF^\bullet(L,L)$ allows us to define \emph{continuation elements}.
\begin{definition}
For any positive isotopy $L_t$, we associate a \emph{continuation element}
\begin{equation}
c(L_t)\in HF^\bullet(L_1,L_0)
\end{equation}
defined as follows.
For sufficiently small isotopies, $c(L_t)$ is simply the the image of the unit under the deformation isomorphism $HF^\bullet(L_0,L_0):=HF^\bullet(L_0^+,L_0)=HF^\bullet(L_1,L_0)$.
For an arbitrary isotopy, break it up into smaller isotopies, and define $c(L_t)$ as the composition of the continuation elements for each of these smaller isotopies.
\end{definition}

\begin{lemma} \label{continuationproperties}
The continuation elements satisfy the following properties:
\begin{itemize}
\item\emph{(Well-definedness)} $c(L_t)$ is independent of how the given isotopy is divided into small isotopies.
\item\emph{(Deformation invariance)} $c(L_t)$ is invariant under deformation of the isotopy $L_t$ fixed at endpoints.
\item\emph{(Composition)} A composition of continuation elements coincides with the continuation element of the concatenation of isotopies.
\item\emph{(Naturality)} Continuation elements are preserved by inclusions of Liouville sectors.
Namely, if $L_t$ is a positive isotopy contained in $X \subseteq X'$, then $c(L_t)$ does not depend on whether it is computed in $X$ or $X'$ (with respect to the corresponding invariance of $HF^\bullet(L_1, L_0)$ from Lemma \ref{lemma:localityHF}).
\item For any positive isotopy $L_t$ and negative isotopy $K_t$ such that $L_t$ and $K_t$ disjoint at infinity for all $t\in[0,1]$, the map $HF^\bullet(L_0,K_0)\to HF^\bullet(L_1,K_1)$ given by composing with continuation elements on either side coincides with the deformation isomorphism described in Lemma \ref{simultaneousdeformation}.
\end{itemize}
\end{lemma}

\begin{proof}
Given a positive isotopy $L_t$ for $t\in[0,1]$, denote by $c_{s,t}\in HF^\bullet(L_t,L_s)$ the units from Proposition \ref{wrappingunits} for $0\leq s<t\leq 1$ and $\left|s-t\right|$ sufficiently small.
The fact that the $c_{s,t}$ represent the units means that $c_{x,y}\cdot c_{y,z}=c_{x,z}$ for $0\leq x<y<z\leq 1$ and $\left|x-z\right|$ sufficiently small.
It follows that the continuation element $c(L_t):=c_{t_0,t_1}\cdots c_{t_{N-1},t_N}$ is independent of the choice of sufficiently fine partition $0=t_0<\cdots<t_N=1$.
This proves well-definedness.
The proofs of deformation invariance and composition are similar.
Naturality is immediate from the fact that $HF^\bullet$ and its product are preserved under inclusions of Liouville sectors (and hence the units are also preserved).
The proof of the final assertion is similar to the proofs of well-definedness, deformation invariance, and composition, using in addition the unitality of the modules $HF^\bullet(L,K)$ and $HF^\bullet(L,K)$ over $HF^\bullet(L,L)$ from Proposition \ref{wrappingunits}.
\end{proof}

For any positive isotopy $L_t$ and $K$ transverse to $L_0$ and $L_1$, there is an associated \emph{continuation map}, namely composition with $c(L_t)$:
\begin{equation}\label{continuationonHF}
HF^\bullet(L_0, K) \xrightarrow{c(L_t) \cdot  } HF^\bullet(L_1, K).
\end{equation}
In fact, such a continuation map is defined for any non-negative isotopy $L_t$, using the properties from Lemma \ref{continuationproperties} and the fact that a non-negative isotopy $L_0\leadsto L_1$ can be perturbed to a positive isotopy $L_0\leadsto L_1^+$.
These continuation maps \eqref{continuationonHF} compose with each other as expected under concatenation of isotopies.

\subsection{Wrapped Floer cohomology}\label{wrappedhfsection}

Intuitively, the wrapped Floer cohomology $HW^\bullet(L, K)$ is the usual Floer cohomology $HF^\bullet(L^w, K)$ where $L^w$ is the image of $L$ under a large positive Hamiltonian flow.
To avoid choosing a particular such Hamiltonian, the actual definition involves a direct limit.

To this end, for a Lagrangian $L$ (equipped with a $\Spin$ structure), we define the \emph{positive wrapping category} $(L\leadsto-)^+$ as follows.
The objects of $(L\leadsto-)^+$ are isotopies of exact Lagrangians $\phi:L\leadsto L^w$ (equipped with $\Spin$ structures). The morphisms $(\phi:L\leadsto L^w)\to(\phi':L\leadsto L^{w'})$ are homotopy classes of \emph{positive} isotopies of exact Lagrangians $\psi:L^w\leadsto L^{w'}$ (equipped with $\Spin$ structures) such that $\phi\#\psi=\phi'$. We also define the \emph{negative wrapping category} $(L\leadsto-)^-$ in the same manner, save that the morphisms are homotopy classes of \emph{negative} isotopies.
It is the opposite category of the positive wrapping category.

As we will be taking direct limits over wrapping categories, it is crucial to point out that they are filtered (this is a standard elementary fact for which we know no reference).
Recall that a category $\C$ is called \emph{filtered} iff
\begin{enumerate}
\item\label{filteredI}it is non-empty,
\item\label{filteredII}for every $x,y\in\C$, there exists $z\in\C$ and morphisms $x\to z$ and $y\to z$, and
\item\label{filteredIII}for every pair of morphisms $f,g:x\to y$ in $\C$, there exists a morphism $h:y\to z$ such that $h\circ f=h\circ g$.
\end{enumerate}
Direct limits over filtered categories are exact.
Also recall that a functor $F:\C\to\D$ between filtered categories is called \emph{cofinal} iff
\begin{enumerate}
\item\label{cofinalI}for every $d\in\D$ there exists $c\in\C$ and a morphism $d\to F(c)$, and
\item\label{cofinalII}for every pair of morphisms $f,g:d\to F(c)$, there exists a morphism $h:c\to c'$ such that $F(h)\circ f=F(h)\circ g$.
\end{enumerate}
Pulling back a directed system under a cofinal functor preserves the direct limit.
A filtered category $\C$ is said to have \emph{countable cofinality} iff there exists a cofinal functor $\ZZ_{\geq 0}\to\C$.

\begin{lemma}\label{wrappingcatfiltered}
The wrapping category $(L\leadsto-)^+$ is filtered.
\end{lemma}

\begin{proof}
\ref{filteredI} The wrapping category is non-empty since we may take $L^w=L$.

\ref{filteredII} Suppose $L^w$ and $L^{w'}$ are objects of $(L\leadsto-)^+$, meaning $L^w$ is equipped with an isotopy to $L$, and the same for $L^{w'}$.
Choose isotopies $L^w_t$ and $L^{w'}_t$ (parameterized by $t\in[0,1]$) starting at $L^w_0=L^w$ and $L^{w'}_0=L^{w'}$ ending at the same Lagrangian $L^w_1=L^{w'}_1$, such that the resulting isotopies from $L$ to $L^w_1=L^{w'}_1$ are homotopic rel endpoints.
It suffices to show that these isotopies $L^w_t$ and $L^{w'}_t$ can be modified to be positive.
To do this, simply consider $\Phi_tL^w_t$ and $\Phi_tL^{w'}_t$ for $\Phi:[0,1]\to\Ham(X)$ starting at $\Phi_0=\id$ (for the present purpose, the Lie algebra of $\Ham(X)$ consists of those Hamiltonians which are linear at infinity and vanish along with their first derivatives over $\partial X$).
As long as $\Phi_t^{-1}\frac{\partial\Phi_t}{\partial t}$ is sufficiently large (as a section of $T\partial_\infty X/\xi$) over $L^w_t$ and $L^{w'}_t$, these modified paths are positive.
Fix a one-parameter subgroup $\Phi^\circ:\RR_{\geq 0}\to\Ham(X)$ which wraps positively over the subset of $\partial_\infty X$ swept out by $L^w_t$ and $L^{w'}_t$ for $t\in[0,1]$.
It suffices now to take $\Phi_t=\Phi^\circ_{Nt}$ for sufficiently large $N<\infty$, since then $\Phi_t^{-1}\frac{\partial\Phi_t}{\partial t}=N\frac{\partial\Phi_t^\circ}{\partial t}|_{t=0}$.

\ref{filteredIII} Suppose we are given two positive isotopies from an object $L^w$ to another $L^{w'}$ which are homotopic as (not necessarily positive) isotopies.
Denote this situation by $L_{s,t}$ where $L_{s,0}=L^w$, $L_{s,1}=L^{w'}$, and $L_{0,t}$ and $L_{1,t}$ are the given positive isotopies.
As above, for sufficiently positive $\Phi:[0,1]\to\Ham(X)$ starting at $\Phi_0=\id$, the isotopies $\{\Phi_tL_{s,t}\}_{t\in[0,1]}$ are positive for all $s\in[0,1]$.
Now note that $\Phi_tL_{0,t}$ and $\Phi_tL_{0,t}$ are homotopic to the compositions of the given positive isotopies $L_{0,t}$ and $L_{1,t}$ with the positive isotopy $\Phi_tL^{w'}$.
\end{proof}

\begin{remark}\label{cofinaltransversality}
Note that the collection of Lagrangians satisfying any given countable collection of transversality conditions is cofinal in the wrapping category $(L\leadsto-)^+$ by general position arguments.
In particular, the full subcategory consisting of such Lagrangians is also filtered.
\end{remark}

A positive isotopy $\{L_t\}_{t\geq 0}$ starting at $L=L_0$ gives rise to a functor $\RR_{\geq 0} \to (L\leadsto-)^+$ given by $s \mapsto (L\leadsto L_s)$ (the isotopy $\{L_t\}_{0 \leq t \leq s}$).
We now give a criterion under which (the functor associated to) such an isotopy is cofinal:

\begin{lemma}\label{cofinalintegrationcriterion}
Let $\{L_t\}_{t\geq 0}$ be a positive Lagrangian isotopy of $L=L_0$.
If there exists a contact form $\alpha$ on the interior of $\partial_\infty X$ such that
\begin{equation}\label{cofinalintegral}
\int_0^\infty\min_{\partial_\infty L_t}\alpha(\partial_t\partial_\infty L_t)\,dt=\infty,
\end{equation}
then $\{L_t\}_{t\geq 0}$ is cofinal in wrapping category $(L\leadsto-)^+$.
\end{lemma}

\begin{proof}
We begin by arguing that for any contact form $\alpha$, there exists a function $f:(\partial_\infty X)^\circ\to\RR_{\geq 1}$ such that the Reeb flow of $f\alpha$ is complete (note that since $f\geq 1$, replacing $\alpha$ with $f\alpha$ preserves divergence of the integral \eqref{cofinalintegral}).
To ensure that the Reeb flow of $f\alpha$ is complete, it suffices to fix an infinite sequence of disjoint shells $S_1,S_2,\ldots\subseteq(\partial_\infty X)^\circ$ separating larger and larger compact subsets of $(\partial_\infty X)^\circ$ from infinity, and choose $f|_{S_i}$ so that it takes at least unit time under the Reeb flow of $f\alpha$ to travel between the inner and outer boundaries of $S_i$ (this minimum time is scaled by $r>0$ if we scale $f|_{S_i}$ by $r>0$).
To produce for any given $L$ a positive isotopy and a contact form satisfying \eqref{cofinalintegral}, simply choose a contact form $\alpha$ with complete Reeb flow, and define $L_t$ by flowing at infinity by the Reeb flow of $\alpha$.

Now suppose $\{L_t\}_{t\geq 0}$ satisfies \eqref{cofinalintegral}, and let us show that it is cofinal.
By the paragraph above, we may assume without loss of generality that the Reeb flow of $\alpha$ is complete.
Let $\Phi:\RR_{\geq 0}\to\Ham(X)$ be a one-parameter subgroup which at infinity corresponds to the Reeb flow of $\alpha$.
By reparameterizing $L_t$, we may assume that $\alpha(\partial_t\partial_\infty L_t)\geq 2$ pointwise on $\partial_\infty L_t$ for all $t\geq 0$.
It follows that $t\mapsto\Phi_t^{-1}L_t$ is a positive path of Lagrangians.
Using this property, we may verify properties \ref{cofinalI} and \ref{cofinalII} of cofinality as follows.

To verify \ref{cofinalI}, we need to produce a morphism in $(L\leadsto-)^+$ from an arbitrary $L\leadsto L^w$ to some $L\leadsto L_t$.
The reversed isotopy $L^w\leadsto L$ gives a positive isotopy $L^w\leadsto\Phi_tL$ for sufficiently large $t$, and now there is a positive isotopy $\Phi_tL\leadsto L_t$ (by positivity of $t\mapsto\Phi_t^{-1}L_t$).

To verify \ref{cofinalII}, we need to show that any two given morphisms $(L\leadsto L^w)\to(L\leadsto L_t)$ coincide after composing with $L_t\leadsto L_s$ for some $s\geq t$.
The argument for \ref{filteredIII} from Lemma \ref{wrappingcatfiltered} shows that two such morphisms coincide after composing with $L_t\leadsto\Phi_NL_t$ for sufficiently large $N$.
Now simply further compose with the positive isotopy $\Phi_NL_t\leadsto L_{t+N}$.
\end{proof}

The proof of Lemma \ref{cofinalintegrationcriterion} produced a contact form $\alpha$ with complete Reeb flow; hence for every $L$, the positive isotopy $\{L_t\}_{t\geq 0}$ defined by flowing under this Reeb vector field satisfies \eqref{cofinalintegral} and hence is cofinal.
Since $\ZZ_{\geq 0} \to \RR_{\geq 0}$ is cofinal, we conclude:

\begin{corollary}
The wrapping category $(L\leadsto-)^+$ has countable cofinality.\qed
\end{corollary}

\begin{remark}\label{bdrycofinal}
Note that the hypothesis of Lemma \ref{cofinalintegrationcriterion} is satisfied for trivial reasons if $L_t$ approaches $\partial(\partial_\infty X)$ as $t\to\infty$ (in the sense that for every compact subset $K$ of the interior of $\partial_\infty X$, there exists $T<\infty$ such that $L_t\cap K=\varnothing$ for all $t\geq T$), since we are free to make $\alpha$ grow as fast as we like near $\partial(\partial_\infty X)$.
\end{remark}

To define wrapped Floer cohomology, we consider the following covariant functor defined on the positive wrapping category of $L$:
\begin{equation}\label{wrappingfunctorhf}
(\phi: L \leadsto L^w) \mapsto HF^\bullet(L^w, K)
\end{equation}
in which to a given positive isotopy $\psi:L^w \leadsto L^{w'}$, we associate the map $HF^\bullet(L^w, K)\to HF^\bullet(L^{w'}, K)$ given by multiplication by the continuation element $c(\psi)\in HF^\bullet(L^{w'},L^w)$.
The composition property for continuation elements from Lemma \ref{continuationproperties} implies that this defines a functor.
Strictly speaking, this functor \eqref{wrappingfunctorhf} is defined only on the full subcategory spanned by those $L^w$ which are transverse to $K$, however we will elide this point in the present discussion (recall Remark \ref{cofinaltransversality}).

Wrapped Floer cohomology is defined as the direct limit
\begin{equation}\label{hwdirectlimit}
HW^\bullet(L,K):=\varinjlim_{(L\leadsto L^w)^+}HF^\bullet(L^w,K).
\end{equation}
Wrapping the first factor forwards yields the same cohomology groups as wrapping the second factor backwards, or 
as doing both.  That is, the following maps are both isomorphisms:
\begin{equation}
\varinjlim_{(L\leadsto L^w)^+}HF^\bullet(L^w,K)\xrightarrow\sim\varinjlim_{\begin{smallmatrix}(L\leadsto L^w)^+\\(K\leadsto K^w)^-\end{smallmatrix}}HF^\bullet(L^w,K^w)\xleftarrow\sim\varinjlim_{(K\leadsto K^w)^-}HF^\bullet(L,K^w).
\end{equation}
From the middle direct limit, it is apparent 
that isotopies $L_t$ and $K_t$ induce isomorphisms $HW^\bullet(L_0,K_0)=HW^\bullet(L_1,K_1)$.
We may define an associative product on $HW^\bullet$ via
\begin{equation}
\varinjlim_{(L_1\leadsto L_1^w)^+}HF^\bullet(L_1^w,L_2)\otimes\varinjlim_{(L_3\leadsto L_3^w)^-}HF^\bullet(L_2,L_3^w)\to\varinjlim_{\begin{smallmatrix}(L_1\leadsto L_1^w)^+\\(L_3\leadsto L_3^w)^-\end{smallmatrix}}HF^\bullet(L_1^w,L_3^w).
\end{equation}
Wrapped Floer cohomology is covariantly functorial under inclusions of Liouville sectors.
That is, an inclusion of Liouville sectors $X\hookrightarrow X'$ induces a map
\begin{equation}
    HW^\bullet(L,K)_X\to HW^\bullet(L,K)_{X'}
\end{equation}
defined as follows.
First, there is an obvious functor $(L\leadsto-)^+_X\to(L\leadsto-)^+_{X'}$.
Applying $HF^\bullet(-, K)$, note that for any $L^w \subset X$, there are isomorphisms $HF^\bullet(L^w,K)_X=HF^\bullet(L^w,K)_{X'}$ by Lemma \ref{lemma:localityHF}.
Moreoever, for any positive isotopy $\psi: L^w \leadsto L^{w'}$ in $X$, the isomorphisms are compatible with multiplication by the induced continuation maps (in $X$ versus in $X'$) by naturality (Lemma \ref{continuationproperties}).
This defines the map $HW^\bullet(L,K)_X\to HW^\bullet(L,K)_{X'}$.

Similar disc confining arguments establish that the maps $HW^\bullet(L,K)_X\to
HW^\bullet(L,K)_{X'}$ are compatible with composition and identity morphisms,
and so induce a functor between cohomological wrapped Fukaya categories
(or ``wrapped Donaldson--Fukaya categories''). We will describe a chain level
implementation of this functor in the next section.

\begin{remark}\label{cwhomotopydiagram}
The most direct (though not necessarily the easiest) way to upgrade the above discussion to the chain level is as follows.
One defines the ``wrapping $\infty$-category of $L$'' by removing the words ``homotopy class of'' from the definition of the wrapping category of $L$ (this defines a topological category, though technically speaking it is probably more convenient to define it as a quasi-category and work with that instead).
The argument of Lemma \ref{wrappingcatfiltered} shows that the wrapping $\infty$-category of $L$ is filtered in the $\infty$-categorical sense (see \S\ref{inftycatssec}).
One can then define wrapped Floer cochains $CW^\bullet(L,K)$ as the homotopy direct limit of $CF^\bullet(L^w,K)$ over the wrapping $\infty$-category of $L$, provided one upgrades the construction of continuation elements to a construction of coherent continuation cycles (or defines the continuation maps in a different way, e.g.\ by non-negatively moving Lagrangian boundary conditions).
\end{remark}

\begin{lemma}\label{wrappedtrivialiso}
A trivial inclusion of Liouville sectors induces an isomorphism on $HW^\bullet$.
\end{lemma}

\begin{proof}
Fix $L,K\subseteq X$, let $X\hookrightarrow X'$ be a trivial inclusion of Liouville sectors, and let us show that $HW^\bullet(L,K)_X\to HW^\bullet(L,K)_{X'}$ is an isomorphism.

Fix collar coordinates $\partial X'\times\RR_{t\geq 0}$ near $\partial X'$ in $X'$ using the flow of the Hamiltonian vector field of a defining function for $X'$, so that in coordinates this Hamiltonian vector field is $-\frac\partial{\partial t}$.
We can reduce to the case $X=\{t\geq a\}$ as follows.
For any sufficiently small inward shrinking $X\subseteq X'$, we have $\{t\geq a\}\subseteq X\subseteq X'$ for some small $a>0$, and hence the result for $\{t\geq a\}\hookrightarrow X'$ implies surjectivity of the corresponding map for $X\hookrightarrow X'$.
Similarly, there is a chain of inclusions from $X$ into $X'$ into the image of $X$ under the flow of $-\frac\partial{\partial t}$ (which is also the Hamiltonian vector field of a defining function for $X$) for some small time $a>0$, which gives injectivity.
This derives the desired result for sufficiently small inward shrinkings $X\subseteq X'$ from the case of inclusions $\{t\geq a\}\subseteq X'$; moreover, the smallness required is uniform over any compact family of sectors $X'$, from which we may deduce the general case.
It is thus enough to prove the desired result in the special case $X=\{t\geq a\}$ for some (say small) $a>0$.

Fix coordinates on $\partial_\infty X'$ near $\partial(\partial_\infty X')$ as in \S\ref{reebdynamicsboundary}, so $\partial_\infty X=\{t\geq a\}$ and $\partial_\infty X'=\{t\geq 0\}$.
Let $M:\RR_{\geq 0}\to\RR_{\geq 0}$ be admissible as in \S\ref{reebdynamicsboundary}, and let $N:\RR_{\geq a}\to\RR_{\geq 0}$ be such that $N(t+a)$ is admissible and $N\leq M$, with equality for $t$ large.
Extend $M$ and $N$ to all of $X'$ so that they coincide except in these collar coordinates.
Now $L\leadsto\Phi_M^TL$ is cofinal in the wrapping category of $L$ inside $X'$ as $T\to\infty$, and the same for $L\leadsto\Phi_N^TL$ inside $X$.
We claim that the non-negative isotopy of Lagrangians $\Phi_{tM+(1-t)N}^TL$ parameterized by $t\in[0,1]$ stays disjoint from $K$ near infinity (this being assumed to hold for $t=0$).
In other words, we claim that there are no time $T$ Reeb chords of the flow of $tM+(1-t)N$ from $L$ to $K$ for any $t\in[0,1]$, where this is assumed to hold for $t=0$.
This claim follows from the discussion of the dynamics in \S\ref{reebdynamicsboundary}, specifically Proposition \ref{cutofflemma}.
Indeed, the vector field is only changing in a collar neighborhood of the boundary, and chords entering this neighborhood cannot subsequently escape into the rest of $X'$.

To conclude the proof, it suffices to argue that the map $HW^\bullet(L,K)_X\to HW^\bullet(L,K)_{X'}$ is the direct limit of the continuation maps $HF^\bullet(\Phi_N^TL,K)\to HF^\bullet(\Phi_M^TL,K)$ as $T\to\infty$ and appeal to the last part of Lemma \ref{continuationproperties} to see that each of these continuation maps is an isomorphism (specifically, the one from Lemma \ref{simultaneousdeformation}).
Note that since $\Phi_{tM+(1-t)N}^TL$ is a non-negative (as opposed to a positive) isotopy, the relevant continuation map is as defined below \eqref{continuationonHF} (rather than being literally multiplication by a continuation element), however the last part of Lemma \ref{continuationproperties} remains applicable (by perturbing the non-negative isotopy $\Phi_{tM+(1-t)N}^TL$ to a positive isotopy from $\Phi_N^TL$ to a small positive pushoff of $\Phi_M^TL$).
\end{proof}

\subsection{Wrapped Fukaya category}

The \emph{wrapped Fukaya category} is an $\ainf$-category whose objects are Lagrangians (equipped with $\Spin$ structures) and
the cohomology of whose morphisms spaces is wrapped Floer cohomology.  One way to define such a category
would be to lift the constructions from the previous subsections to the chain level as sketched in Remark \ref{cwhomotopydiagram}.   We follow instead an approach of Abouzaid--Seidel \cite{abouzaidseidelunpublished} which avoids this by a clever use of the theory of localization of $\ainf$-categories recalled in \S\ref{alocal}.

Let $X$ be a Liouville sector.
Fix a collection of (not necessarily mutually transverse) cylindrical exact Lagrangians in $X$ (with $\Spin$ structures), indexed by a countable (possibly finite) set $I$.
We suppose also that $I$ contains at least one representative in every isotopy class of exact Lagrangian equipped with a $\Spin$ structure (without this assumption, the only difference is that the category we get will depend, of course, on which isotopy classes are present).

\begin{definition}[Poset $\OO$]\label{posetOdef}
For each Lagrangian $L\in I$, choose a cofinal sequence
\begin{equation}
L=L^{(0)}\leadsto L^{(1)}\leadsto L^{(2)} \leadsto\cdots
\end{equation}
of morphisms in the wrapping category $(L\leadsto-)^+$.
Let $\OO:=\ZZ_{\geq 0}\times I$ be the set of all such $L^{(i)}$, and equip $\OO$ with the partial order inherited from the order on $\ZZ_{\geq 0}$, namely $L^{(i)}<K^{(i')}$ iff $i<i'$.
By choosing the $L^{(i)}$ generically, we ensure that any \emph{finite totally ordered collection} of Lagrangians in $\OO$ are mutually transverse.
\end{definition}

\begin{definition}[$\ainf$-category $\OO$]\label{categoryOdef}
We turn $\OO$ into a strictly unital $\ainf$-category with the following morphism spaces:
\begin{equation}
\OO(L_0,L_1):=\begin{cases}CF^\bullet(L_0,L_1)&L_0>L_1\\\ZZ&L_0=L_1\\0&\text{otherwise.}\end{cases}
\end{equation}
To define the operations
\begin{equation}
\mu^k:\OO(L_0,L_1)\otimes\cdots\otimes\OO(L_{k-1},L_k)\to\OO(L_0,L_k)[2-k]
\end{equation}
for $L_0>\cdots>L_k\in\OO$,
we count holomorphic disks using compatible choices of universal strip-like coordinates and families of almost complex structures
\begin{align}
\label{OcoordsI}\xi_{L_0,\ldots,L_k;j}^+:[0,\infty)\times[0,1]\times\Rbar_{k,1}&\to\Sbar_{k,1}\quad j=1,\ldots,k\\
\label{OcoordsII}\xi_{L_0,\ldots,L_k}^-:(-\infty,0]\times[0,1]\times\Rbar_{k,1}&\to\Sbar_{k,1}\\
\label{Oacs}J_{L_0,\ldots,L_k}:\Sbar_{k,1}&\to\J(X)
\end{align}
as in \S\ref{wcurvessec}.
Note that the only other possibly nontrivial operations $\mu^k$, namely when some $L_i=L_{i+1}$, are formally fixed by strict unitality.
\end{definition}

\begin{definition}[$\ainf$-category $\W$]\label{Cdef}
Denote by $C$ the set of all continuation elements $c\in HF^0(L^{(i+1)},L^{(i)})$ as defined in \S\ref{continuationsection}, for all $L^{(i)} \in \OO$.
Define the \emph{wrapped Fukaya category of $X$} to be the localized category $\W:=\OO[C^{-1}]$ (recall Definition \ref{localizationdef}).
\end{definition}

This category $\W$ depends on a number of choices as specified above (the collection of Lagrangians $I$, the poset $\OO$, the almost complex structures, etc.).
We will see in Proposition \ref{lem:winvariant} that $\W$ is well-defined up to quasi-equivalence.

\begin{lemma}[Abouzaid--Seidel \cite{abouzaidseidelunpublished}]\label{wrappingcalculatesW}
The natural maps
\begin{equation}
HW^\bullet(L,K)=\varinjlim_iH^\bullet\OO(L^{(i)},K)\xrightarrow\sim\varinjlim_iH^\bullet\W(L^{(i)},K)\xleftarrow\sim H^\bullet\W(L,K)
\end{equation}
are both isomorphisms.
\end{lemma}

\begin{proof}
Since the continuation map $L^{(i+1)}\to L^{(i)}$ is in $C$, each map $\W(L^{(i)},K)\to\W(L^{(i+1)},K)$ is a quasi-isomorphism (by Lemma \ref{quotientannihilatesA}), and hence the rightmost arrow is an isomorphism.

To show that the middle arrow is an isomorphism, it is enough (by Lemma \ref{localgoodpro}) to show that ``$\varprojlim_iL^{(i)}$ is left $C$-local'', meaning that every continuation map $M^{(j+1)}\to M^{(j)}$ induces an isomorphism
\begin{equation}
\varinjlim_iHF^\bullet(L^{(i)},M^{(j+1)})\to\varinjlim_iHF^\bullet(L^{(i)},M^{(j)}).
\end{equation}
This map is simply $HW^\bullet(L,M^{(j+1)})\to HW^\bullet(L,M^{(j)})$, which we know is an isomorphism.
\end{proof}

\begin{corollary}\label{wrappingcalculatesanylocalM}
For any left $\OO$-module $\M$, the natural maps
\begin{equation}
\varinjlim_iH^\bullet\M(L^{(i)})\xrightarrow\sim\varinjlim_iH^\bullet{}_{C^{-1}}\M(L^{(i)})\xleftarrow\sim H^\bullet{}_{C^{-1}}\M(L^{(0)})
\end{equation}
are both isomorphisms.
\end{corollary}

\begin{proof}
We may replace $\M$ with $\OO\otimes_\OO\M$ (Lemma \ref{identityidempotent}).
The maps in question then become
\begin{equation}
\varinjlim_iH^\bullet(\OO(L^{(i)},-)\otimes_\OO\M)\rightarrow\varinjlim_iH^\bullet(\W(L^{(i)},-)\otimes_\OO\M)\leftarrow H^\bullet(\W(L^{(0)},-)\otimes_\OO\M)
\end{equation}
which are isomorphisms by Lemma \ref{wrappingcalculatesW}.
\end{proof}

The construction of $\W$ given above can be made a bit more flexible, as we now describe.
We will take advantage of this flexibility both to show that $\W$ is well-defined up to quasi-equivalence and to define pushforward functors on the wrapped Fukaya category for inclusions of Liouville sectors.

Let $\OO$ be any countable poset of Lagrangians (equipped with $\Spin$ structures) in $X$ containing every isotopy class and with the property that every totally ordered collection of Lagrangians in $\OO$ is mutually transverse.
Let $C$ be any collection of elements of $HF^\bullet(L,K)$ for $L>K\in\OO$ consisting only of continuation elements for various positive isotopies from $K$ to $L$.
Now assume that the pair $(\OO,C)$ satisfies the following property:
\begin{itemize}
\item For every $L\in\OO$, there is a cofinal sequence $L=L^{(0)}<L^{(1)}<\cdots$ in $\OO$ together with positive isotopies $L^{(i)}\leadsto L^{(i+1)}$ such that $L=L^{(0)}\leadsto L^{(1)}\leadsto\cdots$ is cofinal in the Lagrangian wrapping category of $L$ and each of the resulting continuation elements in $HF^\bullet(L^{(i+1)},L^{(i)})$ is in $C$.
\end{itemize}
We may construct strip-like coordinates $\underline\xi$ and almost complex structures $\underline J$ as in \eqref{OcoordsI}--\eqref{Oacs} achieving transversality by induction on the collection of finite totally ordered subsets of $\OO$, thus turning $\OO$ into a directed $\ainf$-category.
We can now define the wrapped Fukaya category of $X$ as the localization $\W:=\OO[C^{-1}]$ as before.
Note that the proof of Lemma \ref{wrappingcalculatesW} applies to this category $\W$, so it has the correct cohomology category.

\begin{proposition}\label{lem:winvariant}
The category $\W$, as defined just above in terms of choices of $(\OO,C,\underline\xi,\underline J)$, is well-defined up to quasi-equivalence.
\end{proposition}

By ``are quasi-equivalent'', we mean ``are connected by a zig-zag of quasi-equivalences'' (we do not wish to discuss the question of whether quasi-equivalences of $\ainf$-categories over $\ZZ$ are invertible under our cofibrancy assumptions).

\begin{proof}
If $\OO\subseteq\OO'$ and $C\subseteq C'$ such that $\underline\xi'|_\OO=\underline\xi$ and $\underline J'|_\OO=\underline J$, then there is a natural functor $\W\to\W'$.
By Lemma \ref{wrappingcalculatesW}, this functor is a quasi-equivalence.

It thus suffices to show thay any two quadruples $(\OO,C,\underline\xi,\underline J)$ and $(\OO',C',\underline\xi',\underline J')$ can be included into a third.
Consider the disjoint union $\OO\sqcup\OO'$ with no order relations between $\OO$ and $\OO'$.
Now let $\OO'':=\ZZ_{\geq 0}\times(\OO\sqcup\OO')$ with its lexicographical partial order, choosing for every Lagrangian in $\OO\sqcup\OO'$ a cofinal sequence in its Lagrangian wrapping category.
Let $C''$ denote the union of the resulting continuation elements together with $C$ and $C'$.
The given strip-like coordinates and almost complex structures for $\OO$ and $\OO'$ can be extended to the same for $\OO''$ by induction.
We thus obtain quasi-equivalences $\W\xrightarrow\sim\W''\xleftarrow\sim\W'$ as desired.
\end{proof}

It is natural to expect that $\W(X)$ satisfies a K\"unneth formula as we formulate below.
A careful proof of this is, however, beyond the scope of this paper (for Liouville manifolds, results in this direction have been proven by Gao \cite{gaokunneth,gaokunneth2}).

\begin{conjecture}
There is a natural (cohomologically) fully faithful bilinear $\ainf$-functor $\W(X)\times\W(X')\hookrightarrow\W(X\times X')$.
\end{conjecture}

\subsection{Inclusion functors and deformation invariance}

Let $X\hookrightarrow X'$ be an inclusion of Liouville sectors.
Fix collections of Lagrangians $I$ and $I'$ inside $X$ and $X'$, respectively (not necessarily mutually transverse), containing all isotopy classes. Using the flexibility granted by Proposition \ref{lem:winvariant}, we can pick a ``nice'' (adapted to $X \subseteq X'$) chain model of $\W(X')$ in order to realize (in a particularly simple way) the inclusion functor $\W(X) \to \W(X')$, as we now describe.

Define the poset of Lagrangians $\OO:=\ZZ_{\geq 0}\times I$ following Definition \ref{posetOdef} as before.
Define $\OO':=\ZZ_{\geq 0}\times[I'\sqcup\OO]$, meaning that for each Lagrangian $L\in(I'\sqcup\OO)$, we choose a cofinal sequence of morphisms in its wrapping category inside $X'$.
We equip $\OO'$ with the partial order defined lexicographically, where the second factor $I'\sqcup\OO$ has only the order relations coming from $\OO$.
By choosing the cofinal sequences generically, we can ensure that any finite totally ordered subset of $\OO'$ consists of mutually transverse Lagrangians.
There is a natural inclusion $\OO\hookrightarrow\OO'$.

We turn $\OO$ and $\OO'$ into $\ainf$-categories following Definition \ref{categoryOdef} as before.
More precisely, we first choose almost complex structures \eqref{Oacs} for $\OO$.
We then choose strip-like coordinates and almost complex structures for $\OO'$ whose restriction to $\OO$ and $X$ are those for $\OO$; the inductive construction of these follows \S\ref{wcurvessec}.
Since $\pi$ forces holomorphic disks for $\OO$ to stay within $X$, the inclusion $\OO\hookrightarrow\OO'$ is the inclusion of a full subcategory.

Let $C^\circ$ denote the continuation morphisms for $\OO$, and let $C^{\prime\circ}$ denote the continuation morphisms for $\OO'$, following Definition \ref{Cdef}.
Now let $C:=C^\circ$ and $C':=C^{\prime\circ}\cup C^\circ$.
We set $\W(X):= \W:=\OO[C^{-1}]$ and $\W(X'):=\W':=\OO'[C^{\prime-1}]$.
Note that the proof of Lemma \ref{wrappingcalculatesW} applies just as well to $C'$ as it does to $C^{\prime\circ}$.
Since $C\subseteq C'$, there is a functor 
\begin{equation}
    \W(X)\to\W(X').
\end{equation}
We remark that this functor $F=\{F^k\}_{k\geq 1}$ is somewhat special: the map on objects is injective, the map $F^1$ on morphism complexes is injective, and the maps $F^k$ vanish for $k\geq 2$.

\begin{lemma}\label{wtriv}
Let $X\hookrightarrow X'$ be a trivial inclusion.  Then the functor $\W(X)\to\W(X')$ is a quasi-equivalence.
\end{lemma}

\begin{proof}
It follows from Lemmas \ref{wrappingcalculatesW} and \ref{wrappedtrivialiso} that this functor is a quasi-equivalence onto its image.

To show essential surjectivity, note first that every exact Lagrangian isotopy $L_t$ induces an identification between the Lagrangian wrapping categories of $L_0$ and $L_1$, and thus an isomorphism between $L_0$ and $L_1$ in $\W$.
Hence it is enough to show that every exact Lagrangian in $X'$ is isotopic to an exact Lagrangian inside $X$.
It is enough to show this for sufficiently small trivial inclusions, where it follows by flowing under $-X_I$ where $I:\Nbd^Z\partial X'\to\RR$ is a defining function.
\end{proof}

\subsection{Diagram of wrapped Fukaya categories}\label{wrappedfulldiagram}

Now suppose we have a diagram of Liouville sectors
$\{X_\sigma\}_{\sigma\in\Sigma}$ indexed by a finite poset $\Sigma$.  Fix
collections of Lagrangians $I_\sigma$ inside $X_\sigma$ (not necessarily
mutually transverse) containing all isotopy classes. As in the previous
section, we will make convenient
adapted choices of chain models of the categories $\W(X_{\sigma})$ in order to
obtain an associated (strict) diagram of $\ainf$-categories.

We inductively define posets of Lagrangians
\begin{equation}\label{oosigmasplitting}
\OO_\sigma:=\ZZ_{\geq 0}\times\Bigl[I_\sigma\sqcup\colim_{\sigma'<\sigma}\OO_{\sigma'}\Bigr].
\end{equation}
In other words, the set of Lagrangians $\OO_\sigma$ is defined by
(1) starting with all Lagrangians comprised in $\OO_{\sigma'}$ for $\sigma'<\sigma$,
(2) adding the chosen Lagrangians $I_\sigma$ inside $X_\sigma$, and then
(3) choosing cofinal sequences in the wrapping categories of each of these Lagrangians so that every totally ordered collection of Lagrangians in $\OO_\sigma$ are mutually transverse.
To define the partial order on $\OO_\sigma$, equip $I_\sigma$ with the trivial partial order, equip the $\colim$ with the colimit partial order, equip the $\sqcup$ with the coproduct partial order (no additional order relations), and equip the $\times$ with the lexicographical partial order.

To define the $\ainf$ operations on $\OO_\sigma$, we must specify the compatiblity relations we impose on the choices of almost complex structures.  First, note that there is a weakly order preserving map $\OO_\sigma\to\Sigma_{\leq\sigma}$ which associates to an element $L\in\OO_\sigma$ the unique minimal $\sigma'\leq\sigma$ for which $L\in\im(\OO_{\sigma'}\to\OO_\sigma)$.
Thus for any chain $L_0>\ldots>L_k\in\OO_\sigma$, there is an associated chain $\sigma_{L_0}\geq\cdots\geq\sigma_{L_k}\in\Sigma_{\leq\sigma}$.
We choose strip-like coordinates and almost complex structures
\begin{align}
\label{WcoordsI}\xi_{L_0,\ldots,L_k;j}^+:[0,\infty)\times[0,1]\times\Rbar_{k,1}&\to\Sbar_{k,1}\quad j=1,\ldots,k\\
\label{WcoordsII}\xi_{L_0,\ldots,L_k}^-:(-\infty,0]\times[0,1]\times\Rbar_{k,1}&\to\Sbar_{k,1}\\
\label{Wacs}J_{L_0,\ldots,L_k}:\Sbar_{k,1}&\to\J(X_{\sigma_{L_0}})
\end{align}
for $L_0>\cdots>L_k\in\OO_\sigma$ which are compatible in the natural way and which make $\pi_{\sigma_{L_0}}$ holomorphic (as usual $\pi_\sigma:\Nbd^Z\partial X_\sigma\to\CC_{\Re\geq 0}$ denotes the projection associated to the Liouville sector $X_\sigma$, recalling Convention \ref{choosingpi}).
They may be constructed by induction on $\sigma$ as usual (though note that for the inductive step to work, it is crucial that we have taken the target of $J_{L_0,\ldots,L_k}$ to be $\J(X_{\sigma_{L_0}})$ rather than $\J(X_\sigma)$).

Now we have a diagram
\begin{align}
\Sigma&\to\Ainftycat\\
\sigma&\mapsto\OO_\sigma
\end{align}
meaning that for every $\sigma\in\Sigma$, we have an $\ainf$-category $\OO_\sigma$, for every pair $\sigma\leq\sigma'\in\Sigma$, we have an $\ainf$-functor $F_{\sigma'\sigma}:\OO_\sigma\to\OO_{\sigma'}$, and for every triple $\sigma\leq\sigma'\leq\sigma''\in\Sigma$, we have $F_{\sigma''\sigma}=F_{\sigma''\sigma'}\circ F_{\sigma'\sigma}$.\footnote{It would be somewhat better to call this a ``strict diagram'' to contrast it with the notion of a ``homotopy diagram'' in which the functors only compose up to coherent homotopy. Fortunately, this latter notion of a ``homotopy diagram of $\ainf$-categories'' is not needed for this paper.}
Furthermore, each map $\OO_\sigma\to\OO_{\sigma'}$ is simply the inclusion of a full subcategory.

Let $C_\sigma^\circ$ denote the class of continuation morphisms in $H^0\OO_\sigma$ following Definition \ref{Cdef} (i.e.\ maps $(i+1,x)\to(i,x)$ in terms of the product decomposition \eqref{oosigmasplitting}), and let $C_\sigma:=\bigcup_{\sigma'\leq\sigma}C_{\sigma'}^\circ$, so $C_\sigma\subseteq C_{\sigma'}$ for $\sigma\leq\sigma'$.
Let $\W_\sigma:= \W(X_{\sigma}):=\OO_\sigma[C_\sigma^{-1}]$.  There is thus a diagram
\begin{align}
\Sigma&\to\Ainftycat\\
\sigma&\mapsto\W_\sigma
\end{align}
as desired.
Each inclusion functor $\W_\sigma\to\W_{\sigma'}$ is a naive inclusion, namely an inclusion on the level of morphism complexes with no higher order functor operations.

Note that these categories $\W_\sigma$ fall under the scope of Proposition \ref{lem:winvariant}.
An argument similar to the proof of Proposition \ref{lem:winvariant} shows moreover that the entire diagram $\{\W_\sigma\}_{\sigma\in\Sigma}$ is well-defined up to quasi-equivalence.

\subsection{Functorial wrapped Fukaya categories}

We now give a strictly functorial definition of the wrapped Fukaya category of a Liouville sector.
In fact, this definition also applies to define strictly functorial wrapped Fukaya categories of open Liouville sectors in the sense of Remark \ref{opensector}.
There is a trade off between this construction and that from \S\ref{wrappedfulldiagram}: in exchange for strict functoriality over all Liouville sectors at once, we are forced to work with very large (i.e.\ uncountable) collections of Lagrangians (though on the other hand, we no longer need to appeal to the fact that there are only countably many isotopy classes of exact cylindrical Lagrangians).

Given a Liouville sector (or an open Liouville sector) $X$, we consider \emph{decorated posets over $X$}, namely tuples
\begin{equation}
\vec P = (P,\{X_p\}_{p\in P},\{L_p\}_{p\in P},\underline\xi,\underline J)
\end{equation}
where $P$ is a poset and the remaining data is as follows:
\begin{itemize}
\item Each $X_p\subseteq X$ is a Liouville subsector, such that $X_p\subseteq X_{p'}$ for $p\leq p'$.
There is no need to record the data of a Liouville form on $X_p$ possibly differing on a compact set from the restriction of the Liouville form on $X$, however in accordance with Convention \ref{choosingpi} we do record a choice of $\pi_p:\Nbd\partial X_p\to\CC_{\Re\geq 0}$ such that the inclusions $X_p\subseteq X_{p'}$ and $X_p\subseteq X$ conform to Convention \ref{choosingpi}.
In the case that $X$ is an open Liouville sector, the $X_p$ remain ordinary Liouville sectors.
\item Each $L_p\subseteq X_p$ is a Lagrangian, such that every chain $L_{p_0},\ldots,L_{p_k}$ for $p_0>\cdots>p_k\in P$ is mutually transverse.
\item The $\underline\xi$ are choices of universal strip-like coordinates for each chain $p_0>\cdots>p_k\in P$, which are compatible with gluing in the sense of \S\ref{wcurvessec}.
\item The $\underline J$ are a choice of, for each chain $p_0 > \cdots > p_k \in P$, families of almost complex structures on $X_{p_0}$ making $\pi_{p_0}$ holomorphic, compatible with gluing via $\underline\xi$ in the sense of \S\ref{wcurvessec}, such that the associated moduli spaces of Fukaya $\ainf$ disks are cut out transversely.
\end{itemize}
Given any decorated poset $\vec P$, we have a strictly unital directed $\ainf$-category $\OO_{\vec P}$ as in Definition \ref{categoryOdef}.
Namely, its set of objects is $P$ (though of as the Lagrangians $L_p$), its morphism spaces are $\OO_{\vec P}(p,p)=\ZZ$ and $\OO_{\vec P}(p,p')=CF^\bullet(L_p,L_{p'})$ for $p>p'$ (vanishing otherwise), and the $\ainf$ operations count holomorphic disks using the almost complex structures $\underline J$.
We also have a category $\W_{\vec P}:=\OO_{\vec P}[C_{\vec P}^{-1}]$, namely the localization of $\OO_{\vec P}$ at the set $C_{\vec P}$ of all morphisms in $HF^0(L_{p'},L_p)$ for $p<p'$ which are the continuation element associated to some positive isotopy $L_p\leadsto L_{p'}$ inside $X_{p'}$.
We emphasize that $\W_{\vec P}$ will \emph{not} be quasi-equivalent to $\W(X)$ except under additional assumptions on $\vec P$ (a sufficient condition is given in Proposition \ref{lem:winvariant}).
Given an inclusion of decorated posets $\vec P'\hookrightarrow\vec P$ (meaning an inclusion of underlying posets $P'\hookrightarrow P$ such that the decorations on $P'$ are obtained by restricting those on $P$), there are induced functors $\OO_{\vec P'}\to\OO_{\vec P}$ and $\W_{\vec P'}\to\W_{\vec P}$.

Now for any Liouville sector $X$, we would like to argue that there is a \emph{universal} decorated poset $\vec P_X$ over $X$ and that the associated category $\W_{\vec P_X}$ is a model of $\W(X)$.
To turn this into a statement we can prove, we impose two additional conditions on the decorated posets we consider: we require that $P$ must be \emph{cofinite} (meaning that for all $p\in P$, the subposet $P^{\leq p}$ is finite) and must \emph{have no duplicates}, meaning that the decorated posets $\vec P^{\leq p}$ are pairwise non-isomorphic as $p$ ranges over all elements of $P$; we also restrict attention to inclusions of posets $P'\hookrightarrow P$ which are \emph{downward closed} (meaning that if $p\in P$ is in the image, then so is $P^{\leq p}$).
Let $\Pos_X$ denote the category whose objects are decorated posets $\vec P$ over $X$ which are cofinite and without duplicates and whose morphisms are downward closed inclusions (respecting decorations).
The sense in which there is a (necessarily unique up to unique isomorphism) universal decorated poset $\vec P_X$ over $X$ is that:

\begin{lemma}\label{univposet}
The category $\Pos_X$ has a final object $\vec P_X \in \Pos_X$.
\end{lemma}

\begin{proof}
The first observation is that for $\vec P,\vec Q\in\Pos_X$, there is at most one morphism $\vec P\to\vec Q$.
Indeed, note that if $f:P\hookrightarrow Q$ is a downward closed inclusion, then $\vec P^{\leq p}=\vec Q^{\leq f(p)}$.
Since $\vec Q$ is without duplicates, there is at most one $q\in Q$ satisfying $\vec P^{\leq p}\cong\vec Q^{\leq q}$, and thus $f$ is unique if it exists.

This reasoning may be taken further to construct the terminal object $\vec P_X\in\Pos_X$.
Namely, the underlying poset $P_X$ is defined to be the full subcategory of $\Pos_X$ spanned by the objects $\vec Q\in\Pos_X$ which have a maximum $q\in Q$ (meaning $Q=Q^{\leq q}$); we should note that these objects form a \emph{set} (e.g.\ in view of cofiniteness).
For $p\in P_X$, denote by $\vec Q(p)\in\Pos_X$ the corresponding object.
Now for any $p\in P_X$, it is easy to check that $P_X^{\leq p}=Q(p)$.
We define the decorations on $\vec P_X$ by the requirement that $\vec P_X^{\leq p}=\vec Q(p)$ as decorated posets; it is straightforward to check that this defines unique decorations on $P_X$.
The decorated poset $\vec P_X$ is cofinite and without duplicates by definition, so $\vec P_X\in\Pos_X$.

Now given any $\vec Q\in\Pos_X$, we would like to argue that there is a (necessarily unique) map $\vec Q\to\vec P_X$.
There is only one possible choice for this map, namely it must map $q\in Q$ to the element of $P_X$ corresponding to the decorated poset $\vec Q^{\leq q}$, and it is straightforward to check that this does indeed define a map $\vec Q\to\vec P_X$ respecting decorations.
We conclude that $\vec P_X\in\Pos_X$ is a final object, as desired.
\end{proof}

We now define $\W(X):=\W_{\vec P_X}$.  For any inclusion of Liouville sectors
$X\hookrightarrow X'$ there is a tautological functor $\Pos_X\to\Pos_{X'}$ (by observing that any decorated poset over $X$ defines one over $X'$), and hence an induced map between their final objects $\vec P_X\to\vec P_{X'}$, thus inducing a canonical functor $\W(X)\to\W(X')$ (and these functors
compose with each other as expected).  The situation is identical for open
Liouville sectors $X$.

\begin{proposition}
The category $\W_{\vec P_X}$ is quasi-equivalent to the wrapped Fukaya categories defined in Proposition \ref{lem:winvariant}.
\end{proposition}

\begin{proof}
Note that Proposition \ref{lem:winvariant} does not apply directly to the decorated poset $\vec P_X$, since $\vec P_X$ does not have countable cofinality (the existence of cofinal wrapping \emph{sequences} was used in an essential way in the proof of Proposition \ref{lem:winvariant}).
Instead, we will argue using a direct limit over countable subposets of $\vec P_X$ to which Proposition \ref{lem:winvariant} does apply.

For any decorated poset $\vec P$ and any Lagrangians $L,K\in \vec P$ the map
\begin{equation}\label{posetdirlim}
\varinjlim_{\begin{smallmatrix}\{L,K\}\subseteq  \vec Q\subseteq \vec P\\Q \text{ countable}\end{smallmatrix}}\W_{\vec Q}(L,K)\xrightarrow\sim\W_{\vec P}(L,K)
\end{equation}
is an isomorphism, where $\varinjlim$ is the direct limit of chain complexes (all maps in the directed system are simply inclusions of subcomplexes, and they are strictly compatible with each other).
This is of course also true for the direct limit over finite $\vec Q$, however it is the case of countable $\vec Q$ that is relevant for our present purpose.
If $\vec P$ is cofinite, then we may restrict the direct limit \eqref{posetdirlim} to those $\vec Q\subseteq \vec P$ which are downward closed, as these are cofinal.

Now Proposition \ref{lem:winvariant} provides a sufficient condition on countable $\vec Q$ to imply that $\W_{\vec Q}$ models the wrapped Fukaya category of $X$, namely it is sufficient that:
\begin{itemize}
\item For every $q\in\vec Q$ there exists a sequence $q=q_0<q_1<\cdots\in Q$ which is cofinal in $Q$ and positive isotopies $L_q=L_{q_0}\leadsto L_{q_1}\leadsto\cdots$ which are cofinal in the wrapping category of $L_q$.
\end{itemize}
(The condition that $\vec Q$ should contain all isotopy classes of Lagrangians can be safely ignored since we need only to check that the morphism space in $\W_{\vec P_X}$ between a fixed pair of Lagrangians is correct.)
Hence it suffices to show that those $\vec Q$ satisfying this bulleted condition are cofinal in the direct limit \eqref{posetdirlim} for $\vec P=\vec P_X$.
Now downward closed $\vec Q\subseteq\vec P_X$ are the same thing as objects $\vec Q\in\Pos_X$ (i.e.\ cofinite decorated posets without duplicates).
It thus suffices to show that every countable $\vec Q\in\Pos_X$ admits a morphism to (i.e.\ a downward closed inclusion into) a countable $\vec Q'\in\Pos_X$ satisfying the bulleted condition above.

Fix an exhaustion of $Q$ by downward closed finite subsets $Z_0\subseteq Z_1\subseteq\cdots$.
We define $Q ' :=  Q \sqcup \ZZ_{\geq 0}$, equipped with the order induced by the given order on $Q$, the usual order on $\ZZ_{\geq 0}$, along with the declaration that $i \in \ZZ_{\geq 0}$ is greater than all elements of $Z_i \subseteq Q$.
Note that $Q'$ is cofinite and that $Q \hookrightarrow Q'$ is downward closed.

We define the Lagrangians $L_{q'}$ for $q'\in Q'\setminus Q=\ZZ_{\geq 0}$ as follows.
For every $q\in Q$, we choose an order preserving injection $f_q:\ZZ_{\geq 1}\hookrightarrow Q'\setminus Q=\ZZ_{\geq 0}$, such that the images of $f_{q_1}$ and $f_{q_2}$ are disjoint for $q_1\ne q_2$ and such that $q\leq f_q(1)$ in $Q'$ (since $Q$ is countable, such a family of injections $f_q$ may be constructed by induction on any enumeration of $Q$).
We now revise our definition of $Q'$ to $Q':=Q\sqcup\bigcup_{q\in Q}\im f_q\subseteq Q\sqcup\ZZ_{\geq 0}$ (with the restriction of the originally defined partial order).
Now finally, we define the Lagrangians $L_{q'}$ for $q'\in Q'$ by declaring that for $q\in Q$, there should be a cofinal sequence of positive wrappings $L_q\leadsto L_{f_q(1)}\leadsto L_{f_q(2)}\leadsto\cdots$.
By choosing these wrappings generically, we may ensure that all totally ordered subsets of $Q'$ are mutually transverse and that the Lagrangians in $Q'\setminus Q$ are distinct from each other and from the Lagrangians in $Q$ (this ensures that $\vec Q'$ has no duplicates).
By construction, $Q'$ satisfies the bulleted property above.

Finally, note that the Liouville sectors $X_p$ for $p\in Q'\setminus Q$ may be constructed by induction (using crucially that $Q'$ is cofinite), and the strip-like coordinates and almost complex structures may be constructed by induction as in Lemma \ref{stripacsinduct}.
\end{proof}

\subsection{Geometric criterion for properness}

We observe here that if $\partial_\infty X$ is deformation equivalent to a contactization, then $\W(X)$ is proper (compare Lemma \ref{cylinderstopped}).

Recall that a $\ZZ$-graded complex is called \emph{perfect} iff it is quasi-isomorphic to a bounded complex of finitely generated projective modules.
In general, a complex is called perfect iff it is (up to quasi-isomorphism) a direct summand of a finite iterated extension of finite free modules (regarded as complexes concentrated in a single degree with trivial differential).
An $\ainf$-category $\C$ is called \emph{proper} iff $\C(X,Y)$ is perfect for all $X,Y\in\C$.
(Note that \S\S\ref{continuationsection}--\ref{wrappedhfsection} not only define homology groups $HF^\bullet$ and $HW^\bullet$, they in fact define quasi-isomorphism types $CF^\bullet$ and $CW^\bullet$.)

\begin{lemma}\label{wproper}
If $\partial_\infty X$ is deformation equivalent to a contactization $F_0\times[0,1]$, then $\W(X)$ is proper (equivalently, $CW^\bullet(L,K)$ is perfect for all $L,K\subseteq X$).
\end{lemma}

\begin{proof}
By Lemma \ref{wrappedtrivialiso}, a deformation of Liouville domains induces a quasi-isomorphism on $CW^\bullet$.
By Lemma \ref{deformboundary}, deformations of $\partial_\infty X$ lift to deformations of $X$.
So, without loss of generality, we may deform $X$ so that at infinity it is of the form considered in either of the proofs of Lemma \ref{cylinderstopped}.
We consider now the associated cutoff Reeb vector fields on $\partial_\infty X$.
Under the flow of any such cutoff Reeb vector field, any compact subset of $\partial_\infty X$ (in particular $\partial_\infty L$) converges to the boundary of $\partial_\infty X$.
In particular, this gives a cofinal wrapping of $L$ (by Remark \ref{bdrycofinal}) which after finite time never again passes through $\partial_\infty K$ at infinity, and so we conclude by the last property from Lemma \ref{continuationproperties} and Lemma \ref{simultaneousdeformation} that $CW^\bullet(L,K)$ is quasi-isomorphic to $CF^\bullet(L^w,K)$ (which is a perfect complex, as it is action filtered and generated by finitely many intersections points) for some finite wrapping $L\leadsto L^w$.
\end{proof}

\section{Symplectic cohomology of Liouville sectors}\label{secsymplecticcohomology}

For any Liouville sector $X$, we define a symplectic cohomology group $SH^\bullet(X,\partial X)$, and we show that $SH^\bullet(X,\partial X)$ is covariantly functorial with respect to inclusions of Liouville sectors.
The key to the functoriality of $SH^\bullet$ is Lemma \ref{scconfinecurves}, which shows that for an inclusion of Liouville sectors $X\hookrightarrow X'$ and a Hamiltonian $H:X'\to\RR$ adapted to both $X$ and $X'$, a Floer trajectory with input inside $X$ must lie entirely inside $X$.
In fact, we define a cochain complex $SC^\bullet(X,\partial X)_\pi$ functorial in $X$ and $\pi:\Nbd^Z\partial X\to\CC$ as in Definition \ref{cprojdef}, which computes the functor $SH^\bullet(X,\partial X)$ (this chain level information is crucial for the proof of Theorem \ref{localtoglobalnondegenerate}).
The relevant higher homotopical data is defined directly in terms of holomorphic curve counts (as opposed to the quotient category construction of the wrapped Fukaya category given in \S\ref{secwrapped}).

The notion of a ``homotopy coherent diagram'' plays an important role in this section (and the next) to keep track of chain level information.
To formalize this notion, we use (in a very elementary way) the language of quasi-categories (aka $\infty$-categories) introduced by Joyal \cite{joyal} and developed further by Lurie \cite{luriehtt,lurieha}.

Rather than choosing consistent Floer data for all Liouville sectors and all homotopies at once, we prefer to simply take a (homotopy) colimit over the ``space'' of all allowable Floer data.
This allows for more flexibility in this and subsequent constructions, and it is convenient in that ``independence of choice'' is built into the definition itself.
In this framework, it is of course crucial to show that this space of Floer data is contractible in the relevant sense.
In the present ``wrapped'' context, the relevant sense is that the simplicial set of Floer data should be a filtered $\infty$-category.

Proving compactness for the moduli spaces of holomorphic curves we wish to consider is nontrivial, and for this purpose we adapt Groman's \cite{groman} notion of dissipative Floer data (Definition \ref{dissipative}) to our setting.
We also adopt Groman's construction of dissipative Floer data (Proposition \ref{hjcontractible}) and Groman's proof of compactness for dissipative Floer data (Proposition \ref{compactness}).

Convention \ref{choosingpi} will be in effect for the remainder of this section.

\subsection{Moduli spaces of domains}\label{scdomainmodulisec}

\begin{figure}[hbt]
\centering
\includegraphics{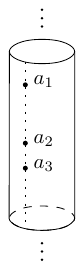}
\caption{Riemann surfaces used to define symplectic cohomology.}\label{domainSC}
\end{figure}

We consider the compactified moduli space of $n$-tuples of points $a_1\geq\cdots\geq a_n\in\RR$ up to translation.
It is helpful to view these points as lying on $\RR\times\{0\}\subseteq\RR\times S^1$ as in Figure \ref{domainSC}.
The points are allowed to collide with each other (meaning $a_i=a_{i+1}$) and no ``bubble'' is considered to have been formed (this makes a difference if at least three collide as in $a_i=a_{i+1}=a_{i+2}$).
On the other hand, if the consecutive spacing $b_i:=a_i-a_{i+1}$ approaches infinity for some $i$, then we regard $\RR\times S^1$ as splitting into two copies of $\RR\times S^1$, the first containing $a_1,\ldots,a_i$ and the second containing $a_{i+1},\ldots,a_n$.
The resulting moduli spaces may be described topologically as
\begin{equation}
\Mbar^{SC}_n=\begin{cases}\pt/\RR&n=0\\ 
    [0,\infty]^{n-1}&n\geq 1
\end{cases}
\end{equation}
using the coordinates $b_1,\ldots,b_{n-1}\in[0,\infty]$.
Denote by $\Cbar^{SC}_n$ the universal curve over $\Mbar^{SC}_n$, so $\Cbar^{SC}_0=S^1$, $\Cbar^{SC}_1=\RR\times S^1$, etc.
The spaces $\Mbar^{SC}_n$ (and correspondingly $\Cbar^{SC}_n$) come with natural inclusions of codimension one boundary strata
\begin{align}
\label{scboundaryI}\Mbar^{SC}_k\times\Mbar^{SC}_{n-k}&\hookrightarrow\Mbar^{SC}_n\\
\label{scboundaryII}\Mbar^{SC}_{n-1}&\hookrightarrow\Mbar^{SC}_n
\end{align}
for $0<k<n$, corresponding to setting $b_k=\infty$ and $b_k=0$, respectively.

There are tautological cylindrical coordinates on $\Cbar^{SC}_n\to\Mbar^{SC}_n$, as $\RR\times S^1$ is itself a cylinder (the translational ambiguity does not concern us, as we always view strip-like/cylindrical coordinates as well-defined up to translation anyway); we point out the obvious fact that these coordinates are compatible with each other in the sense of \S\ref{wcurvessec}.
Gluing via these coordinates defines a collar
\begin{equation}
\label{sccollarI}\Mbar^{SC}_k\times\Mbar^{SC}_{n-k}\times(0,\infty]\hookrightarrow\Mbar^{SC}_n
\end{equation}
covered by a map of universal curves.
Moreover, these collars \eqref{sccollarI} are compatible with each other in the sense that, for any boundary stratum (possibly of higher codimension) of $\Mbar^{SC}_n$, every curve over a neighborhood of the stratum has a well-defined identification with a well-defined curve in the stratum via the above gluing operation.

\begin{figure}[hbt]
\centering
\includegraphics{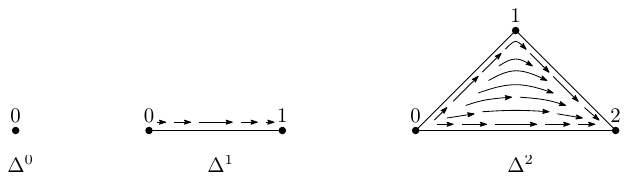}
\caption{The vector field $V_n$ from Remark \ref{adamspathsremark} for $n=0,1,2$.}\label{simplexflowsfigure}
\end{figure}

\begin{remark}\label{adamspathsremark}
The spaces $\Mbar^{SC}_n$ can also be described as spaces of Morse flow lines on $\Delta^n$ for a particular choice of Morse function.
Namely, following \cite[Definition 10.1.4 and \S C.13.1]{pardonimplicitatlas} we consider the gradient-like pair (illustrated in Figure \ref{simplexflowsfigure})
\begin{equation}\label{simplexmorsefunctiongradient}
(F_n,V_n):=\biggl(-\sum_{i=1}^n\cos(\pi x_i),\sum_{i=1}^n\sin(\pi x_i)\frac\partial{\partial x_i}\biggr)
\end{equation}
on the $n$-simplex with coordinates
\begin{equation}
\Delta^n=\{\underline x\in[0,1]^n:0\leq x_1\leq\cdots\leq x_n\leq 1\},
\end{equation}
where the $i$th vertex is given by $x_{n-i}=0$ and $x_{n-i+1}=1$.
The critical locus of $F_n$ consists of the vertices of $\Delta^n$, and the Morse index of vertex $i$ equals $i$.
The vector field $V_n$ is compatible with all simplicial maps $\Delta^n\to\Delta^m$.

The space $\F(\Delta^n)$ of broken flow lines of $V_n$ from vertex $0$ to
vertex $n$ is (a convenient variation on) the ``Adams family of paths''
\cite{adamscobar}.
Namely, we consider maps $\ell:\RR\to\Delta^n$ satisfying $\ell'(s)=-V_n(\ell(s))$ with $\ell(+\infty)=0$ and $\ell(-\infty)=n$.
Every such flow line is of the form $\ell(t)=(f(a_n-t),\ldots,f(a_1-t))$ for
some $a_1\geq\cdots\geq a_n\in\RR$ (unique up to the addition of an overall
constant), where $f:\RR\to[0,1]$ denotes the unique solution to the initial
value problem $f(0)=\frac 12$ and $f'(x)=\sin(\pi f(x))$.  Thus the space of
flow lines is parameterized by $(b_1,\ldots,b_{n-1})\in[0,\infty)^{n-1}$, where
$b_i=a_i-a_{i+1}$; moreover, this parameterization extends continuously to
a homeomorphism
\begin{equation}
[0,\infty]^{n-1}\xrightarrow\sim\F(\Delta^n).
\end{equation}
In these coordinates, $b_k=\infty$ iff the flow line
is broken at vertex $k$, and $b_k=0$ iff the flow line factors through
$\Delta^{[0\ldots\hat k\ldots n]}\subseteq\Delta^n$.  More generally, the
natural inclusions 
\begin{align}
\label{Fproduct}\F(\Delta^{[0\ldots k]})\times\F(\Delta^{[k\ldots n]})&\to\F(\Delta^n)\\
\label{Fface}\F(\Delta^{[0\ldots\hat k\ldots n]})&\to\F(\Delta^n)
\end{align}
admit a simple description in terms of the $b$-coordinates.
In fact, any simplicial map $f:\Delta^n\to\Delta^m$ with $f(0)=0$ and $f(n)=m$ induces a map $f_\ast:\F(\Delta^n)\to\F(\Delta^m)$.
\end{remark}

Remark \ref{adamspathsremark} gives an important conceptual understanding of the meaning of the moduli spaces $\Mbar^{SC}_n$.
A given moduli space $\Mbar^{SC}_n$ should be regarded as associated to an $n$-simplex $\Delta^n$, the points $a_i$ should be regarded as associated to the edges $(i-1)\to i$ of $\Delta^n$, and the intervals $(a_1,\infty),(a_2,a_1),\ldots,(a_n,a_{n-1}),(-\infty,a_n)$ should be regarded as associated to the vertices $0,\ldots,n$ of $\Delta^n$.

\subsection{Hamiltonians and almost complex structures}\label{subsec:nsimplexfloerdata}

We now introduce the technical conditions we impose on Hamiltonians and almost complex structures in order to define symplectic cohomology for Liouville sectors.
We organize the collection of all allowable Floer data into a simplicial set, in which a $0$-simplex specifies Floer data for the differential, a $1$-simplex specifies Floer data for a continuation map, a $2$-simplex specifies Floer data for a homotopy between a continuation map and a composition of two continuation maps, etc.

Let $\H(X)$ denote the space of Hamiltonians $H:X\to\RR$, and recall that $\J(X)$ denotes the space of $\omega$-compatible cylindrical almost complex structures on $X$.

\begin{definition}\label{nsimplexfloer}
An $n$-simplex of Floer data $(H,J)$ on $X$ consists of a collection of maps
\begin{align}
H_{v_0\cdots v_m}:\Cbar^{SC}_m&\to\H(X)\\
J_{v_0\cdots v_m}:\Cbar^{SC}_m&\to\J(X)
\end{align}
for all integers $0\leq v_0<\cdots<v_m\leq n$.
These maps must be \emph{compatible with gluing and forgetting vertices} in the natural way with respect to the boundary collars \eqref{sccollarI} and the inclusions of strata \eqref{scboundaryII}.
Namely, $H_{v_0\cdots v_m}$ must agree with $H_{v_0}$ for $s\gg0$ and with $H_{v_m}$ for $s\ll 0$, the restriction of $H_{v_0\cdots v_m}$ to the image of \eqref{sccollarI} must agree with the obvious splicing of $H_{v_0\cdots v_k}$ and $H_{v_k\cdots v_m}$ (note that the former condition may be interpreted as the $k=0,m$ cases of the latter), and the restriction of $H_{v_0\cdots v_m}$ to the image of \eqref{scboundaryII} must coincide with $H_{v_0\cdots\widehat{v_k}\cdots v_m}$ (the meaning of these conditions should be compared with Remark \ref{adamspathsremark}).

The same requirements are imposed on $J$ as well.
See \cite[II (9i)]{seidelbook} for similar conditions.
\end{definition}

Note that, in the above definition, the maps $H_{0\cdots n}$ and $J_{0\cdots n}$ determine all the rest, so we could have equivalently defined an $n$-simplex of Floer data as a pair of maps
\begin{align}
\label{nsimplexalternateI}H:\Cbar^{SC}_n&\to\H(X)\\
\label{nsimplexalternateII}J:\Cbar^{SC}_n&\to\J(X)
\end{align}
satisfying certain analogous properties.
Note also that a $0$-simplex of Floer data is simply a time-dependent Hamiltonian $H:S^1\to\H(X)$ and a time-dependent family of almost complex structures $J:S^1\to\J(X)$.

An $n$-simplex of Floer data can be pulled back to an $m$-simplex of Floer data under any simplicial map $\Delta^m\to\Delta^n$.
It follows that the collections of all $n$-simplices of Floer data, for all $n$, form a simplicial set.
Concretely, thinking of an $n$-simplex of Floer data as a pair of maps \eqref{nsimplexalternateI}--\eqref{nsimplexalternateII}, the \emph{face maps} are
\begin{equation}\label{floerdatafacemaps}
    d_i(H, J) = (H,J)|_{\Cbar^{SC}_{n-1}}
\end{equation}
where $\Cbar^{SC}_{n-1}\hookrightarrow\Cbar^{SC}_n$ is the stratum where $a_i=a_{i+1}$ for $0<i<n$ (for $i=0$ and $i=n$, we instead pull back under $\Mbar^{SC}_1\times\Mbar^{SC}_{n-1}\to\Mbar^{SC}_n$ and $\Mbar^{SC}_{n-1}\times\Mbar^{SC}_1\to\Mbar^{SC}_n$, respectively).
The \emph{degeneracy maps} for $0 \leq i \leq n$ are given by
\begin{equation}\label{floerdatadegeneracymaps}
s_i(H,J) = \pi_i^\ast(H, J)
\end{equation}
where $\pi_i: \Cbar^{SC}_{n+1} \to \Cbar^{SC}_{n}$ denotes the map that forgets the marked point $a_{i+1}$.

\begin{definition}\label{adapted}
An $n$-simplex of Floer data $(H,J)$ is said to be \emph{adapted to $\partial X$} when both $H=\Re\pi$ and $\pi$ is $J$-holomorphic over $\pi^{-1}(\CC_{\left|\Re\right|\leq\varepsilon})$ for some $\varepsilon>0$.
(These conditions are crucial for constraining holomorphic curves near $\partial X$ and, in particular, for the functoriality of $SH^\bullet$; specifically, they are used in Lemma \ref{scconfinecurves}.)
\end{definition}

\begin{figure}[hbt]
\centering
\includegraphics{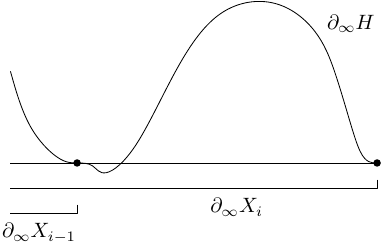}\qquad\qquad\qquad\includegraphics{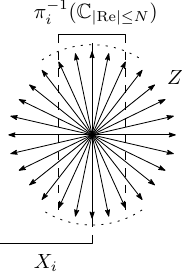}
\caption{Behavior of admissible Hamiltonians near infinity (left), and the region $\pi^{-1}(\CC_{\left|\Re\right|\leq N})$ (right).}\label{graphHhamregionsfig}
\end{figure}

\begin{definition}\label{admissible}
A Hamiltonian $H:S^1\to\H(X_r)$ is said to be \emph{admissible} with respect to a chain of Liouville sectors $X_0\subseteq\cdots\subseteq X_r$ iff $H$ is linear at infinity, except for near $\partial X_i$ where it instead is subject to the following requirements:
\begin{itemize}
\item Outside $\pi_i^{-1}(\CC_{\left|\Re\right|\leq N})$ for some $N<\infty$, we must have $H=H_i$ for some (necessarily unique) linear $H_i:\Nbd^Z\partial X_i\to\RR$ defined near infinity.
Here $H_i$ must be smooth on the closure of each component of $(\Nbd^Z\partial X_i)\setminus\partial X_i$.
Over $\partial X_i$, we must have $H_i=0$ and $dH_i=0$, and furthermore its second derivative over $\partial X_i$ must be positive on the ``inside'' $X_i$ and negative on the ``outside'' $X_r\setminus X_i$, as illustrated in Figure \ref{graphHhamregionsfig} left.
\item Inside $\pi^{-1}_i(\CC_{\left|\Re\right|\leq N})$ for all $N<\infty$, we must have $H$ bounded uniformly in $C^\infty$ (with respect to some, equivalently any, Riemannian metric $g$ satisfying $\sL_Zg=g$, e.g.\ one induced by a cylindrical almost complex structure).
Note that these strips $\pi_i^{-1}(\CC_{\left|\Re\right|\leq N})$ around $\partial X_i$ limit to $\partial\partial_\infty X_i$ at infinity, see Figure \ref{graphHhamregionsfig} right.
\end{itemize}
\end{definition}

The definition of admissibility above is somewhat complicated, so let us explain the motivation behind it.
To define the symplectic cohomology of Liouville manifolds, it is usually convenient to use Hamiltonians which are linear at infinity.
Unfortunately, linearity at infinity is incompatible with the condition of being adapted (Definition \ref{adapted}) to the boundary of a Liouville sector (recall that $Z\Re\pi=\frac 12\Re\pi$ rather than $Z\Re\pi=\Re\pi$).
Being adapted is, however, crucial for the necessary confinement results for holomorphic curves (specifically, ensuring that holomorphic curves do not approach the boundary and, more generally, that for a chain $X_0\subseteq\cdots\subseteq X_r$, holomorphic curves with positive asymptotic in $X_i$ do not pass through $\partial X_i$; see Lemma \ref{scconfinecurves}, which is the key to establishing $d^2=0$ and functoriality under inclusions of Liouville sectors).
The notion of an admissible Hamiltonian is a compromise: the linearity constraint is weakened to allow admissible Hamiltonians to also be  adapted, yet enough regularity is imposed to ensure that admissible Hamiltonians remain well behaved at infinity (for example, the flow of any admissible Hamiltonian adapted to $\partial X_r$ is complete in both directions---this follows from the reasoning used in the proof of Lemma \ref{bddgeoforsh}).
Later, we will construct a sufficient supply of simultaneously adapted and admissible Hamiltonians by perturbing linear Hamiltonians which satisfy a certain condition called being `pre-admissible'; see Definition \ref{preadmissible} and the surrounding discussion.
Note that it is not clear \emph{a priori} that there are any admissible Hamiltonians at all---the essential reason they exist is that (a smoothing of) the function $\frac{\Re\pi}{\left|\Re\pi\right|}(\Re\pi)^2$ satisfies both conditions of admissibility.

\begin{definition}[Adapted from Groman \cite{groman}]\label{dissipative}
An $n$-simplex of Hamiltonians $H$ on $X$ will be called \emph{dissipative} iff the following conditions are satisfied (these conditions are relevant for the proof of compactness in Proposition \ref{compactness}):
\begin{itemize}
\item(Non-degenerate fixed points) For every vertex $v\in\Delta^n$, the flow map $\Phi_{H_v}$ of $H_v:S^1\to\H(X)$ has non-degenerate fixed points.

\item(No fixed points at infinity) For every vertex $v\in\Delta^n$, the flow map $\Phi_{H_v}$ satisfies $d(x,\Phi_{H_v}(x))>\varepsilon>0$ for some $\varepsilon>0$ and all $x$ outside a compact subset of $X$ (distance is measured with respect to some/any $g$ satisfying $\sL_Zg=g$).

(This lower bound $d(x,\Phi_{H_v}(x))>\varepsilon>0$ is used to prove \emph{a priori} $C^0$-estimates over the thin parts of the domain, i.e.\ long cylinders $I\times S^1$ with constant Floer data (meaning independent of $s\in I$).)

\item(Boundedness below of wrapping) $\inf_{\Cbar^{SC}_k\times X}(-\frac\partial{\partial s}H_{v_0\cdots v_k})>-\infty$.

(Boundedness below of wrapping ensures that the geometric energy of a Floer trajectory is bounded above by its topological energy plus a constant, see \eqref{poswrappinggood}.)

\item(Dissipation data) We require dissipation data in the following sense \emph{to be specified}.

Dissipation data for a family $H:\Cbar^{SC}_n\to\H(X)$ ($n\geq 1$) consists of an open set $A_v\subseteq\Cbar^{SC}_n$ for each vertex $v\in\Delta^n$ and a finite collection of quadruples $(v_i,B_i,\{K_{ij}\}_{j\geq 1},\{U_{ij}\}_{j\geq 1})$ where $v_i\in\Delta^n$ is a vertex, $B_i\subseteq\Cbar^{SC}_n$ is open, $K_{ij}\subseteq X$ are compact, $U_{ij}\subseteq X$ are open, and $K_{i1}\subseteq U_{i1}\subseteq K_{i2}\subseteq U_{i2}\subseteq\cdots$ is an exhaustion of $X$, such that
\begin{align}
A_v&\subseteq\Cbar^{SC}_n\text{ contains the thin part associated to }v,\\
\Cbar^{SC}_n&=\bigcup_vA_v\cup\bigcup_iB_i,\\
H&=H_v\phantom{H_{v_i}}\text{over }(\Nbd\overline{A_v})\times X,\\
\label{dissipationshellcondition}H&=H_{v_i}\phantom{H_v}\text{over }(\Nbd\overline{B_i})\times\bigcup_{j=1}^\infty\Nbd(\overline{U_{ij}\setminus K_{ij}}),\\
\label{dissipationshellsdistancesum}\smash{\sum_{j=1}^\infty d(X\setminus U_{ij}^-,K_{ij}^+)^2}&=\infty,
\end{align}
where $K_{ij}^+$ denotes the locus of points $p\in X$ such that the forwards/backwards Hamiltonian trajectory of $H_{v_i}$ for time $\leq\frac 12$ starting at $(p,t)$ for some $t\in S^1$ intersects $K_{ij}$, and $U_{ij}^-$ denotes the locus of points for which all such trajectories stay inside $U_{ij}$.
As usual, distance is measured with some/any $g$ satisfying $\sL_Zg=g$.
We will often refer to the regions $U_{ij}^-\setminus K_{ij}^+$ and/or $U_{ij}\setminus K_{ij}$ as ``shells''.

Dissipation data for an $n$-simplex of Hamiltonians $\{H_{v_0\cdots v_m}\}_{0\leq v_0<\cdots<v_m\leq n}$ consists of dissipation data for each family $H_{v_0\cdots v_m}:\Cbar^{SC}_m\to\H(X)$ ($m\geq 1$) which is compatible in the following sense.
For each open set $A_v'$ or $B_i'$ of $\Cbar^{SC}_{m-1}$ specified for $H_{v_0\cdots\widehat{v_a}\cdots v_m}$ ($0<a<m$), there must exist a corresponding open set $A_v$ (the same $v$ and with $A_v'\subseteq A_v$) or $B_j$ (with $v_i'=v_j$, $K_{ik}'=K_{jk}$, $U_{ik}'=U_{jk}$, and $B_i'\subseteq B_j$).
The same condition is imposed for every open set $A_v'$ or $B_i'$ of $\Cbar^{SC}_a$ specified for $H_{v_0\cdots v_a}$ or $H_{v_{m-a}\cdots v_m}$ ($0<a<m$), except we require $A_v'\times\Mbar^{SC}_{m-a}\subseteq A_v$ or $B_i'\times\Mbar^{SC}_{m-a}\subseteq B_j$.

(Dissipation data is used to prove \emph{a priori} $C^0$-estimates over the thick parts of the domain, i.e.\ cylinders $I\times S^1$ of bounded length with possibly varying Floer data (meaning depending upon $s\in I$).)
\end{itemize}
Note that for $n=0$, the latter two conditions are vaccuous, so a single Hamiltonian $H:S^1\to\H(X)$ is dissipative if and only if $\Phi_H$ has non-degenerate fixed points and $d(x,\Phi_H(x))>\varepsilon>0$ near infinity.
Also note that for $n>0$, being dissipative is extra structure rather than simply a property.
\end{definition}

It causes no difference in our arguments to restrict consideration to $S^1$-invariant dissipation data (meaning each open set $A_v$ and $B_i$ is $S^1$-invariant).

\begin{definition}\label{hjscdef}
We define a simplicial set $\HJ_\bullet(X_0,\ldots,X_r)$ as follows for any chain of Liouville sectors $X_0\subseteq\cdots\subseteq X_r$.
An $n$-simplex of $\HJ_\bullet(X_0,\ldots,X_r)$ consists of an $n$-simplex of Floer data
\begin{align}
\label{Hsc}H:\Cbar^{SC}_n&\to\H(X_r)\\
\label{Jsc}J:\Cbar^{SC}_n&\to\J(X_r)
\end{align}
satisfying the following properties:
\begin{itemize}
\item $(H,J)$ is adapted to $\partial X_i$ for $0\leq i\leq r$ (Definition \ref{adapted}).
\item $H_v$ is admissible (Definition \ref{admissible}) for all $v\in\Delta^n$ with respect to a \emph{specified} chain of Liouville sectors
\begin{equation}\label{Ychain}
\varnothing=:X_{-1}\subseteq Y_1^{(0)}\subseteq\cdots\subseteq Y_{a_0}^{(0)}\subseteq X_0\subseteq Y_1^{(1)}\subseteq\cdots\subseteq Y_{a_1}^{(1)}\subseteq X_1\subseteq\cdots\subseteq X_r
\end{equation}
depending on $v$.
We require that the chain \eqref{Ychain} specified at vertex $v+1$ be obtained from that specified at vertex $v$ by removing some of the $Y^{(i)}_b$'s.

(The purpose of allowing Hamiltonians which are admissible with respect to such a chain \eqref{Ychain} is so that we can define the forgetful maps \eqref{hjforget}.
Note that a Hamiltonian admissible for $X_0\subseteq\cdots\subseteq X_r$ will usually not be admissible for the chain with $X_i$ removed.)
\item $H$ is dissipative with specified dissipation data (Definition \ref{dissipative}).
\end{itemize}
The face and degeneracy maps on the simplicial set $\HJ_\bullet(X_0, \ldots, X_r)$ are given by
the operations \eqref{floerdatafacemaps} and \eqref{floerdatadegeneracymaps},
which tautologically preserve the condition of being in $\HJ_\bullet(X_0, \ldots, X_r)$.
\end{definition}

There are forgetful maps (of simplicial sets) 
\begin{equation}\label{hjforget}
\HJ_\bullet(X_0,\ldots,X_r)\to\HJ_\bullet(X_0,\ldots,\widehat{X_i},\ldots,X_r)
\end{equation}
for $0\leq i\leq r$; the only non-obvious part of their definition is that the
chains \eqref{Ychain} retain the forgotten $X_i$ if $i<r$ (note that the
case of $i = r$, in which one restricts $(H,J)$ to $X_{r-1}$, is somewhat special compared to the cases $i<r$, in which one considers the same $(H,J)$ on $X_r$).

\subsection{Construction of Hamiltonians and almost complex structures}

We now introduce constructions of Hamiltonians and almost complex structures suitable for defining symplectic cohomology of Liouville sectors, i.e.\ constructions of simplices of $\HJ_\bullet$.
Lemma \ref{constructH} provides a ready supply of vertices of $\HJ_\bullet$ (i.e.\ Floer data suitable for defining a Floer complex $CF^\bullet(X;H)$ and its differential).
Proposition \ref{hjcontractible} produces sufficiently many higher simplices in $\HJ_\bullet$ (i.e.\ Floer data suitable for defining continuation maps, etc.).

\begin{definition}\label{preadmissible}
A linear Hamiltonian $H:S^1\to\H(X_r)$ defined near infinity is said to be \emph{pre-admissible} with respect to a chain of Liouville sectors $X_0\subseteq\cdots\subseteq X_r$ iff it satisfies the following conditions.

\begin{itemize}
    \item We require $H$ to be smooth on the complement of $\bigcup_i\partial X_i$.
    \item Near each $\partial X_i$, we require the restriction of $H$ to the closure of each ``side'' of $\partial X_i$ to be smooth, meaning it admits a smooth extension to an open neighborhood (it thus makes sense to evaluate the derivatives of $H$ over $\partial X_i$, though we must specify whether we compute them from the ``inside'' or the ``outside'').
    \item Over $\partial X_i$, we require that $H=0$ and $dH=0$ (from both sides)
    and that the second derivative of $H$ be positive on the inside and negative on the outside.
\end{itemize}
\end{definition}

Every admissible (in the sense of Definition \ref{admissible}) Hamiltonian $\tilde H:S^1\to\H(X_r)$ determines a unique linear Hamiltonian $H:S^1\to\H(X_r)$ defined near infinity, defined by the property that $H=\tilde H$ outside $\bigcup_i\pi_i^{-1}(\CC_{\left|\Re\right|\leq N})$ for some $N<\infty$.
Clearly such $H$ is pre-admissible.

Being linear at infinity, pre-admissible Hamiltonians are easy to construct.
In contrast, an adapted admissible Hamiltonian fails to be linear at infinity over the small strips $\pi_i^{-1}(\CC_{\left|\Re\right|\leq\varepsilon})$, where it must coincide with $\Re \pi_i$.
However, the next lemma shows we can modify most pre-admissible Hamiltonians to make them adapted and admissible.

\begin{lemma}\label{constructH}
Let $H:S^1\to\H(X_r)$ be pre-admissible with respect to $X_0\subseteq\cdots\subseteq X_r$.
Assume also that $\Phi_H$ has no fixed points other than $\bigcup_i\partial X_i$.
There exists an admissible $\tilde H:S^1\to\H(X_r)$ corresponding
to $H$ at infinity,
which is dissipative and adapted to all $\partial X_i$.
If $H$ is $S^1$-independent, then we may take $\tilde H$ to be as well.
\end{lemma}

\begin{proof}
It is enough to modify $H$ near each $\partial X_i$.
It is furthermore enough to discuss this modification on the ``inner'' side of $X_i$, as the situation on the ``outer'' side is the same upon negation.
Hence, we may forget about the chain of Liouville sectors altogether and simply modify a given pre-admissible $H:S^1\to\H(X)$ near $\partial X$ to make it admissible, adapted to $X$, and dissipative.

We define $\tilde H$ by smoothing $\max(H,R)$, where $R:=\Re\pi$.
We may write $H=FR^2$ for some $Z$-invariant function $F:\Nbd^Z\partial X\to\RR_{>0}$ defined near infinity.
In particular, we have
\begin{equation}
\max(H,R)=\begin{cases}R&R\leq\frac 1N\\H&R\geq N\end{cases}
\end{equation}
for sufficiently large $N<\infty$.
The smoothing will take place over the strip $\{\frac 1N\leq R\leq N\}$.
(Note that $ZR=\frac 12R$, so $R$ covers the entirety of $[0,N]$ near infinity over any $\Nbd^Z\partial X$.)

Since $H$ is linear at infinity, $\Phi_H$ is cylindrical at infinity, and hence satisfies $d(x,\Phi_H(x))\geq\varepsilon>0$ near infinity, except possibly over $\Nbd^Z\partial X$.
We must smooth $\max(H,R)$ so it retains this displacement property and is bounded in $C^\infty$.
The key to doing this is to note that the desired lower bound $d(x,\Phi_{\tilde H}(x))\geq\varepsilon>0$ is implied by the stronger property $X_{\tilde H}I\geq\varepsilon>0$ over $\Nbd^Z\partial X$, where $I=\Im\pi$.
Indeed, $\left|dI\right|$ is $Z$-invariant (and nonzero), and hence is of constant order near infinity.

To produce $\tilde H$ satisfying $X_{\tilde H}I\geq\varepsilon>0$ over $\Nbd^Z\partial X$, argue as follows.
We have that
\begin{equation}\label{XRIlbforH}
X_RI\equiv 1
\end{equation}
by definition.
We have that $X_HI$ vanishes to first order over $\partial X$ with positive inward derivative (see the proof of Lemma \ref{compactcutoff}), so since $Z(X_HI)=\frac 12X_HI$ and $ZR=\frac 12R$, we have
\begin{equation}\label{XHIlbforH}
X_HI\geq c\cdot R
\end{equation}
over $\Nbd^Z\partial X$ for some $c>0$.
Now, simply observe that over the locus $\{0\leq R\leq N\}$, the metric $g$ has bounded geometry and the functions $I$, $R$, and $H=FR^2$ are uniformly bounded in $C^\infty$.
It thus follows from \eqref{XRIlbforH}--\eqref{XHIlbforH} that we may smooth $\max(H,R)$ over the strip $\{\frac 1N\leq R\leq N\}$ to obtain $\tilde H$ (also uniformly bounded in $C^\infty$) such that $X_{\tilde H}I\geq\varepsilon>0$.
\end{proof}

\begin{proposition}\label{hjcontractible}
For $n\geq 2$, every map $\partial\Delta^n\to\HJ_\bullet(X_0,\ldots,X_r)$ extends to $\Delta^n$.
For $n=1$, a sufficient (and obviously necessary) condition for an extension to exist is that $H_0\leq H_1+C$ for some $C<\infty$ and that the chain \eqref{Ychain} at $0$ be a superset of that at $1$.
\end{proposition}

\begin{proof}
The input data of a map $\partial\Delta^n\to\HJ_\bullet(X_0,\ldots,X_r)$ amounts to all of the data of an $n$-simplex of Floer data $(H,J)$ except for the ``top-dimensional'' maps $H_{0\cdots n}$ and $J_{0\cdots n}$.
Note that $H_{0\cdots n}$ and $J_{0\cdots n}$ are determined uniquely over (the inverse image of) $\partial\Mbar^{SC}_n$, the images of all collars \eqref{sccollarI}, and the ends $s\gg 0$ and $s\ll 0$.
Extension of $J_{0\cdots n}$ to all of $\Cbar^{SC}_n$ is trivial by contractibility of $\J(X)$.
Our main task is to show that $H_{0\cdots n}$ extends so that dissipativity is satisfied.

For this purpose of extending $H_{0\cdots n}$, we may as well forget about the meaning of $\Cbar^{SC}_n$ and remember only its topology.
Namely, we replace the original extension problem on $\Cbar^{SC}_n$ (rel boundary) with an extension problem on $D^{n-1}\times[0,1]\times S^1$ (rel boundary), where, for the purpose of making sense of boundedness below of wrapping, the role of the vector field $-\partial_s$ is played by differentiation in the $[0,1]$-coordinate direction.

By assumption, the given $H$ (defined near the boundary) satisfies boundedness below of wrapping.
Dissipation data can be defined covering $\Nbd\partial(D^{n-1}\times[0,1]\times S^1)$ by extending from the boundary (note that this uses crucially the compatibility properties of the dissipation data chosen for each facet of $\Delta^n$).
To complete this to a cover of $D^{n-1}\times[0,1]\times S^1$, we add two more open sets $B_i$ given by $(D^{n-1}\setminus\Nbd\partial D^{n-1})\times[\delta,\frac 23]\times S^1$ (for vertex $0$) and $(D^{n-1}\setminus\Nbd\partial D^{n-1})\times[\frac 13,1-\delta]\times S^1$ (for vertex $n$), with any exhaustion satisfying \eqref{dissipationshellsdistancesum} (which exists since the flow of $H_v$ is complete).
For reasons which will become apparent below, we actually duplicate the open sets $B_i$ which arose from extending from the boundary, where the first copy covers the boundary (but is disjoint from the two new added $B_i$'s) and the second copy is disjoint from the boundary.
We illustrate the resulting cover of $D^{n-1}\times[0,1]\times S^1$ in Figure \ref{dataextensionfig} (the regions $D^{n-1}\times[0,\delta)\times S^1$ and $D^{n-1}\times(1-\delta,1]\times S^1$ are $A_0$ and $A_n$, respectively).

\begin{figure}[hbt]
\centering
\includegraphics{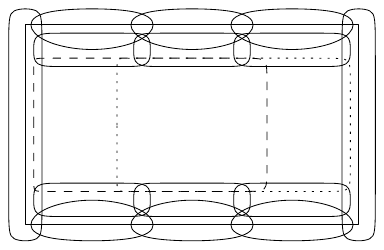}
\caption{Open cover of $S^1\times D^{n-1}\times[0,1]$ consisting of a small extension of the dissipation data given on the boundary (solid) and two more open sets $B_i$ (dashed/dotted).}\label{dataextensionfig}
\end{figure}

Unfortunately, this dissipation data may be self-contradictory: the two new open sets $B_i$ may have associated shells which intersect those from the dissipation data extended from the boundary, which is problematic in view of \eqref{dissipationshellcondition}.
The key observation (due to Groman \cite{groman}, following Cieliebak--Eliashberg \cite{cieliebakeliashberg}) is that we can remedy this defect by forgetting some of the shells associated to the two new sets $B_i$ and the second copies of the dissipation data lifted from the boundary.
To achieve this, we simply iterate over these $B_i$'s infinitely many times, where at each iteration we choose a finite number of shells $U_{ij}\setminus K_{ij}$ to ``keep'' which are disjoint from already chosen shells and for which the corresponding sum of $d(X\setminus U_{ij}^-,K_{ij}^+)^2$ is at least $1$.

Now that the dissipation data has been ``thinned out'' so as to be consistent, it is straightforward to extend $H$, consistently with this dissipation data, maintaining boundedness below of wrapping.
\end{proof}

Together, Lemma \ref{constructH} and Proposition \ref{hjcontractible} imply the following result, which formalizes the vague statements that the space of Floer data is contractible and that wrapping is filtered.
The notion of a filtered $\infty$-category is reviewed in \S\ref{inftycatssec} below.

\begin{corollary}\label{hjfiltered}
The simplicial set $\HJ_\bullet(X_0,\ldots,X_r)$ is a filtered $\infty$-category, and the forgetful map $\HJ_\bullet(X_0,\ldots,X_r)\to\HJ_\bullet(X_0,\ldots,X_{r-1})$ is cofinal.
\end{corollary}

\begin{proof}
Note that Proposition \ref{hjcontractible} implies that $\HJ_\bullet$ satisfies the hypotheses of Lemma \ref{filteredlifting}; in particular, the necessary and sufficient condition for a pair of vertices of Floer data $H_0$ and $H_1$ to extend to a $1$-simplex from $H_0$ to $H_1$ is obviously transitive; denote this condition, which gives a partial ordering on vertices, as $H_0 \unlhd H_1$.
Lemma \ref{filteredlifting} therefore implies that, to show $\HJ_\bullet$ is
filtered, it is sufficient to show that the poset induced by the relation $\unlhd$ on the vertices
of $\HJ_\bullet$ is filtered (directed).
Directedness of $\unlhd$ follows from Lemma \ref{constructH} since pre-admissible Hamiltonians $H$ are trivial to construct (note that the hypothesis of $\Phi_H$ having no fixed points other than $\bigcup_i\partial X_i$ holds generically in view of Lemma \ref{compactcutoff}).

Cofinality of the forgetful map also follows immediately from Lemmas \ref{filteredlifting}, \ref{constructH}, and \ref{compactcutoff}.
\end{proof}

\subsection{Filtered \texorpdfstring{$\infty$-categories}{infinity-categories}}\label{inftycatssec}

\begin{definition}[{Lurie \cite[Definition 1.1.2.4]{luriehtt}}]
A simplicial set $\C$ is called an $\infty$-category iff it satisfies the extension property for all inclusions $\Lambda^n_i\hookrightarrow\Delta^n$ for $0<i<n$ (recall that $\Lambda^n_i\subseteq\Delta^n$ denotes the union of all faces containing the vertex $i$).
A functor between $\infty$-categories is simply a map of simplicial sets.
An $\infty$-category is called an $\infty$-groupoid iff it satisfies the extension property for all inclusions $\Lambda^n_i\hookrightarrow\Delta^n$ for $0\leq i\leq n$.
\end{definition}

\begin{example}
Every category $\C$ gives rise to a simplicial set $N\C$, called its \emph{nerve}, which is an $\infty$-category.
A functor $\C\to\D$ is the same thing as a map of simplicial sets $N\C\to N\D$.
We will hence make no distinction between a category and its nerve.
\end{example}

\begin{definition}
For every $\infty$-category $X_\bullet$, there is an associated category $hX_\bullet$ called its \emph{homotopy category}.
The objects of $hX_\bullet$ are the $0$-simplices of $X_\bullet$, and the morphisms of $hX_\bullet$ are equivalence classes of $1$-simplices in $X_\bullet$, where an equivalence between $1$-simplices $x\to y$ is a map $\Delta^1\times\Delta^1\to X_\bullet$ for which $\Delta^1\times\{0\}$ and $\Delta^1\times\{1\}$ are the degenerate $1$-simplices over $x$ and $y$, respectively.
There is a canonical map $X_\bullet\to hX_\bullet$, which is initial in the category of all maps from $X_\bullet$ to (the nerve of) a category.
\end{definition}

\begin{definition}\label{def:diagram}
A \emph{diagram} in an $\infty$-category $\C$ is a map of simplicial sets $p:K\to\C$; the simplicial set $K$ is called the indexing simplicial set (or indexing ($\infty$-)category as the case may be).
\end{definition}

\begin{definition}[{Lurie \cite[Notation 1.2.8.4]{luriehtt}}]\label{trirightdef}
Given any simplicial set $K$, denote by $K^\vartriangleleft$ the simplicial set obtained by adding an initial vertex $\ast$ to $K$.
More formally,
\begin{equation}
\Hom(\Delta^p,K^\vartriangleleft):=\Hom(\Delta^p,K)\cup\Hom(\Delta^{p-1},K)\cup\cdots\cup\Hom(\Delta^0,K) \cup \{\ast\}
\end{equation}
where a simplex $\Delta^p\to K^\vartriangleleft$ consists of the map sending the initial $\Delta^{[0\ldots k]}\subseteq\Delta^p$ to the initial vertex $\ast$ and some map from the final $\Delta^{[k+1\ldots p]}$ to $K$, for $-1\leq k\leq p$.
Similarly, define $K^\vartriangleright$ by adding a terminal vertex to $K$.
\end{definition}

\begin{definition}[{Lurie \cite[\S 1.2.9]{luriehtt}}]
For an $\infty$-category $\C$ and an object $c\in\C$ (i.e.\ a vertex $c\in\C_0$ of the simplicial set $\C_\bullet$), the under-category $\C_{c/}$ is the $\infty$-category defined by the universal property that $\Hom(K,\C_{c/})\subseteq\Hom(K^\vartriangleleft,\C)$ is the subset sending the initial vertex $\ast\in K^\vartriangleleft$ to $c$.
The over-category $\C_{/c}$ is defined similarly via $\Hom(K,\C_{/c})\subseteq\Hom(K^\vartriangleright,\C)$.
\end{definition}

There is a canonical map $\C_{c/} \to \C$ (respectively $\C_{/c} \to \C$). 

\begin{definition}[{Lurie \cite[Definition 5.3.1.7]{luriehtt}}]
An $\infty$-category $\C$ is called \emph{filtered} iff it satisfies the extension property for all inclusions $K\hookrightarrow K^\vartriangleright$ for finite simplicial sets $K$.
\end{definition}

An ordinary category is filtered (see \S\ref{wrappedhfsection}) iff it (or rather its nerve) is filtered as an $\infty$-category.
The homotopy category of a filtered $\infty$-category is filtered.
An $\infty$-groupoid is filtered iff it is contractible.
(There is a notion of when an arbitrary simplicial set is filtered \cite[Remark 5.3.1.11]{luriehtt}, but it is not the naive generalization of the above.)

\begin{definition}
A functor $F:\C\to\D$ between filtered $\infty$-categories is said to be \emph{cofinal} iff for every vertex $d\in\D$, the fiber product $\C\times_\D\D_{d/}$ is a filtered $\infty$-category.
\end{definition}

A functor between filtered categories is cofinal (see \S\ref{wrappedhfsection}) iff it is cofinal as a functor between filtered $\infty$-categories.
A cofinal functor between filtered $\infty$-categories induces a cofinal map between their homotopy categories.
(Recalling that any filtered $\infty$-category is weakly contractible \cite[Lemma 5.3.1.18]{luriehtt}, the definition of cofinal functor given above appears stronger than the definition given in \cite[Theorem 4.1.3.1]{luriehtt} for functors between arbitrary $\infty$-categories; it is in fact equivalent, though we will not need to appeal to this fact.
There is a notion of when a map of arbitrary simplicial sets is cofinal due to Joyal \cite[Definition 4.1.1.1]{luriehtt}.)

We leave the following result as an exercise.

\begin{lemma}\label{filteredlifting}
Let $X_\bullet$ be a simplicial set, and suppose that:
\begin{itemize}
\item$X_\bullet$ satisfies the extension property for all inclusions $\partial\Delta^n\hookrightarrow\Delta^n$ for $n\geq 2$.
\item The relation $\leq$ on $X_0$ defined by $p\leq q$ iff there is a $1$-simplex from $p$ to $q$ is transitive.
\end{itemize}
Then, $X_\bullet$ is an $\infty$-category, its homotopy category $hX_\bullet$ is a poset (i.e.\ for all $x,y\in hX_\bullet$ there is at most one morphism $x\to y$), and $X_\bullet\to hX_\bullet$ is a trivial Kan fibration.\qed
\end{lemma}

In particular, if $X_\bullet$ satisfies the hypotheses of Lemma \ref{filteredlifting}, then $X_\bullet$ is a filtered $\infty$-category if and only if $hX_\bullet$ is filtered (directed) as a poset.
Specializing even further, such $X_\bullet$ is a contractible $\infty$-groupoid if and only if $hX_\bullet=\ast$ (equivalently, for every $x,y\in hX_\bullet$ there exists a morphism $x\to y$).
It also follows that if $X_\bullet$ and $Y_\bullet$ satisfy the hypotheses of Lemma \ref{filteredlifting} and both $hX_\bullet$ and $hY_\bullet$ are filtered (so $X_\bullet$ and $Y_\bullet$ are filtered $\infty$-categories), then $X_\bullet\to Y_\bullet$ is cofinal iff $hX_\bullet\to hY_\bullet$ is cofinal.

\subsection{Holomorphic curves}

\begin{definition}
For any $(H,J)\in\HJ_n(X)$ and periodic orbits $\gamma^+:S^1\to X$ of $H_0$ and $\gamma^-$ of $H_n$, we define 
\begin{equation}
\M_n(H,J,\gamma^+,\gamma^-)
\end{equation}
to be the moduli space of sequences $a_1 \geq \cdots \geq a_n \in \RR$ and maps $u:\RR\times S^1\to X$
with asymptotics $u(+\infty,t)=\gamma^+(t)$, $u(-\infty,t)=\gamma^-(t)$, 
satisfying Floer's equation with respect to the given $(H,J)$
\begin{equation}\label{floereqnforsh}
(du-X_H\otimes dt)^{0,1}_J=0
\end{equation}
modulo simultaneous $\RR$-translation of $a_1 \geq \cdots \geq a_n$ and the domain of $u$ 
(recall that $(H,J)$ are maps $\Cbar^{SC}_n \to \H(X)$, which we pull back to $\RR\times S^1$ by identifying $(\RR\times S^1,a_1,\ldots,a_n)$ with a unique fiber of $\Cbar^{SC}_n$).
The associated ``compactified'' moduli space
\begin{equation}\label{mbarscmoduli}
\Mbar_n(H,J,\gamma^+,\gamma^-)
\end{equation}
includes all stable broken trajectories as well.
\end{definition}

For future reference, we record here the two notions of energy for trajectories in $\Mbar_n(H,J,\gamma^+,\gamma^-)$.
The topological energy is defined as
\begin{align}
E^\top(u)&:=\int_{\RR\times S^1}u^\ast\omega-dH\wedge dt-\partial_sH\,ds\wedge dt\\
&=\int_{\RR\times S^1}d(u^\ast\lambda-H\,dt)\\
&=a(\gamma^+)-a(\gamma^-)
\end{align}
where the action of $\gamma$ is $a(\gamma):=\int_{S^1}\gamma^\ast\lambda-H(\gamma)\,dt$.
The geometric energy is
\begin{equation}
E^\geo(u):=\int_{\RR\times S^1}u^\ast\omega-dH\wedge dt = \int_{\RR \times S^1}\frac 12 \| du - X_H \otimes dt\|^2,
\end{equation}
where the norm $\|\cdot\|$ comes from $\omega(\cdot, J\cdot)$ (the equality of
the two expressions above holds since $du-X_H\otimes dt$ is complex linear).
Note that the integrand of $E^\geo(u)$ is $\geq 0$, and if
$E^\geo(u)=0$ then $\partial_su\equiv 0$ (i.e.\ the solution $u$ is a ``trivial
cylinder'').
We have
\begin{equation}\label{poswrappinggood}
E^\geo(u)=E^\top(u)+\int_{\RR\times S^1}\partial_sH\,ds\wedge dt.
\end{equation}
Boundedness below of wrapping $\inf(-\partial_sH)>-\infty$ provides an upper bound on the last term (note that the projection to $\RR$ of the support of $\partial_sH$ has \emph{a priori} bounded length), and hence $E^\geo(u)$ is bounded above by $E^\top(u)$ plus a constant.
If in fact $-\partial_sH\geq 0$, then $E^\geo(u)\leq E^\top(u)$.

\begin{lemma}\label{scconfinecurves}
For Floer data $(H,J)\in\HJ_\bullet(X_0,\ldots,X_r)$, any trajectory in $\Mbar_n(H,J,\gamma^+,\gamma^-)$ with positive end $\gamma^+$ in $X_i$ must lie entirely inside $X_i$.
\end{lemma}

\begin{proof}
Fix a bump function $\varphi:\RR\to\RR_{\geq 0}$ supported in $[-\frac\varepsilon 2,\frac\varepsilon 2]$ and define the $\varphi(x)\omega_\CC$-energy
\begin{equation}
E^{\varphi(x)\omega_\CC}(u):=\int_{\RR\times S^1}(du-X_H\otimes dt)^\ast\pi_i^\ast(\varphi(x)\omega_\CC).
\end{equation}
Since $\pi_i$ is holomorphic over $\pi_i^{-1}(\CC_{\left|\Re\right|\leq\varepsilon})$, this energy is $\geq 0$, and equality implies that the trajectory is disjoint from the support of $\varphi\circ\pi_i$ (which in particular contains $\partial X_i$).

Now the $2$-form $\varphi(x)\omega_\CC$ on $\CC_{\left|\Re\right|\leq\varepsilon}$ is exact relative to a neighborhood of the boundary, namely $\varphi(x)\omega_\CC=d\kappa$ for some $1$-form $\kappa$ on $\CC$ supported inside $\CC_{\left|\Re\right|\leq\varepsilon}$.
Furthermore, the flow of $X_{\Re}$ on $\CC$ preserves $d\kappa=\varphi(x)\omega_\CC$, so in view of Cartan's formula, the $1$-form $\beta:=d\kappa(X_H,\cdot)$ is closed; the class $[\beta]\in H^1(\CC_{\left|\Re\right|\leq\varepsilon},\partial)$ is called the \emph{flux}.
We may thus express the $\varphi(x)\omega_\CC$-energy as a ``flux pairing'', namely
\begin{equation}
E^{\varphi(x)\omega_\CC}(u)=\int_{\RR\times S^1}d\kappa\Bigl(X_H,\frac{\partial u}{\partial s}\Bigr)\,ds\,dt=\langle[\beta],(\pi\circ u)_\ast([\RR]\times\{0\})\rangle
\end{equation}
where the latter is the pairing between $H^1(\CC_{\left|\Re\right|\leq\varepsilon},\partial)$ and $H_1(\CC_{\left|\Re\right|\leq\varepsilon},\partial)$.
Keeping careful track of signs, we conclude that if $\gamma^-$ lies outside $X_i$, then $E^{\varphi(x)\omega_\CC}(u)<0$, a contradiction.
Thus $\gamma^-$ lies inside $X_i$, implying $E^{\varphi(x)\omega_\CC}(u)=0$, which means the trajectory must be disjoint from $\partial X_i$ and hence entirely inside $X_i$.
\end{proof}

\subsection{Compactness}

The proof of compactness we give here is similar to Proposition \ref{wcompactness}, but now crucially incorporates arguments of Groman \cite{groman}.
We need the following result about geometric boundedness (so that we can make use of the monotonicity inequalities recalled in \S\ref{inftyescape}).

\begin{lemma}\label{bddgeoforsh}
Let $H:S^1\to\H(X_r)$ be admissible with respect to $X_0\subseteq\cdots\subseteq X_r$ in the sense of Definition \ref{admissible}.
Let $D^2\subseteq\RR\times S^1$ be a disk centered at $p\in D^2$, and let $J:D^2\to\J(X)$ be cylindrical.
Then $(D^2\times X,\omega_{D^2}+\omega_X,j_{D^2}\oplus\bar J)$ has bounded geometry, where $\bar J:=(\Phi_H^{t(p)\to t})^\ast J$.
\end{lemma}

\begin{proof}
In the case that $H$ is linear at infinity (i.e.\ in the very special case $\partial X_0=\cdots=\partial X_r=\varnothing$), the family of almost complex structures $\bar J:D^2\to\J(X)$ is cylindrical at infinity, and so we can simply appeal to Lemma \ref{domainacsbddgeo}.

To push this reasoning a bit further, let $N<\infty$ be such that $H$ is linear at infinity outside $\bigcup_i\pi_i^{-1}(\CC_{\left|\Re\right|\leq N})$.
Now over the subset of $X$ which never hits $\bigcup_i\pi_i^{-1}(\CC_{\left|\Re\right|\leq N})$ under the flow of $X_H$ over any time interval $[t(p),t(q)]$ for $q\in D^2$, Lemma \ref{domainacsbddgeo} applies to show the desired geometric boundedness statement, since there the gauge transformed almost complex structure $\bar J$ is cylindrical.
We claim that the ``bad'' subset of $X$ which does hit $\bigcup_i\pi_i^{-1}(\CC_{\left|\Re\right|\leq N})$ is contained within $\bigcup_i\pi_i^{-1}(\CC_{\left|\Re\right|\leq M})$ for some $M<\infty$.
To show this, it is enough to show that $\left|X_HR\right|\leq c\cdot\left|R\right|$ for $R=\Re\pi_i$ and some $c<\infty$ (so $R$ grows at most exponentially under the flow of $X_H$).
This inequality holds since $X_HR=0$ over $\{R=0\}$ (by admissibility of $H$) and both $X_HR$ and $R$ have ``square root growth at infinity'', i.e.\ $Z(X_HR)=\frac 12X_HR$ and $ZR=\frac 12R$.

It remains only to prove geometric boundedness over $\bigcup_i\pi_i^{-1}(\CC_{\left|\Re\right|\leq M})$ for $M<\infty$.
This follows immediately using the fact that $H$ is $C^\infty$-bounded over such regions and again the fact that flowing under $X_H$ for finite time stays within yet a larger region $\bigcup_i\pi_i^{-1}(\CC_{\left|\Re\right|\leq M'})$.
\end{proof}

\begin{proposition}\label{compactness}
The moduli spaces $\Mbar_n(H,J,\gamma^+,\gamma^-)$ are compact.
\end{proposition}

\begin{proof}
In view of the inclusion $\HJ_\bullet(X_0,\ldots,X_r)\subseteq\HJ_\bullet(X_r)$, it is enough to treat the case $r=0$, writing $X_0=X$.

It is enough to show that there is a compact subset of $X \backslash \partial X$ containing the image of all trajectories in $\Mbar_n(H,J,\gamma^+,\gamma^-)$ (then the usual Gromov compactness arguments apply).
By Lemma \ref{scconfinecurves}, trajectories in $X$ stay away from $\partial X$, so it suffices to construct such a compact subset of $X$.
Note that $E^\geo$ is bounded above on $\Mbar_n(H,J,\gamma^+,\gamma^-)$ since $E^\top=a(\gamma^+)-a(\gamma^-)$ is fixed and wrapping is bounded below (see \eqref{poswrappinggood}).
The compact set we produce will depend only on $(H,J)$ and this upper bound on $E^\geo$.
Note that since $E^\geo\geq 0$ and is additive along broken trajectories, we may assume that our trajectory is not broken (though this is only for psychological comfort; the argument below applies verbatim to broken trajectories as well).

Let a trajectory $u:\RR\times S^1\to X$ be given.
We decompose $\RR$ into $\leq 2n+1$ (possibly infinite or half-infinite) overlapping intervals, each of one of the following two types:
\begin{itemize}
\item Intervals of \emph{a priori} bounded length (the \emph{thick parts}; concretely, these can be taken to be the union of the intervals of length $N$ centered at each of the points $a_i\in\RR$, for some sufficiently large $N<\infty$).
\item Intervals over which $H$ and $J$ are independent of the $s$-coordinate (that is, the \emph{thin parts}, namely where the domain curve is close to breaking or where $s$ is close to $\pm\infty$).
\end{itemize}

Let us first show that the image under $u$ of any component of the thick part (i.e.\ any interval of the first type above) is bounded \emph{a priori} in terms of $(H,J)$, $E^\geo$, and $u(p)$ for any point $p$ in the component.
Let dissipation data $\{(B_i,v_i,\{K_{ij}\}_j,\{U_{ij}\}_j)\}_i$ be given.
Since we are presently interested only in the particular domain $\RR\times S^1\subseteq\Cbar^{SC}_n$, we will abuse notation and denote by $B_i$ what would more accurately be written $B_i\cap(\RR\times S^1)$.
Since the length of the interval is bounded, it suffices to show that for every small ball $B_{2\varepsilon}(p)\subseteq B_i$ (let us assume $0<\varepsilon<\frac 18$), the image $u(B_\varepsilon(p))\subseteq X$ is bounded in terms of $u(p)$ (of course, we should also prove the same statement for $B_{2\varepsilon}(p)\subseteq A_v$, however this case is strictly easier, hence omitted).
To prove this, we consider the map $\bar u:B_{2\varepsilon}(p)\to X$ defined by the relation $u=\Phi_{H(s,\cdot)}^{t(p)\to t}\circ\bar u$ (i.e.\ we do a change of coordinates on the target using the Hamiltonian flow of $H$).
For any $j$, we have
\begin{equation}
(d\bar u)^{0,1}_{\bar J}=0\quad\text{over }\bar u^{-1}(U_{ij}^-\setminus K_{ij}^+)
\end{equation}
where $\bar J=(\Phi_{H_{v_i}}^{t(p)\to t})^\ast J$ (note that $\Phi_{H(s,\cdot)}^{t(p)\to t}=\Phi_{H_{v_i}}^{t(p)\to t}$ over $U_{ij}^-\setminus K_{ij}^+$).
Now the graph of $\bar u$ over $B_{2\varepsilon}(p)\cap\bar u^{-1}(U_{ij}^-\setminus K_{ij}^+)$ is a properly embedded holomorphic curve in $B_{2\varepsilon}(p)\times(U_{ij}^-\setminus K_{ij}^+)$ with respect to $j_{B_{2 \varepsilon}(p)} \oplus \bar J$.
Using the fact that $H_{v_i}$ is admissible, one sees that this almost complex structure has bounded geometry by Lemma \ref{bddgeoforsh}. 
We claim that monotonicity (see \S\ref{inftyescape}) applied to the graph of $\bar u$ implies that if $\bar u|B_\varepsilon(p)$ crosses a ``shell'' $U_{ij}^-\setminus K_{ij}^+$, then the graph of $\bar u$ over $B_{2\varepsilon}(p)\cap\bar u^{-1}(U_{ij}^-\setminus K_{ij}^+)$ has geometric energy bounded below by a constant (depending on $\varepsilon>0$ and the geometry of $(X,\omega,\bar J)$) times $\min(d,d^2)$ where $d=d(X\setminus U_{ij}^-,K_{ij}^+)$ (compare Groman \cite{groman}).
Indeed, suppose that $\bar u|B_\varepsilon(p)$ crosses a ``shell'' $U_{ij}^-\setminus K_{ij}^+$ (meaning, precisely, that $\bar u(B_\varepsilon(p))$ intersects both $K_{ij}^+$ and $X\setminus U_{ij}^-$).
Then by connectedness of $B_\varepsilon(p)$, we may find a set of $\max(1,\left\lfloor d\right\rfloor)$ points in $B_\varepsilon(p)$ such that the balls of radius $\min(\frac 12,\frac 12d)$ around their images in $X$ under $\bar u$ are disjoint and contained inside $U_{ij}^-\setminus K_{ij}^+$.
Now we consider the corresponding points on the graph of $\bar u$, and we observe that the balls of radius $\min(\frac 12,\frac 12d,\varepsilon)$ centered at these points are disjoint and contained inside $B_{2\varepsilon}(p)\times(U_{ij}^-\setminus K_{ij}^+)$, and moreoever the boundary of the graph of $\bar u$ over $B_{2 \varepsilon}(p) \times (U_{ij}^-\setminus K_{ij}^+)$ lies outside these balls (recall that their centers are contained in $B_{\varepsilon}(p) \times (U_{ij}^-\setminus K_{ij}^+)$).
We can therefore apply monotonicity to these balls to produce the desired lower bound on the energy of the graph of $\bar u$ over $B_{2\varepsilon}(p)\cap\bar u^{-1}(U_{ij}^-\setminus K_{ij}^+)$.
Now, on the other hand, the sum of these geometric energies (over all $j$) is bounded above by the geometric energy of the graph of $\bar u$ over $B_{2\varepsilon}(p)$, which is bounded above by $E^\geo(u)+(2\varepsilon)^2\pi$.  This shows that $\bar u|B_\varepsilon(p)$ can cross at most
finitely many shells by divergence of \eqref{dissipationshellsdistancesum}, and hence that $u(B_\varepsilon(p))$ is bounded depending
on $u(p)$, as desired.
This proves the desired claim about intervals in the thick part of the domain.

Now let us examine the thin parts of the domain (i.e.\ intervals of the second type above).
Define $\bar u$ by the relation $u(s,t)=\Phi_H^{0\to t}(\bar u(s,t))$ for $0\leq t<2\pi$, so
\begin{equation}\label{gaugerelation}
\partial_tu-X_H=(\Phi_H^{0\to t})_\ast\partial_t\bar u.
\end{equation}
Fix a large compact set $K\subseteq X$ outside which $d(x,\Phi_H(x))\geq\varepsilon>0$; note that this implies that if $u(s,0)\notin K$, then
\begin{equation}
\int_0^{2\pi}\left|\partial_t\bar u\right|\geq d(u(s,0),\Phi_H(u(s,0)))\geq\varepsilon.
\end{equation}
Recall that we measure length $\left|\cdot\right|$ with respect to any Riemannian metric $g$ satisfying $\sL_Zg=g$ near infinity (they are all uniformly commensurable), such as the metric induced by any cylindrical $J$.
Now for any interval $[a,b]\subseteq\RR$ in the thin part such that $u(s,0)\notin K$ for all $s\in[a,b]$, we have
\begin{align}\label{csdistanceintegral}
E^\geo\geq\int_{[a,b]}\int_{S^1}\left|\partial_tu-X_H\right|^2\gtrsim{}&\int_{[a,b]}\biggl(\int_{S^1}\left|\partial_tu-X_H\right|\biggr)^2\\
{}\asymp{}&\int_{[a,b]}\biggl(\int_0^{2\pi}\left|\partial_t\bar u\right|\biggr)^2\geq\varepsilon^2(b-a).\nonumber
\end{align}
The last inequality on the first line is Cauchy--Schwarz.
To justify the next relation $\asymp$ (meaning equality up to a positive constant bounded uniformly away from $0$ and $\infty$), in view of \eqref{gaugerelation} we just need to argue that $\Phi_H^{0\to t}:X\to X$ distorts lengths by at most a constant uniform in $t$.
Over the region where the flow of $H$ never hits $\bigcup_i\pi_i^{-1}(\CC_{\left|\Re\right|\leq N})$, this follows from linearity of $H$; the region where the flow does hit this region is contained in $\bigcup_i\pi_i^{-1}(\CC_{\left|\Re\right|\leq M})$ for some $M<\infty$ (see the proof of Lemma \ref{bddgeoforsh}) and so the desired statement follows from $C^\infty$-boundedness of $H$ over this region.
Having justified \eqref{csdistanceintegral}, it thus follows that there exists $L<\infty$ such that in every interval $[a,b]\subseteq\RR$ in the thin part of length $\geq L$, there exists an $s$ such that $u(s,0)\in K$.
Now we apply the monotonicity argument (used above in the thick part) to conclude that $u(s,0)\in K$ implies an \emph{a priori} bound on $u$ over $[s-10L,s+10L]\times S^1$ as well.
This shows that over the thin part of the domain, $u$ is bounded uniformly away from infinity.
The same then follows for in the thick part as well using the claim proved in the previous paragraph and the fact that every component of the thick part is adjacent to the thin part.
\end{proof}

\subsection{Transversality}\label{sctranssec}

Let $\HJ_\bullet^\reg\subseteq\HJ_\bullet$ denote the collection of data for which all moduli spaces $\Mbar_n(H,J,\gamma^+,\gamma^-)$ are cut out transversely, including those associated to $(H,J)|_{\Delta^k}$ for facets $\Delta^k\hookrightarrow\Delta^n$.
Note that $\HJ_\bullet^\reg\subseteq\HJ_\bullet$ is indeed closed under degeneracy maps (note that this involves appealing to the fact that trivial cylinders are cut out transversely).

Standard Sard--Smale transversality arguments suffice to show that given any map $(\Delta^k,\partial\Delta^k)\to(\HJ_\bullet,\HJ_\bullet^\reg)$, the $J$ component can be perturbed away from $\partial\Delta^k$ to obtain a map $\Delta^k\to\HJ_\bullet^\reg$.
It is enough to perturb away from $\partial\Delta^k$ because gluing for holomorphic curves shows that transversality over $\partial\Delta^k$ implies transversality for $\Mbar_k(H,J,\gamma^+,\gamma^-)$ over $\Nbd\partial\Mbar^{SC}_k$.

Now this ``perturbed lifting property'' for maps $(\Delta^k,\partial\Delta^k)\to(\HJ_\bullet,\HJ_\bullet^\reg)$ implies (following Corollary \ref{hjfiltered}) that each $\HJ_\bullet^\reg$ is a filtered $\infty$-category and that each inclusion map $\HJ_\bullet^\reg\hookrightarrow\HJ_\bullet$ is cofinal; it follows that each forgetful map $\HJ_\bullet^\reg(X_0,\ldots,X_r)\to\HJ_\bullet^\reg(X_0,\ldots,X_{r-1})$ is cofinal.

\subsection{Homotopy coherent diagrams}\label{ndgdefsec}

We now formalize the notion of a homotopy coherent diagram of chain complexes.
Denote by $\Ch=\Ch_\ZZ^{\ZZ/2}$ the category of $\ZZ/2$-graded cochain complexes of $\ZZ$-modules.
A (strict) diagram of chain complexes is simply a map of simplicial sets $K\to\Ch$ (recall that we conflate any category with its nerve).
A homotopy coherent diagram of chain complexes is a map $K\to\Ndg\Ch$, where $\Ndg\Ch$ is the \emph{differential graded nerve} of $\Ch$, defined as follows (compare Lurie \cite[Construction 1.3.1.6]{lurieha}).
Recall the cubes $\F(\Delta^n)$ of broken Morse flow lines from Remark \ref{adamspathsremark}, and denote by $C_\bullet(\F(\Delta^n))$ the cubical chain complex of $\F(\Delta^n)=[0,1]^{n-1}$ (namely free of rank $3^{n-1}$).

\begin{definition}\label{ndgdef}
Denote by $\Ndg\Ch$ the simplicial set whose $p$-simplices are $(p+1)$-tuples of objects $A_0^\bullet,\ldots,A_p^\bullet\in\Ch$ along with (degree zero) chain maps
\begin{equation}
f_\sigma:A_{\sigma(0)}^\bullet\otimes C_{-\bullet}(\F(\Delta^q))\to A_{\sigma(q)}^\bullet
\end{equation}
for every map $\sigma:\Delta^q\to\Delta^p$ ($q\geq 1$) such that
\begin{itemize}
\item For $0<k<q$, we have $f_{\sigma|k\cdots q}\circ f_{\sigma|0\cdots k}=f_\sigma|_{C_{-\bullet}(\F(\Delta^k))\otimes C_{-\bullet}(\F(\Delta^{q-k}))}$ with respect to the natural map $\F(\Delta^k)\times\F(\Delta^{q-k})\to\F(\Delta^q)$ from \eqref{Fproduct}.
\item For every $\tau:\Delta^r\to\Delta^q$ with $\tau(0)=0$ and $\tau(r)=q$, we have $f_{\sigma\circ\tau}=f_\sigma\circ\tau_\ast$, where $\tau_\ast:\F(\Delta^r)\to\F(\Delta^q)$ is induced from $\tau$ as in Remark \ref{adamspathsremark}.
In the degenerate case $q=0$ (so $f_\sigma$ is not defined), we interpret $f_\sigma$ as the identity map above.
\end{itemize}
More generally, for any simplicial set $X_\bullet$, a map $X_\bullet\to\Ndg\Ch$ consists of an object $A^\bullet_v\in\Ch$ for every vertex $v\in X_0$ and maps $f_\sigma:A^\bullet_{\sigma(0)}\otimes C_{-\bullet}(\F(\Delta^p))\to A^\bullet_{\sigma(p)}$ for every $p$-simplex $\sigma\in X_p$ with $p\geq 1$ satisfying the two conditions above.
\end{definition}

Note the tautological functor from the classical nerve to the dg-nerve $\Ch\to\Ndg\Ch$, corresponding to taking $f_\sigma$ to factor through $C_\bullet(\F(\Delta^q))\to C_\bullet(\pt)=\ZZ$.

\begin{remark}
Definition \ref{ndgdef} extends immediately to any dg-category $\C$ in place of $\Ch$, though we should remark that in this more general context, rather than saying $f_\sigma:A_{\sigma(0)}^\bullet\otimes C_{-\bullet}(\F(\Delta^q))\to A_{\sigma(q)}^\bullet$ is a chain map, we should instead say $f_\sigma\in\Hom^\bullet(A_{\sigma(0)},A_{\sigma(q)})\otimes C^\bullet(\F(\Delta^q))$ is a cycle of degree zero.
\end{remark}

\begin{remark}
The set of $p$-simplices of $\Ndg\C$ can be equivalently described as the set of strictly unital $\ainf$-functors from the $A_{p+1}$-quiver $0\to\cdots\to p$ to $\C$. 
In particular, the face and degeneracy maps on a $p$-simplex $\Ndg\C$ are simply induced by composition with the co-face and co-degeneracy maps of the cosimplicial dg-category $\Delta_\bullet=\{\Delta_p=0 \to \cdots \to p\}_{p\geq 0}$.
\end{remark}

\begin{remark}
It is not difficult to check that $\Ndg\C$ is an $\infty$-category for any dg-category $\C$ \cite[Proposition 1.3.1.10]{lurieha}, though this fact is not essential to our arguments.
\end{remark}

\subsection{Homotopy colimits}\label{subsechocolim}

Let $X_\bullet$ be a simplicial set, and let $bX_\bullet$ denote the \emph{barycentric subdivision} of $X_\bullet$, namely a $p$-simplex of $bX_\bullet$ is a chain $\Delta^{a_p}\hookrightarrow\cdots\hookrightarrow\Delta^{a_0}\to X_\bullet$ (note that this is the ``opposite'' of what is usually called the barycentric subdivision).

For any diagram $A:bX_\bullet\to\Ndg\Ch$, we denote by
\begin{equation}
C_\bullet(X_\bullet;A):=\bigoplus_{\sigma:\Delta^n\to X_\bullet}A(\sigma)[n]
\end{equation}
the complex of ``simplicial chains on $X_\bullet$ with coefficients in $A$''.

\begin{remark}
The span of all degenerate simplices is a subcomplex and it is moreover easily seen to be acyclic (this is easy once it is observed that every degenerate simplex lies over a unique non-degenerate simplex).
It is thus equivalent to consider the quotient complex of ``normalized chains''.
In everything which follows, we could just as easily work with the normalized complex in place of the non-normalized version above.
\end{remark}

For any diagram $A:X_\bullet\to\Ndg\Ch$, we denote by
\begin{equation}\label{vanillahocolim}
\hocolim_{X_\bullet}A:=C_\bullet(X_\bullet;A\circ r)
\end{equation}
the ``homotopy colimit of $A$'', where $r:bX_\bullet\to X_\bullet$ is the canonical map sending a $p$-simplex $\Delta^{a_p}\hookrightarrow\cdots\hookrightarrow\Delta^{a_0}\xrightarrow\sigma X_\bullet$ in $bX_\bullet$ to the $p$-simplex $\sigma(0\in\Delta^{a_0},\ldots,0\in\Delta^{a_p})$ of $X_\bullet$.

Note that the diagrams $A:bX_\bullet\to\Ndg\Ch$ obtained by pre-composition with $r:bX_\bullet\to X_\bullet$ satisfy the property that $A(\sigma)\xrightarrow\sim A(\sigma|0\cdots\hat k\cdots n)$ is an isomorphism for any $k>0$ and any simplex $\sigma:\Delta^n\to X_\bullet$.
By abuse of notation, we also use
\begin{equation}\label{hocolimbarycentricabuse}
\hocolim_{X_\bullet}A
\end{equation}
to denote $C_\bullet(X_\bullet;A)$ for any diagram $A:bX_\bullet\to\Ndg\Ch$ with the property that $A(\sigma)\xrightarrow\sim A(\sigma|0\cdots\hat k\cdots n)$ is a \emph{quasi-}isomorphism for any $k>0$ (trivial but frequently used fact: it is equivalent to require this map be a quasi-isomorphism just for $k=n$).

\begin{remark}
The discussion above generalizes immediately to any dg-category $\C$ in place of $\Ch$.
Namely, for any diagram $A:bX_\bullet\to\Ndg\C$ we can form $C_\bullet(X_\bullet;A)\in\Tw^\oplus\C$ (the category of infinite direct sum twisted complexes of $\C$), we can form $\hocolim_{X_\bullet}A\in\Tw^\oplus\C$ for any $A:X_\bullet\to\Ndg\C$, etc.
\end{remark}

\begin{remark}
The homotopy colimit as defined above satisfies a universal property in the $\infty$-category $\Ndg\C$, though this fact is not essential to our arguments.
\end{remark}

The following result is standard.

\begin{lemma}\label{cofinalityI}
For any cofinal map $f:X_\bullet\to Y_\bullet$ between filtered $\infty$-categories and any diagram $A:Y_\bullet\to\C$, the natural map
\begin{equation}
\hocolim_{X_\bullet}A\circ f\xrightarrow\sim\hocolim_{Y_\bullet}A
\end{equation}
is a quasi-isomorphism.\qed
\end{lemma}

\begin{lemma}\label{finalobjectII}
Let $K$ be a simplicial set, and let $K^\vartriangleright$ be as in Definition \ref{trirightdef}.
The natural inclusion
\begin{equation}
A(\ast)\xrightarrow\sim\hocolim_{K^\vartriangleright}A
\end{equation}
is a quasi-isomorphism (even for the homotopy colimit in the generalized sense of \eqref{hocolimbarycentricabuse}).
\end{lemma}

\begin{proof}
Let $A:b(K^\vartriangleright)\to\Ndg\Ch$ be the diagram in question, and filter
\begin{equation}
\hocolim_{K^\vartriangleright}A=\bigoplus_{\sigma:\Delta^n\to K^\vartriangleright}A(\sigma)[n]
\end{equation}
by subcomplexes $(\hocolim_{K^\vartriangleright}A)_{\leq k}$ spanned by those simplices $\sigma:\Delta^n\to K^\vartriangleright$ for which $\sigma|k\ldots n$ is the degenerate $(n-k)$-simplex over $\ast$.

The subcomplex $(\hocolim_{K^\vartriangleright}A)_{\leq 0}$ is given by $A(\ast)$ tensored with simplicial chains on the point.
Each subsequent associated graded piece $(\hocolim_{K^\vartriangleright}A)_{\leq k}/(\hocolim_{K^\vartriangleright}A)_{\leq k-1}$ with $k\geq 1$ is the direct sum over all $(k-1)$-simplices $\sigma:\Delta^{k-1}\to K$ of (something quasi-isomorphic to) $A(\sigma(0))$ tensored with reduced simplicial chains on the point.
\end{proof}

\subsection{Symplectic cochain diagrams}\label{secsympcochaindiagrams}

By counting holomorphic curves in the moduli spaces $\Mbar_n(H,J,\gamma^+,\gamma^-)$, we produce a diagram of symplectic cochain complexes over the simplicial set $\HJ_\bullet^\reg(X)$, namely a map
\begin{equation}\label{hjdiagram}
\HJ_\bullet^\reg(X)\to\Ndg{\Ch}
\end{equation}
where $\Ndg\Ch$ is as in \S\ref{ndgdefsec}.

To a vertex $(H,J)\in\HJ_0^\reg(X)$, we associate the Floer complex $CF^\bullet(X;H)$ equipped with the differential arising from counting the zero-dimensional moduli spaces $\Mbar_0(H,J,\gamma^+,\gamma^-)$.
The Floer complex $CF^\bullet(X;H)$ is isomorphic as a $\ZZ$-module to the free abelian group on the fixed points of the flow map $\Phi_H:X\to X$ obtained from integrating $X_H$ around $S^1$.
More canonically, it is defined as the direct sum of orientation lines
\begin{equation}
CF^\bullet(X;H):=\bigoplus_{\Phi_H(x)=x}\oo_{\Phi_H,x}
\end{equation}
(compare the analogous discussion in \S\ref{wcurvessec}).
We recall briefly the definition of the orientation line $\oo_{\Phi,x}$ associated to a non-degenerate fixed point $x$ of a symplectomorphism $\Phi:X\to X$.  Consider the linearized
$\bar\partial$-operator at the constant map
$u:(-\infty,0]\times[0,1]\to X$ sending everything to $x$ and subject to $u(s,1)=\Phi(u(s,0))$.  We may view this operator as living on $(-\infty,0]\times S^1$ by gluing
$(s,1)\sim_\Phi(s,0)$ (if $\Phi=\Phi_H$, then this is equivalently described as the linearized operator associated to the equation $(du-X_H\otimes dt)^{0,1}_J=0$ at the constant solution $u:(-\infty,0]\times S^1\to X$ at $x$).
We may further glue on a disk to obtain a Cauchy--Riemann operator $D$ on the Riemann sphere with one
negative puncture (see Figure \ref{orientationlinessc}), and $\oo_{\Phi,x}:=\oo_D^\vee$ is defined as the dual of the Fredholm orientation line of $D$.
(The obstruction to extending the Cauchy--Riemann operator to the glued on disk lies in $\pi_1(BU(n))=0$; there is an ambiguity of $\pi_2(BU(n))=\ZZ$ in choosing such an extension (the relative Chern class), however the resulting orientation lines are all canonically isomorphic.)

\begin{figure}[hbt]
\centering
\includegraphics{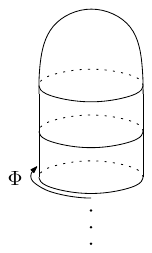}
\caption{The orientation line $\oo_{\Phi,x}$ is the dual of the Fredholm orientation line of the Cauchy--Riemann operator the Riemann surface illustrated above, extending the linearized $\bar\partial$-operator at the constant map $u:(-\infty,0]\times[0,1]\to X$ subject to $u(s,1)=\Phi(u(s,0))$ sending everything to $x$.}\label{orientationlinessc}
\end{figure}

A $1$-simplex $(H,J)\in\HJ_1^\reg(X)$ defines a chain map
\begin{equation}
F_{(H,J)}:CF^\bullet(X;H(0))\to CF^\bullet(X;H(1))
\end{equation}
by counting the zero-dimensional moduli spaces $\Mbar_1(H,J,\gamma^+,\gamma^-)$.
More generally, suppose $(H,J)\in\HJ_n^\reg(X)$ is an $n$-simplex with $n\geq 1$.
Now $\Mbar_n^{SC}=\F(\Delta^n)$ is a cube, and we can count (zero-dimensional components of) the inverse image in $\Mbar_n(H,J,\gamma^+,\gamma^-)$ of any of the $3^{n-1}$ strata of $\Mbar_n^{SC}$.
These counts define a chain map
\begin{equation}
\label{SChomotopyoperation}
F_{(H,J)}:CF^\bullet(X;H(0))\otimes C_{-\bullet}(\F(\Delta^n))\to CF^\bullet(X;H(n)).
\end{equation}

\begin{lemma}\label{scdiagramdegenerate}
The collection of maps \eqref{SChomotopyoperation} defines a diagram $\HJ_\bullet^\reg(X)\to\Ndg\Ch$ in the sense of Definition \ref{ndgdef}.
\end{lemma}

\begin{proof}
The compatibility conditions are all tautological except for the assertion that for any degenerate $(n+1)$-simplex $(H',J')$, we have
\begin{equation}\label{degenerateclaim}
F_{(H',J')}(-\otimes[\F(\Delta^{n+1})])=\begin{cases}0&n>0\\\id&n=0\end{cases}
\end{equation}
where $[\F(\Delta^{n+1})]$ denotes the top-dimensional generator of $C_{-\bullet}(\F(\Delta^{n+1}))$.
To prove \eqref{degenerateclaim}, argue as follows.

Say $(H',J')$ is obtained by pulling back an $n$-simplex $(H,J)$ under a surjection $\kappa_j:\Delta^{n+1}\to\Delta^n$, say mapping vertices $j+1$ and $j$ of $\Delta^{n+1}$ to the same vertex $j$ of $\Delta^n$ (any $0\leq j\leq n$).
Concretely, this means that our $(n+1)$-simplex $(H',J')$ is given by $\pi_j^\ast(H,J)$, where $\pi_j:\Cbar^{SC}_{n+1}\to\Cbar^{SC}_n$ is the map forgetting $a_{j+1}$ (see \eqref{floerdatadegeneracymaps}).
Thus every solution $(a_1,\ldots,a_{n+1},u)$ to \eqref{floereqnforsh} for $(H',J')=\pi_j^\ast(H,J)$ gives rise to a solution for $(H,J)$ simply by forgetting $a_{j+1}$.
This almost gives a map
\begin{equation}
\Mbar_{n+1}(\pi_j^\ast(H,J),\gamma^+,\gamma^-)\to\Mbar_n(H,J,\gamma^+,\gamma^-),
\end{equation}
except for the fact that forgetting $a_{j+1}$ might make the trajectory unstable.
Now we are interested in the case $\dim\Mbar_{n+1}(\pi_j^\ast(H,J),\gamma^+,\gamma^-)=0$, which means $\dim\Mbar_n(H,J,\gamma^+,\gamma^-)=-1$ and thus $\Mbar_n(H,J,\gamma^+,\gamma^-)=\varnothing$.
Thus to show \eqref{degenerateclaim}, we just need to analyze when forgetting $a_{j+1}$ can make the trajectory unstable.
Since $\Mbar_{n+1}(\pi_j^\ast(H,J),\gamma^+,\gamma^-)$ is cut out transversely and of dimension zero, it contains no split trajectories.
So, if there are any $a_i$ other than $a_{j+1}$, a given trajectory will remain stable upon forgetting $a_{j+1}$.
Thus the only remaining case is when $n=0$.
Now an unstable trajectory is one with an $\RR$-symmetry, i.e.\ a trivial cylinder, and these contribute exactly the identity map, as desired in \eqref{degenerateclaim}.
\end{proof}

Having completed the definition of the diagram of symplectic cochains \eqref{hjdiagram} over $\HJ_\bullet^\reg(X)$, we now define the symplectic cochain complex of a Liouville sector $X$ as the homotopy colimit (in the sense of \eqref{vanillahocolim} in \S\ref{subsechocolim}) of this diagram over $\HJ_\bullet^\reg(X)$, namely
\begin{equation}\label{scchainspre}
SC^\bullet(X,\partial X):=\hocolim_{\HJ_\bullet^\reg(X)}CF^\bullet(X;-).
\end{equation}
Note that since $\HJ_\bullet^\reg$ is filtered, the homology $SH^\bullet(X,\partial X)$ may be computed by taking the direct limit of $HF^\bullet(X;H)$ over any cofinal collection of $(H,J)$.

To define pushforward maps on $SH^\bullet$ for inclusions of Liouville sectors, observe that these diagrams of symplectic cochain complexes over $\HJ_\bullet^\reg(X)$ generalize directly to $\HJ_\bullet^\reg(X_0,\ldots,X_r)$.
Namely, using the forgetful maps $\HJ_\bullet^\reg(X_0,\ldots,X_r)\to\HJ_\bullet^\reg(X_i)$, we obtain via pullback diagrams $CF^\bullet(X_i;-)$ over $\HJ_\bullet^\reg(X_0,\ldots,X_r)$ for each $0\leq i \leq r$. By Lemma \ref{scconfinecurves}, there are inclusions of diagrams (over $\HJ_\bullet^\reg(X_0,\ldots,X_r)$)
\begin{equation}\label{shsubcomplexinclusion}
CF^\bullet(X_0;-)\subseteq\cdots\subseteq CF^\bullet(X_r;-).
\end{equation}
Now the functoriality of $SH^\bullet$ for inclusions of Liouville sectors may be defined by considering
\begin{equation}\label{scsimplefunctoriality}
\hocolim_{\HJ_\bullet^\reg(X)}CF^\bullet(X;-)\xleftarrow\sim\hocolim_{\HJ_\bullet^\reg(X,X')}CF^\bullet(X;-)\to\hocolim_{\HJ_\bullet^\reg(X')}CF^\bullet(X';-).
\end{equation}
The left arrow is a quasi-isomorphism by Lemma \ref{cofinalityI} since the forgetful map $\HJ_\bullet^\reg(X,X')\to\HJ_\bullet^\reg(X)$ is cofinal by Corollary \ref{hjfiltered}, and hence this defines a map $SH^\bullet(X,\partial X)\to SH^\bullet(X',\partial X')$.

\newlength{\lengthofbiggerhocolim}
\settowidth\lengthofbiggerhocolim{$\displaystyle\hocolim_{\HJ_\bullet^\reg(X_0,\ldots, \widehat{X_i}, \ldots, X_r)}$}

To define symplectic cochain complexes which are functorial under inclusions of Liouville sectors, note that there are natural maps
\begin{align}
\label{schjforgetxi}\hocolim_{\HJ_\bullet^\reg(X_0,\ldots,X_r)}CF^\bullet(X_0;-)&\to{\hocolim_{\HJ_\bullet^\reg(X_0,\ldots, \widehat{X_i}, \ldots, X_r)}}\,CF^\bullet(X_0;-)\qquad i>0\\
\label{schjforgetxifirst}\hocolim_{\HJ_\bullet^\reg(X_0,\ldots,X_r)}CF^\bullet(X_0;-)&\to\makebox[\lengthofbiggerhocolim][c]{$\displaystyle\hocolim_{\HJ_\bullet^\reg(X_1,\ldots, X_r)}$}\,CF^\bullet(X_1;-)
\end{align}
induced by the forgetful maps
$\HJ_\bullet^\reg(X_0,\ldots,X_r) \to \HJ_\bullet^\reg(X_0,\ldots,\widehat{X_i},
\ldots,X_r)$ (for $i>0$ this map of indexing simplicial sets is covered by ``the identity map'' on diagrams, and for $i=0$ it is covered by the natural inclusion $CF^\bullet(X_0; -) \subseteq CF^\bullet(X_1; -)$).
Since $\HJ_\bullet^\reg(X_0,\ldots,X_r)\to\HJ_\bullet^\reg(X_0,\ldots,X_{r-1})$ is cofinal, it follows that \eqref{schjforgetxi} is a quasi-isomorphism for $i=r$, which in turn implies that \eqref{schjforgetxi} is a quasi-isomorphism for all $i>0$.

We now consider the alternative definition
\begin{equation}\label{scchains}
SC^\bullet(X,\partial X):=\hocolim_{\begin{smallmatrix}X_0\subseteq\cdots\subseteq X_r\subseteq X\\X_i\text{ Liouville sectors}\end{smallmatrix}}\hocolim_{\HJ_\bullet^\reg(X_0,\ldots,X_r)}CF^\bullet(X_0;-),
\end{equation}
in which the outer $\hocolim$ (over the poset of Liouville subsectors of $X$) is taken in the sense of \eqref{hocolimbarycentricabuse}, which applies in this case since \eqref{schjforgetxi} is a quasi-isomorphism for $i>0$.
Concretely, \eqref{scchains} is the direct sum of $\hocolim_{\HJ_\bullet^\reg(X_0,\ldots,X_r)}CF^\bullet(X_0;-)$ over all chains $X_0\subseteq\cdots\subseteq X_r$ of Liouville subsectors of $X$, equipped with the differential which is (the internal differential of each direct summand, plus) maps \eqref{schjforgetxi}--\eqref{schjforgetxifirst} forgetting some $X_i$ (with appropriate signs).
It follows from Lemma \ref{finalobjectII} that the inclusion of the former version \eqref{scchainspre} of $SC^\bullet$ into the latter version \eqref{scchains} as the subcomplex $r=0$ and $X_0=X$ is a quasi-isomorphism.

The latter version \eqref{scchains} of $SC^\bullet$ is emminently functorial with respect to inclusions of Liouville sectors, and the induced maps on cohomology coincide with those defined by \eqref{scsimplefunctoriality}.

\subsection{Properties}

Let $X$ be a Liouville sector.
Let $V$ be a contact vector field on $\partial_\infty X$ which near the boundary is a cutoff Reeb vector field (see \S\ref{reebdynamicsboundary}) and whose time $2\pi$ flow has no fixed points (on the symplectization).
There is an invariant
\begin{equation}
SH^\bullet(X,\partial X)^{<V}
\end{equation}
namely the Hamiltonian Floer cohomology of some (any) Hamiltonian arising from
applying Lemma \ref{constructH} to $V$ (this is well-defined by our earlier
arguments).  There are continuation maps
\begin{equation}
SH^\bullet(X,\partial X)^{<V_1}\to SH^\bullet(X,\partial X)^{<V_2}
\end{equation}
for $V_1\leq V_2$ (meaning $\alpha(V_1)\leq\alpha(V_2)$ for any/all contact forms $\alpha$).
This order on $V$ is filtered (obviously), and we have, by definition, that
\begin{equation}
SH^\bullet(X,\partial X)=\varinjlim_VSH^\bullet(X,\partial X)^{<V}.
\end{equation}
There are also pushforward maps
\begin{equation}\label{filteredpushforward}
SH^\bullet(X,\partial X)^{<V}\to SH^\bullet(X',\partial X')^{<V'}
\end{equation}
for $X\hookrightarrow X'$ and $V\leq V'|_{X}$ with strict inequality over $\partial X$ (strict inequality over $\partial X$ is used to ensure that the admissible Hamiltonians associated to $V'$ will be larger than the admissible Hamiltonians associated to $V$ plus a constant).

\begin{lemma}\label{isoifnoneworbits}
Suppose $V_1\leq V_2$ and $tV_1+(1-t)V_2$ has no periodic orbits of period $2\pi$ for any $t\in[0,1]$.
Then, the continuation map $SH^\bullet(X,\partial X)^{<V_1}\to SH^\bullet(X,\partial X)^{<V_2}$ is an isomorphism.
\end{lemma}

\begin{proof}
Let $H_1:X\to\RR$ be linear at infinity corresponding to $V_1$.
We will construct a Hamiltonian $H_2:X\to\RR$ corresponding to $V_2$ at infinity along with smoothings $\tilde H_1,\tilde H_2$ as in Lemma \ref{constructH} satisfying the following properties:
\begin{itemize}
\item $\tilde H_1\leq\tilde H_2$.
\item $\tilde H_1=\tilde H_2$ over a neighborhood of their periodic orbits.
\end{itemize}
Now the map $SH^\bullet(X,\partial X)^{<V_1}\to SH^\bullet(X,\partial X)^{<V_2}$ is realized by the continuation map $HF^\bullet(X;\tilde H_1)\to HF^\bullet(X;\tilde H_2)$.
We may choose Floer data for this continuation map (i.e.\ a $1$-simplex in $\HJ_\bullet^\reg$) for which $-\partial_sH\geq 0$ (non-negative wrapping).
Energy considerations and \eqref{poswrappinggood} thus imply that this continuation map is the identity map plus a map which strictly decreases action, and hence is an isomorphism.
It is thus enough to construct $\tilde H_1$ and $\tilde H_2$ as above.

Fix a contact form $\alpha$ on $Y=\partial_\infty X$, thus fixing symplectization coordinates $(\RR\times Y,e^s\alpha)\to(X,\lambda)$ near infinity.
Let us denote the Hamiltonians for $V_1$ and $V_2$ by $e^sA=:H_1$ and $e^sB$ for $A,B:Y\to\RR$, which by assumption have no periodic orbits of period $2\pi$.
We now consider the interpolation $H_2:=e^s((1-\varphi(s))A+\varphi(s)B)$ for some $\varphi:\RR\to[0,1]$ such that $\varphi(s)=0$ for $s$ sufficiently negative and $\varphi(s)=1$ for $s$ sufficiently positive.
We claim that for $\varphi'(s)$ sufficiently small, this interpolation also has no periodic orbits of period $2\pi$.
To see this, we calculate
\begin{equation}\label{interpolatedvf}
X_{e^s((1-\varphi(s))A+\varphi(s)B)}=(1-\varphi(s))X_{e^sA}+\varphi(s)X_{e^sB}+\varphi'(s)(B-A)\R_\alpha.
\end{equation}
Note that each of the terms $X_{e^sA}$, $X_{e^sB}$, and $(B-A)\R_\alpha$ are $Z$-invariant vector fields.
As in the proof of Lemma \ref{compactcutoff}, we fix a defining function $I$ and consider its derivatives with respect to each of these vector fields.
Since $X_I$ is outward pointing, we conclude that $X_{e^sA}I$ and $X_{e^sB}I$ vanish transversally on $\partial X$ and are positive in the interior, and $(B-A)\R_\alpha I$ vanishes to second order on $\partial X$.
Since these functions are all linear at infinity, we conclude that the derivative of $I$ by \eqref{interpolatedvf} is positive over $(\Nbd^Z\partial X)\setminus\partial X$, where the size of the neighborhood depends only on an upper bound on $\varphi'(s)$.
In particular, there are no periodic orbits contained entirely inside this neighborhood $(\Nbd^Z\partial X)\setminus\partial X$.
For $\varphi'(s)$ sufficiently small, this vector field \eqref{interpolatedvf} is also very close to $\varphi(s)V_1+(1-\varphi(s))V_2$ and hence has no periodic orbits of period $2\pi$ which intersect an arbitrarily large compact subset of the interior of $\partial_\infty X$.
Thus we have shown that this interpolated Hamiltonian $H_2$ coincides with $H_1=e^sA$ for $s$ small, equals $e^sB$ for $s$ large, and has no periodic orbits of period $2\pi$.

Finally, we need to apply Lemma \ref{constructH} to produce admissible $\tilde H_1\leq\tilde H_2$ from our pre-admissible $H_1\leq H_2$.
To see that this creates no new orbits near the boundary, note that the reasoning in Lemma \ref{constructH} about $X_{\tilde H}I$ for $\frac 12$-defining functions $I$ applies equally well to ($1$-)defining functions $I$ as considered above and shows that $X_{\tilde H}I\geq c\cdot\max(1,R)e^{\frac 12s}$ over $\Nbd^Z\partial X$ for some $c>0$.
Finally, to ensure that the region where $\tilde H_1<\tilde H_2$ is sufficiently close to infinity and thus away from the periodic orbits, we should replace $\varphi(s)$ with $\varphi(s-s_0)$ for sufficiently large $s_0$.
\end{proof}

\begin{corollary}
Up to canonical isomorphism, $SH^\bullet(X,\partial X)^{<V}$ is invariant under deformation of $V$ through contact vector fields with no periodic orbits of period $2\pi$ on the symplectization.
\end{corollary}

\begin{proof}
Note that the property of having no periodic orbits of period $2\pi$ is an open condition among contact vector fields which are cutoff Reeb near the boundary by Lemma \ref{compactcutoff}.
Using this, we can define the deformation isomorphism in terms of the pushforward maps \eqref{filteredpushforward} and their inverses, using the fact that these are isomorphisms by Lemma \ref{isoifnoneworbits}.
\end{proof}

\begin{lemma}\label{calculateSHwithsimpleboundaryreeb}
Fix coordinates near $\partial(\partial_\infty X)$ as in \S\ref{reebdynamicsboundary} and fix $M:\RR_{\geq 0}\to\RR_{\geq 0}$ (defined in a neighborhood of $[0,t_0]$) which is admissible in the sense of \S\ref{reebdynamicsboundary}.
The natural map
\begin{equation}\label{goodtoarbitraryboundary}
\varinjlim_{\begin{smallmatrix}V\\V|_{[0,t_0+\varepsilon]}=V_M\end{smallmatrix}}SH^\bullet(X,\partial X)^{<V}\to SH^\bullet(X,\partial X)
\end{equation}
is an isomorphism, where the direct limit on the left is over those $V$ which coincide with $V_M$ (the cutoff Reeb vector field defined by $M(t)$) over $\partial(\partial_\infty X)\times[0,t_0+\varepsilon]$ for some $\varepsilon>0$.
\end{lemma}

\begin{proof}
Let $N:\RR_{\geq 0}\to\RR_{\geq 0}$ be admissible, and suppose $N(t)\geq M(t)$.
Fix any $V_2$ which is given by the contact Hamiltonian $N(t)$ over $[0,t_0+\varepsilon]$ for some $\varepsilon>0$.
Now let $\varphi:\RR\to[0,1]$ be a cutoff function with $\varphi(x)=0$ for $x\leq t_0$ and $\varphi(x)=1$ for $x\geq t_0+\varepsilon$.
We define $V_1$ to coincide with $V_2$ away from $[0,t_0+\varepsilon]$ and to be given by the contact Hamiltonian $M(t)+\varphi(t)\cdot(N(t)-M(t))$ over $[0,t_0+\varepsilon]$.

With these choices of $V_1$ and $V_2$, the continuation map
\begin{equation}\label{isoconstructed}
SH^\bullet(X,\partial X)^{<V_1}\to SH^\bullet(X,\partial X)^{<V_2}
\end{equation}
is an isomorphism by Lemma \ref{isoifnoneworbits}.
Indeed, $V_1=V_2$ except over a neighborhood of $\partial(\partial_\infty X)$, and trajectories passing through this neighborhood cannot be closed by Proposition \ref{cutofflemma}.
Taking the direct limit of \eqref{isoconstructed} over $N(t)$ and $V_2$, we obtain \eqref{goodtoarbitraryboundary} by cofinality.
\end{proof}

\begin{proposition}\label{trivsciso}
Let $X\subseteq X'$ be a trivial inclusion of Liouville sectors.
The induced map $SH^\bullet(X,\partial X)\to SH^\bullet(X',\partial X')$ is an isomorphism.
\end{proposition}

\begin{proof}
We claim that it is enough to produce, for every Liouville sector $X$, a (globally defined) defining function $I:X\to\RR$ (linear at infinity) and a cofinal collection of admissible Hamiltonians $\tilde H:S^1\times X\to\RR$, each satisfying
\begin{equation}\label{supbound}
\sup_{S^1\times X}X_I\tilde H<\infty.
\end{equation}
Indeed, suppose such an $I$ and collection of $\tilde H$ are given.
The Hamiltonian flow of $I$ defines a coordinate system $\partial X\times\RR_{t\geq 0}\to X$ near $\partial X$ and moreover defines Liouville sectors $X_a:=\{t\geq-a\}$ by ``flowing out of the boundary for time $a$''.
A sandwiching argument as in the proof of Lemma \ref{wrappedtrivialiso} shows that it is enough to show that the inclusions $X\hookrightarrow X_a$ induce isomorphisms on $SH^\bullet$.
Now consider one of the given Hamiltonians $\tilde H$ on $X$.
Using the Hamiltonian flow of $I$, we push $\tilde H$ forward to $X_a$, and denote the result by $\tilde H_a$.
By \eqref{supbound}, we have $\tilde H_a\geq\tilde H-C$ for some $C<\infty$, and hence there is a well-defined continuation map
\begin{equation}\label{contisofortriv}
HF^\bullet(X;\tilde H)\to HF^\bullet(X_a;\tilde H_a)
\end{equation}
for any $a>0$.
As $\tilde H$ varies over a cofinal collection of admissible Hamiltonians for $X$, so does $\tilde H_a$ for $X_a$.
By a direct limit argument, it is thus enough to show that \eqref{contisofortriv} is an isomorphism for all $a>0$.
We can moreover assume that $a>0$ is sufficiently small, since the maps \eqref{contisofortriv} compose as expected (note that for any $b\geq 0$, the map $HF^\bullet(X_b;\tilde H_b)\to HF^\bullet(X_{b+a};\tilde H_{b+a})$ is isomorphic to \eqref{contisofortriv}).

To show that \eqref{contisofortriv} is an isomorphism for sufficiently small $a>0$, argue as follows.
First, note that we may perturb $\tilde H$ so that no two of its (finitely many!) periodic orbits have the same action.
Now we cutoff $I$ in a neighborhood of all periodic orbits of $\tilde H$ to obtain $I'$, and we use the Hamiltonian flow of $I'$ to push $\tilde H$ forward to $X_a$, denoting the result by $\tilde H_a'$.
By \eqref{supbound}, for sufficiently small $a>0$ we have $\tilde H_a'\geq\tilde H-\varepsilon$, where $\varepsilon>0$ is smaller than the smallest action difference of any pair of periodic orbits of $\tilde H$.
It follows that the continuation map
\begin{equation}
CF^\bullet(X;\tilde H)\to CF^\bullet(X_a;\tilde H_a')
\end{equation}
is (for appropriate choice of extension of $\tilde H$ from $X$ to $X_a$) the ``identity'' plus a map which strictly decreases action, and hence is an isomorphism.
Since $\tilde H_a'-\tilde H_a$ has compact support, this map coincides with \eqref{contisofortriv}, giving the desired result.
We conclude that, as claimed, it suffices to construct $I$ and cofinal $\tilde H$ satisfying \eqref{supbound}; in fact, our construction below ensures that $X_I\tilde H\leq 0$ near infinity.

To construct $I$ and $\tilde H$, argue as follows.
Fix a linear $I_0:\Nbd^Z\partial X\to\RR$ defined near infinity with $X_{I_0}$ outward pointing along $\partial X$, and fix coordinates $\partial(\partial_\infty X)\times\RR_{t\geq 0}\to\partial_\infty X$ as in \S\ref{reebdynamicsboundary} in which $\frac\partial{\partial t}$ is the contact vector field induced by $-X_{I_0}$.
Define $I:=N\cdot I_0$, for a cutoff function $N(t)$ which is $1$ in a neighborhood of zero and is $0$ outside a small neighborhood of zero, with $N'(t)\leq 0$.
We now consider $H:\Nbd^Z\partial X\to\RR$ given by an admissible contact Hamiltonian $M(t)$ as in \S\ref{reebdynamicsboundary}, and we claim that $X_IH\leq 0$ near infinity.
First, note that
\begin{align}
X_IH&=NX_{I_0}H+I_0X_NH\\
&=NX_{I_0}H-I_0X_HN.
\end{align}
We now calculate in coordinates $\RR_s\times\eqref{YcoordsnearGamma}=\RR_s\times\Gamma_{\partial Y}\times\RR_{\left|u\right|\leq\varepsilon}\times\RR_{t\geq 0}$ with Liouville form $\lambda=e^s\cdot\eqref{keycontactform}=\frac{e^s}{\psi(u)}(\mu+u\,dt)$ (which tell the whole story).
We have $H=e^sM(t)$ (by definition) and $X_{I_0}=-\frac\partial{\partial t}$ (since this preserves the Liouville form and lifts $-\frac\partial{\partial t}$ on the contact boundary), so
\begin{equation}
NX_{I_0}H=-e^sN(t)M'(t)\leq 0.
\end{equation}
Now note that $I_0=ZI_0=\omega(Z,X_{I_0})=\lambda(X_{I_0})=-e^su/\psi(u)$.
Hence by \eqref{complicatedreeb}, we have
\begin{equation}
-I_0X_HN=e^s\frac u{\psi(u)}\psi'(u)M(t)N'(t)\leq 0
\end{equation}
since $u\psi'(u)\geq 0$.
This shows $X_IH\leq 0$ near infinity as claimed.

Now we claim that $X_I\tilde H\leq 0$ near infinity as well, where $\tilde H$ is obtained from $H$ by the procedure of Lemma \ref{constructH} (clearly this is enough, since such $\tilde H$ are cofinal among admissible Hamiltonians as $M(t)$ varies).
Since the locus where $\tilde H\ne H$ is disjoint from where we cut off $I_0$ to obtain $I$, it is equivalent to show that
\begin{equation}
X_{I_0}\tilde H\leq 0
\end{equation}
over this locus.

Recall from the proof of Lemma \ref{constructH} that $\tilde H$ is defined by smoothing $\max(R,H)$ over the strip $\{\frac 1N\leq R\leq N\}$, over which both $R$ and $H$ are $C^\infty$ bounded (we should warn the reader that here $R=\Re\pi$, but $I$ and $I_0$ will continue to denote the functions defined above, not $\Im\pi$).
It thus would be enough to show that $-X_{I_0}H>\varepsilon>0$, $-X_{I_0}R>\varepsilon>0$, and that $X_{I_0}$ is $C^\infty$ bounded over $\{\frac 1N\leq R\leq N\}$, for some $\varepsilon>0$.
The vector field $X_{I_0}$ is, however, not even $C^0$ bounded in the required sense, as $\sL_ZX_{I_0}=0$ and we are measuring with respect to metrics $g$ satisfying $\sL_Zg=g$.
To fix this, however, it suffices to simply consider $e^{-\frac 12s}X_{I_0}$ in its place.
The scaling behavior $\sL_Z(e^{-\frac 12s}X_{I_0})=-\frac 12(e^{-\frac 12s}X_{I_0})$ together with $\sL_Zg=g$ implies easily that $e^{-\frac 12s}X_{I_0}$ is $C^\infty$ bounded over the strip $\{\frac 1N\leq R\leq N\}$.
It is thus enough to show lower bounds
\begin{equation}\label{weirdlowerboundI}
\begin{cases}-e^{-\frac 12s}X_{I_0}H>\varepsilon>0\\-e^{-\frac 12s}X_{I_0}R>\varepsilon>0\end{cases}\quad\text{over the strip }{\textstyle\{\frac 1N\leq R\leq N\}}\text{ near infinity.}
\end{equation}
Since $X_{I_0}$ is outward pointing along $\partial X$, it follows that
\begin{equation}\label{weirdlowerboundII}
\begin{cases}-X_{I_0}R>\varepsilon e^{\frac 12s}\\-X_{I_0}H>\varepsilon e^{\frac 12s}R\end{cases}\quad\text{over }\Nbd^Z\partial X.
\end{equation}
Indeed, since in each of these inequalities, both sides have the same scaling behavior under $Z$, it is enough that $\varepsilon>0$ exist over a neighborhood of any compact subset of $\partial X$, which holds since $X_{I_0}$ is outward pointing (recall that $H$ vanishes to order two along $\partial X$ with positive second derivative).
Clearly \eqref{weirdlowerboundII} implies \eqref{weirdlowerboundI}.
\end{proof}

\begin{corollary}\label{scdefinvariance}
The invariant $SH^\bullet(X,\partial X)$ is invariant under deformation of $X$ up to canonical isomorphism.
\end{corollary}

\begin{proof}
An arbitrary deformation may be factored as a composition of trivial inclusions and their inverses, so the result follows from Proposition \ref{trivsciso}.
\end{proof}

It is natural to expect that $SH^\bullet(X,\partial X)$ satisfies a K\"unneth formula generalizing that for symplectic cohomology of Liouville manifolds proved by Oancea \cite{oanceakunneth} (see also Groman \cite{groman}).
A careful proof of this is, however, beyond the scope of this paper.

\begin{conjecture}
There is a natural quasi-isomorphism $SC^\bullet(X,\partial X)\otimes SC^\bullet(X',\partial X')=SC^\bullet(X\times X',\partial(X\times X'))$.
\end{conjecture}

\subsection{From cohomology to symplectic cohomology}

Recall that when $H$ and $J$ are both $S^1$-invariant, Floer trajectories with $\partial_tu\equiv 0$ are simply Morse trajectories of $H$ with respect to the metric induced by $J$.
In particular, there is a map from Morse trajectories to Floer trajectories.
For sake of clarity in the following discussion, we may indicate the choice of projection $\pi:\Nbd^Z\partial X\to\CC$ (recall Convention \ref{choosingpi}) by adding a subscript, as in $X_\pi$.
Note that the rescalings $a\cdot\pi$ ($a\in\RR_{>0}$) of any given projection $\pi$ are also valid projections.

\begin{proposition}\label{morsefloer}
Let $(H,J)\in\HJ_n(X_\pi)$ be $S^1$-invariant, with $H(i)$ Morse for vertices $i\in\Delta^n$, and suppose that all Morse trajectories are cut out transversely.
For sufficiently small $\delta>0$, the map from Morse trajectories to Floer trajectories for $(\delta\cdot H,J)$ is bijective and all Floer trajectories are regular (i.e.\ $(\delta\cdot H,J)\in\HJ_n^\reg(X_{\delta\pi})$).
\end{proposition}

\begin{proof}
For $X$ closed, this is a fundamental result due to Floer \cite[Theorem 2]{floerwitten} (also presented conceptually in Salamon--Zehnder \cite[Theorem 7.3 0.1]{salamonzehnder}).
To extend this result to the present setting, it is enough to show that in the limit $\delta\to 0$, all Floer trajectories are contained in a fixed compact subset of $X$.
In other words, we must show that the compact subset of $X$ produced by the proof of Proposition \ref{compactness} (compactness) can be made uniform as $\delta\to 0$.

The only part of the proof of Proposition \ref{compactness} which is potentially problematic as $\delta\to 0$ is the argument surrounding \eqref{csdistanceintegral}, specifically the upper bound $L\asymp E^\geo/\varepsilon^2$ on the length of an interval $[a,b]$ with the property that $u(s,0)\notin K$ for all $s\in[a,b]$.
Both $\varepsilon$ and $E^\top$ clearly scale linearly with $\delta$, and the same holds for $E^\geo$ by \eqref{poswrappinggood}.
The upper bound $E^\geo/\varepsilon^2$ is therefore unfortunately not uniform in $\delta$.
However, this upper bound does lead to
\begin{equation}
\int_{[a,b]\times S^1}\left|\partial_su\right|\lesssim\biggl((b-a)\int_{[a,b]\times S^1}\left|\partial_su\right|^2\biggr)^{1/2}\lesssim\biggl(\frac{E^\geo}{\varepsilon^2}E^\geo\biggr)^{1/2}=\frac{E^\geo}\varepsilon
\end{equation}
for any interval $[a,b]$ with the property that $u(s,0)\notin K$ for all $s\in[a,b]$.
Thus on any such interval, there exists a $t\in S^1$ such that
\begin{equation}\label{lineintegralinequality}
\int_{[a,b]\times\{t\}}\left|\partial_s u\right|\leq\frac{E^\geo}\varepsilon,
\end{equation}
where the right hand side is now bounded uniformly as $\delta\to 0$.
This upper bound \eqref{lineintegralinequality} gives the desired uniform bound on the image of $u$ as $\delta\to 0$.

Alternatively, uniform compactness as $\delta\to 0$ follows from the following covering trick.
A trajectory for $\delta\cdot H$ pulls back under the $N$-fold covering $S^1\to S^1$ to a trajectory for $N\delta\cdot H$.
The proof of Proposition \ref{compactness} is obviously uniform over $a\cdot H$ for (say) $a\in[1,2]$, and for all sufficiently small $\delta>0$ there exists a integer $N$ such that $N\delta\in[1,2]$.
\end{proof}

\begin{proposition}\label{shpss}
Proposition \ref{morsefloer} defines a canonical map $H^\bullet(X,\partial X)\to SH^\bullet(X,\partial X)$, and this map is functorial with respect to inclusions of Liouville sectors.
\end{proposition}

\begin{proof}
For $(H,J)\in\HJ_0(X_\pi)$ which is $S^1$-invariant and Morse with transversally cut out Morse trajectories, Proposition \ref{morsefloer} guarantees that for sufficiently small $\delta>0$, we have $(\delta H,J)\in\HJ_0^\reg(X_{\delta\pi})$ and all Floer trajectories are Morse trajectories.
Hence we obtain a map $H^\bullet(X,\partial X) = HF^\bullet(X; \delta H, J) \to SH^\bullet(X,\partial X)$ (note that there is no need for a subscript indicating the choice of projection to $\CC_{\Re\geq 0}$ on the target $SH^\bullet$ in view of Proposition \ref{scdefinvariance}).
This map $H^\bullet(X,\partial X)\to SH^\bullet(X,\partial X)$ depends \emph{a priori} on the choice of $(H,J)\in\HJ_0(X_\pi)$ and the choice of sufficiently small $\delta>0$.

Next, given an inclusion of Liouville sectors $X\hookrightarrow X'$, along with the data of a vertex $(H,J)\in\HJ_0(X_\pi)$, a vertex $(H_e,J_e)\in\HJ_0(X_\pi,X'_{\pi'})$ whose restriction to $X$ is $(H,J)$, a vertex $(H',J')\in\HJ_0(X'_{\pi'})$, and a $1$-simplex from (the image in $\HJ_0(X'_{\pi'})$ of) $(H_e,J_e)$ to $(H',J')$ (which are, as before, all $S^1$-invariant and Morse with transversally cut out Morse trajectories), Proposition \ref{morsefloer} implies that for sufficiently small $\delta > 0$, there is a commutative diagram
\begin{equation}\label{morsefloerpushforward}
\begin{tikzcd}[column sep = small]
HF^\bullet(X; \delta H,J) \ar{r} \ar[equal]{d}& HF^\bullet(X'; \delta H_e, J_e)\ar{r} \ar[equal]{d} & HF^\bullet(X'; \delta H', J') \ar[equal]{d} \\
H^\bullet(X, \partial X) \ar{r} & H^\bullet(X', \partial X') \ar{r} & H^\bullet(X', \partial X')
\end{tikzcd}
\end{equation} 
where the left horizontal maps are induced by the subcomplex inclusions \eqref{shsubcomplexinclusion} and the right horizontal maps are the Floer (respectively Morse) continuation maps induced by the $1$-simplex with Hamiltonian term scaled by $\delta$.
This implies that for such sufficiently small $\delta>0$, the maps $H^\bullet(X,\partial X)\to SH^\bullet(X,\partial X)$ and $H^\bullet(X',\partial X')\to SH^\bullet(X',\partial X')$ defined by $(H,J,\delta)$ and $(H',J',\delta)$, respectively, are compatible with pushforward under $X\hookrightarrow X'$.

Now the key step is to consider a \emph{$1$-parameter family} of data as in the paragraph above, where $X_\pi$, $(H,J)$, and $X'$ are fixed, but $\pi'$ and $(H',J')$ vary in the parameter $t\in(-\varepsilon,\varepsilon)$ as $(1+t)\pi'$ and $((1+t)H',J')\in\HJ_0(X'_{(1+t)\pi'})$, and the remaining data of $(H_e,J_e)\in\HJ_0(X_\pi,X'_{(1+t)\pi'})$ and the $1$-simplex are chosen arbitrarily.
The estimate on $\delta>0$ in Proposition \ref{shpss} is uniform over $t$ in a sufficiently small neighborhood of zero, and hence we conclude that the maps $H^\bullet(X,\partial X)\to SH^\bullet(X,\partial X)$ and $H^\bullet(X',\partial X')\to SH^\bullet(X',\partial X')$ defined by $(H,J,\delta)$ and $(H',J',(1+t)\delta)$, respectively, are compatible with pushforward under $X\hookrightarrow X'$ \emph{for every $\left|t\right|\leq\varepsilon$}.
Taking $X\hookrightarrow X'$ to be a trivial inclusion (or in fact any map for which $H^\bullet(X,\partial X)\to H^\bullet(X',\partial X')$ is surjective), it follows that the map $H^\bullet(X',\partial X')\to SH^\bullet(X',\partial X')$ defined by $(H',J',\delta)$ is independent of $\delta$ for sufficiently small $\delta>0$.

Knowing that the map $H^\bullet(X,\partial X)\to SH^\bullet(X,\partial X)$ for a given $(H,J)$ is independent of the choice of sufficiently small $\delta>0$, it is now straightforward to check independence of $(H,J)$ (use continuation maps between different $(H,J)$) and compatibility with pushforward (use the reasoning surrounding \eqref{morsefloerpushforward}).
\end{proof}

In fact the above argument shows that there is a canonical isomorphism $H^\bullet(X,\partial X)=SH^\bullet(X,\partial X)^{<V}$ for sufficiently small cutoff Reeb vector fields $V$ (note that Lemma \ref{compactcutoff} implies that every cutoff Reeb vector field has no small time periodic orbits).

\begin{proposition}\label{hhshiso}
If $\partial_\infty X$ admits (up to deformation) a cutoff Reeb vector field with no periodic orbits (for example, this holds if $\partial_\infty X$ is deformation equivalent to $F\times[0,1]$ by Lemma \ref{cylinderstopped}), then the natural map $H^\bullet(X,\partial X)\to SH^\bullet(X,\partial X)$ is an isomorphism.
\end{proposition}

\begin{proof}
Let $V$ be the given cutoff Reeb vector field on $\partial_\infty X$ with no periodic orbits.
As remarked above, the map $H^\bullet(X,\partial X)\to SH^\bullet(X,\partial X)$ arises from an isomorphism $H^\bullet(X,\partial X)=SH^\bullet(X,\partial X)^{<\delta\cdot V}$ for $\delta>0$ sufficiently small.
Now the continuation map $SH^\bullet(X,\partial X)^{<\delta\cdot V}\to SH^\bullet(X,\partial X)^{<N\cdot V}$ is an isomorphism for all $N<\infty$ by Lemma \ref{isoifnoneworbits} since $V$ has no closed orbits.
The result now follows by taking the direct limit as $N\to\infty$.
\end{proof}

\begin{proposition}\label{mfchain}
For any diagram of Liouville sectors $\{X_\sigma\}_{\sigma\in\Sigma}$ indexed by a finite poset $\Sigma$, there is a corresponding map $C^\bullet(X_\sigma,\partial X_\sigma)\to SC^\bullet(X_\sigma,\partial X_\sigma)$ of diagrams $\Sigma\to\Ndg\Ch$.
\end{proposition}

\begin{proof}
For every chain $\sigma_0\leq\cdots\leq\sigma_r\in\Sigma$, we choose as follows a finite subcomplex
\begin{equation}
\HJ_\bullet^\morse(\sigma_0,\ldots,\sigma_r)\subseteq\HJ_\bullet(X_{\sigma_0},\ldots,X_{\sigma_r})
\end{equation}
consisting of simplices satisfying the hypotheses of Proposition \ref{morsefloer}.
For chains with strict inequalities $\sigma_0<\cdots<\sigma_r$, we define $\HJ_\bullet^\morse$ by downwards induction on $r$ as
\begin{equation}
\HJ_\bullet^\morse(\sigma_0,\ldots,\sigma_r):=\left(\colim_{\begin{smallmatrix}\underline\tau_0<\sigma_0<\underline\tau_1<\sigma_1<\cdots<\sigma_r<\underline\tau_{r+1}\\\text{some }\underline\tau_i\text{ nonempty}\end{smallmatrix}}\HJ_\bullet^\morse(\underline\tau_0,\sigma_0,\underline\tau_1,\sigma_1,\ldots,\sigma_r,\underline\tau_{r+1})\right)^\vartriangleright
\end{equation}
where each $\underline\tau_i$ is a chain inside $\Sigma$.
Concretely, this means we consider the union of all possible $\HJ_\bullet^\morse(\underline\tau_0,\sigma_0,\underline\tau_1,\sigma_1,\ldots,\sigma_r,\underline\tau_{r+1})$ and adjoin a final object in $\HJ_\bullet(\sigma_0,\ldots,\sigma_r)$, which exists since $\HJ_\bullet(\sigma_0,\ldots,\sigma_r)$ is filtered.
We choose this final vertex to restrict to a cutoff Reeb vector field on $\partial_\infty X_{\sigma_0}$.
For general chains $\sigma_0\leq\cdots\leq\sigma_r$, simply forget the repetitions.

Since $\Sigma$ is finite and each $\HJ_\bullet^\morse(\sigma_0,\ldots,\sigma_r)$ is finite, sufficiently small $\delta>0$ satisfy the conclusion of Proposition \ref{morsefloer} for all simplices in all $\HJ_\bullet^\morse$ at once.
Since each $\HJ_\bullet^\morse$ has a final object, the complex
\begin{equation}
\hocolim_{\delta\cdot\HJ_\bullet^\morse(\sigma_0,\ldots,\sigma_r)}CF^\bullet(X_{\sigma_0};-)
\end{equation}
calculates $C^\bullet(X_{\sigma_0},\partial X_{\sigma_0})$ for sufficiently small $\delta>0$.
Thus the obvious inclusion of
\begin{equation}
\hocolim_{\sigma_0\leq\cdots\leq\sigma_r\leq\sigma}\hocolim_{\delta\cdot\HJ_\bullet^\morse(\sigma_0,\ldots,\sigma_r)}CF^\bullet(X_{\sigma_0};-)
\end{equation}
into \eqref{scchains} is the desired diagram.
\end{proof}

\begin{corollary}\label{schomologyhypercover}
Let $\{X_\sigma\}_{\sigma\in\Sigma}$ be a collection of Liouville sectors indexed by a poset $\Sigma$ which form a homology hypercover of a Liouville manifold $X$.
The induced map
\begin{equation}
\hocolim_{\sigma\in\Sigma}SC^\bullet(X_\sigma,\partial X_\sigma)\to SC^\bullet(X)
\end{equation}
hits the unit.
\end{corollary}

\begin{proof}
We consider the diagram of Liouville sectors over $\Sigma^\vartriangleright$ obtained from the given diagram over $\Sigma$ by adjoining $X$ as a final object.
We apply Proposition \ref{mfchain} to this diagram.
Taking homotopy colimits over $\Sigma$ produces the right half of the diagram \eqref{keydiagram}, which implies the desired result (noting that $H^\bullet(X)\to SH^\bullet(X)$ sends the unit to the unit).
\end{proof}

\section{Open-closed map for Liouville sectors}\label{secopenclosed}

For any Liouville sector $X$, we define an open-closed map $\OC:HH_\bullet(\W(X))\to SH^{\bullet+n}(X,\partial X)$, which respects the functoriality of both sides with respect to inclusions of Liouville sectors.
In fact, for any diagram of Liouville sectors $\{X_\sigma\}_{\sigma\in\Sigma}$ indexed by a finite poset $\Sigma$, we define a corresponding diagram of maps $\{\OC:CC_\bullet(\W(X_\sigma))\to SC^{\bullet+n}(X_\sigma,\partial X_\sigma)\}_{\sigma\in\Sigma}$.

To define the open-closed map, we use a method introduced by Abouzaid--Ganatra \cite{abouzaidganatra} whereby we take the domain of $\OC$ to be $CC_\bullet(\OO,\B)$, where $\OO$ is as in \S\ref{secwrapped} and $\B$ is a geometrically defined $\OO$-bimodule quasi-isomorphic as $\OO$-bimodules to $\W$ (a similar construction of $\ainf$-bimodules appears in Seidel \cite{seidelfukayalefschetzI}).
General properties of localization imply that $HH_\bullet(\OO,\B)=HH_\bullet(\OO,\W)=HH_\bullet(\W)$ (due to Abouzaid--Ganatra).

\subsection{Hochschild homology}

We recall some standard results concerning the Hochschild homology of $\ainf$-categories.

\begin{definition}
For any $\C$-bimodule $\M$, its Hochschild homology $HH_\bullet(\C,\M)$ is
the homology of the Hochschild complex
\begin{equation}
CC_\bullet(\C,\M)=\bigoplus_{\begin{smallmatrix}p\geq 0\\X_0,\ldots,X_p\in\C\end{smallmatrix}}
\M(X_p,X_0)\otimes\C(X_0,X_1)[1]\otimes\cdots\otimes\C(X_{p-1},X_p)[1]=\text{``}\begin{tikzcd}\M\otimes_\C\ar[dash, rounded corners, to path = {(\tikztostart) -- ([xshift=0.5ex]\tikztostart.east) -- ([xshift=0.5ex,yshift=2ex]\tikztostart.east) -- ([xshift=-0.5ex,yshift=2ex]\tikztotarget.west) -- ([xshift=-0.5ex]\tikztotarget.west) -- (\tikztotarget)}]{}\end{tikzcd}\text{''}
\end{equation}
with the usual Hochschild differential (see e.g.\ \cite{seidelsubalgebras, abouzaidcriterion, ganatrathesis, sheridanformulae}, and note again that our conventions are closest to \cite{seidelsubalgebras}).
We use the usual abbreviation $HH_\bullet(\C):=HH_\bullet(\C,\C)$.
A word on gradings: the grading appears in the subscript to indicate Hochschild \emph{homology}, however the grading is \emph{cohomological} (i.e.\ the differential has degree one).
\end{definition}

If $f: \D \to \C$ is an $\ainf$-functor and $\M$ is a $\C$-bimodule, we will frequently use the abuse of notation
\begin{equation}
    HH_\bullet(\D, \M) := HH_\bullet(\D, f^* \M)
\end{equation}
where $f^*\M := (f,f)^*\M$ is the two-sided pullback of $\M$ along $f$ as in \S\ref{secbimodules} (in all
cases we consider, $f$ will be a naive inclusion functor, justifying this
notation).

\begin{lemma}\label{hochschildobviousiso}
Let $j:\C\to\D$ be cohomologically fully faithful and split-generating (e.g.\ the map $j$ could be a quasi-equivalence).
For any $\D$-bimodule $\B$, the map $HH_\bullet(\C,j^\ast\B)\to HH_\bullet(\D,\B)$ is an isomorphism; in particular $HH_\bullet(\C)\to HH_\bullet(\D)$ is an isomorphism.
\end{lemma}

\begin{proof}
This is almost the same statement as Lemma \ref{splitgeneratorquasiisomorphism}, and may be reduced to that result as follows.
Since the natural map $\D\otimes_\D\B\to\B$ is a quasi-isomorphism, we may replace $\B$ with $\D\otimes_\D\B$.
The map in question can then be written as
\begin{equation}
HH_\bullet(\D,\B\otimes_\C\D)\to HH_\bullet(\D,\B\otimes_\D\D).
\end{equation}
Now the result follows from Lemma \ref{splitgeneratorquasiisomorphism} applied to the map $\B\otimes_\C\D\to\B\otimes_\D\D$.
\end{proof}

\begin{lemma}[Abouzaid--Ganatra \cite{abouzaidganatra}]\label{HHquotientisomorphism}
For any $(\C/\A)$-bimodule $\M$, the natural map $HH_\bullet(\C,\M)\to
HH_\bullet(\C/\A,\M)$ is a quasi-isomorphism.
\end{lemma}

\begin{proof}
This follows from Lemma \ref{quotienttensorproperty} by the same argument used to prove Lemma \ref{hochschildobviousiso}.
\end{proof}

\subsection{Preparatory remarks}\label{ocprep}

The main construction of this section takes as input a diagram of Liouville sectors $\{X_\sigma\}_{\sigma\in\Sigma}$ indexed by a finite poset $\Sigma$.
A choice of such input diagram shall be regarded as fixed for most of this section (up through \S\ref{mapocsec}, to be precise).
In keeping with Convention \ref{choosingpi}, we regard each $X_\sigma$ as coming with a choice of $\pi_\sigma:\Nbd^Z\partial X_\sigma\to\CC$ such that for $\sigma\leq\sigma'$, either $X_\sigma\subsetneqq X_{\sigma'}$ or $X_\sigma=X_{\sigma'}$ and $\pi_\sigma=\pi_{\sigma'}$.
The main constructions of \S\ref{secwrapped} and \S\ref{secsymplecticcohomology} now give, respectively, diagrams $\{\W_\sigma\}_{\sigma\in\Sigma}$ and $\{SC^\bullet(X_\sigma,\partial X_\sigma)\}_{\sigma\in\Sigma}$, and the goal of this section is to define a map of diagrams $\{\OC_\sigma:CC_\bullet(\W_\sigma)\to SC^{\bullet+n}(X_\sigma,\partial X_\sigma)\}_{\sigma\in\Sigma}$.
Actually, we slightly tweak the definitions from \S\ref{secwrapped} and \S\ref{secsymplecticcohomology} as we now describe, in order to better accomodate the construction of $\OC$.

Regarding the construction of $\{\W_\sigma\}_{\sigma\in\Sigma}$ in \S\ref{wrappedfulldiagram}, we make the following modification.
We fix slight inward pushoffs $X_\sigma^{--}\subsetneqq X_\sigma^-\subsetneqq X_\sigma$ (trivial inclusions) together with $\pi_\sigma^-$ and $\pi_\sigma^{--}$, such that for $\sigma\leq\sigma'$ we have $X_\sigma^{--}\subseteq X_{\sigma'}^{--}$ and $X_\sigma^-\subseteq X_{\sigma'}^-$ (strict iff $X_\sigma\subseteq X_{\sigma'}$ is strict).
When defining $\OO_\sigma$, we use only Lagrangians inside $X_\sigma^{--}$, we wrap them only inside $X_\sigma^{--}$ (namely in Definition \ref{posetOdef}, we consider the wrapping category inside $X_\sigma^{--}$ instead of inside $X_\sigma$), and we use almost complex structures defined over $X_\sigma^-$ (and adapted to $\pi_\sigma^-$).
This modification is completely harmless; it follows easily from the results of \S\ref{secwrapped} that this produces a quasi-equivalent diagram of $\ainf$-categories.
We henceforth fix such choices of $X_\sigma^-$, $X_\sigma^{--}$, $\OO_\sigma$, strip-like coordinates, and almost complex structures
\begin{align}
\label{ocWcoordsI}\xi_{L_0,\ldots,L_k;j}^+:[0,\infty)\times[0,1]\times\Rbar_{k,1}&\to\Sbar_{k,1}\quad j=1,\ldots,k\\
\label{ocWcoordsII}\xi_{L_0,\ldots,L_k}^-:(-\infty,0]\times[0,1]\times\Rbar_{k,1}&\to\Sbar_{k,1}\\
\label{ocWacs}J_{L_0,\ldots,L_k}:\Sbar_{k,1}&\to\J(X_{\sigma_{L_0}}^-)
\end{align}
used to define $\W_\sigma$ for the diagram $\{X_\sigma\}_{\sigma\in\Sigma}$ (recalling that $\sigma_L$ denotes the unique minimal element of $\Sigma$ such that $L\in\OO_{\sigma_L}$).
For reasons of technical convenience, we will also assume that the negative strip-like coordinates \eqref{ocWcoordsII} extend to a biholomorphism
\begin{equation}\label{ocWcoordsIII}
\RR\times[0,1]\times\RRR_{k,1}\to\SSS_{k,1}.
\end{equation}
Such strip-like coordinates may be constructed by induction exactly as in \S\ref{wcurvessec}.

Regarding $\{SC^\bullet(X_\sigma,\partial X_\sigma)\}_{\sigma\in\Sigma}$ as defined in \eqref{scchains} from \S\ref{secsympcochaindiagrams}, we make the following modification.
We let $\HJ_\bullet^{SC}(\sigma_0,\ldots,\sigma_r)$ for $\sigma_0\leq\cdots\leq\sigma_r\in\Sigma$ denote the following variant of $\HJ_\bullet(X_{\sigma_0},\ldots,X_{\sigma_r})$.
To specify an $n$-simplex, we specify for every vertex $v\in\Delta^n$ a chain
\begin{equation}\label{scchainforoc}
\mu_1^{(0)}\leq\cdots\leq\mu_{a_0}^{(0)}\leq\sigma_0\leq\mu_1^{(1)}\leq\cdots\leq\mu_{a_1}^{(1)}\leq\sigma_1\leq\cdots\leq\sigma_r
\end{equation}
in $\Sigma$ (instead of a chain \eqref{Ychain} in the category of Liouville sectors), and the rest of the definition remains the same.
In this section, we use the definition
\begin{equation}\label{scmodelforoc}
SC^\bullet(X_\sigma,\partial X_\sigma):=\hocolim_{\sigma_0\leq\cdots\leq\sigma_r\leq\sigma}\hocolim_{\HJ_\bullet^{SC,\reg}(\sigma_0,\ldots,\sigma_r)}CF^\bullet(X_{\sigma_0};-),
\end{equation}
which maps quasi-isomorphically to \eqref{scchains} by virtue of the natural cofinal map $\HJ_\bullet^{SC}(\sigma_0,\ldots,\sigma_r)\to\HJ_\bullet(X_{\sigma_0},\ldots,X_{\sigma_r})$.
We should admit that this replacement is for aesthetic, not mathematical (perhaps even just notational), reasons.

\subsection{Moduli spaces of domains for \texorpdfstring{$\B$}{B}}\label{Bmodulisubsec}

\begin{figure}[hbt]
\centering
\includegraphics{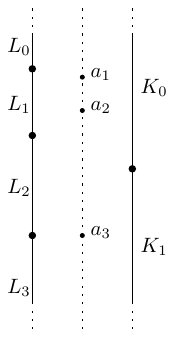}
\caption{Such holomorphic maps define the diagram of bimodules $\B$.}\label{Bstrip}
\end{figure}

We consider the compactified moduli space of strips $\RR\times[0,1]$ with $k$ boundary marked points on $\RR\times\{0\}$ and $\ell$ boundary marked points on $\RR\times\{1\}$, together with marked points $a_1\geq\cdots\geq a_n\in\RR$ (allowed to collide with each other as in \S\ref{scdomainmodulisec}) up to simultaneous translation (we view these points $a_i$ as lying on $\RR\times\{\frac 12\}\subseteq\RR\times[0,1]$).
We denote this moduli space and its universal family by $\Cbar^\B_{k,\ell,n}\to\Mbar^\B_{k,\ell,n}$.
Of course, for $n=0$ this is simply $\Sbar_{k+1+\ell,1}\to\Rbar_{k+1+\ell,1}$ from \S\ref{wcurvessec}, and for $k=\ell=0$, we have $\Mbar^\B_{0,0,n}=\Mbar^{SC}_{0,0,n}$ (but of course the universal curves are not the same).

To help orient the reader, we list the codimension one boundary strata of $\Mbar^\B_{k,\ell,n}$.
There are boundary strata $\Mbar^\B_{k,\ell,n-1}$ corresponding to when $a_i=a_{i+1}$, and there are boundary strata $\Mbar^\B_{k',\ell',n'}\times\Mbar^\B_{k'',\ell'',n''}$ for $k=k'+k''$, $\ell=\ell'+\ell''$, $n=n'+n''$ corresponding to when $a_{n'}-a_{n'+1}=\infty$.
There are also boundary strata $\Mbar^\B_{k',\ell,n}\times\Rbar_{k'',1}$ for $k''+k'=k+1$ corresponding to when some of the marked points on $\RR\times\{0\}$ collide, and similarly $\Mbar^\B_{k,\ell',n}\times\Rbar_{\ell'',1}$.

The operation of forgetting the $k+\ell$ boundary marked points and remembering only the $n$ marked points gives a map $\Mbar^\B_{k,\ell,n}\to\Mbar^{SC}_n$.
This forgetful map can be covered by a map of universal curves $\Cbar^\B_{k,\ell,n}\to\Cbar^{SC}_n$ (up to handling unstable components correctly) defined by $(s,t)\mapsto(s,\varphi(t))$ for any fixed $\varphi:[0,1]\to S^1$.
We fix once and for all such a map
\begin{equation}\label{ocvarphi}
\varphi:[0,1]\to S^1
\end{equation}
with $\varphi'(t)\geq 0$, $\varphi(t)=0$ for $t$ in a neighborhood of $0,1\in[0,1]$, and $\varphi$ of degree one.
This map will be used to define the Hamiltonian terms in the Floer equations we are about to consider.
That is, our Hamiltonians will be specified on $\Cbar^{SC}_m$ and then pulled back under this map to $\Cbar^\B_{k,\ell,m}$.
This may seem like a somewhat hacked approach, but it allows us to reuse wholesale the definition and construction of dissipative Hamiltonians from \S\ref{secsymplecticcohomology} rather than adapting these to the present setting (which would be time consuming, though likely not difficult).

One warning about the moduli spaces $\Mbar^\B_{k,\ell,n}$ is in order.
Given strip-like coordinates at the various boundary marked points, the standard gluing operation does not in general define a collar neighborhood of a given boundary stratum (such as $\Mbar^\B_{k',\ell',n'}\times\Mbar^\B_{k'',\ell'',n''}$, $\Mbar^\B_{k',\ell,n}\times\Rbar_{k'',1}$, or $\Mbar^\B_{k,\ell',n}\times\Rbar_{\ell'',1}$ as described above) due to the simple fact that, on the glued curve, the images of the marked points $a_i$ need not lie on the line $\RR\times\{\frac 12\}\subseteq\RR\times[0,1]$.
This issue does not arise, however, as long as we use the tautological strip-like coordinates near $s=\pm\infty$ and as long as we require the strip-like coordinates \eqref{ocWcoordsII} to extend as in \eqref{ocWcoordsIII}.
(The same warning applies to the moduli spaces $\Mbar^{SC}_m$ from \S\ref{secsymplecticcohomology}, though for them there is little reason to discuss anything other than tautological cylindrical coordinates near $s=\pm\infty$.)

\subsection{Floer data for \texorpdfstring{$\B$}{B}}

For every chain $\tau\leq\sigma_0\leq\cdots\leq\sigma_r\in\Sigma$, we define a simplicial set $\HJ_\bullet^\B(\tau;\sigma_0,\ldots,\sigma_r)$ of Floer data.
An $n$-simplex of $\HJ_\bullet^\B(\tau;\sigma_0,\ldots,\sigma_r)$ consists of a specification of strip-like coordinates, almost complex structures, and Hamiltonians
\begin{align}
\label{BchoicescoordsI}\xi^{v_0\cdots v_m}_{K_k,\ldots,K_0;L_0,\ldots,L_\ell;0,j}:\RR_{\geq 0}\times[0,1]\times\Mbar^\B_{k,\ell,m}&\to\Cbar^\B_{k,\ell,m}\quad j=1,\ldots,k\\
\label{BchoicescoordsII}\xi^{v_0\cdots v_m}_{K_k,\ldots,K_0;L_0,\ldots,L_\ell;1,j}:\RR_{\geq 0}\times[0,1]\times\Mbar^\B_{k,\ell,m}&\to\Cbar^\B_{k,\ell,m}\quad j=1,\ldots,\ell\\
\label{Bchoicesacs}J^{v_0\cdots v_m}_{K_k,\ldots,K_0;L_0,\ldots,L_\ell}:\Cbar^\B_{k,\ell,m}&\to\J(X_{\sigma_r}^-)\\
\label{BchoicesH}H^{v_0\cdots v_m}:\Cbar^{SC}_m&\to\H(X_{\sigma_r}^-)
\end{align}
for $K_k>\cdots>K_0\in\OO_{\sigma_r}$, $L_0>\cdots>L_\ell\in\OO_{\sigma_r}$, and $0\leq v_0<\cdots<v_m\leq n$, satisfying the following properties:
\begin{itemize}
\item The strip-like coordinates $\xi$, the tautological strip-like coordinates at $s=\pm\infty$, and the strip-like coordinates \eqref{ocWcoordsI}--\eqref{ocWcoordsII} fixed for $\OO_{\sigma_i}$ must be \emph{compatible with gluing and forgetting vertices} in the sense of \S\ref{wcurvessec} and Definition \ref{nsimplexfloer}.
Namely, over a neighborhood of each glued boundary stratum of $\Mbar^\B_{k,\ell,m}$, the strip-like coordinates must coincide with those defined by ``gluing'' via the strip-like coordinates given on each fiberwise irreducible component (this makes sense in view of the extension \eqref{ocWcoordsIII}), and over each forgotten vertex boundary stratum, they must coincide with those specified for the relevant copy of $\Mbar^\B_{k,\ell,m-1}$.
\item The almost complex structures $J$ must be \emph{compatible with gluing and forgetting vertices} in the sense of \S\ref{wcurvessec} and Definition \ref{nsimplexfloer}.
Namely, (1) they must be $s$-invariant in the thin parts, (2) over a neighborhood of each glued boundary stratum, they must be extended via the usual gluing construction, and (3) over each (glued or forgotten vertex) boundary stratum, they must agree with the almost complex structures specified on each fiberwise irreducible component.
\item The almost complex structures $J$ must be \emph{adapted to $\partial X_i^-$}, meaning that $\pi_{\sigma_i}^-$ must be holomorphic with respect to \eqref{Bchoicesacs} over $(\pi_{\sigma_i}^-)^{-1}(\CC_{\left|\Re\right|\leq\varepsilon})$ for some $\varepsilon>0$ whenever $K_k,L_0\in\OO_{\sigma_i}\subseteq\OO_{\sigma_r}$ (compare Definition \ref{adapted}).
\item The Hamiltonians $H$ must be \emph{compatible with gluing and forgetting vertices} in the sense of Definition \ref{nsimplexfloer}.
\item The Hamiltonians $H$ must be \emph{adapted to $\partial X_i^-$}, meaning that $H\equiv 0$ over $(\pi_{\sigma_i}^-)^{-1}(\CC_{\left|\Re\right|\leq\varepsilon})$ for some $\varepsilon>0$ and $0\leq i\leq r$ (compare Definition \ref{adapted}).
\item The Hamiltonians $H^v$ must be \emph{linear at infinity} (compare Definition \ref{admissible}).
\item The Hamiltonians $H$ must be \emph{dissipative} with specified dissipation data in the sense of Definition \ref{dissipative}, except the ``non-degenerate fixed points'' and ``no fixed points at infinity'' conditions (which in the present context are both irrelevant and never satisfied) are replaced with the requirement that $\Phi_{H^v}K\pitchfork L$ for all $K,L\in\OO_\tau$ (note that it is not reasonable to impose this condition for all $K,L\in\OO_{\sigma_j}$ for any $j>0$, since $H^v$ vanishes near $\partial X_{\sigma_i}^-$ and thus cannot ensure that $\Phi_{H^v}L\pitchfork L$ if $L\in\OO_{\sigma_j}$ intersects $\partial X_{\sigma_i}^-$ nontrivially).
\end{itemize}
Note the forgetful maps
\begin{alignat}{2}
\label{BforgetI}\HJ_\bullet^\B(\tau;\sigma_0,\ldots,\sigma_r)&\to\HJ_\bullet^\B(\tau;\sigma_0,\ldots,\widehat{\sigma_i},\ldots,\sigma_r)&\quad&0\leq i\leq r,\\
\label{BforgetII}\HJ_\bullet^\B(\tau';\sigma_0,\ldots,\sigma_r)&\to\HJ_\bullet^\B(\tau;\sigma_0,\ldots,\sigma_r)&\quad&\tau\leq\tau'.
\end{alignat}

It is straightforward to construct vertices of $\HJ_\bullet^\B(\tau;\sigma_0,\ldots,\sigma_r)$ (recall from \S\ref{ocprep} that the Lagrangians in $\OO_\tau$ are contained in $X_\tau^{--}\subseteq X_{\sigma_0}^{--}\subsetneqq X_{\sigma_0}^-$, so the requirement that $H$ vanish over $\Nbd^Z\partial X_{\sigma_0}^-$ does not conflict with the need to perturb all Lagrangians in $\OO_\tau$).
Moreover, the proofs of Proposition \ref{hjcontractible} and Corollary \ref{hjfiltered} apply essentially verbatim to show that each $\HJ_\bullet^\B(\tau;\sigma_0,\ldots,\sigma_r)$ is a filtered $\infty$-category.
The forgetful map \eqref{BforgetI} forgetting $\sigma_i$ is easily seen to be cofinal (even surjective on vertices) for $i=r$.
The forgetful map \eqref{BforgetII} decreasing $\tau$ is also easily seen to be cofinal.

\subsection{Holomorphic curves for \texorpdfstring{$\B$}{B}}

Given an $n$-simplex $(H,J,\xi)$ of $\HJ_\bullet^\B(\tau;\sigma_0,\ldots,\sigma_r)$, we consider the moduli space of maps $u:\RR\times[0,1]\to X_{\sigma_r}$ and points $a_1\geq\cdots\geq a_n\in\RR$ (up to simultaneous translation) as in Figure \ref{Bstrip}, with boundary conditions in $\OO_\tau$, satisfying
\begin{equation}\label{Beqn}
(du-X_{H(s,\varphi(t))}\otimes d\varphi(t))^{0,1}_J=0
\end{equation}
where $\varphi$ is the universally fixed function \eqref{ocvarphi}.
To interpret the Hamiltonian term in \eqref{Beqn}, recall the map $\Cbar^\B_{k,\ell,n}\to\Cbar^{SC}_n$ determined by $\varphi$ as mentioned at the end of \S\ref{Bmodulisubsec}.
Such solutions $u$ are contained in $X_{\sigma_0}^-$ since $H$ vanishes and $\pi_{\sigma_0}^-$ is $J$-holomorphic over $\pi_{\sigma_0}^{-1}(\CC_{\left|\Re\right|\leq\varepsilon})$ (see Lemma \ref{pibarrierlemma}).
Denote by
\begin{equation}\label{Bmodulispaces}
\Mbar^\B_{k,\ell,n}(H,J,\xi;p_k,\ldots,p_1,x^+,q_1,\ldots,q_\ell,x^-)
\end{equation}
the associated compactified moduli spaces (i.e.\ including all stable broken trajectories as well), where $p_i\in K_i\cap K_{i-1}$, $q_i\in L_{i-1}\cap L_i$, $x^+\in\Phi_{H(0)}K_0\cap L_0$, and $x^-\in\Phi_{H(n)}K_k\cap L_\ell$.

\begin{proposition}\label{bcompactness}
The moduli spaces \eqref{Bmodulispaces} are compact.
\end{proposition}

\begin{proof}
As mentioned above, the projection map $\pi_{\sigma_0}^-$ keeps trajectories away from $\partial X_{\sigma_0}^-$. To prevent escape to infinity in $X_{\sigma_0}^-$, the proofs of Propositions \ref{wcompactness} and \ref{compactness} apply without significant modification.
Namely, in any end $[0,\infty)\times[0,1]$ (at either the top or bottom of the strip or at one of the boundary punctures) or thin part, the reasoning surrounding \eqref{wcsdistanceintegral} and \eqref{csdistanceintegral} shows that there exists a compact subset $K\subseteq X_{\sigma_0}^-$ and a real number $L<\infty$ such that for every interval $[a,b]$ of length $\geq L$, there exists a point in $[a,b]\times[0,1]$ which is mapped into $K$ by $u$.
We can now apply monotonicity inequalities to the graph of $u$ and use dissipativity of $H$ to conclude that the image of $u$ is contained \emph{a priori} in a compact subset of $X_{\sigma_0}^-$ depending only on the Floer data and the topological energy.
\end{proof}

Let $\HJ_\bullet^{\B,\reg}(\tau;\sigma_0,\ldots,\sigma_r)\subseteq\HJ_\bullet^\B(\tau;\sigma_0,\ldots,\sigma_r)$ consist of those $n$-simplices for which all moduli spaces \eqref{Bmodulispaces} associated to facets of $\Delta^n$ are cut out transversely.
A standard perturbation argument as in \S\ref{sctranssec} shows that $\HJ_\bullet^{\B,\reg}$ is also a filtered $\infty$-category and that the inclusion $\HJ_\bullet^{\B,\reg}\hookrightarrow\HJ_\bullet^\B$ is cofinal.

\subsection{Bimodules \texorpdfstring{$\B$}{B}}\label{Bdiagramsec}

We now define a diagram of $\OO_\tau$-bimodules over $\HJ_\bullet^{\B,\reg}(\tau;\sigma_0,\ldots,\sigma_r)$ by counting the moduli spaces \eqref{Bmodulispaces}.
That is, we define a map of simplicial sets
\begin{equation}\label{Bdiagram}
\HJ_\bullet^{\B,\reg}(\tau;\sigma_0,\ldots,\sigma_r)\to\Ndg{[\OO_\tau,\OO_\tau]},
\end{equation}
where $[\OO_\tau,\OO_\tau]$ denotes the dg-category of $\OO_\tau$-bimodules from \S\ref{secbimodules} and $\Ndg$ denotes the differential graded nerve from \S\ref{ndgdefsec}.
This map \eqref{Bdiagram} encodes the various operations
\begin{multline}\label{Bdiagramoperations}
\OO(K_k,K_{k-1})\otimes\cdots\otimes\OO(K_1,K_0)\otimes CF^\bullet(K_0,L_0;H(0))\otimes\OO(L_0,L_1)\otimes\cdots\otimes\OO(L_{\ell-1},L_\ell)\\
\to CF^\bullet(K_k,L_\ell;H(n))[1-\ell-k-n]
\end{multline}
defined by counting the moduli spaces \eqref{Bmodulispaces} for any $n$-simplex $(H,J,\xi)$ of $\HJ_\bullet^{\B,\reg}$ and any $K_k>\cdots>K_0\in\OO_\tau$ and $L_0>\cdots>L_\ell\in\OO_\tau$ (and dictated by strict unitality when some $K_{i+1}=K_i$ or $L_i=L_{i+1}$) and the relations they satisfy with \eqref{muoperation}.
The construction of \eqref{Bdiagram} follows \S\ref{secsympcochaindiagrams} very closely.

To a vertex $(H,J,\xi)\in\HJ_0^{\B,\reg}(\tau;\sigma_0,\ldots,\sigma_r)$, we associate the $\OO_\tau$-bimodule $CF^\bullet(-,-;H)$, equipped with the $\ainf$-bimodule structure maps coming from the operations \eqref{Bdiagramoperations} (with respect to suitable signs/orientations as in \cite{seidelfukayalefschetzI}); the $\ainf$ relations follow from the usual boundary analysis.
A $1$-simplex $(H,J,\xi)\in\HJ_1^{\B,\reg}(\tau;\sigma_0,\ldots,\sigma_r)$ induces an $\ainf$-bimodule morphism
\begin{equation}
F_{(H,J,\xi)}:CF^\bullet(-,-;H(0))\to CF^\bullet(-,-;H(1))
\end{equation}
for the same reason; more generally, an $n$-simplex $(H,J,\xi)\in\HJ_1^{\B,\reg}(\tau;\sigma_0,\ldots,\sigma_r)$ with $n\geq 1$ induces an $\ainf$-bimodule morphism
\begin{equation}
F_{(H,J,\xi)}:CF^\bullet(-,-;H(0))\otimes C_{-\bullet}(\F(\Delta^n))\to CF^\bullet(-,-;H(n)).
\end{equation}
Arguing as in the proof of Lemma \ref{scdiagramdegenerate}, we conclude that these maps define a diagram \eqref{Bdiagram} as desired.

We define various $\OO_\tau$-bimodules as the homotopy colimits of the various diagrams \eqref{Bdiagram}, namely
\begin{equation}
\B_{\tau;\sigma_0,\ldots,\sigma_r}(-,-):=\hocolim_{\HJ_\bullet^{\B,\reg}(\tau;\sigma_0,\ldots,\sigma_r)}CF^\bullet(-,-;-)\in\Tw^\oplus{[\OO_\tau,\OO_\tau]}.
\end{equation}
Note that since $\HJ_\bullet^{\B,\reg}(\tau;\sigma_0,\ldots,\sigma_r)$ is a filtered $\infty$-category, the cohomology of $\B_{\tau;\sigma_0,\ldots,\sigma_r}$ can be computed by taking an ordinary direct limit of $HF^\bullet(-,-;H)$ over any cofinal collection of $H$.
The homotopy colimit $\B_{\tau;\sigma_0,\ldots,\sigma_r}$ is an object of $\Tw^\oplus{[\OO_\tau,\OO_\tau]}$ (infinite direct sum twisted complexes), but it may also be regarded as an honest bimodule by composing with the natural functor $\Tw^\oplus{[\OO_\tau,\OO_\tau]}\to[\OO_\tau,\OO_\tau]$.
It makes, however, essentially no difference which of these perspectives we take.

There are natural maps
\begin{alignat}{2}
\label{BmodforgetI}\B_{\tau;\sigma_0,\ldots,\sigma_r}&\to\B_{\tau;\sigma_0,\ldots,\widehat{\sigma_i},\ldots,\sigma_r}&\quad&0\leq i\leq r,\\
\label{BmodforgetII}\B_{\tau';\sigma_0,\ldots,\sigma_r}|_{\OO_\tau}&\to\B_{\tau;\sigma_0,\ldots,\sigma_r}&\quad&\tau\leq\tau'.
\end{alignat}
induced by the forgetful maps \eqref{BforgetI}--\eqref{BforgetII}.
The first of these is a quasi-isomorphism for $i>0$ since \eqref{BforgetI} is cofinal for $i=r$, and the second is (always) a quasi-isomorphism since \eqref{BforgetII} is cofinal.

\subsection{Cohomology of \texorpdfstring{$\B$}{B}}\label{Bprop}

We now calculate the cohomology of the $\OO_\tau$-bimodules $\B_{\tau;\sigma_0,\ldots,\sigma_r}$.
In view of when \eqref{BmodforgetI}--\eqref{BmodforgetII} are quasi-isomorphisms, it is enough to fix $\sigma$ and consider the case of $\B_{\sigma;\sigma}$.
We will thus abbreviate $X=X_\sigma$, $\OO=\OO_\sigma$, $\HJ_\bullet^\B=\HJ_\bullet^\B(\sigma;\sigma)$, $\B=\B_{\sigma;\sigma}$, etc.

Let us begin by observing that there are natural isomorphisms
\begin{equation}\label{Bhomologyiso}
HF^\bullet(L,K;H)=HF^\bullet(\Phi_HL,K),
\end{equation}
where the left side denotes cohomology of the $\OO$-bimodule $CF^\bullet(-,-;H)$ and the right side denotes Floer cohomology as defined in \S\S\ref{wcurvessec}--\ref{continuationsection}.
The notation on the left is justified because $HF^\bullet(L,K;H)$ is independent of the choice of Floer data, due to the fact that the subcomplex of $\HJ_\bullet^{\B,\reg}$ with fixed $H$ is a contractible $\infty$-groupoid.
To see the isomorphism \eqref{Bhomologyiso}, simply change coordinates on $[0,1]\times X$ using the flow of $X_{H(\varphi(t))}d\varphi(t)$ (here it is crucially important that $X_{H(\varphi(t))}d\varphi(t)$ is a function of $t\in[0,1]$ only and that $H$ is linear at infinity, so this change of coordinates preserves cylindricity of $J$).
This same argument applies to show that the composition operations
\begin{align}
HF^\bullet(L,K;H)\otimes HF^\bullet(K,K')\to HF^\bullet(L,K';H)\\
HF^\bullet(L',L)\otimes HF^\bullet(L,K;H)\to HF^\bullet(L',K;H)
\end{align}
from the bimodule diagram \eqref{Bdiagram} coincide (under the isomorphism \eqref{Bhomologyiso}) with the usual compostion operations in Lagrangian Floer theory
\begin{align}
HF^\bullet(\Phi_HL,K)\otimes HF^\bullet(K,K')\to HF^\bullet(\Phi_HL,K')\\
HF^\bullet(L',L)\otimes HF^\bullet(\Phi_HL,K)\to HF^\bullet(\Phi_HL',K)
\end{align}
from \S\S\ref{wcurvessec}--\ref{continuationsection}.
We now address a less tautological comparison.

\begin{lemma}
For $H,H':S^1\to\H(X)$ linear at infinity, vanishing over $\Nbd^Z\partial X$, and satisfying $H\leq H'$ near infinity, the ``continuation map defined via dissipative Hamiltonians''
\begin{equation}\label{Bopcontinuation}
HF^\bullet(L,K;H)\to HF^\bullet(L,K;H')
\end{equation}
from the bimodule diagram \eqref{Bdiagram} coincides under \eqref{Bhomologyiso} with the continuation map
\begin{equation}\label{elementcontinuation}
HF^\bullet(\Phi_HL,K)\to HF^\bullet(\Phi_{H'}L,K)
\end{equation}
from \eqref{continuationonHF} associated to the non-negative Lagrangian isotopy $\Phi_{(1-a)H+aH'}L$ for $a\in[0,1]$.
\end{lemma}

\begin{proof}
In the case $H=H'$ at infinity, both maps \eqref{Bopcontinuation} and \eqref{elementcontinuation} coincide with the continuation map isomorphism from Lemma \ref{simultaneousdeformation}, and thus agree.
In addition, both classes of continuation maps \eqref{Bopcontinuation} and \eqref{elementcontinuation} compose as expected for triples $H,H',H''$.
Hence to check that \eqref{Bopcontinuation} and \eqref{elementcontinuation} agree for all $H$ and $H'$, it will suffice to first prove it for some particular ``nice'' pairs $(H,H')$ (to be specified below) and then show that for an arbitrary pair $(H,H'')$, we can always find a ``nice'' $(H,H')$ with $H' = H''$ at infinity.

The first ``niceness'' condition we impose on $(H,H')$ is that $H=H'$ over a sufficiently large (in terms of $H$) compact subset of $X$ (in particular, one outside of which $H$ is already linear), which implies that there is an inclusion $\Phi_HL\cap K\subseteq\Phi_{H'}L\cap K$, and hence a direct sum decomposition of $\ZZ$-modules
\begin{equation}\label{eq:hhprimedecomposition}
CF^\bullet(L,K;H')=CF^\bullet(L,K;H)\oplus(\text{additional generators}).
\end{equation}
Now the component of the differential on $CF^\bullet(L,K;H')$ mapping $CF^\bullet(L,K;H)$ to itself coincides with the differential on $CF^\bullet(L,K;H)$ (say we fix an almost complex structure $J:S^1\to\J(X)$ achieving transversality), since the proof of Proposition \ref{wcompactness} provides that holomorphic curves between intersection points $\Phi_HL\cap K$ cannot escape an \emph{a priori} determined compact subset of $X$, and we may require that $\Phi_HL=\Phi_{H'}L$ over this subset.
By the same reasoning (and using the same almost complex structure, or rather the induced $\RR$-invariant family $\RR\times S^1\to\J(X)$), the component of both (chain level) continuation maps \eqref{Bopcontinuation} and \eqref{elementcontinuation} mapping into the first direct summand is the identity map (only constant disks contribute).

Now we say a pair $(H,H')$ is ``nice'' iff (in addition to the condition above) $H\leq H'$ everywhere and the additional generators in the decomposition \eqref{eq:hhprimedecomposition} have action strictly greater than the generators of $CF^\bullet(L,K; H)$.
Recall that the action functional on generators of $CF^\bullet(\Phi_HL,K)=CF^\bullet(L,K;H)$ is given by
\begin{align}
a(x) = f_{\Phi_HL}(x)-f_K(x)&=f_L(\gamma(0))-f_K(\gamma(1))+\int_0^1\gamma^\ast\lambda-H(\gamma(t))\,dt\\
\label{isotopyactionintegral}&=f_L(\gamma(0))-f_K(\gamma(1))+\int_0^1(ZH-H)(\gamma(t))\,dt
\end{align}
where $x\in\Phi_HL\cap K$ corresponds to the Hamiltonian chord (of $H$) $\gamma:[0,1]\to X$ from $L$ to $K$, and the functions $f$ are the chosen primitives for $\lambda$ restricted to the Lagrangians.
To make sense of this equation, we should declare that as $(L,f_L)$ undergoes Hamiltonian isotopy, the primitive changes according to the formula $\frac{\partial f_L}{\partial t}=ZH-H$, so $f_{\Phi_HL}$ is determined by $f_L$. 

It is straightforward to prove that \eqref{Bopcontinuation} and \eqref{elementcontinuation} coincide when $(H,H')$ is ``nice''; first, the fact that the additional generators have greater action and the fact that the differential decreases action together imply that the direct sum decomposition \eqref{eq:hhprimedecomposition} exhibits $CF^\bullet(L,K;H)$ as a subcomplex of $CF^\bullet(L,K;H')$.
The additional generators having strictly greater action also implies that the continuation maps \eqref{Bopcontinuation} and \eqref{elementcontinuation} are both simply the tautological inclusion of this subcomplex, provided we show that these maps also both weakly decrease action.
For the continuation map \eqref{Bopcontinuation}, since $H\leq H'$ everywhere, we may choose a dissipative family $\bar H:\RR\times S^1\to\H(X)$ from $H$ to $H'$ which satisfies $\partial_s\bar H\leq 0$, and hence conclude that this continuation map weakly decreases action since the geometric and topological energies are related as in \eqref{poswrappinggood}, and Stokes' theorem implies that the topological energy of a trajectory from $x$ to $y$ is the difference $a(x) - a(y)$.
For the continuation map \eqref{elementcontinuation}, first observe that it may be defined using the moving Lagrangian boundary conditions given by the isotopy $\Phi_{(1-a)H+aH'}L$ (compactness is justified as in the proof of Lemma \ref{wrappingunits}, which requires us to assume that our almost complex structure $J$ is of contact type over the locus swept out by the non-negatively moving Lagrangian boundary conditions).
Now this isotopy $\Phi_{(1-a)H+aH'}L$ is given by a non-negative Hamiltonian since $H\leq H'$, and thus the identity \eqref{geoenergymoving} implies in a similar fashion that this continuation map also weakly decreases action.
Hence, both maps \eqref{Bopcontinuation} and \eqref{elementcontinuation} are given by the identity map onto the first factor of the decomposition \eqref{eq:hhprimedecomposition}; in particular they coincide as desired for ``nice'' $(H,H')$.

To finish the proof, it remains (as noted in the first paragraph) to argue that, for a general pair $(H,H'')$, we can always find a ``nice'' $(H,H')$ with $H' = H''$ at infinity.
We define $H'$ as $H+\varphi(s-s_0)(H''-H)$ for sufficiently large $s_0<\infty$, where we have fixed a smooth function
\begin{equation}\label{varphicutoffdefn}
\varphi(s):=\begin{cases}0&s\leq 0,\\\in(0,1)&0<s<N,\\1&s\geq N,\end{cases}\qquad\varphi(s)\in(0,1)\implies\varphi'(s)>0,
\end{equation}
and we have fixed symplectization coordinates $(\RR_s\times\partial_\infty X,e^s\alpha)$ on $X$ near infinity.
Clearly $H\leq H'$ since the difference $H''-H$ is non-negative at infinity, and by taking $s_0\to\infty$ we have $H=H'$ over arbitrarily large compact subsets of $X$.
It thus remains to check that the ``additional generators'' in the decomposition \eqref{eq:hhprimedecomposition} have strictly greater action.

For each of these additional generators, the last term of \eqref{isotopyactionintegral} scales exponentially in $s_0$ (since varying $s_0$ simply translates any additional generators by the Liouville vector field); hence it is enough to show that this integral
\begin{equation}
\int_0^1\varphi'(s-s_0)(H''-H)(\gamma(t))\,dt
\end{equation}
is positive for every additional generator $\gamma$ (and then take $s_0$ sufficiently large).
The integrand is $\geq 0$, so we just need to exclude the possibility that an additional generator $\gamma$ maps entirely to the region where the integrand vanishes.
Suppose for sake of contradiction that $\gamma$ is such an additional generator, i.e.\ $\gamma$ maps to the region where either $s\leq 0$, $s\geq N$, or $H=H''$.
Note that at any point where $H=H''$, their Hamiltonian vector fields are also equal (because of the inequality $H\leq H''$); thus whenever $\gamma$ lies in the region $\{s\in(0,N)\}$, it is a Hamiltonian trajectory of both $H$ and $H''$.
Now by taking $N$ sufficiently large, we may ensure that no Hamiltonian trajectory of $H$ or $H''$ of length $\leq 1$ has $s$ varying by more than $N$.
It thus follows that $\gamma$ is contained in one of the regions $\{s\leq N\}$ or $\{s\geq 0\}$, and hence that it is a Hamiltonian chord of $H$ or $H''$ (respectively) from $\partial_\infty L$ to $\partial_\infty K$, which does not exist by assumption on $H$ and $H''$.
\end{proof}

Equipped with the identification $HF^\bullet(L,K;H)=HF^\bullet(\Phi_HL,K)$ in \eqref{Bhomologyiso} and the knowledge that the continuation maps between these groups in \eqref{Bdiagram} coincide with the usual continuation maps from \S\ref{secwrapped}, we may now deduce the following results/properties about the homotopy colimit $\B$.

\begin{lemma}\label{BisWonH}
For $L,K\in\OO$, there is a natural isomorphism $H^\bullet\B(L,K)=HW^\bullet(L,K)$.
Moreover, these isomorphisms intertwine the $H^\bullet\OO$-bimodule structure with the product on $HW^\bullet$ under the natural maps $H^\bullet\OO(L,K)=HF^\bullet(L,K)\to HW^\bullet(L,K)$.
In particular, $\B$ is $C$-local on both sides, so the localization maps $\B\to{}_{C^{-1}}\B$ and $\B\to\B_{C^{-1}}$ are both quasi-isomorphisms.
\end{lemma}

\begin{proof}
As $H$ varies over vertices of $\HJ_\bullet^{\B,\reg}$, the wrapped Lagrangians $\Phi_HL$ are cofinal in the wrapping category of $L$ inside $X$ (or rather $X^-$ as fixed in \S\ref{ocprep}, though this difference matters little at this point).
Moreover, the natural map $HF^\bullet(\Phi_HL,K)\to HW^\bullet(L,K)$ is compatible with multiplication by $HF^\bullet(K,K')$ on the right and multiplication by $HF^\bullet(L',L)$ on the left.
It is therefore enough to use the fact (proved just above) that the continuation maps \eqref{Bopcontinuation} involved in defining $\B$ agree with the continuation maps involved in defining $HW^\bullet$ (namely multiplication by continuation elements).

To prove the last assertion, note that the property of being $C$-local can be checked at the level of homology, and apply Lemma \ref{localgood}.
\end{proof}

\subsection{Quasi-isomorphism \texorpdfstring{$\B=\W$}{B = W}}

We now upgrade the isomorphism $H^\bullet\B_{\tau;\sigma_0,\ldots,\sigma_r}=H^\bullet\W_{\sigma_0}$ of $H^\bullet\OO_\tau$-bimodules from Lemma \ref{BisWonH} to a quasi-isomorphism of $\OO_\tau$-bimodules $\B_{\tau;\sigma_0,\ldots,\sigma_r}$ and $\W_{\sigma_0}$ (recall that the localization functor $\OO_{\sigma_0}\to\W_{\sigma_0}$ allows one to consider $\W_{\sigma_0}(-,-)$ as an $\OO_{\sigma_0}$-bimodule, which we may further restrict to $\OO_\tau\subseteq\OO_{\sigma_0}$).
The essential point is to define a (continuation) map of $\OO_\tau$-bimodules $\OO_\tau^-\to\B_{\tau;\sigma_0,\ldots,\sigma_r}$ inducing the usual map $HF^\bullet\to HW^\bullet$ on homology, where $\OO_\tau^-$ denotes the $\OO_\tau$-bimodule given by
\begin{equation}
\OO_\tau^-(L,K)=\begin{cases}\OO_\tau(L,K)&L>K\\0&\text{else}\end{cases}
\end{equation}
i.e.\ $\OO_\tau^-(L,K):=\OO_\tau(L,K)$ except for $\OO_\tau^-(L,L):=0$.

Floer data for the continuation map $\OO_\tau^-\to\B_{\tau;\sigma_0,\ldots,\sigma_r}$ is encoded in an enlarged simplicial set $\HJ_\bullet^{\OO\B}(\tau;\sigma_0,\ldots,\sigma_r)$ defined as follows.
An $n$-simplex of $\HJ_\bullet^{\OO\B}$ is defined identically to an $n$-simplex of $\HJ_\bullet^\B$, except for the following modifications:
\begin{itemize}
\item An integer $-1\leq e\leq n$ is specified.  Vertices $0,\ldots,e$ of $\Delta^n$ are called ``$\OO$-vertices'' and vertices $e+1,\ldots,n$ of $\Delta^n$ are called ``$\B$-vertices'' (so $e=-1$ means all vertices are $\B$-vertices, and in this case the simplex in question is simply a simplex of $\HJ_\bullet^\B$).
\item For tuples $0\leq v_0<\cdots<v_m\leq n$ with $v_0$ an $\OO$-vertex, in each of \eqref{BchoicescoordsI}--\eqref{BchoicesH} we replace $m$ with $m-f$ where $v_f\leq e<v_{f+1}$.
In other words, whereas before the marked points $a_i$ for $1\leq i\leq m$ corresponded to the edges $v_{i-1}\to v_i$, now we only use marked points $a_i$ for the edges $v_{i-1}\to v_i$ in which $v_i$ is a $\B$-vertex.
\item For tuples $0\leq v_0<\cdots<v_m\leq n$ with $v_0$ an $\OO$-vertex, we only consider chains $K_k>\cdots>K_0\in\OO_{\sigma_r}$ and $L_0>\cdots>L_\ell\in\OO_{\sigma_r}$ with $K_0>L_0$.
\item We additionally choose strip-like coordinates at $s=+\infty$ when $v_0$ is an $\OO$-vertex and at $s=-\infty$ when $v_m$ is an $\OO$-vertex.
\item For tuples $0\leq v_0<\cdots<v_m\leq n$ with $v_m$ an $\OO$-vertex, the target of \eqref{Bchoicesacs} is $\J(X_{\sigma_{K_k}}^-)$ (recalling that $\sigma_L$ denotes the unique minimal element of $\Sigma$ such that $L\in\OO_{\sigma_L}$).
\item Over the maximal $\OO$-simplex $\Delta^{0\cdots e}\subseteq\Delta^n$, we require that $H\equiv 0$ and that the strip-like coordinates $\xi$ and almost complex structures $J$ coincide with those specified for defining $\OO_{\sigma_r}$, namely \eqref{ocWcoordsI}--\eqref{ocWacs}, under the identifications $\Cbar^\B_{k,\ell,0}=\Sbar_{k+1+\ell,1}$ (observe that $\xi$ and $J$ defined in this way are indeed compatible as required).
Furthermore, we require $H$ to vanish for $s\leq a_i$ whenever the corresponding edge $v_{i-1}\to v_i$ has $v_{i-1}$ an $\OO$-vertex, and dissipation data is chosen only over facets with at least one $\B$-vertex and only to cover the locus $s\leq a_i$ in $\Cbar^{SC}_m$.
\item The transversality condition $\Phi_{H^v}K\pitchfork L$ for $K,L\in\OO_\tau$ (part of dissipativity of $H$) is imposed only for $\B$-vertices $v$ (of course, for $\OO$-vertices $v$ we have $\Phi_{H^v}K\pitchfork L$ for $K>L\in\OO_\tau$ by the definition of $\OO_\tau$ since $H^v\equiv 0$).
\end{itemize}
There is a tautological inclusion $\HJ_\bullet^\B\subseteq\HJ_\bullet^{\OO\B}$ as the set of simplices all of whose vertices are $\B$-vertices (i.e.\ those for which $e=-1$).
We denote by $\HJ_\bullet^\OO\subseteq\HJ_\bullet^{\OO\B}$ the set of simplices all of whose vertices are $\OO$-vertices (it is a consequence of the definition that $\HJ_\bullet^\OO$ has a unique $n$-simplex for every $n\geq 0$, i.e.\ $\HJ_\bullet^\OO=\Delta^0$).

For any $n$-simplex $(H,J,\xi)$ of $\HJ_\bullet^{\OO\B}$, we consider the same moduli spaces \eqref{Bmodulispaces}, subject to the requirement that if $v=0$ is an $\OO$-vertex then $K_0>L_0$ (the marked points $a_1,\cdots,a_e$ are still present, they just play no role in determining $\xi$, $J$, $H$).
We denote by $\HJ_\bullet^{\OO\B,\reg}\subseteq\HJ_\bullet^{\OO\B}$ the set of Floer data for which all such moduli spaces are transverse.
As in \S\ref{Bdiagramsec}, we obtain a diagram
\begin{equation}
\HJ_\bullet^{\OO\B,\reg}(\tau;\sigma_0,\ldots,\sigma_r)\to\Ndg{[\OO_\tau,\OO_\tau]}
\end{equation}
which to $\B$-vertices associates $CF^\bullet(L,K;H)$ and to the unique $\OO$-vertex associates $CF^\bullet(L,K)$ for $L>K$ and zero for all other pairs.
The inclusion $\HJ_\bullet^{\B,\reg}\subseteq\HJ_\bullet^{\OO\B,\reg}$ is covered by a tautological identification of diagrams.
The $\OO$-vertex is regular and the $\OO_\tau$-bimodule over it is canonically identified with $\OO_\tau^-$ (simply because the moduli spaces under consideration are exactly the same as those used to define the $\ainf$ operations for $\OO_\tau$).

The usual reasoning shows that $\HJ_\bullet^{\OO\B}$ is a filtered $\infty$-category, $\HJ_\bullet^{\OO\B,\reg}$ is a filtered $\infty$-category, and the inclusions $\HJ_\bullet^{\B,\reg}\subseteq\HJ_\bullet^{\OO\B,\reg}\subseteq\HJ_\bullet^{\OO\B}$ are cofinal.

We now consider the following diagram of $\OO_\tau$-bimodules
\begin{equation}\label{OtoB}
\hspace{-1in}
\begin{tikzcd}[column sep = small]
{}&\vphantom{C}\hocolim\limits_{\HJ_\bullet^{\OO}(\tau;\sigma_0,\ldots,\sigma_r)}CF^\bullet(-,-;-)\ar{r}\ar{d}{\sim}&\vphantom{C}\hocolim\limits_{\HJ_\bullet^{\OO\B,\reg}(\tau;\sigma_0,\ldots,\sigma_r)}CF^\bullet(-,-;-) \\
\OO_\tau&\ar{l}\OO_\tau^-&\vphantom{C}\hocolim\limits_{\HJ_\bullet^{\B,\reg}(\tau;\sigma_0,\ldots,\sigma_r)}CF^\bullet(-,-;-)\ar[equals]{r}\ar{u}[swap]{\sim}&\B_{\tau;\sigma_0,\ldots,\sigma_r}
\end{tikzcd}
\hspace{-1in}
\end{equation}
where the leftmost vertical map is the quotient map collapsing the $\hocolim$ (noting that $\HJ_\bullet^\OO$ is a single vertex).
On the level of cohomology, the three rightmost bimodules are $HW^\bullet_{\sigma_0}$ by Lemma \ref{BisWonH}, and we have by definition $H^\bullet\OO_\tau^-=HF^\bullet$ for $L>K$ and zero otherwise (similarly for $H^\bullet\OO_\tau$).
We also see from Lemma \ref{BisWonH} that the maps between these groups and their bimodule structure are, on the cohomology level, the natural ones.
Using this knowledge, we may now prove the main result of this subsection.

\begin{lemma}\label{OtoBlocalquasiiso}
For $\tau=\sigma_0$, the maps in \eqref{OtoB} all become quasi-isomorphisms after localization on the left at $C_\tau=C_{\sigma_0}$.
\end{lemma}

\begin{proof}
Applying $\E\mapsto\varinjlim_i\E(L^{(i)},K)$ to \eqref{OtoB} results (on cohomology) in all groups being $HW^\bullet_\tau=HW^\bullet_{\sigma_0}$ and all maps being the identity map; in particular all the maps are quasi-isomorphisms.
The result now follows from Corollary \ref{wrappingcalculatesanylocalM} (which says that the direct limit over $L^{(i)}$ calculates localization on the left at $C_\tau=C_{\sigma_0}$).
\end{proof}

Lemma \ref{OtoBlocalquasiiso} shows that $\W_{\sigma_0} = {}_{C_{\sigma_0}^{-1}}(\OO_{\sigma_0})$ and ${}_{C_{\sigma_0}^{-1}}(\B_{\sigma_0;\sigma_0,\ldots,\sigma_r})$ are quasi-isomorphic.
The latter is quasi-isomorphic to $\B_{\sigma_0;\sigma_0,\ldots,\sigma_r}$ by Lemma \ref{BisWonH} and Lemma \ref{localgood}.
Using the quasi-isomorphism \eqref{BmodforgetII}, we conclude that $\B_{\tau;\sigma_0,\ldots,\sigma_r}$ is quasi-isomorphic to $\W_{\sigma_0}$ as $\OO_\tau$-bimodules.
(By ``are quasi-isomorphic'', we mean ``are connected by a zig-zag of quasi-isomorphisms''; we do not wish to discuss the question of whether quasi-isomorphisms of $\ainf$-bimodules over $\ZZ$ are invertible under our cofibrancy assumptions.)

\subsection{Hochschild homology of \texorpdfstring{$\B$}{B}}

Following Abouzaid--Ganatra \cite{abouzaidganatra}, we would like to take $CC_\bullet(\OO,\B)$ in place of $CC_\bullet(\W)$ as the domain of the open-closed map.
It follows immediately from  Lemmas \ref{OtoBlocalquasiiso}, \ref{BisWonH}, \ref{HHquotientisomorphism} above that these two complexes are quasi-isomorphic.

\begin{corollary}[Abouzaid--Ganatra \cite{abouzaidganatra}]\label{corollaryCCOB}
There is a canonical zig-zag of quasi-isomorphisms between $CC_\bullet(\OO_{\sigma_0}, \B_{\sigma_0;\sigma_0,\ldots,\sigma_r})$ and $CC_\bullet(\W_{\sigma_0})$.\qed
\end{corollary}

We should now specify a particular chain model for $CC_\bullet(\OO,\B)$ which is functorial in $\sigma$.
Unfortunately, $CC_\bullet(\OO_\sigma,\B_{\sigma;\sigma})$ is not functorial in $\sigma$, rather for $\sigma\leq\sigma'$ there are only maps
\begin{equation}
\begin{tikzcd}[column sep = tiny]
CC_\bullet(\OO_\sigma,\B_{\sigma;\sigma,\sigma'})
\ar{d}{\sim}\ar{dr}&
CC_\bullet(\OO_\sigma,\B_{\sigma';\sigma'}|_{\OO_\sigma})
\ar{d}{\sim}\ar{dr}\\
CC_\bullet(\OO_\sigma,\B_{\sigma;\sigma})&CC_\bullet(\OO_\sigma,\B_{\sigma;\sigma'})&CC_\bullet(\OO_{\sigma'},\B_{\sigma';\sigma'})
\end{tikzcd}
\end{equation}
which act as desired on homology.
As in \eqref{scsimplefunctoriality}, there are two ``problems'' preventing the existence of a naive pushforward map, namely we must consider the correspondence $\HJ_\bullet^\B(\sigma;\sigma)\leftarrow\HJ_\bullet^\B(\sigma;\sigma,\sigma')\to\HJ_\bullet^\B(\sigma;\sigma')\leftarrow\HJ_\bullet^\B(\sigma';\sigma')$ due to the ``problem'' of extending Floer data from $X_\sigma$ to $X_{\sigma'}$ and the ``problem'' of the restricting to Floer data making the moduli spaces regular for the larger category $\OO_{\sigma'}$.
In contrast to the case of \eqref{scsimplefunctoriality}, these problems cannot be dealt with simultaneously by defining ``$\HJ_\bullet^\B(\sigma';\sigma,\sigma')$'': indeed, this simplicial set is empty outside of trivial cases, since no Floer data for $X_\sigma\subseteq X_{\sigma'}$ can ensure $\Phi_HL\pitchfork L$ for all $L\in\OO_{\sigma'}$ and simultaneously satisfy $H\equiv 0$ near $\partial X_\sigma^-$.
Thus to solve these two problems, we need two additional homotopy colimits (as opposed to a single additional homotopy colimit as in \eqref{scchains}).

We are thus led to take as the domain of the open-closed map the complex
\begin{equation}\label{hochschildcolimit}
\hocolim_{\sigma_0\leq\cdots\leq\sigma_r\leq\sigma}\hocolim_{\tau_0\leq\cdots\leq\tau_s\leq\sigma_0}CC_\bullet(\OO_{\tau_0},\B_{\tau_s,\sigma_0,\ldots,\sigma_r}|_{\OO_{\tau_0}}),
\end{equation}
which is strictly functorial in $\sigma\in\Sigma$.
The inner $\hocolim$ is taken over the subposet of $\Sigma$ consisting of elements $\leq\sigma_0$, using the structure maps
\begin{equation}
CC_\bullet(\OO_{\tau_0},\B_{\tau_s,\sigma_0,\ldots,\sigma_r}|_{\OO_{\tau_0}})\to CC_\bullet(\OO_{\tau_0'},\B_{\tau_s',\sigma_0,\ldots,\sigma_r}|_{\OO_{\tau_0'}})
\end{equation}
for $\tau_0\leq\tau_0'\leq\tau_s'\leq\tau_s$ (imagining $(\tau_0',\ldots,\tau_s')$ is obtained from $(\tau_0,\ldots,\tau_s)$ by forgetting some $\tau_i$); the notation $\hocolim$ is justified since \eqref{BmodforgetII} is a quasi-isomorphism.
The outer $\hocolim$ is over the subposet of $\Sigma$ consisting of elements $\leq\sigma$; the notation $\hocolim$ is justified since \eqref{BmodforgetI} is a quasi-isomorphism for $i>0$.

By construction, \eqref{hochschildcolimit} is strictly functorial in $\sigma\in\Sigma$.
The inclusion of the subcomplex $CC_\bullet(\OO_\sigma,\B_{\sigma;\sigma})$ (corresponding to $r=s=0$ and $\tau_0=\sigma_0=\sigma$) into \eqref{hochschildcolimit} is a quasi-isomorphism by two applications of Lemma \ref{finalobjectII}.

We will use \eqref{hochschildcolimit} as the domain of our open-closed map.
The main result of this subsection is:

\begin{proposition}
There is a canonical zig-zag of quasi-isomorphisms (of diagrams $\Sigma\to\Ch$) between $CC_\bullet(\W_\sigma)$ and \eqref{hochschildcolimit}.
\end{proposition}

\begin{proof}
We consider the following quasi-isomorphisms functorial in $\sigma$:
\begin{align}
&\hocolim_{\sigma_0\leq\cdots\leq\sigma_r\leq\sigma}\hocolim_{\tau_0\leq\cdots\leq\tau_s\leq\sigma_0}CC_\bullet(\OO_{\tau_0},\B_{\tau_s;\sigma_0,\ldots,\sigma_r}|_{\OO_{\tau_0}})\\
&\qquad\qquad=\nonumber\\
&\label{hoIB}\hocolim_{\sigma_0\leq\cdots\leq\sigma_r\leq\sigma}\hocolim_{\tau_0\leq\cdots\leq\tau_s\leq\sigma_0}CC_\bullet(\OO_{\tau_0},\hocolim_{\HJ_\bullet^{\B,\reg}(\tau_s;\sigma_0,\ldots,\sigma_r)}CF^\bullet(-,-;-))\\
&\qquad\qquad\downarrow\sim\nonumber\\
&\label{hoIIBlocal}\hocolim_{\sigma_0\leq\cdots\leq\sigma_r\leq\sigma}\hocolim_{\tau_0\leq\cdots\leq\tau_s\leq\sigma_0}CC_\bullet(\OO_{\tau_0},\Bigr._{C_{\tau_0}^{-1}}\Bigl(\hocolim_{\HJ_\bullet^{\B,\reg}(\tau_s;\sigma_0,\ldots,\sigma_r)}CF^\bullet(-,-;-)\Bigr))\\
&\qquad\qquad\downarrow\sim\nonumber\\
&\label{hoIIIOB}\hocolim_{\sigma_0\leq\cdots\leq\sigma_r\leq\sigma}\hocolim_{\tau_0\leq\cdots\leq\tau_s\leq\sigma_0}CC_\bullet(\OO_{\tau_0},\Bigr._{C_{\tau_0}^{-1}}\Bigl(\hocolim_{\HJ_\bullet^{\OO\B,\reg}(\tau_s;\sigma_0,\ldots,\sigma_r)}CF^\bullet(-,-;-)\Bigr))\\
&\qquad\qquad\uparrow\sim\nonumber\\
&\label{hoIVO}\hocolim_{\sigma_0\leq\cdots\leq\sigma_r\leq\sigma}\hocolim_{\tau_0\leq\cdots\leq\tau_s\leq\sigma_0}CC_\bullet(\OO_{\tau_0},\Bigr._{C_{\tau_0}^{-1}}\Bigl(\hocolim_{\HJ_\bullet^\OO(\tau_s;\sigma_0,\ldots,\sigma_r)}CF^\bullet(-,-;-)\Bigr))\\
&\qquad\qquad\downarrow\sim\nonumber\\
&\hocolim_{\sigma_0\leq\cdots\leq\sigma_r\leq\sigma}\hocolim_{\tau_0\leq\cdots\leq\tau_s\leq\sigma_0}CC_\bullet(\OO_{\tau_0},{}_{C_{\tau_0}^{-1}}(\OO_{\tau_0}^-))\\
&\qquad\qquad\downarrow\sim\nonumber\\
&\hocolim_{\sigma_0\leq\cdots\leq\sigma_r\leq\sigma}\hocolim_{\tau_0\leq\cdots\leq\tau_s\leq\sigma_0}CC_\bullet(\OO_{\tau_0},\W_{\tau_0}).
\end{align}
The map $\eqref{hoIB}\to\eqref{hoIIBlocal}$ is a quasi-isomorphism by Lemma \ref{BisWonH} and Lemma \ref{localgood}.
The map $\eqref{hoIIBlocal}\to\eqref{hoIIIOB}$ is a quasi-isomorphism by cofinality.
The map $\eqref{hoIVO}\to\eqref{hoIIIOB}$ is a quasi-isomorphism by Lemma \ref{finalobjectII} and since
\begin{equation}
\Bigr._{C_{\sigma_0}^{-1}}\Bigl(\hocolim_{\HJ_\bullet^\OO(\sigma_0;\sigma_0,\ldots,\sigma_r)}CF^\bullet(-,-;-)\Bigr)\to\Bigr._{C_{\sigma_0}^{-1}}\Bigl(\hocolim_{\HJ_\bullet^{\OO\B,\reg}(\sigma_0;\sigma_0,\ldots,\sigma_r)}CF^\bullet(-,-;-)\Bigr)
\end{equation}
is a quasi-isomorphism by Lemma \ref{OtoBlocalquasiiso}.
The remaining quasi-isomorphisms are clear.

Finally, note that there are quasi-isomorphisms functorial in $\sigma$:
\begin{multline}
\hocolim_{\sigma_0\leq\cdots\leq\sigma_r\leq\sigma}\hocolim_{\tau_0\leq\cdots\leq\tau_s\leq\sigma_0}CC_\bullet(\OO_{\tau_0},\W_{\tau_0})\xrightarrow\sim\hocolim_{\sigma_0\leq\cdots\leq\sigma_r\leq\sigma}CC_\bullet(\OO_{\sigma_0},\W_{\sigma_0})\xrightarrow\sim CC_\bullet(\OO_\sigma,\W_\sigma)\\
\xrightarrow\sim CC_\bullet(\W_\sigma,\W_\sigma).
\end{multline}
Each of the first two maps is given by collapsing the relevant $\hocolim$, i.e.\ they are given by the obvious pushforward on the vertices of the indexing simplicial set and by zero on all positive dimensional simplices (note that $CC_\bullet(\OO_\sigma,\W_\sigma)$ is strictly functorial in $\sigma$).
These are quasi-isomorphisms by Lemma \ref{finalobjectII}.
The last map is a quasi-isomorphism by Lemma \ref{HHquotientisomorphism}.
We have thus defined the desired quasi-isomorphism between \eqref{hochschildcolimit} and $CC_\bullet(\W_\sigma)$ as diagrams $\Sigma\to\Ch$.
\end{proof}

\subsection{Moduli spaces of domains for \texorpdfstring{$\OC$}{OC}}\label{ocdomainsec}

\begin{figure}[hbt]
\centering
\includegraphics{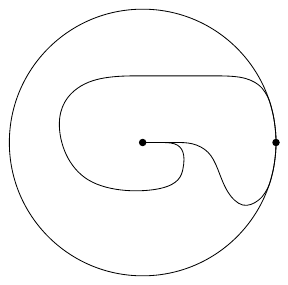}
\caption{Embedding $\RR\times[0,1]\to D^2$.}\label{embedmap}
\end{figure}

Let $D^2\subseteq\CC$ denote the unit disk.
We fix once and for all a map (illustrated in Figure \ref{embedmap})
\begin{equation}
\label{diskstrip}\RR\times[0,1]\to D^2
\end{equation}
whose restriction to $s\gg 0$ gives positive strip-like coordinates at $1\in D^2$, and whose restriction to $s\ll 0$ gives negative cylindrical coordinates at $0\in D^2$ (via the map $[0,1]\to S^1$ given by $t\mapsto 2\pi t$).
Note that these coordinates give a way of gluing strips $\RR\times[0,1]$ and cylinders $\RR\times S^1$ at $1\in D^2$ and $0\in D^2$, respectively, and identifying the result back with $D^2$, in such a way which is the ``identity map'' on points of the original $D^2$ away from $0$ and $1$.

We consider the compactified moduli space of decorations of $D^2$ with $p+1$ boundary marked points including $1\in\partial D^2$, an interior marked point $0\in D^2$, and points $a_1\geq\cdots\geq a_{n_0}\geq a_{n_0+1}=0\geq a_{n_0+1+1}\geq\cdots\geq a_{n_0+1+n_1}\in\RR$ (regarded as the image of $\RR\times\{\frac 12\}$ under the map \eqref{diskstrip}).
The marked points $a_i$ are allowed to collide with each other, though bubbles are formed at $0,1\in D^2$ if they approach $\pm\infty$ (a bubble at $0\in D^2$ is thus an object parameterized by $\Mbar^{SC}$, and a bubble at $1\in D^2$ is an object parameterized by $\Mbar^\B$).
We denote this moduli space and its universal family by $\Cbar^\OC_{p,n_0,n_1}\to\Mbar^\OC_{p,n_0,n_1}$.
By the observation at the end of the previous paragraph (and the extension property \eqref{ocWcoordsIII}), gluing at $0,1\in D^2$ (which should be regarded as punctures) via the coordinates \eqref{diskstrip} and via any strip-like coordinates at the remaining $p$ boundary punctures gives boundary collars for $\Mbar^\OC_{p,n_0,n_1}$.

Also fix once and for all a family of maps
\begin{equation}\label{diskcalibration}
\varphi_s:S^1\to S^1
\end{equation}
parameterized by $s\in\RR$, satisfying $\varphi_s'(t)\geq 0$, independent of $s$ for $\left|s\right|\gg 0$, such that $\varphi_\infty(2\pi t)=\varphi(t)$ for the function $\varphi$ fixed in \eqref{ocvarphi}, and such that $\varphi_{-\infty}(t)=t$.
In fact, let us assume that $\varphi_s=\varphi_\infty$ as long as $s$ is above the region over which \eqref{diskstrip} descends to a smooth embedding $\RR\times S^1\to D^2$.
The forgetful map $\Mbar^\OC_{p,n_0,n_1}\to\Mbar^{SC}_{n_0+1+n_1}$ (remembering just the points $a_i$) is covered by a map $\Cbar^\OC_{p,n_0,n_1}\to\Cbar^{SC}_{n_0+1+n_1}$ (defined over the image of \eqref{diskstrip}) given by $(s,t)\mapsto(s,\varphi_s(2\pi t))$.
This map will be used to define the Hamiltonian terms in the Floer equations we are about to consider.

\subsection{Floer data for \texorpdfstring{$\OC$}{OC}}

The Floer data for defining the open-closed map $\OC$ is organized into simplicial sets
\begin{equation}
\HJ_\bullet^\OC(\tau;\sigma_0,\ldots,\sigma_r;\sigma_0',\ldots,\sigma_{r'}')
\end{equation}
for every chain $\tau\leq\sigma_0\leq\cdots\leq\sigma_r\leq\sigma_0'\leq\cdots\leq\sigma_{r'}'\in\Sigma$, each fitting into a fiber diagram (aka pullback square):
\begin{equation}\label{hjocfiber}
\hspace{-.1in}
\begin{tikzcd}[column sep = tiny]
\HJ_\bullet^\B(\tau;\sigma_0,\ldots,\sigma_r;\sigma_0',\ldots,\sigma_{r'}')\sqcup\HJ_\bullet^{SC}(\sigma_0',\ldots,\sigma_{r'}')\ar{r}\ar{d}&\HJ_\bullet^\OC(\tau;\sigma_0,\ldots,\sigma_r;\sigma_0',\ldots,\sigma_{r'}')\ar{d}\\
\Delta^0\sqcup\Delta^0\ar{r}&\Delta^1
\end{tikzcd}
\hspace{-.1in}
\end{equation}
where $\HJ_\bullet^\B$ maps to the initial vertex of $\Delta^1$ and $\HJ_\bullet^{SC}$ maps to the final vertex of $\Delta^1$.
There is a forgetful map
\begin{equation}\label{hjbforgetforoc}
\HJ_\bullet^\B(\tau;\sigma_0,\ldots,\sigma_r;\sigma_0',\ldots,\sigma_{r'}')\to\HJ_\bullet^\B(\tau;\sigma_0,\ldots,\sigma_r)
\end{equation}
which will also be explained below.

To motivate this setup, in particular as compared with the simpler cases of $\HJ_\bullet^{SC}$ and $\HJ_\bullet^\B$, note that, whereas $\HJ_\bullet^{SC}$ and $\HJ_\bullet^\B$ parameterize Floer data needed to define certain \emph{Floer cohomology groups} (together with higher coherences), the simplicial sets $\HJ_\bullet^\OC$ are supposed to parameterize Floer data to define a \emph{map between Floer cohomology groups} (together with higher coherences).
In particular, this explains why $\HJ_\bullet^\OC$ has no more vertices other than those already in $\HJ_\bullet^{SC}\sqcup\HJ_\bullet^\B$ (this follows from the fiber diagram \eqref{hjocfiber}), as there are no new groups, just new maps.
The smallest simplices in $\HJ_\bullet^\OC$ which are not contained in $\HJ_\bullet^{SC}\sqcup\HJ_\bullet^\B$ are $1$-simplices mapping surjectively to $\Delta^1$.
These encode Floer data for defining a single open-closed map, and higher-dimensional simplices mapping surjectively to $\Delta^1$ encode Floer data for defining higher homotopies between open-closed maps and their compositions with the various continuation maps for $SC^\bullet$ and $\B$.
The final result will be a diagram $\HJ_\bullet^\OC\to\Ndg\Ch$, which over $\HJ_\bullet^{SC}\subseteq\HJ_\bullet^\OC$ coincides with the diagram of symplectic cochains constructed in \S\ref{secsympcochaindiagrams} and which over $\HJ_\bullet^\B\subseteq\HJ_\bullet^\OC$ is given by the Hochschild chains of the diagram of bimodules \eqref{Bdiagram} over $\HJ_\bullet^\B$.
The most interesting part of this diagram is, of course, the simplices of $\HJ_\bullet^\OC$ which are not contained in either $\HJ_\bullet^{SC}$ or $\HJ_\bullet^\B$ (i.e.\ those which map surjectively to $\Delta^1$), which are the ones encoding the open-closed map.
Given this diagram over $\HJ_\bullet^\OC$, we will eventually take an appropriate ``fiberwise homotopy colimit'' over $\Delta^1$ to obtain a single map (which is the open-closed map).

\begin{remark}
One might have expected that we would define $\HJ_\bullet^\OC$ so that
there is a diagram
$\HJ_\bullet^\OC\to\operatorname{Fun}(\Delta^1,\Ndg\Ch)$, which by
definition is just a diagram $\HJ_\bullet^\OC\times\Delta^1\to\Ndg\Ch$; we
see no problem with this approach, though we found the setup above to be much more convenient.
\end{remark}

Let us now define $\HJ_\bullet^\OC$ (by the fiber diagram \eqref{hjocfiber}, this also defines the variant of $\HJ_\bullet^\B$).
An $n$-simplex of $\HJ_\bullet^\OC$ consists first of a map $r:\Delta^n\to\Delta^1$, corresponding to the map $\HJ_\bullet^\OC\to\Delta^1$.
The map $r$ induces an isomorphism $\Delta^n=\Delta^{n_0}\ast\Delta^{n_1}$ (the simplicial join), where $\Delta^{n_i}=r^{-1}(\{i\})$ (define $\Delta^{-1}:=\varnothing$) and $n=n_0+1+n_1$.
The remaining data specifying an $n$-simplex of $\HJ_\bullet^\OC(\tau;\sigma_0,\ldots,\sigma_r;\sigma_0',\ldots,\sigma_{r'}')$ is the following:
\begin{itemize}
\item An $n_0$-simplex of $\HJ_\bullet^\B(\tau;\sigma_0,\ldots,\sigma_r)$, namely choices of \eqref{BchoicescoordsI}--\eqref{BchoicesH}, except that the Hamiltonians $H$ are now specified on $X_{\sigma_{r'}'}$ and (in addition to the requirements of $\HJ_\bullet^\B$) must be adapted to $\partial X_{\sigma_i'}$ in the sense of Definition \ref{adapted} and must be admissible with respect to a specified chain \eqref{scchainforoc} with $\mu_1^{(0)}\geq\sigma_r$ at each vertex in the sense of Definition \ref{admissible}.
\item An $n_1$-simplex of $\HJ_\bullet^{SC}(\sigma_0',\ldots,\sigma_{r'}')$, namely choices of \eqref{Hsc}--\eqref{Jsc} and chains \eqref{scchainforoc} (with the usual deletion relation between such chains for pairs of vertices $0\leq v<v+1\leq n$).
\item Choices of strip-like coordinates, almost complex structures, and Hamiltonians
\begin{align}
\label{OCchoicescoordsI}\xi_{L_0,\ldots,L_p;j}^{v_0\cdots v_{m_0},v_0'\cdots v_{m_1}'}:\RR_{\geq 0}\times[0,1]\times\Mbar^\OC_{p,m_0,m_1}&\to\Cbar^\OC_{p,m_0,m_1}\quad j=1,\ldots,p\\
\label{OCchoicesacs}J_{L_0,\ldots,L_p}^{v_0\cdots v_{m_0},v_0'\cdots v_{m_1}'}:\Cbar^\OC_{p,m_0,m_1}&\to\J(X_{\sigma_{r'}'})\\
\label{OCchoicesH}H^{v_0\cdots v_{m_0},v_0'\cdots v_{m_1}'}:\Cbar^{SC}_{m_0+1+m_1}&\to\H(X_{\sigma_{r'}'})
\end{align}
for $L_0>\cdots>L_p\in\OO_{\sigma_r}$, and $0\leq v_0<\cdots<v_{m_0}\leq n_0<n_0+1\leq v_0'<\cdots<v_{m_1}'\leq n_0+1+n_1$.
\end{itemize}
satisfying the following properties:
\begin{itemize}
\item The strip-like coordinates $\xi$, almost complex structures $J$, and Hamiltonians $H$ must be \emph{compatible with gluing and forgetting vertices} in the usual sense.
\item The almost complex structures $J$ and Hamiltonians $H$ must be \emph{adapted to $\partial X_{\sigma_i'}$} in the sense of Definition \ref{adapted}.
\item The Hamiltonians $H$ must be \emph{dissipative} with specified dissipation data in the sense of Definition \ref{dissipative} (modified over $\Delta^{n_0}$ as in $\HJ_\bullet^\B$).
\end{itemize}
This completes the definition of $\HJ_\bullet^\OC(\tau;\sigma_0,\ldots,\sigma_r;\sigma_0',\ldots,\sigma_{r'}')$.

Note the forgetful maps on $\HJ_\bullet^\OC$, namely decreasing $\tau$ and forgetting any $\sigma_i$ or $\sigma_i'$, as well as the forgetful map \eqref{hjbforgetforoc}.

The proof of Proposition \ref{hjcontractible} applies to $\HJ_\bullet^\OC$ and $\HJ_\bullet^\B$ (as defined above) to show that they are filtered $\infty$-categories.
It is easy to see that the forgetful map \eqref{hjbforgetforoc} is cofinal, and it is obvious from the fiber diagram \eqref{hjocfiber} that $\HJ_\bullet^{SC}\to\HJ_\bullet^\OC$ is cofinal.

\subsection{Holomorphic curves for \texorpdfstring{$\OC$}{OC}}

\begin{figure}[hbt]
\centering
\includegraphics{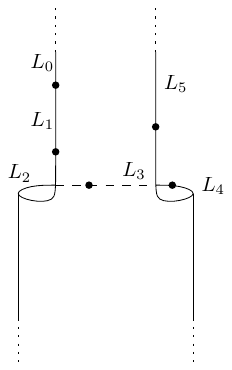}
\caption{The Riemann surface for the open-closed map.}\label{openclosed}
\end{figure}

Given an $n$-simplex of $\HJ^\OC_\bullet(\tau;\sigma_0,\ldots,\sigma_r;\sigma_0',\ldots,\sigma_{r'}')$, we consider the moduli space of maps $u:D^2\setminus\{0,1\}\to X_{\sigma_{r'}'}$ along with marked points $a_i$ as in \S\ref{ocdomainsec}, where $u$ has boundary conditions $L_i\in\OO_\tau$ as in Figure \ref{openclosed} and satisfies
\begin{equation}\label{openclosedeqn}
(du-X_{H(s,\varphi_s(2\pi t))}\otimes d(\varphi_s(2\pi t)))^{0,1}_J=0,
\end{equation}
where $(s,\varphi_s(2\pi t))$ denotes the partially defined map of universal curves $\Cbar^\OC_{p,n_0,n_1}\to\Cbar^{SC}_{n_0+1+n_1}$ from the end of \S\ref{ocdomainsec} (the Hamiltonian term is declared to be zero outside the image of \eqref{diskstrip}).
The argument of Lemma \ref{scconfinecurves} shows that all such trajectories $u$ are contained in $X_{\sigma_0'}$.
We denote by
\begin{equation}\label{OCmodulispaces}
\Mbar^\OC_{p,n_0,n_1}(H,J,\xi;x,y_1,\ldots,y_p)
\end{equation}
the associated compactified moduli space (i.e.\ including all stable broken trajectories), where $(H,J,\xi)$ stands for an $n$-simplex of $\HJ_\bullet^\OC(\tau;\sigma_0,\ldots,\sigma_r;\sigma_0',\ldots,\sigma_{r'}')$, and $x\in\Phi_{H(0)}L_p\cap L_0$, $y_i\in L_{i-1}\cap L_i$.

\begin{proposition}\label{occompactness}
The moduli spaces \eqref{OCmodulispaces} are compact.
\end{proposition}

\begin{proof}
    As in Proposition \ref{bcompactness}, the discussion above shows
    trajectories avoid $\partial X_{\sigma_0'}$; in the non-compact directions
    the proofs of Propositions \ref{wcompactness} and \ref{compactness} adapt
    without any trouble.
\end{proof}

Let $\HJ_\bullet^{\OC,\reg}\subseteq\HJ_\bullet^\OC$ denote the subset where all moduli spaces \eqref{OCmodulispaces}, \eqref{Bmodulispaces}, \eqref{mbarscmoduli} are cut out transversely.
Note that there is again (by definition) a fiber diagram
\begin{equation}
\begin{tikzcd}
\HJ_\bullet^{\B,\reg}(\tau;\sigma_0,\ldots,\sigma_r;\sigma_0',\ldots,\sigma_{r'}')\sqcup\HJ_\bullet^{SC,\reg}(\sigma_0',\ldots,\sigma_{r'}')\ar{r}\ar{d}&\HJ_\bullet^{\OC,\reg}(\tau;\sigma_0,\ldots,\sigma_r;\sigma_0',\ldots,\sigma_{r'}')\ar{d}\\
\Delta^0\sqcup\Delta^0\ar{r}&\Delta^1
\end{tikzcd}
\end{equation}
and that the usual perturbation argument shows that $\HJ_\bullet^{\OC,\reg}$ is also filtered $\infty$-category, cofinal inside $\HJ_\bullet^\OC$ (and the same for $\HJ_\bullet^{\B,\reg}\subseteq\HJ_\bullet^\B$).

\subsection{Diagram \texorpdfstring{$\OC$}{OC}}

We now argue that counting the moduli spaces \eqref{OCmodulispaces} defines a diagram
\begin{equation}\label{OCdiagram}
\HJ_\bullet^{\OC,\reg}(\tau;\sigma_0,\ldots,\sigma_r;\sigma_0',\ldots,\sigma_{r'}')\to\Ndg\Ch
\end{equation}
extending the diagrams $CC_\bullet(\OO_\tau,CF^\bullet(-,-;-))$ over $\HJ_\bullet^{\B,\reg}(\tau;\sigma_0,\ldots,\sigma_r;\sigma_0',\ldots,\sigma_{r'}')$ (pulled back from $\HJ_\bullet^{\B,\reg}(\tau;\sigma_0,\ldots,\sigma_r)$) and $CF^{\bullet+n}(X_{\sigma_0'};-)$ over $\HJ_\bullet^{SC,\reg}(\sigma_0',\ldots,\sigma_{r'}')$ defined earlier (more precisely, the former diagram is obtained from the diagram of bimodules \eqref{Bdiagram} by composing with the dg-functor $CC_\bullet(\OO_\tau,-):[\OO_\tau,\OO_\tau]\to\Ch$).
This diagram \eqref{OCdiagram} encodes the various operations 
\begin{equation}\label{ocoperations}
CF^\bullet(L_p,L_0;H(0))\otimes\OO(L_0,L_1)\otimes\cdots\otimes\OO(L_{p-1},L_p)\to CF^\bullet(X_{\sigma_0'};H(n))[{\textstyle\frac 12}\dim X-p-n_0-n_1]
\end{equation}
defined by counting the moduli spaces \eqref{OCmodulispaces} for $L_0>\cdots>L_p\in\OO_\tau$ and an $n$-simplex $(H,J,\xi)$ of $\HJ_\bullet^{\OC,\reg}$ with $n_0,n_1\geq 0$, i.e.\ which projects surjectively onto $\Delta^1$ (set to vanish for $L_i=L_{i+1}$) and the identities they satisfy with \eqref{muoperation}, \eqref{SChomotopyoperation}, \eqref{Bdiagramoperations}.
The map \eqref{ocoperations} involves trivializing the Fredholm orientation line of the $\bar\partial$-operator on $D^2$ with $\Spin$ Lagrangian boundary conditions.

To a $1$-simplex of $\HJ_\bullet^{\OC,\reg}$ mapping surjectively to $\Delta^1$ (namely with $n_0=n_1=0$), the operations \eqref{ocoperations} (with signs as in \cite[eq (5.24)]{abouzaidcriterion}) define a chain map
\begin{equation}
CC_\bullet(\OO_\tau;CF^\bullet(-,-;H(0)))\to CF^{\bullet+n}(X_{\sigma_0'};H(n)),
\end{equation}
and more generally the same argument associates to any $n$-simplex mapping surjectively to $\Delta^1$ a chain map
\begin{equation}
CC_\bullet(\OO_\tau;CF^\bullet(-,-;H(0)))\otimes C_{-\bullet}(\F(\Delta^n))\to CF^{\bullet+n}(X_{\sigma_0'};H(n)).
\end{equation}
These maps are compatible in the natural way under face and degeneracy maps and thus define the desired diagram.

\subsection{Map \texorpdfstring{$\OC$}{OC}}\label{mapocsec}

The diagram \eqref{OCdiagram} allows us to define the open-closed map $\eqref{hochschildcolimit}\to\eqref{scmodelforoc}$ as follows.
Namely, we consider the following maps:
\begin{align}
\eqref{hochschildcolimit}=&\hocolim_{\sigma_0\leq\cdots\leq\sigma_r\leq\sigma}\hocolim_{\tau_0\leq\cdots\leq\tau_s\leq\sigma_0}CC_\bullet(\OO_{\tau_0},\B_{\tau_s;\sigma_0,\ldots,\sigma_r})\\
&\qquad\uparrow\sim\nonumber\\
\hocolim_{\sigma_0'\leq\cdots\leq\sigma_{r'}'\leq\sigma}&\hocolim_{\sigma_0\leq\cdots\leq\sigma_r\leq\sigma_0'}\hocolim_{\tau_0\leq\cdots\leq\tau_s\leq\sigma_0}CC_\bullet(\OO_{\tau_0},\B_{\tau_s;\sigma_0,\ldots,\sigma_r})\\
&\qquad=\nonumber\\
\hocolim_{\sigma_0'\leq\cdots\leq\sigma_{r'}'\leq\sigma}&\hocolim_{\sigma_0\leq\cdots\leq\sigma_r\leq\sigma_0'}\hocolim_{\tau_0\leq\cdots\leq\tau_s\leq\sigma_0}\hocolim_{\HJ_\bullet^{\B,\reg}(\tau_s;\sigma_0,\ldots,\sigma_r)}CC_\bullet(\OO_{\tau_0},CF^\bullet(-,-;-))\\
&\qquad\uparrow\sim\nonumber\\
\hocolim_{\sigma_0'\leq\cdots\leq\sigma_{r'}'\leq\sigma}&\hocolim_{\sigma_0\leq\cdots\leq\sigma_r\leq\sigma_0'}\hocolim_{\tau_0\leq\cdots\leq\tau_s\leq\sigma_0}\hocolim_{\HJ_\bullet^{\B,\reg}(\tau_s;\sigma_0,\ldots,\sigma_r;\sigma_0',\ldots,\sigma_{r'}')}CC_\bullet(\OO_{\tau_0},CF^\bullet(-,-;-))\\
&\qquad\downarrow\nonumber\\
\label{ochocolim}\hspace{-0.5in}\hocolim_{\sigma_0'\leq\cdots\leq\sigma_{r'}'\leq\sigma}&\hocolim_{\sigma_0\leq\cdots\leq\sigma_r\leq\sigma_0'}\hocolim_{\tau_0\leq\cdots\leq\tau_s\leq\sigma_0}\hocolim_{\HJ_\bullet^{\OC,\reg}(\tau_s;\sigma_0,\ldots,\sigma_r;\sigma_0',\ldots,\sigma_{r'}')}\begin{cases}CC_\bullet(\OO_{\tau_0},CF^\bullet(-,-;-))\\CF^{\bullet+n}(X_{\sigma_0'};-)\end{cases}\hspace{-0.5in}\\
&\qquad\uparrow\sim\nonumber\\
\hocolim_{\sigma_0'\leq\cdots\leq\sigma_{r'}'\leq\sigma}&\hocolim_{\sigma_0\leq\cdots\leq\sigma_r\leq\sigma_0'}\hocolim_{\tau_0\leq\cdots\leq\tau_s\leq\sigma_0}\hocolim_{\HJ_\bullet^{SC,\reg}(\sigma_0',\ldots,\sigma_{r'}')}CF^{\bullet+n}(X_{\sigma_0'};-)\\
&\qquad\downarrow\sim\nonumber\\
\hocolim_{\sigma_0'\leq\cdots\leq\sigma_{r'}'\leq\sigma}&\hocolim_{\HJ_\bullet^{SC,\reg}(\sigma_0',\ldots,\sigma_{r'}')}CF^{\bullet+n}(X_{\sigma_0'};-)=\eqref{scmodelforoc}
\end{align}
This is the desired open-closed map (to explain the funny notation in \eqref{ochocolim}, recall that vertices of $\HJ_\bullet^\OC$ are just vertices of $\HJ_\bullet^{SC}$ or $\HJ_\bullet^\B$).
Note that the second arrow from the bottom is a quasi-isomorphism since the inclusion of $\HJ_\bullet^{SC,\reg}$ into $\HJ_\bullet^{\OC,\reg}$ is cofinal.

Note that $HH_\bullet(\W(X))$ as well as $HH_\bullet(\F)$ for any collection $\F$ of isotopy classes of Lagrangians are both well-defined (i.e.\ independent of the choices made in their construction), as is the open-closed map on homology $HH_\bullet(\W(X))\to SH^\bullet(X,\partial X)$.
This follows from Lemma \ref{hochschildobviousiso}, the inductive nature of the choices involved in the construction of $\W$ and $\OC$, and the compatibility with pushforward under inclusions.

\begin{remark}
Most of the algebraic complication of the above definition is due to the fact that we insist on constructing chain level diagrams, which require many ``higher coherences'' to define completely.
On the other hand, the result of applying the open-closed map as defined above to any particular Hochschild cycle may be calculated as usual by considering only the (small number of) relevant moduli spaces.
\end{remark}

\subsection{Proof of Theorem \ref{localtoglobalnondegenerate}}\label{finalproof}

\begin{proof}[Proof of Theorem \ref{localtoglobalnondegenerate}]
Since every homology hypercover has a finite homology subhypercover, we may assume that $\Sigma$ is finite.
We consider the diagram of Liouville sectors over $\Sigma^\vartriangleright$ obtained from the given diagram over $\Sigma$ by adjoining $X$ as a final object.
We apply the construction of \S\ref{secwrapped} to this diagram of Liouville sectors and the given collections of Lagrangians to obtain categories $\OO_\sigma$ and $\F_\sigma:=\OO_\sigma[C_\sigma^{-1}]$.
We apply the construction of this section to obtain a diagram of open-closed maps over $\Sigma^\vartriangleright$.
Taking homotopy colimits over $\Sigma$, we obtain the left half of the key diagram \eqref{keydiagram}.
Since the local open-closed maps are isomorphisms by hypothesis, the top left horizontal arrow in \eqref{keydiagram} is a quasi-isomorphism.
Hence, we conclude that the image of the map $HH_{\bullet-n}(\bigcup_{\sigma\in\Sigma}\F_\sigma)\to SH^\bullet(X)$ contains the image of the map $\hocolim_{\sigma\in\Sigma}SC^\bullet(X_\sigma,\partial X_\sigma)\to SC^\bullet(X)$.
This map in turn hits the unit in $SH^\bullet(X)$ by Corollary \ref{schomologyhypercover} since $\{X_\sigma\}_{\sigma\in\Sigma}$ is a homology hypercover of $X$.
\end{proof}

\subsection{Compatibility with Morse theory}

The following compatibility result is well-known for Liouville manifolds.

\begin{proposition}\label{ocmorse}
For any Liouville sector $X$ and Lagrangian $i:L\to X$, the following diagram commutes:
\begin{equation}\label{ocmorseeq}
\begin{tikzcd}[column sep = large]
H^\bullet(L)\ar{r}{\text{Rmk \ref{hfmorse}}}\ar{d}{i_!}&HW^\bullet(L,L)\ar{d}{\OC}\\
H^{\bullet+n}(X,\partial X)\ar{r}{\text{Prop \ref{shpss}}}&SH^{\bullet+n}(X,\partial X).
\end{tikzcd}
\end{equation}
\end{proposition}

\begin{proof}
The two compositions in \eqref{ocmorseeq} are defined by counting holomorphic maps as illustrated in Figure \ref{pssocfigure}.
This illustration makes apparent the obvious neck stretching / gluing argument which should show that \eqref{ocmorseeq} commutes.
To turn this into a proof, we just need specify the relevant Hamiltonian terms and show that compactness is maintained throughout the deformation.

\begin{figure}[hbt]
\centering
\includegraphics{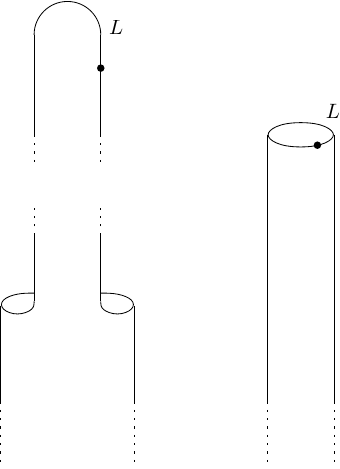}
\caption{The two compositions in \eqref{ocmorseeq}, namely $H^\bullet(L)\to HW^\bullet(L,L)\to SH^{\bullet+n}(X,\partial X)$ (left) and $H^\bullet(L)\to H^{\bullet+n}(X,\partial X)\to SH^{\bullet+n}(X,\partial X)$ (right).}\label{pssocfigure}
\end{figure}

We start by describing the moduli spaces defining the composition $H^\bullet(L)\to HW^\bullet(L,L)\to SH^{\bullet+n}(X,\partial X)$ (Figure \ref{pssocfigure} left), which we then slowly transform into moduli spaces which define the composition $H^\bullet(L)\to H^{\bullet+n}(X,\partial X)\to SH^{\bullet+n}(X,\partial X)$ (Figure \ref{pssocfigure} right).

To begin with, let us recall the definition of the map $H^\bullet(L)\to HF^\bullet(L,L)$ from Remark \ref{hfmorse}.
We choose an almost complex structure which is of contact type near infinity over $\Nbd^ZL$.
We then count holomorphic disks with one negative end (i.e.\ output) and moving Lagrangian boundary conditions following a sufficiently small positive at infinity Lagrangian isotopy $L\leadsto L^+$ (in the clockwise direction).
Energy/action considerations along with monotonicity imply that such holomorphic disks stay within the small cylindrical neighborhood of $L$ over which the almost complex structure is of contact type at infinity, and hence by the maximum principle Lemma \ref{maxprinciple} such disks are bounded \emph{a priori} away from infinity.
The count of such disks defines the unit $\1_L\in HF^\bullet(L,L)$, and adding a point constraint on $L$ defines the map $H^\bullet(L)\to HF^\bullet(L,L)$ (namely, this construction defines a map from an appropriate model of ``locally finite chains on $L$'', which calculates $H^\bullet(L)$ by Poincar\'e duality).

For the present purpose, we will need another description of the map $H^\bullet(L)\to HF^\bullet(L,L)$ using dissipative families of Hamiltonians, as we now describe.
Fix a linear Hamiltonian $H_1\in\H(X)$ vanishing near $\partial X$, positive over $L$ near infinity, such that the projection of $X_{H_1}$ to $TX/TL$ over $L$ is transverse to zero (equivalently, the restriction of $H_1$ to $L$ is Morse).
Now using the procedure of Proposition \ref{hjcontractible}, let $H:\RR_{\leq 0}\to\H(X)$ be a dissipative family interpolating between $H(s)=0$ near $s=0$ and $H(s)=H_1$ for $s\ll 0$, such that $\inf-\partial_sH>-\infty$ (wrapping bounded below).
Denoting by $\Phi_{H(s)}$ the integral of $X_{H(s)}$ over $S^1$, the isotopy of Lagrangians $\Phi_{H(s)}L$ for $s\in\RR_{\leq 0}$ (from $L$ to $\Phi_{H_1}L$) is generated by a family of Hamiltonians which are bounded below.
Now for any $a\in[0,1]$, any sufficiently small $\delta>0$, and any choice of negative strip-like coordinates $\RR_{\leq 0}\times[0,1]\to D^2\setminus\{1\}$ at the puncture, consider the moduli space of maps $u:D^2\setminus\{1\}\to X$ satisfying
\begin{align}
(du-X_{(1-a)\delta H(s)}\otimes d\varphi(t))^{0,1}_{J_a}&=0\\
u(s,1)&\in L\\
u(s,0)&\in\Phi_{a\delta H(s)}L
\end{align}
(away from the strip-like coordinates, there is no Hamiltonian term and the boundary condition is $L$), where $J_a:D^2\setminus\{1\}\to\J(X)$ for $a\in[0,1]$ satisfies $J_a(s,t)=(\Phi_{(1-a)\delta H_1}^{1-\varphi(t)})^\ast J(t)$ for $s\ll 0$, and $\varphi:[0,1]\to S^1$ is as fixed in \eqref{ocvarphi}.
For $a=0$, the boundary conditions are constant and the Hamiltonian term is supported away from the boundary, so the arguments from Propositions \ref{wcompactness} and \ref{compactness} apply to show these moduli spaces are compact.
In fact, this compactness argument applies for general $a\in[0,1]$; to see this, we just need to observe that dissipativity of $H$ means that the monotonicity part of the argument applies near the moving Lagrangian boundary condition (geometric energy is bounded by \eqref{geoenergymoving} and \eqref{poswrappinggood} since $\inf-\partial_sH>-\infty$).
Indeed, dissipativity of $H$ implies that, in a neighborhood of each point on the boundary of the domain curve, there is a sequence of shells whose widths diverge as in \eqref{dissipationshellsdistancesum} and inside which the Lagrangian boundary conditions are cylindrical and stationary, and hence the monotonicity arguments from Propositions \ref{wcompactness} and \ref{compactness} go through.
We thus obtain a map
\begin{equation}\label{Lpsssecond}
H^\bullet(L)\to HF^\bullet(\Phi_{\delta H_1}L,L)=HF^\bullet(L,L;\delta H_1)
\end{equation}
for every $a\in[0,1]$.
Notice that for $s\ll 0$, the $a$-dependence is just a change of gauge, so the moduli spaces are well-behaved as a family over $a\in[0,1]$.
It follows that this map \eqref{Lpsssecond} is independent of $a\in[0,1]$.

We now show that the two maps $H^\bullet(L)\to HF^\bullet(L,L)$ defined in the two paragraphs directly above coincide.
Consider the second construction at $a=1$, namely the holomorphic curve equation $(du)^{0,1}_J=0$ with moving Lagrangian boundary conditions $\Phi_{\delta H(s)}L$.
Now consider replacing $H$ with an alternative interpolation $\tilde H:\RR_{\leq 0}\to\H(X)$ between $\tilde H(s)=0$ near $s=0$ and $\tilde H(s)=H_1$ for $s\ll 0$ with wrapping bounded below $\inf-\partial_s\tilde H>-\infty$ (in fact, with the same lower bound as before) which coincides with $H$ over a large compact set and which near infinity equals $\psi(s)H_1$ for some cutoff function $\psi:\RR_{\leq 0}\to[0,1]$ with $\psi'(s)\leq 0$.
If we take $\tilde H=H$ over a sufficiently large compact set, then the proof of compactness applies equally well to $\tilde H$ as it did to $H$, and shows that solutions $u$ are confined to a compact set over which $H=\tilde H$ (both $H$ and $\tilde H$ enjoy the same upper bound on geometric energy coming from boundedness below of wrapping in view of \eqref{geoenergymoving}).
Hence, when we replace $H$ with $\tilde H$, the moduli spaces of solutions remain the same, and hence define the same map $H^\bullet(L)\to HF^\bullet(L,L)$.
On the other hand, $\Phi_{\delta\tilde H(s)}L$ is a cylindrical isotopy, non-negative at infinity, and thus defines the map from the first construction.
We conclude that the two maps $H^\bullet(L)\to HF^\bullet(L,L)$ defined above are in fact the same.

In view of the above discussion, we may regard the map $H^\bullet(L)\to HF^\bullet(L,L)$ as counting maps $u:D^2\setminus\{1\}\to X$ with boundary on $L$ (and a boundary point constraint) satisfying
\begin{equation}\label{lagpssfinalgood}
(du-X_{\delta H(s)}\otimes d\varphi(t))^{0,1}_J=0
\end{equation}
with respect to strip-like coordinates $\RR_{\leq 0}\times[0,1]\to D^2\setminus\{1\}$ and a dissipative family $H(s):\RR_{\leq 0}\to\H(X)$ from $H(0)=0$ to $H(-\infty)=H_1$ (linear and non-negative at infinity, vanishing near $\partial X$, and whose restriction to $L$ is positive at infinity and Morse) and any family $J:D^2\setminus\{1\}\to\J(X)$ which is $s$-invariant for $s\ll 0$ and achieves transversality.

\begin{figure}[hbt]
\centering
\includegraphics{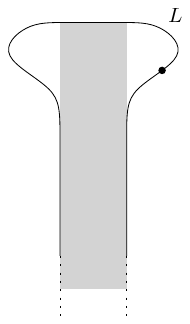}
\caption{Map $\RR_{\leq 0}\times[0,1]\to D^2\setminus\{1\}$.}\label{halfstripd}
\end{figure}

To facilitate the interaction of the moduli spaces of solutions to \eqref{lagpssfinalgood} with the open-closed map and symplectic cohomology, we modify $H$ over a small neighborhood of $\partial X$ so that it equals $\Re\pi$ (independent of $s$) over $\pi^{-1}(\CC_{\left|\Re\right|\leq\varepsilon})$.
Note that for the equation \eqref{lagpssfinalgood} to be well-behaved after this modification, we should, before modifying $H$, homotope the strip-like coordinates so that $\{0\}\times[0,1]\subseteq\partial D^2$ as in Figure \ref{halfstripd}.
For such choice of strip-like coordinates, the spaces of solutions of \eqref{lagpssfinalgood} with respect to the original and modified $H$ are the same, since the projection $\pi^-$ blocks disks from reaching the region where $H$ is modified (assume $H$ vanishes near $\partial X^-$ and $J$ makes $\pi^-$ holomorphic).

\begin{figure}[hbt]
\centering
\includegraphics{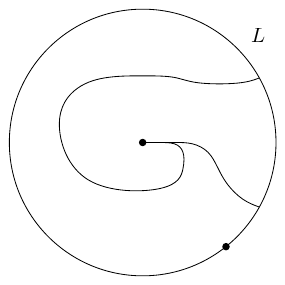}
\caption{Riemann surface defining composition $H^\bullet(L)\to HW^\bullet(L,L)\to SH^{\bullet+n}(X,\partial X)$.}\label{pssocflatham}
\end{figure}

Finally, let us compose the map $H^\bullet(L)\to HF^\bullet(L,L)$ with the open-closed map.
In other words, we count pairs consisting of a solution to \eqref{lagpssfinalgood} (with domain illustrated in Figure \ref{halfstripd}) and a solution to \eqref{openclosedeqn} (with domain illustrated in Figures \ref{embedmap} and \ref{openclosed}) asymptotic (at the negative/positive ends, respectively) to the same element of $\Phi_{\delta H_1}L\cap L$.
The resulting map $H^\bullet(L)\to SH^\bullet(X,\partial X)$ is unaffected by gluing together the positive/negative ends of these domains to form a finite length strip.
The result is that we count maps $u:D^2\setminus\{0\}\to X$ with boundary on $L$ satisfying
\begin{equation}\label{secondfinalequation}
(du-X_{\delta H(s)}\otimes d(\varphi_{s+s_0}(2\pi t)))^{0,1}_J=0.
\end{equation}
Here we use coordinates $\RR_{\leq 0}\times[0,1]\to D^2\setminus\{0\}$ which for $s\ll 0$ is strip-like coordinates at $0\in D^2$ (via the map $[0,1]\to S^1$ given by $t\mapsto 2\pi t$) and with $\{0\}\times[0,1]\subseteq\partial D^2$ (see Figure \ref{pssocflatham}), the map $\varphi_s:S^1\to S^1$ is from \eqref{diskcalibration}, $s_0\in\RR$ is a large positive number, $J:D^2\setminus\{0\}\to\J(X)$ is $s$-invariant for $s\ll 0$, and $H:\RR_{\leq 0}\to\H(X)$ (different from the $H$ used earlier) is dissipative and satisfies $H=\Re\pi$ over $\pi^{-1}(\CC_{\left|\Re\right|\leq\varepsilon})$, with $H=0$ near $s=0$ (except near $\partial X$) and $H(-\infty)$ admissible in the sense of Definition \ref{admissible}.
Compactness follows, using crucially the fact that $H(0)$ vanishes in a neighborhood of $L$.

We further deform the equation \eqref{secondfinalequation} as follows.
We first deform the coordinates $\RR_{\leq 0}\times[0,1]\to D^2\setminus\{0\}$ into the standard biholomorphism $\RR_{\leq 0}\times[0,1]/(s,0)\sim(s,1)\to D^2\setminus\{0\}$ given by $e^{s+2\pi it}$.
Next, we send $s_0$ to $-\infty$ so that $\varphi_{s+s_0}$ simply becomes the identity map.
We thus conclude that the map $H^\bullet(L)\to SH^\bullet(X,\partial X)$ is given by counting maps $u:D^2\setminus\{0\}\to X$ with boundary on $L$ satisfying
\begin{equation}\label{finalequation}
(du-X_{\delta H(s)}\otimes dt)^{0,1}_J=0
\end{equation}
with respect to the standard coordinates $z=e^{s+it}\in D^2$, where $H:\RR_{\leq 0}\to\H(X)$ is dissipative, interpolating from $H(0)=0$ (modified near $\partial X$ to equal $\Re\pi$ over $\pi^{-1}(\CC_{\left|\Re\right|\leq\varepsilon})$) to $H(-\infty)$ which is admissible in the sense of Definition \ref{admissible}.

We may choose $H(-\infty)$ to be Morse, and by Proposition \ref{morsefloer}, for $\delta>0$ sufficiently small, we may take $J$ to be $t$-invariant for $s\ll 0$.
In fact, the argument of Proposition \ref{morsefloer} shows that for $\delta>0$ sufficiently small, we may take $J$ to be $t$-invariant on all of $D^2\setminus\{0\}$ and all solutions $u$ are $t$-invariant.
We conclude that the map $H^\bullet(L)\to SH^\bullet(X,\partial X)$ factors as $H^\bullet(L)\to H^{\bullet+n}(X,\partial X)\to SH^{\bullet+n}(X,\partial X)$, where the first map counts Morse half-trajectories $\ell:\RR_{\leq 0}\to X$ with $\ell(0)\in L$ constrained to lie on a given locally finite chain on $L$.
This is the standard Morse model for the pushforward $i_!:H^\bullet(L)\to H^{\bullet+n}(X,\partial X)$, thus concluding the proof.
\end{proof}

\bibliographystyle{amsplain}
\bibliography{sectorsoc}
\addcontentsline{toc}{section}{References}

\end{document}